\documentclass[fmp,times]{cuparticle}

\usepackage{latexsym,amssymb, amsmath, amscd}

\newdefinition{definition}{Definition}
\newproof{proof}{Proof}

\volume{}
\doi{}

\newtheorem{thm}[subsubsection]{\bf Theorem}
\newtheorem{mainthm}[subsubsection]{\bf Main Theorem}
\newtheorem{defi}[subsubsection]{\bf Definition}
\newtheorem{coro}[subsubsection]{\bf Corollary}
\newtheorem{cor}[subsubsection]{\bf Corollary}
\newtheorem{lem}[subsubsection]{\bf Lemma}
\newtheorem{hyp}[subsubsection]{\bf Hypothesis}
\newtheorem{hype}[subsubsection]{\bf Hypotheses}
\newtheorem{ass}[subsubsection]{\bf Assumption}
\newtheorem{fact}[subsubsection]{\bf Fact}
\newtheorem{prop}[subsubsection]{\bf Proposition}
\newtheorem{rmk}[subsubsection]{\it Remark}

\newcommand{\divides}{\mid}
\newcommand{\ndivides}{\nmid}
\newcommand{\nin}{\notin}

\newcommand{\A}{{\mathbb A}}

\newcommand{\C}{{\mathbb C}}
\newcommand{\Cp}{{\C_p}}

\newcommand{\G}{{\mathbb G}}

\newcommand{\Q}{{\mathbb Q}}
\newcommand{\Ql}{{\Q_\ell}}
\newcommand{\Qp}{\Q_p}
\newcommand{\Qptimes}{\Q_p^\times}
\newcommand{\bQ}{{\overline{\Q}}}
\newcommand{\bQp}{{\bQ_p}}
\newcommand{\bQptimes}{{\bQ_p^\times}}

\newcommand{\R}{{\mathbb R}}

\newcommand{\bT}{{\mathbf{T}}}
\newcommand{\bX}{{\mathbf{X}}}

\newcommand{\Z}{{\mathbb Z}}
\newcommand{\Zp}{{\Z_p}}
\newcommand{\bZ}{{\overline\Z}}

\newcommand{\bZpp}{{\bZ_{(p)}}}
\newcommand{\wZ}{{\widehat\Z}}

\newcommand{\CA}{\mathcal A}

\newcommand{\CE}{\mathcal E}
\newcommand{\CF}{\mathcal F}

\newcommand{\CK}{\mathcal K}
\newcommand{\CL}{\mathcal L}

\newcommand{\CO}{\mathcal O}
\newcommand{\CP}{\mathcal P}
\newcommand{\CS}{\mathcal S}
\newcommand{\CV}{\mathcal V}

\newcommand{\CI}{\mathcal I}

\newcommand{\veps}{\varepsilon}
\newcommand{\bmu}{\mathbf{\mu}}
\newcommand{\kap}{\kappa}
\newcommand{\lap}{\kappa^c}
\newcommand{\vphi}{\varphi}

\newcommand{\s}{{\sigma}}

\newcommand{\grg}{{\mathfrak{g}}}
\newcommand{\grb}{{\mathfrak{b}}}
\newcommand{\grz}{{\mathfrak{z}}}

\newcommand{\grk}{{\mathfrak{k}}}
\newcommand{\grt}{{\mathfrak{t}}}
\newcommand{\grm}{{\mathfrak{m}}}
\newcommand{\grp}{{\mathfrak{p}}}
\newcommand{\grX}{{\mathfrak{X}}}
\newcommand{\grV}{{\mathfrak{V}}}

\newcommand{\alg}{{\mathrm{alg}}}

\newcommand{\cond}{{\mathrm{cond}}}
\newcommand{\diag}{\mathrm{diag}}
\newcommand{\D}{\mathrm{D}}
\newcommand{\Disc}{\mathrm{Disc}}
\newcommand{\End}{\mathrm{End}}
\newcommand{\nullset}{{\emptyset}}
\newcommand{\Gal}{{\mathrm{Gal}}}
\newcommand{\GL}{\mathrm{GL}}
\newcommand{\Gm}{{\mathbf{G}_m}}
\newcommand{\Hom}{\mathrm{Hom}}
\newcommand{\Ind}{\mathrm{Ind}}
\newcommand{\Isom}{\mathrm{Isom}}
\newcommand{\K}{\CK}
\newcommand{\la}{\langle}
\newcommand{\Lie}{{\mathrm{Lie}}}

\newcommand{\M}{{\mathrm M}}
\newcommand{\Meas}{{\mathrm{Meas}}}

\renewcommand{\O}{\CO}
\newcommand{\oM}{\overline{\M}}
\newcommand{\ord}{\mathrm{ord}}
\newcommand{\padic}{{p\mathrm{-adic}}}
\newcommand{\pair}{{\langle\cdot,\cdot\rangle}}
\newcommand{\ppair}{{(\cdot,\cdot)}}

\newcommand{\ra}{{\rangle}}
\newcommand{\rank}{\mathrm{rank}}

\newcommand{\Res}{\mathrm{Res}}
\renewcommand{\SS}{\mathrm{S}}
\newcommand{\sq}{{\square}}

\newcommand{\supp}{\mathrm{supp}}
\renewcommand{\t}{{}^t}
\newcommand{\T}{\mathrm{T}}

\newcommand{\tor}{\mathrm{tor}}
\newcommand{\trace}{\mathrm{trace}}

\newcommand{\vol}{\mathrm{vol}}
\newcommand{\op}{\mathrm{op}}
\newcommand{\gl}{\GL}
\newcommand{\gln}{\GL_n}
\newcommand{\Oe}{\mathcal{O}_{\K^+}}
\newcommand{\incl}{\mathit{incl}}
\newcommand{\cusp}{\mathrm{cusp}}
\newcommand{\IQ}{\mathbb{Q}}
\newcommand{\IC}{\mathbb{C}}
\newcommand{\Oev}{\mathcal{O}_{\K^+_v}}

\newcommand{\adeles}{\A}
\newcommand{\ad}{\mathbf{A}}

\newcommand{\OK}{\CO}
\newcommand{\OCp}{\CO_{\Cp}}
\newcommand{\ocp}{\OCp}

\newcommand{\tb}{\tilde{r}}
\newcommand{\tc}{\tilde{s}}
\newcommand{\ub}{\underline{\tb}}
\newcommand{\uc}{\underline{\tc}}

\newcommand{\ZZ}{\mathbb{Z}}
\newcommand{\ci}{C^\infty}

\renewcommand{\Im}{\mathrm{Im}}
\newcommand{\IR}{\mathbb{R}}
\newcommand{\isomto}{\overset{\sim}{\rightarrow}}

\newcommand{\ram}{\mathrm{ram}}
\newcommand{\hern}{\mathrm{Her}_n}

\newcommand{\rar}{{~\rightarrow~}}
\newcommand{\st}{{\star}}
\newcommand{\grP}{{\mathfrak{P}}}
\newcommand{\Tb}{\mathbf{T}}
\newcommand{\TT}{\mathbb{T}}
\newcommand{\DD}{\mathbb{D}}
\newcommand{\Dc}{\mathbb{D}_c}

\newcommand{\aord}{\text{\rm a-ord}}
\newcommand{\Frac}{\mathrm{Frac}}

\newcommand{\Serre}{\mathrm{Ser}}
\newcommand{\Pet}{\mathrm{Pet}}
\newcommand{\pairS}{\pair^\Serre}

\newcommand{\pairP}{\pair^\Pet}
\newcommand{\vpair}[1]{\langle {#1} \rangle}

\newcommand{\unitary}{\mathrm{unitary}}
\newcommand{\Eis}{\mathrm{Eis}}
\newcommand{\Eisab}{\Eis_{a, b}}
\newcommand{\Vol}{\mathrm{Vol}}

\newcommand{\isoarrow}{{~\overset\sim\longrightarrow~}}

\begin{document}

\authorheadline{Ellen Eischen and Michael Harris and Jianshu Li and Christopher Skinner}
\runningtitle{$p$-adic $L$-functions for unitary groups}
 
\begin{frontmatter}

\title{$p$-adic $L$-functions for unitary groups}
\author[1]{Ellen Eischen}
 
\address[1]{Department of Mathematics, University of Oregon,
Eugene, OR 97403,
USA
  \ead{eeischen@uoregon.edu}}

\author[2]{Michael Harris}
 
\address[2]{Department of Mathematics, Columbia University, New York, NY  10027, USA
  \ead{harris@math.columbia.edu}
}
  
\author[3]{Jianshu Li}

\address[3]{
Institute for Advanced Study in Mathematics,
Zhejiang University,
Hangzhou, China
\ead{matom@ust.hk}
\ead{jianshu@sjtu.edu.cn}}

\author[4]{Christopher Skinner}

\address[4]{
Department of Mathematics,
Princeton University,
 Fine Hall, Washington Road,
Princeton, NJ 08544-1000,
USA
\ead{cmcls@math.princeton.edu}
}

\received{}
 
\begin{abstract}
This paper completes the construction of $p$-adic $L$-functions for unitary groups.  More precisely, in \cite{HLS}, three of the authors proposed an approach to constructing such $p$-adic $L$-functions (Part I).  Building on more recent results, including the first named author's construction of Eisenstein measures and $p$-adic differential operators \cite{apptoSHL,EDiffOps}, Part II of the present paper provides the calculations of local $\zeta$-integrals occurring in the Euler product (including at $p$).  Part III of the present paper develops the formalism needed to pair Eisenstein measures with Hida families in the setting of the doubling method.
\end{abstract}
 
\end{frontmatter}

\section{Introduction}

This paper completes the construction of $p$-adic $L$-functions for unitary groups.  More precisely, in \cite{HLS}, three of the authors proposed an approach to constructing such $p$-adic $L$-functions (Part I).  Building on more recent results, including the first named author's construction of Eisenstein measures and $p$-adic differential operators \cite{apptoSHL,EDiffOps}, Part II of the present paper provides the calculations of local $\zeta$-integrals occurring in the Euler product (including at $p$).  Part III of the present paper develops the formalism needed to pair Eisenstein measures with Hida families in the setting of the doubling method.

The construction of $p$-adic $L$-functions consists of several significant steps, including studying certain $\zeta$-integrals occurring in the Euler products of the corresponding $\IC$-valued $L$-functions (one of the main parts of this paper, which involves certain careful choices of local data and which is the specific step about which we are most frequently asked by others in the field) and extending and adapting earlier constructions of $p$-adic $L$-functions (e.g. Hida's work in \cite{hidasearch}, which recovers Katz's construction from \cite{kaCM} as a special case). We also note that the last three named authors had already computed local zeta integrals for sufficiently regular data as far back as 2003, but the computations were not 
included in \cite{HLS} for lack of space.  Since then, a new approach to choosing local data and computing local zeta integrals at primes dividing $p$ has allowed us to treat the general case.  These are the computations presented here.

In Section  \ref{about}, we put this paper in the context of the full project to construct $p$-adic $L$-functions (which comprises the present paper and \cite{HLS}), and we describe the key components and significance of the broader project.  The exposition in the present paper, especially the description of the geometry, was written especially carefully to provide a solid foundation for future work both by the authors of this paper and by other researchers in the field.

\subsection{About the project}\label{about}

Very precise and orderly conjectures predict how certain integer values of $L$-functions of motives over number fields, suitably modified, 
fit together into $p$-adic analytic functions (e.g. \cite{coatesII,CPR,pan, hidasearch}).     These functions directly generalize the $p$-adic zeta function of Kubota and Leopoldt 
that has played a central role in algebraic number theory, through its association with Galois cohomology, in the form of Iwasawa's Main Conjecture.
Such $p$-adic $L$-functions have been defined in a number of settings.   In nearly all cases they are attached to automorphic forms rather 
than to motives; no systematic way is known to obtain information about special values of motivic $L$-functions unless they can be identified with
automorphic $L$-functions.  However, the procedures for attaching $L$-functions to automorphic forms other than Hecke characters
are by no means orderly; any given $L$-function can generally be obtained by a number of methods that have no relation to one another, and in general no obvious
relation to the geometry of motives.  And while these procedures are certainly precise, they also depend on arbitrary choices:  the $L$-function is
attached abstractly to an automorphic representation, but as an analytic function it can only be written down after choosing a specific automorphic form,
and in general there is no optimal choice.

When Hida developed the theory of analytic families of ordinary modular forms he also expanded the concept of $p$-adic $L$-functions.
Hida's constructions naturally gave rise to analytic functions in which the modular forms are variables, alongside the character of $GL(1)$ that
plays the role of the $s$ variable in the complex $L$-function.  This theory has also been generalized, notably to overconvergent modular forms.  There seems to be a consensus among experts
on how this should go in general, but as far as we know no general conjectures have been made public.  This is in part because constructions
of $p$-adic families are no more orderly than the construction of automorphic $L$-functions, except in the cases Hida originally studied:  
families are realized in the coherent or topological cohomology\footnote{In principle, completed cohomology in Emerton's sense could
also be used for this purpose, and would give rise to more general families.  As far as we know $p$-adic $L$-functions have
not yet been constructed in this setting.} of a locally symmetric space; but the connection of the latter to motives is tenuous and in many cases
purely metaphoric.

The present project develops one possible approach to the construction of $p$-adic $L$-functions.  We study complex $L$-functions
of automorphic representations of unitary groups of $n$-dimensional hermitian spaces, by applying the doubling method of Garrett and Piatetski-Shapiro-Rallis
\cite{ga,GPSR} to the
automorphic representations that contribute to the coherent cohomology of Shimura varieties in degree $0$; in other words, to
holomorphic modular forms.  When $n = 1$, we recover
Katz's theory of $p$-adic $L$-functions of Hecke characters \cite{kaCM}, and much of the analytic theory is an adaptation of Katz's constructions to higher 
dimensions.   For general $n$, the theory of ordinary families of holomorphic modular forms on Shimura varieties of PEL type has been
developed by Hida, under hypotheses on the geometry of compactifications that have subsequently been proved by
Lan.  It is thus no more difficult to construct $p$-adic $L$-functions of Hida families than to study the $p$-adic versions of complex
$L$-functions of individual automorphic representations.  Interpreting our results poses a special challenge, however.    The conjectures on
motivic $p$-adic $L$-functions are formulated in a framework in which the Betti realization plays a central role, in defining complex as well as
$p$-adic periods used to normalize the special values.  Betti cohomology exists in the automorphic setting as well, but it cannot be detected by
automorphic methods.   The doubling method provides a substitute:  the cup product in coherent cohomology.  Here one needs to exercise some
care.  Shimura proved many years ago that the critical value at $s = 1$ of the adjoint $L$-function attached to a holomorphic modular form $f$ 
equals the Petersson square norm $\la f, f \ra$, multiplied by an elementary factor.   If one takes this quantity as the normalizing period, the resulting 
$p$-adic adjoint $L$-function is identically equal to $1$.  Hida observed that the correct normalizing period is not $\la f, f \ra$ but rather the product
of (normalized) real and imaginary periods; using this normalization, one obtains a $p$-adic adjoint $L$-function whose special values measures
congruences between $f$ and other modular forms.   This is one of the fundamental ideas in the theory of deformations of modular forms and Galois
representations; but it seems to be impossible to apply in higher dimensions, because the real and imaginary periods are defined by means of Betti
cohomology.  One of the observations in the present project is that the integral information provided by these Betti periods can naturally be recovered
in the setting of the doubling method, provided one works with Hida families that are free over their corresponding Hecke algebras, and one assumes
that the Hecke algebras are Gorenstein.  These hypotheses are not indispensable, but they make the statements much more natural, and we have
chosen to adopt them as a standard; some of the authors plan to indicate in a subsequent paper what happens when they are dropped.

This approach to families is the first of the innovations of the present project, in comparison with the previous work \cite{HLS}.  We stress that the Gorenstein hypothesis, suitably interpreted, is  particularly natural in the setting
of the doubling method.   Our second, most important innovation, is the use of the general Eisenstein measure constructed  in \cite{apptoSHL, apptoSHLvv}.   

In order to explain the contents of this project more precisely, we remind the reader what is expected
of a general theory of $p$-adic $L$-functions.  We are given a $p$-adic analytic space $Y$ and a subset $Y^{class}$
of points such that, for each $y \in Y^{class}$ there is a motive $M_y$, and possibly an additional datum $r_y$ (a {\it refinement})
such that $0$ is a critical value
of the $L$-function $L(s,M_y)$.  The $p$-adic $L$-function is 
then a meromorphic function $L_p$ on $Y$ whose values at $y \in Y^{class}$ can be expressed in terms of $L(0,M_y)$.
More precisely, there is a $p$-adic period $p(M_y,r_y)$ such that $\frac{L_p(y)}{p(M_y,r_y)}$ is an algebraic number, and then we have the relation
\begin{equation}\label{genplfunction}
\frac{L_p(y)}{p(M_y,r_y)} = Z_{\infty}(M_y)Z_p(M_y,r_y)\cdot \frac{L(0,M_y)}{c^+(M_y)}.
\end{equation}
Here $c^+(M_y)$ is the period that appears in Deligne's conjecture on special values of $L$-functions, so that
$\frac{L(0,M_y)}{c^+(M_y)}$ is an algebraic number, while $Z_{\infty}$ and $Z_p$ are correction factors that are
built out of Euler factors and $\varepsilon$-factors of the zeta function of $M_y$ at archimedean primes and primes dividing $p$,
respectively.

In our situation, we start with a CM field $\K$ over $\Q$, a quadratic extension of a totally real field
$\K^+$, and an $n$-dimensional hermitian vector space $V/\K$.
Then $Y$ is the space of pairs $(\lambda,\chi)$, where $\lambda$ runs through the set of ordinary $p$-adic modular forms
on the Shimura variety $Sh(V)$ attached to $U(V)$ and $\chi$ runs through $p$-adic Hecke characters of $\K$; both $\lambda$ and
$\chi$ are assumed to be unramified outside a finite set $S$ of primes of $\K$, including those dividing $p$, and of bounded level 
at primes not dividing $p$.    Because we are working with  ordinary forms -- more precisely, what Hida calls {\it nearly ordinary} forms, though the terminology is used inconsistently in the literature -- the ring $\CO(Y)$ of holomorphic functions on $Y$ is finite
over some Iwasawa algebra, and the additional refinement is superfluous.   In the project, $\lambda$ denotes a character of Hida's (nearly) ordinary
Hecke algebra.
If $(\lambda,\chi) \in Y^{class}$ then 
\begin{itemize}
\item $\lambda = \lambda_{\pi}$
for some automorphic representation $\pi$ of $U(V)$; it is the character of the ordinary Hecke algebra acting on vectors that are spherical outside
$S$ and (nearly) ordinary at primes dividing $p$;
\item $\chi$ is a Hecke character of type $A_0$; 
\item the standard $L$-function $L(s,\pi,\chi)$ has a critical value at $s = 0$.  
\end{itemize}
(By replacing $\chi$ by its multiples by powers of the norm
character, this definition accommodates all critical values of $L(s,\pi,\chi)$.)  Under hypotheses to be discussed below,
the automorphic version of Equation \eqref{genplfunction} is particularly simple to understand:

\begin{equation}\label{plfunction}
L_p(\lambda_{\pi},\chi) = c(\pi)\cdot Z_{\infty}(\pi,\chi)Z_p(\pi,\chi)Z_S\cdot \frac{L(0,\pi,\chi)}{P_{\pi,\chi}}
\end{equation}

The left-hand side is the specialization to the point $(\lambda_{\pi},\chi)$ of an element $L_p \in \CO(Y)$.  The right hand side is purely
automorphic.  The $L$-function is the standard Langlands $L$-function of $U(V) \times GL(1)_{\K}$.   Its analytic and arithmetic
properties have been studied most thoroughly using the doubling method.  If $U(V)$ is the symmetry group of the hermitian form
$\pair_V$ on $V$, let $-V$ be the space $V$ with the hermitian form $-\pair_V$, and let $U(-V)$ and $Sh(-V)$ be the corresponding unitary group
and Shimura variety.  The groups $U(V)$ and $U(-V)$ are canonically isomorphic, but the natural  identification of $Sh(-V)$ with
$Sh(V)$ is anti-holomorphic; thus holomorphic automorphic forms on $Sh(V)$ are identified with anti-holomorphic automorphic forms, or
coherent cohomology classes of top degree,
on $Sh(-V)$, and vice versa.  The space $W = V \oplus (-V)$, endowed with the hermitian form $\pair_V \oplus -\pair_V$, is always
maximally isotropic, so $U(W)$ has a maximal parabolic subgroup $P$ with Levi factor isomorphic to $GL(n)_{\K}$.  To any Hecke
character $\chi$ of $\K$ one associates the family of degenerate principal series
$$I(\chi,s) = \Ind_{P(\adeles)}^{U(W)(\adeles)} \chi\circ \det\cdot \delta_P^{-s/n}$$
and constructs the meromorphic family of Eisenstein series $s \mapsto E(\chi,s,f,g)$ with $f = f(s)$ a section of $I(\chi,s)$ and $g \in U(W)(\adeles)$.  
On the other hand, $U(V) \times U(-V)$ naturally embeds in $U(W)$.  Thus if $\phi$ and $\phi'$ are cuspidal automorphic forms on $U(V)(\adeles)$ and
$U(-V)(\adeles)$, respectively, the integral 
$$I(\phi,\phi',f,s) = \int_{[U(V)\times U(-V)]} E(\chi,s,f, (g_1,g_2))\phi(g_1)\phi'(g_2)\chi^{-1}(\det(g_2)) dg_1 dg_2,$$
defines a meromorphic function of $s$.  
Here $[U(V)\times U(-V)] = U(V)(F)\backslash U(V)(\adeles) \times U(-V)(F)\backslash U(-V)(\adeles)$, $g_1 \in U(V)(\adeles)$, $g_2 \in U(-V)(\adeles)$, 
and $dg_1$ and $dg_2$ are Tamagawa measures.  

The doubling method asserts that, if $\pi$ is a cuspidal automorphic representation
of $U(V)$ and $\phi \in \pi$, then $I(\phi,\phi',f,s)$ vanishes identically unless $\phi' \in \pi^{\vee}$; and if $\la \phi, \phi' \ra \neq 0$,
then the integrals $I(\phi,\phi',f,s)$ unwind and factor as an Euler product whose unramified terms give the standard $L$-function $L(s+ \frac{1}{2},\pi,\chi)$ and
(as $f$, $\phi$, $\phi'$ vary) provide the meromorphic continuation and functional equation of the standard $L$-function.   Another way to
look at this construction is to say that the {\it Garrett map}
$$\phi \mapsto G(f,\phi,s)(g_2) = \chi^{-1}\circ\det(g_2)\cdot \int_{U(V)(F)\backslash U(V)(\adeles)]} E(\chi,s,f, (g_1,g_2))\phi(g_1) dg_1$$
is a linear transformation from the automorphic representation $\pi$ of $U(V)$ to $\pi$ viewed as an automorphic reprepresentation of $U(-V)$; and the matrix
coefficients of this linear transformation give the adelic theory of the standard $L$-function.  We develop
a theory that allows us to interpret these matrix coefficients integrally in Hida families, under special hypotheses on the localized Hecke
algebra described below.  Note that when $\pi$ is an anti-holomorphic representation
of $U(V)$, its image under the Garrett map is $\pi$, but viewed as a {\it holomorphic} representation of $U(-V)$.

The factor $P_{\pi,\chi}$ is a product
of several terms, of which the most important is a normalized Petersson inner product of holomorphic forms on $U(V)$.  Although it arises
naturally as a feature of the doubling method, its definition involves
some choices that are reflected in the other terms.  The local term $Z_S$, in our normalization, is a local volume multiplied
by a local inner product (depending on the choices).  The correction factors $Z_{\infty}$ and $Z_{p}$ are explicit local zeta
integrals given by the doubling method.  The archimedean factor has not been evaluated explicitly, except when $\pi$ is associated
to a holomorphic modular form of scalar weight (by Shimura) or, more generally, of weight that is ``half scalar"  
at every archimedean place (by Garrett) \cite{sh, ga06}.   In the present paper we leave it unspecified; it depends only
on the archimedean data (the weights) and not on the Hecke eigenvalues.

The explicit calculation of the local term $Z_p$ is our third major innovation {\bf and one of the key pieces of the current paper}, and it occupies the longest
single section of this paper (Section  \ref{ESeriesZIntegrals-section}).   It has the expected form:  a quotient of a product of Euler factors (evaluated at $s$)
by another product of Euler factors (evaluated at $1-s$) multiplied by a local $\varepsilon$ factor and a volume factor.  The key observation is that the denominator arises by applying the Godement-Jacquet local functional equation
to the input data.  This is the step in the construction that owes the most to (adelic) representation theory.  The input
data for the Eisenstein measure represent one possible generalization of Katz's construction in \cite{kaCM}.  The local integral
has been designed to apply to overconvergent families as well as to ordinary families; one of us plans to explore this in future work.  The precise form of the local factor at a prime $w$ dividing $p$ depends on the signatures of the
hermitian form at the archimedean places associated to $p$ as part of the ordinary data; this appears mysterious but
in fact turns out to be a natural reflection of the PEL structure at primes dividing $p$, or alternatively of the
embedding of the ordinary locus of the Shimura variety attached to (two copies of) $U(V)$ in that attached to the doubled
group.   

A different calculation of the local term had been carried out at the time of \cite{HLS}.  It was not published
at the time because of space limitations.  It was more ad hoc than the present version and applied only 
when the adelic local components at primes dividing $p$ of an ordinary form could be identified as an
explicit function in a principal series.  The present calculation is more uniform and yields a result in the
expected form.

Before explaining the final factor $c(\pi)$ it is preferable to explain the special hypotheses underlying the formula \ref{plfunction}, which represent the fourth innovation in this project. 
The point $(\lambda_{\pi},\chi)$ belongs to a Hida family, which for the present purposes means a {\it connected} component, which we denote
$Y_{\pi,\chi}$, of the space $Y$; in other contexts one works with an {\it irreducible} component.  The ring of functions on $Y_{\pi,\chi}$ is of the form
$\Lambda\hat{\otimes} \TT_{\pi}$, where $\Lambda$ is an Iwasawa algebra attached to $\chi$ and $\TT_{\pi}$ is the localization of the
big Hecke algebra at the maximal ideal attached to $\pi$.  The principal hypotheses are that $\TT_{\pi}$ is Gorenstein, and that
the module of ordinary modular forms (or its $\Zp$-dual, to be more precise) is free over $\TT_{\pi}$.  There are also local hypotheses that
correspond to the hypothesis of minimal level in the Taylor-Wiles theory of deformations of modular Galois representations.  These hypotheses
make it possible to define $L_p$ as an element of $\CO_{Y_{\pi, \chi}}$.  The presence of the factor $c(\pi)$ is a sign that $L_p$ is not
quite the $p$-adic $L$-function; $c(\pi)$ is a generator of the {\it congruence ideal} which measures congruences between $\lambda_{\pi}$ and
other characters $\lambda_{\pi'}$ of $\TT_{\pi}$ (of the same weight and level).  
The specific generator $c(\pi)$ depends on the same choices used to define $P_{\pi,\chi}$, so that
the product on the right-hand side is independent of all choices.

In the absence of the special hypotheses, it is still possible to define $L_p$ in the fraction field of 
$\Lambda\hat{\otimes} \TT_{\pi}$, but the statement
is not so clean.  In any case, the $p$-adic valuations of $c(\pi)$ are in principle unbounded, and so the $p$-adic interpolation of the normalized
critical values of standard $L$-functions is generally given by a meromorphic function on $Y$.

\subsubsection{Clarifications}

The above discussion has artificially simplified several points.  The Shimura variety is attached not to $U(V)$ but rather to
the subgroup, denoted $GU(V)$, of the similitude group of $V$ with rational similitude factor.  All of the statements above
need to be modified to take this into account, and this is done in the paper.   This detail plagues the paper
from beginning to end, as it seems at least to some degree also to plague every paper on Shimura varieties attached to unitary groups.  One can hope that
a far-sighted colleague will find an efficient way to do away with this.  

What we called the moduli space of PEL type
associated to $V$ is in general a union of several isomorphic Shimura varieties, indexed by the defect of the Hasse principle;
$p$-adic modular forms are most naturally defined on a single Shimura variety rather than on the full moduli space.  We need the
moduli space in order to define $p$-adic modular forms, but in the computations we work with a single fixed Shimura variety.  

Although the $p$-adic $L$-functions are attached to automorphic forms on unitary (similitude) groups, they are best
understood as $p$-adic analogues of the standard $L$-functions of cuspidal automorphic representations of $GL(n)$.
The passage from unitary groups to $GL(n)$ is carried out by means of stable base change.  A version of this
adequate for our applications was developed by Labesse in \cite{lab}.  Complete results, including precise multiplicity formulas,
were proved by Mok for quasi-split unitary groups \cite{mok}; however, we need to work with unitary groups over totally real fields
with arbitrary signatures, and the quasi-split case does not suffice.   The general case is presently being completed by Kaletha, Minguez,
Shin, and White, and we have assumed implicitly that Arthur's multiplicity conjectures are known for unitary groups.
The book \cite{gangof4} works out the multiplicities of tempered representations and is probably sufficient for the purposes of the
present project.

From the standpoint of automorphic representations of $GL(n)$, the ordinary hypothesis looks somewhat special; in fact, the critical
values of $L$-functions of $GL(n)$ can be interpreted geometrically on unitary groups of different signatures, and the ordinary
hypotheses for these different unitary groups represent different branches of a $p$-adic $L$-function that can only be related
to one another in a general overconvergent family.  The advantage of restricting
our attention to ordinary families is that the $p$-adic $L$-functions naturally belong to integral Hecke algebras.  To add to the confusion, however,
Hida's theory of (nearly) ordinary modular forms applies to holomorphic automorphic representations, but the doubling method requires
us to work with {\it anti-holomorphic} representations.  The eigenvalues of the $U_p$-operators on representations do not coincide with
those on their holomorphic duals; for lack of a better terminology, we call these representations {\it anti-ordinary}.  Keeping track of
the normalizations adds to the bookkeeping but involves no essential difficulty.

\subsubsection{What this project does not accomplish}

Although we have made an effort to prove rather general theorems, limitations of 
patience have induced us to impose restrictions on our results.  Here are
some of the topics we have not covered.

First of all, we have not bothered to verify that the local  terms ($Z_p, Z_\infty, Z_S$) and the global terms ($L(0,\pi,\chi), Q_{\pi,\chi}$)  in Equation \eqref{plfunction} 
correspond termwise with those predicted by the general conjectures on $p$-adic $L$-functions for motives.  The correspondence between automorphic representations and (de Rham realizations of) motives is not
straightforward.  In a general sense, comparing  \eqref{genplfunction} with \eqref{plfunction},  we can say that the motive $M_y$ that appears in \eqref{genplfunction} corresponds in \eqref{plfunction} to the hypothetical motive attached to the automorphic representation $\pi$, whose $\ell$-adic realization is the $n$-dimensional  representation of $Gal(\bar{\Q}/\CK)$ constructed in \cite{CH} (among many other places), twisted by the $\ell$-adic Galois character attached to the Hecke character $\chi$.   The local factor $Z_p$ certainly has the same shape as the  local factors that appear in the conjectures of \cite{coatesII,CPR}, but we have not checked that the Frobenius eigenvalues that appear in the latter conjectures are exactly the ones we find.
We expect to address these issue in a subsequent paper; however, until we find a simple way to
compute the archimedean term $Z_{\infty}(\pi,\chi)$ explicitly, we will not be able to compare it with anything motivic.

We have also not attempted to analyze the local factors at ramified finite primes for $\pi$ and $\chi$.  The geometry
of the moduli space has no obvious connection to the local theory of the doubling method.  Moreover, a complete
treatment of ramified local factors requires a $p$-integral version of the doubling method.  This may soon be available,
thanks to work of Minguez, Helm, Emerton-Helm, and Moss, but for the moment we have preferred to simplify
our presentation by choosing local data that give simple volume factors for the local integrals at bad primes.

One of us plans to adapt the methods of the present project to general overconvergent families,
where Hida theory is no longer appropriate.  On the other hand, the methods of Hida theory do apply to more
general families than those we consider.  In \cite{H98}, Hida introduces the notion of {\it $P$-ordinary} modular forms on 
a reductive group $G$, where $P$ denotes a parabolic subgroup of $G$.  One obtains the usual (nearly) ordinary forms 
when $P = B$ is a Borel subgroup; in general, for $P$ of $p$-adic rank $r$, the $P$-ordinary forms vary in an $r$-dimensional
family, up to global adjustments (related to Leopoldt's conjecture in general).  Most importantly, a form can be
$P$-ordinary without being $B$-ordinary.    Our theory applies to $P$-ordinary forms as well; we hope to return to
this point in the future, and \cite{EiMa} is a first step in this direction.

Our $p$-adic $L$-function, when specialized at a classical point corresponding to the automorphic
representation $\pi$, gives the
corresponding value of the classical complex $L$-function,  divided by what appears
to be the correctly normalized complex period invariant, and multiplied by a factor $c(\pi)$ measuring
congruences between $\pi$ and other automorphic representations.  This is a formal consequence of
the Gorenstein hypothesis and is consistent with earlier work of Hida and others on $p$-adic $L$-functions
of families.  It is expected that the factor $c(\pi)$ is the specialization at $\pi$ of the ``genuine'' $p$-adic $L$-function
that interpolates normalized values at $s = 1$ of the adjoint $L$-function (of $\pi$, or one of the Asai $L$-functions
for its base change to $GL(n)$).  As far as we know, no one has constructed this $p$-adic adjoint
$L$-function in general.  We do not know how to construct a $p$-adic analytic function on the ordinary
family whose specialization at $\pi$ equals $c(\pi)$, not least because $c(\pi)$ is only well-defined
up to multiplication by a $p$-adic unit.  Most likely the correct normalization will have to take account of
$p$-adic as well as complex periods.

Finally, we have always assumed that our base field $\K$ is unramified at $p$.   This
hypothesis is unnecessary, thanks to Lan's work in \cite{Lan2}, but it simplifies a number of statements.

\subsection{History}  Work on this paper began in 2001 as a collaboration between two of the authors, around the time of a visit by one of us (M.H.) to the second one (J.-S. L.) in Hong Kong.  The initial objective was to study congruences between endoscopic and stable holomorphic modular forms on unitary groups.  The two authors were soon joined by a third (C. S.), and a report on the results was published in \cite{HLS05}.  The subsequent article \cite{HLS} carried out the first part of the construction of a $p$-adic analytic function for a single automorphic representation.  Because $p$-adic differential operators had not yet been constructed for unitary group Shimura varieties, this function only provided the $p$-adic interpolation for the right-most critical value of the $L$-function, and only applied to scalar-valued holomorphic modular forms.  Moreover, although the local computation of the zeta integrals at primes dividing $p$, which was not included in \cite{HLS}, was based on similar principles to the computation presented here, it had only been completed for ramified principal series and only when the conductors of the local inducing characters were aligned with the slopes of the Frobenius eigenvalues.  After the fourth author (E.E.) had defined $p$-adic differential operators in \cite{EDiffOps, emeasurenondefinite} and constructed the corresponding Eisenstein measure in \cite{apptoSHL,apptoSHLvv}, it became possible to treat general families of holomorphic modular forms and general ramification.  

The delay in completing the paper, for which the authors apologize, can be attributed in large part to the difficulty of reconciling the different notational conventions that had accumulated over the course of the project.   In the meantime, Xin Wan had constructed certain $p$-adic $L$-functions in the same setting in \cite{xinwan}, by a method based on computation of Fourier-Jacobi coefficients, as in \cite{SkUr}.  More recently, Zheng Liu has constructed $p$-adic $L$-functions for symplectic groups \cite{ZL}.  Among other differences, Liu makes consistent use of the theory of nearly overconvergent $p$-adic modular forms, thus directly interpreting nearly holomorphic Eisenstein series as $p$-adic modular forms; and her approach to the local zeta integrals is quite different from ours.

\subsection{Contents and structure of this paper}

After establishing notation and conventions in Section  \ref{notation-section} below, we begin in Section  \ref{padic-unitary-section} by recalling the theory of modular forms on unitary groups, as well as Hida's theory of $p$-adic modular forms on unitary groups.  This section has carefully set up the framework needed for our project and will likely also provide a solid foundation for others working in this area.  In Section  \ref{doublingsetup}, we discuss the geometry of restrictions of automorphic forms, since the restriction of an Eisenstein series is a key part of the doubling method (Section  \ref{doubling16}) used to construct $L$-functions.  In Section  \ref{ESeriesZIntegrals-section}, we discuss the doubling method.  This section also contains the local zeta calculations mentioned at the beginning of the introduction.  The most important of these is the calculation at primes dividing $p$ (Section  \ref{pchoices-section}), which is also the longest single step of this paper.  
Section  \ref{EMeasure-padicLfunctions-section} provides statements about measures, which depend on the local data chosen in Section  \ref{ESeriesZIntegrals-section}.   A formalism for relating duality pairings to complex conjugation and to the action of Hecke algebras is developed in Section  \ref{serreduality-section}; this is extended to Hida families in Section  \ref{families}, which also begins the formalism for construction of $p$-adic $L$-functions in families.  Section  \ref{localtheory} establishes the relation between $p$-adic and $C^{\infty}$-differential operators, and develops the local theory of ordinary and anti-ordinary vectors in representations at $p$-adic places.  Finally, Section  \ref{lastchapter} states and proves the main theorems about the existence of the $p$-adic $L$-function.

\subsection{Notation and conventions}\label{notation-section}

\subsubsection{General notation}

Let $\bQ\subset\C$ be the algebraic closure of $\Q$ in $\C$ and
let the complex embeddings of a number field $F\subset\bQ$ be $\Sigma_F=\Hom(F,\C)$; so $\Sigma_F=\Hom(F,\bQ)$.
Throughout, $\K\subset\bQ$ is a CM field with ring of integers $\O$, and $\K^+$ is the maximal totally real subfield of $\K$.
The non-trivial automorphism in $\Gal(\K/\K^+)$ is denoted by $c$. Given a place $v$ of $\K$,
the conjugate place $c(v)$ is usually denoted $\bar v$.

Let $p$ be a fixed prime that is unramified in $\K$ and such that every place above
$p$ in $\K^+$ splits in $\K$.
Let $\bQp$ be an algebraic closure of $\Qp$ and fix an embedding $\incl_p:\bQ\hookrightarrow\bQp$.
Let $\bZpp\subset\bQ$ be the valuation ring for the
valuation determined by $\incl_p$.  Let $\Cp$ be the completion of $\bQp$ and let $\O_\Cp$ be the valuation ring
of $\Cp$ (so the completion of $\bZpp$). Let $\iota_p:\C\isoarrow \Cp$ be an isomorphism extending $\incl_p$.

When $V$ is a hermitian space over $\CK$, with hermitian form $\langle, \rangle$, we let $GU^+(V)$ denote the group of unitary similitudes of $V$; this is a group scheme over $\K^+$, defined by
$$GU^+(V)(R) = \{g \in GL(V\otimes_{\K^+} R) ~|~ \langle g(v),g(v') \rangle = \nu(g) \langle v, v' \rangle ~ \forall v, v' \in V\otimes_{\K^+} R \}$$
where $\nu(g) \in R^\times$; here $R$ is any $\K^+$-algebra.  
This is the group that is usually denoted $GU(V)$.  However, we prefer to reserve the notation $GU(V)$ for the  $\Q$-subgroup scheme
of $R_{\K^+/\Q} GU^+(V)$ which is the fiber product
$$GU(V) = R_{\K^+/\Q} GU^+(V) \times_{R_{\K^+/\Q} \mathbb{G}_{m,\K^+}} \mathbb{G}_m$$
where the map $\mathbb{G}_m \hookrightarrow R_{\K^+/\Q} \mathbb{G}_{m,\K^+}$ is the canonical inclusion.

For any $\sigma\in \Sigma_\K$ let $\grp_\sigma$ be the prime of $\O$ determined by the embedding $\incl_p\circ\sigma$. Note that $c(\grp_\sigma) = \grp_{\sigma c}$.
For a place $w$ of $\K$ over $p$ we will write $\grp_w$ for the corresponding prime of $\O$.
Let $\Sigma_p$ be a set containing exactly one place
of $\K$ over each place of $\K^+$ over $p$.

\begin{rmk}\label{grouppoints} It is often more convenient to denote an algebraic group over a ring $R$ which is a number field, a $p$-adic field, or an integer ring, by its group of points $G(R)$.  For example, if $V$ is a free $R$-module we may write $G = GL_R(V)$ as shorthand for the group scheme over $Spec(R)$ whose group of $S$-valued points, for any $R$-scheme $S$, is given by $G(V\otimes_R S)$.  
\end{rmk}

Let $\Z(1)\subset\C$ be the kernel of the exponential map $\exp:\C\rightarrow\C^\times$.
This is a free $\Z$-module of rank one with non-canonical basis $2\pi \sqrt{-1}$.  
For any commutative ring $R$ let $R(1)=R\otimes\Z(1)$.

In what follows, 
when $(G,X)$ is a Shimura datum, an automorphic representation of $G$ is defined to be a $(\grg,K)\times G(\A_f)$-submodule of the space of automorphic forms,
where $K$ is {\it the stabilizer of a point} in $X$; in particular, $K$ contains the center of $G(\R)$ but
{\it does not} generally contain a full maximal compact subgroup.  
In this way, holomorphic and anti-holomorphic representations are kept separate.  This is of fundamental importance for applications to
coherent cohomology and thus to our construction of $p$-adic $L$-functions.

\subsubsection{Measures and pairings}

We will need to fix a Haar measure $dg$ on the ad\`ele group of a reductive group $G$ over a number field $F$.  For the sake of definiteness we take $dg$ to be Tamagawa measure.  In this paper we will not be so concerned with the precise choice of measure, because we will not be calculating local zeta integrals at archimedean primes explicitly, but we do want to be consistent.  When we write $dg = \prod_v dg_v$, where $v$ runs over places of $F$ and $dg_v$ is a Haar measure on the $F_v$-points $G(F_v)$, we will want to
make the following additional hypotheses:
\begin{hype}\label{meas1}
\begin{enumerate}
\item At all finite places $v$ at which the group $G$ is unramified, $dg_v$ is the measure that gives volume $1$ to a hyperspecial maximal compact subgroup.\label{hyp1meas1}

\item At all finite places $v$ at which the group $G$ is isomorphic to $\prod_i GL(n_i,F_{i,w_i})$, where
$F_{i,w_i}$ is a finite extension of $F_v$ with integer ring $\CO_i$, 
(whether or not $F_{i,w_i}$ is ramified over the corresponding completion of $\Q$),
$dg_v$ is the measure that gives volume $1$ to the group $\prod_i GL(n_i,\CO_i)$.\label{hyp2meas1}

\item At all finite places $v$, the values of $dg_v$ on open compact subgroups are rational numbers.
\item At archimedean places $v$, we choose measures such that $\prod_v dg_v$ is Tamagawa measure.
\end{enumerate}
\end{hype}

Let $Z_G \subset G$ denote the  center of $G$, and let $Z \subset Z_G(\ad)$
be any closed subgroup such that $Z_G(\ad)/Z$ is compact; for example, one can take $Z$ to be
the group of real points of the maximal $F$-split subgroup of $Z_G$.  We choose a Haar measure on $Z$ that satisfies the conditions of \ref{meas1}
if $Z$ is the group of ad\`eles of an $F$-subgroup of $Z_G$.  
The measure $dg$ defines a bilinear pairing $\la , \ra$ on $L^2(Z\cdot G(F)\backslash G(\ad))$; more generally,
if $f_1(zg)f_2(zg) = f_1(g)f_2(g)$ for all $z \in Z$, we can extend the pairing to write
\begin{equation}\label{bilinearpairing} \la f_1, f_2 \ra_Z = \int_{Z \cdot G(F)\backslash G(\ad)} f_1(g)f_2(g) dg, \end{equation}
and if not, we set $\la f_1, f_2 \ra_Z = 0$.  

Suppose $\pi$ and $\pi^{\vee}$ are irreducible cuspidal automorphic representations of $G$.
Then $\la , \ra_Z:  \pi \otimes \pi^{\vee} \rightarrow \C$ is a canonically defined pairing.  Now suppose we have factorizations
\begin{equation}\label{factor}
fac_{\pi}:  \pi \isoarrow  \otimes'_v \pi_v,~~ fac_{\pi^{\vee}}: \pi^{\vee} \isoarrow \otimes'_v \pi_v^{\vee}
\end{equation}
where $\pi_v$ is an irreducible representation of $G(F_v)$.  Assume moreover that we are given non-degenerate pairings of $G(F_v)$-spaces
\begin{equation}\label{localpair}   \la, \ra_{\pi_v}:  \pi_v \otimes \pi_v^{\vee} \rar \C    \end{equation}
for all $v$.  Then there is a constant $C = C(dg,fac_{\pi},fac_{\pi^{\vee}}, \prod_v \la, \ra_{\pi_v})$ such that, for all
vectors $\varphi \in \pi$, $\varphi^{\vee} \in \pi^{\vee}$ that are factorizable in the sense that
$$fac_{\pi}(\varphi) = \otimes_v \varphi_v; ~ fac_{\pi^{\vee}}(\varphi^{\vee}) = \otimes_v \varphi_v^{\vee}$$
we have

\begin{equation}\label{schur} 
\la \varphi,\varphi^{\vee} \ra_Z = C(dg,fac_{\pi},fac_{\pi^{\vee}}, \prod_v \la, \ra_{\pi_v}) \prod_v  \la \varphi_v, \varphi_v^{\vee}\ra_{\pi_v}.
\end{equation}

When $G$ is quasi-split and unramified over $F_v$ and $\pi_v$ is a principal series representation, induced from a Borel subgroup $B \subset G(F_v)$,
we choose a hyperspecial maximal compact subgroup $K_v \subset G(F_v)$ and define the {\it standard local pairing} to be:
\begin{equation}\label{prin}
\la f, f^{\vee} \ra_{\pi_v} = \int_{K_v} f(g_v)f^{\vee}(g_v) dg_v.
\end{equation}

In situation \eqref{hyp2meas1} of Hypotheses \ref{meas1}, we take $K_v = \prod_i GL(n_i,\CO_i)$; however, the pairing \eqref{prin} does not depend on the choice of $K_v$.

\part*{Part II: zeta integral calculations}

\section{Modular forms and $p$-adic modular forms on unitary groups}\label{padic-unitary-section}

This section introduces details about modular forms and $p$-adic modular forms on unitary groups that we will need for our applications.  For alternate discussions of modular forms and $p$-adic modular forms on unitary groups, see \cite{Hida, CEFMV}.

\subsection{PEL moduli problems: generalities}\label{PELdata}
By a PEL datum we will mean a tuple $P=(B,*,\O_B,L,\pair,h)$ where
\begin{itemize}
\item $B$ is a semisimple $\Q$-algebra with positive involution $*$, the action of which we write
as $b\mapsto b^*$;
\item $\O_B$ is a $*$-stable $\Z$-order in $B$;
\item $L$ is a $\Z$-lattice with a left $\O_B$-action and a
non-degenerate alternating pairing $\pair:L\times L\rightarrow \Z(1)$ such that
$\la bx,y\ra = \la x,b^*y\ra$ for $x,y\in L$ and $b\in \O_B$;
\item $h:\C\rightarrow \End_{\O_B\otimes\R}(L\otimes\R)$ is a homomorphism such that
$\la h(z)x,y\ra = \la x,h(\bar z)y\ra$ for $x,y\in L\otimes\R$ and $z\in\C$ and
$-\sqrt{-1}\la \cdot,h(\sqrt{-1})\cdot\ra$ is positive definite and symmetric.
\end{itemize}

For the purposes of subsequently defining $p$-adic modular forms for unitary groups
we assume that the PEL data considered also satisfy:
\begin{itemize}
\item $B$ has no type $D$ factor;
\item $\pair: (L\otimes\Zp)\times (L\otimes\Zp) \rightarrow\Zp(1)$ is a perfect pairing;
\item $p\nmid \Disc(\O_B)$, where $\Disc(\O_B)$ is the discriminant of $\O_B$ over $\Z$
defined in \cite[Def.~1.1.1.6]{Lan}; this condition implies that $\O_B\otimes\Z_{(p)}$ is a maximal $\Z_{(p)}$-order in $B$ and that $\O_B\otimes\Zp$ is a product of matrix algebras.
\item The technical \cite[Condition 1.4.3.10]{Lan} is satisfied. 
\end{itemize}
We associate a group scheme $G=G_P$ over $\Z$ with such a PEL datum $P$:
for any $\Z$-algebra $R$
$$
G(R)=\{(g,\nu)\in \GL_{\O_B\otimes R}(L\otimes R)\times R^\times\ : \la gx,gy\ra = \nu\la x,y\ra \ \forall x,y\in L\otimes R\}.
$$
Then $G_{/\Q}$ is a reductive group, and by our hypotheses with respect to $p$, $G_{/\Zp}$ is smooth and $G(\Zp)$ is a hyperspecial maximal
compact of $G(\Qp)$.

Let $F\subset\C$ be the reflex field of $(L,\pair,h)$ (or of $P$) as defined in \cite[1.2.5.4]{Lan} and let
$\CO_{F}$ be its ring of integers.
Let $\sq=\{p\}$ or $\nullset$, and let $\Z_{(\sq)}$ be the localization of $\Z$ at the primes
in $\sq$. Let $S_\sq=\O_{F}\otimes\Z_{(\sq)}$. Let $K^\sq\subset G(\A_f^\sq)$
be an open compact subgroup
and let $K\subset G(\A_f)$ be $K^\sq$ if $\sq = \nullset$ and $G(\Zp)K^\sq$ otherwise.
Suppose that $K$ is neat, as defined in \cite[Def.~1.4.1.8]{Lan}. Then, as explained in \cite[Cor.~7.2.3.10]{Lan},
there is a smooth, quasi-projective $S_\sq$-scheme $\M_K=\M_K(P)$ that represents the functor on locally noetherian $S_\sq$-schemes
that assigns to such a scheme $T$ the set of equivalence classes of quadruples $(A,\lambda,\iota,\alpha)$
where
\begin{itemize}
 \item $A$ is an abelian scheme over $T$;
\item $\lambda:A\rightarrow A^\vee$ is a prime-to-$\sq$ polarization;
\item $\iota:\O_B\otimes\Z_{(\sq)}\rightarrow \End_{T}A\otimes\Z_{(\sq)}$ such that
$\iota(b)^\vee\circ\lambda = \lambda\circ\iota(b^*)$;
\item $\alpha$ is a $K^\sq$-level structure: this assigns to a geometric point $t$ on each connected component of
$T$ a $\pi_1(T,t)$-stable $K^\sq$-orbit of $\O_B\otimes\A_f^\sq$-isomorphisms
$$
\alpha_t:L\otimes\A_f^\sq\isoarrow H_1(A_t,\A_f^\sq)
$$ 
that identify
$\pair$ with a $\A_f^{\sq,\times}$-multiple of the symplectic pairing on the Tate module $H_1(A_t,\A_f^\sq)$
defined by $\lambda$ and the Weil-pairing;
\item $\Lie_TA$ satisfies the Kottwitz determinant condition defined by $(L\otimes\R,\pair,h)$
(see \cite[Def.~1.3.4.1]{Lan});
\end{itemize}
and two quadruples $(A,\lambda,\iota,\alpha)$ 
and $(A',\lambda',\iota',\alpha')$ are equivalent if there
exists a prime-to-$\sq$ isogeny $f:A\rightarrow A'$ such that $\lambda$ equals $f^\vee\circ\lambda'\circ f$
up to some positive element in $\Z_{(\sq)}^\times$, $\iota'(b)\circ f = f\circ\iota(b)$ for all $b\in \O_B$,
and $\alpha'=f\circ\alpha$.

\subsection{PEL moduli problems related to unitary groups}\label{PELunitary}
Suppose
\begin{equation*}
\label{unitaryPELdata}
P=(B,*,\O_B,L,\pair,h)
\end{equation*}
is a PEL datum as in Section  \ref{PELdata} with
\begin{itemize}
 \item $B=\K^m$, the product of $m$ copies of $\K$; 
 \item $*$ is the involution acting as $c$ on each factor of $\K$;
\item $\O_B\cap \K =\O$ where $\K$ maps to $B = \K^m$ diagonally.
\end{itemize}
We say such a $P$ is {\it of unitary type.}
By maximality, $\O_B\otimes\Z_{(p)} = \O_1\times \cdots\times \O_m = \O_{(p)}\times\cdots\times\O_{(p)}$
(each $\O_i$ is a maximal $\Z_{(p)}$-order in $\K$),
so $\O_B\otimes\Zp \cong  \prod_{w|p} \prod_{i=1}^m \O_w$.  
Let $e_i\in \O_B\otimes\Z_{(p)}$ be
the idempotent projecting $B$ to the $i$th copy of $\K$. Let 
$n_i = \dim_\K e_i(L\otimes\Q)$.

Over $\Zp$ there is a canonical isomorphism
\begin{equation}\label{GL-iso}
\GL_{\O_{B}\otimes\Zp}(L\otimes\Zp)\isoarrow \prod_{w|p}\prod_{i=1}^m\GL_{\O_{w}}(e_iL_{w}),
\ \ g\mapsto (g_{w,i}),
\end{equation}
induced by the $\O_{B}\otimes\Z_p=\prod_{w|p}\O_{B,w}$-decomposition
$L\otimes\Zp = \prod_{w|p}L_{w}$.
This in turn induces 
\begin{equation}\label{G-iso}
G_{/\Zp} \isoarrow \G_m\times
\prod_{w\in\Sigma_p}\prod_{i=1}^m \GL_{\O_{w}}(e_iL_w), \ \ (g,\nu)\mapsto (\nu,(g_{w,i}))
\end{equation}
(Here and elsewhere we use the convention of Remark \ref{grouppoints}.)

The homomorphism $h$ determines a pure Hodge structure of weight $-1$ on $V=L\otimes\C$. Let
$V^0\subset V$ be the degree $0$ piece of the Hodge filtration; this is an $\O_B\otimes\C$-submodule.
For each $\sigma\in\Sigma_\K$, let $a_{\sigma,i} =\dim_\C e_i(V^0\otimes_{\O\otimes\C,\sigma}\C)$.
Let $b_{\sigma,i} = n_i - a_{\sigma,i}$. We call
the collection of pairs $\{(a_{\sigma,i},b_{\sigma,i})_{\sigma\in\Sigma_\K}\}$, the
signature of $h$. Note that $(a_{\sigma c,i},b_{\sigma c,i})= (b_{\sigma,i},a_{\sigma,i})$.
The following fundamental hypothesis will be assumed throughout:
\begin{hyp}[Ordinary hypothesis]\label{ord-hyp}
$$\grp_\sigma=\grp_{\sigma'} \implies a_{\sigma,i} = a_{\sigma',i}.$$
\end{hyp}
For $w|p$ a place of $\K$, we can then define $(a_{w,i},b_{w,i})=(a_{\sigma,i},b_{\sigma,i})$ for any
$\sigma\in \Sigma_\K$ such that $\grp_w=\grp_\sigma$.
Let $\O_{B,w}=\O_B\otimes_\O\O_w$ and $L_w = L\otimes_\O\O_w$.
We fix an $\O_B\otimes\Zp$-decomposition $L\otimes\Zp = L^+\oplus L^-$ such that
\begin{itemize}
\item $L^+ = \prod_{w|p} L_w^+$ is an
$\O_B\otimes\Zp=\prod_{w|p}\O_{B,w}$-module with
$\rank_{\O_w} (e_iL_w^+) =a_{w,i}$ (so $L^- = \prod_{w|p}L_w^-$ with $\rank_{\O_w}(e_iL_w^-) = b_{w,i}$ and $L_w = L_w^+\oplus L_w^-$);
\item $L_w^\pm$ is the annihilator of $L_{\bar w}^\pm$ for the perfect pairing
$\pair:L_w\times L_{\bar w}\rightarrow\Zp(1)$.
\end{itemize}

We fix a decomposition of
$e_iL_{w}^+$ as a direct sum of copies of $\O_w$. Taking $\Zp$-duals via $\pair$
yields a decomposition of $e_iL_{\bar w}^-$ as a direct sum of copies of $\O_{\bar w} \cong
\Hom_\Zp(\O_{w},\Zp)$ (the $\O_B$-action 
on $\Hom_\Zp(E_{i,w},\Zp)$ factors through $e_i\O_B\otimes\Z_{(p)}$ 
and is given by $b\phi(x) = \phi(b^* x)$).
The choice of these decompositions determines
isomorphisms
\begin{equation}\label{G-iso2}
\begin{split}
\GL_{\O_{i,w}}(e_iL_{w}^+) & \cong \GL_{a_{w,i}}(\O_w), \ \
\GL_{\O_{i,w}}(e_iL_{w}^-) \cong \GL_{b_{w,i}}(\O_w), \\ 
& \text{and}  \ \ \GL_{\O_{i,w}}(e_iL_{w}) \cong \GL_{n_i}(\O_w).
\end{split}
\end{equation}
With respect to these isomorphisms, the embedding
$$
\GL_{\O_{i,w}}(e_iL_{w}^+)\times\GL_{\O_{i,w}}(e_iL_{w}^-) \hookrightarrow
\GL_{\O_{i,w}}(e_iL_{w}) = \GL_{\O_{i,w}}(e_iL_w^+\oplus e_iL_w^-)
$$
is just the block diagonal map $(A,B) \mapsto \left(\smallmatrix A & 0 \\ 0 & B\endsmallmatrix\right)$.

\subsection{Connections with unitary groups and their Shimura varieties}\label{unitarygroupPEL}
We recall how PEL data of unitary type naturally arise from unitary groups.
Let $\CV=(V_i,\pair_{V_i})_{1\leq i\leq m}$ be a collection of hermitian pairs over $\K$: $V_i$ is a finite-dimensional $\K$-space and
$\pair_{V_i}:V_i\times V_i\rightarrow \K$ is a hermitian form
relative to $\K/\K^+$. Let $\delta\in\O$ be totally imaginary and prime to $p$, and
put $\pair_i=\trace_{\K/\Q}\delta\pair_{V_i}$. Let $L_i\subset V_i$ be an $\O$-lattice such that
$\la L_i,L_i\ra_i\subset \Z$ and $\pair_i$ is a perfect pairing on $L_i\otimes\Zp$.
Such an $L_i$ exists because of our hypotheses on $p$ and its prime divisors in $\K$ and on $\delta$.
For each $\sigma\in\Sigma_\K$, $V_{i,\sigma}=V_i\otimes_{\K,\sigma}\C$ has a $\C$-basis with respect to which
$\pair_{i,\sigma}= \pair_{V_{i,\sigma}}$ is given by a matrix of the form
$\diag(1_{r_{i,\sigma}}, -1_{s_{i,\sigma}})$. Fixing such a basis, let
$h_{i,\sigma}:\C\rightarrow \End_{\R}(V_{i,\sigma})$ be $h_{i,\sigma}(z) =
\diag(z1_{r_{i,\sigma}},\bar z 1_{s_{i,\sigma}})$.
Let $\Sigma=\{\sigma\in\Sigma_\K\ : \ \grp_\sigma\in \Sigma_p\}$. Then
$\Sigma$ is a CM type of $\K$, and we let
$h_i = \prod_{\sigma\in\Sigma} h_{i,\sigma}: \C \rightarrow \End_{\K^+\otimes\R}(V_i\otimes\R) =
\prod_{\sigma\in\Sigma} \End_{\R}(V_{i,\sigma})$.
Let $B=\K^m$, $*$ the involution that acts by $c$ on each $\K$-factor of $B$,
$\O_B=\O^m$, $L=\prod_i L_i$ with the $i$th factor of $\O_B=\O^m$ acting by scalar multiplication on the $i$th factor of $L$,
$\pair = \sum_i\pair_i$,
and $h = \prod_i h_i$. Then $P=(B,*,\O_B,L,2\pi\sqrt{-1}\pair,h)$ is a PEL datum of unitary type as defined above.
Note that $(a_{\sigma,i},b_{\sigma,i})$ equals $(r_{i,\sigma},s_{i,\sigma})$ if
$\sigma\in \Sigma$ and otherwise equals $(s_{i,\sigma},r_{i,\sigma})$.
The reflex field of this PEL datum $P$ is just the field
$$
F = \Q[\{\sum_{\sigma\in\Sigma_\K} a_{\sigma,i} \sigma(a) \ : \ a\in\K, i=1,..,m\}]\subset \C.
$$
This follows, for example, from \cite[Cor.~1.2.5.6]{Lan}. Note that $F$ is contained in the Galois
closure $\K'$ of $\K$ in $\C$.

As explained in \cite[\S 8]{Kottwitz} (see also Equations \eqref{cpx-1} below), over the reflex field $F$, a moduli
space ${\M_{K}}_{/F}$ associated with $P$ is the union of $|\ker^1(\Q,G)|$ copies of the canonical model of the Shimura variety $S_K(G,X_P)$
associated to $(G,h_P,K)$; here $(G,X_P)$ is the Shimura datum for which $h_P=h \in X_P$ and 
$\ker^1(\Q,G) := \ker \left(H^1(\Q,G)\rightarrow \prod_{v} H^1(\Q_v,G)\right)$.  
More precisely, the elements of $\ker^1(\Q,G)$ classify isomorphism classes of hermitian tuples $\CV'=(V_j',\pair_{V_j'})_{1\leq j\leq m}$ 
that are locally isomorphic to $\CV$ at every place of $\Q$. Let $\CV = \CV^{(1)}, ...., \CV^{(k)}$ be representatives for these isomorphism classes.
Then ${\M_K}_{/F}$ is naturally a disjoint union of $F$-schemes indexed by the $\CV^{(j)}$: ${\M_K}_{/F} = \sqcup  \M_{K,\CV^{(j)}}$. The scheme
$\M_{K,\CV}=\M_{K,\CV^{(1)}}$ is the canonical model of $S_K(G,X_P)$, and for each $j$ there is  an $F$-automorphism of ${\M_K}_{/F}$ mapping
$\M_{K,\CV}$ isomorphically onto $\M_{K,\CV^{(j)}}$. 
In \cite{Kottwitz}, Kottwitz only treats the case where $m = 1$, but the reasoning is the same in the general case.

If $m = 1$ and $\dim_\K V_1$ is even,
then the group $G$ satisfies the Hasse principle (that is, $\ker^1(\Q,G)=0$).
In this case $\M_{K}$ is an integral model of the Shimura variety $S_K(G,X_P)$.  If $\dim_\K V_1$ is odd or $m\geq 1$, this is no longer
the case.  However, for applications to automorphic forms, we only need one copy of $S_K(G,X_P)$. We let 
$\M_{K,L}$ be the scheme theoretic closure of the $F$-scheme $\M_{K,\CV}$ in $\M_K$; this is a smooth, quasi-projective $S_\sq$-scheme. We let
\begin{equation}\label{hasseprin} 
s_L:  \M_{K,L} \hookrightarrow \M_{K} 
\end{equation} 
be the inclusion.  We will refer to $\M_{K}$ as the {\it moduli space} and $\M_{K,L}$ as the {\it Shimura variety}.

\begin{rmk} 
For any PEL datum $P$, Lan has explained how the canonical model of the Shimura variety $S_K(G,X_P)$ is
realized as an open and closed subscheme of ${\M_K}_{/F}$ \cite[\S 2]{lanalgan},
with a smooth, quasi-projective
$S_0$-model provided by its scheme-theoretic closure in $\M_K$. This is just the model
described above.
\end{rmk}

\subsubsection{Base points}\label{basepoints}

Suppose $m=1$. Let $(V,\pair_V)=(V_1,\pair_{V_1})$, and let $n=\dim_\K V$.
Suppose $\K_1,...,\K_r$ are finite CM extensions 
of $\K$
with $\sum_{i = 1}^r [\K_i:\K] = n$.  For $i=1,...,r$, let $J_{0,i}$ be the Serre subtorus (defined in, e.g.,  \cite[Definition A.4.3.1]{ChCoOo})
of $\Res_{\K_i/\Q}\G_m$ and let $\nu_i:  J_{0,i} \rar \G_m$ be its similitude map.
Let $J'_0 \subset \prod_{i=1}^r J_{0,i}$ be the subtorus defined by equality of all the $\nu_i$.   
Let $V'_i=\K_i$, viewed as a $\K$-space of dimension $[\K_i:\K]$.  
Each $V'_i$ can be given a $\K_i$-hermitian structure such that
$\oplus_i V'_i$ is isomorphic to $V$ as an hermitian space over $\K$. Such an isomorphism determines an embedding of $J'_0$ in $G$.  Moreover, with respect
to such an embedding, there exists a point $h_0 \in X_P$ that factors through the image of $J'_0(\R)$ in $G(\R)$.   The corresponding
embedding of Shimura data $(J'_0,h_0) \rar (G,X)$ defines a CM Shimura subvariety of $\M_{K,L}$.

For the case $\K_i =\K$ for all $i$ (so $r=n$),
we write $J_{0}^{(n)}$ for $J'_0$; this corresponds to a PEL datum as in Section  \ref{PELdata} with
$B = \K^n$.  The base point $h \in X_P$ is called {\it standard} if it factors through an inclusion of $J_{0}^{(n)}$.  We henceforward assume that the 
base point $h$ in the PEL datum $P$ is standard.
This will guarantee that later constructions involving Harish-Chandra modules are rational over the Galois closure of $\K$.

Concretely, the assumption that $h$ is standard just means that $V$ has a $\K$-basis with respect to which $\pair_V$ is diagonalized and that each $h_\sigma$
has image in the diagonal matrices with respect to the induced basis of $V\otimes_{\K,\sigma}\C$.

\subsection{Toroidal compactifications}
One of the main results of \cite{Lan} is the existence of smooth toroidal
compactifications of $\M_K$ over $S_\sq$ associated to certain smooth projective polyhedral cone
decompositions (which we do not make precise here); when $\sq=\nullset$ this was already known.   
We denote such a compactification by $\M_{K,\Sigma}^\tor$.  See \cite[Section 6.4]{Lan} for the main statements used below.
There is a notion of one polyhedral cone decomposition refining another
that partially orders the $\Sigma$'s. If $\Sigma'$ refines $\Sigma$, then there is a canonical proper surjective map $\pi_{\Sigma',\Sigma}:\M_{K,\Sigma'}^\tor\rightarrow\M_{K,\Sigma}^\tor$ that is the identity
on $\M_K$. We write
$\M_K^\tor$ for the tower of compactications $\{\M_{K,\Sigma}^\tor\}_\Sigma$.
In certain situations (e.g., changing the group $K$, defining Hecke operators) it is
 more natural work to work with this tower, avoiding making specific compatible choices of $\Sigma$ or  having to
vary the `fixed' choices.

If $K_1^\sq\subset K_2^\sq$ then the natural map
$\M_{K_1}\rightarrow\M_{K_2}$ extends canonically to a map (of towers) $\M_{K_1}^\tor\rightarrow \M_{K_2}^\tor$. Similarly, if $g\in G(\A^\sq_f)$, then the map
$[g]:\M_{gKg^{-1}}\rightarrow \M_K$, $(A,\lambda,\iota,\alpha)\mapsto (A,\lambda,\iota, \alpha g)$, extends canonically to a map $\M_{gKg^{-1}}^\tor\rightarrow \M_K^\tor$.
This defines a right action of $G(\A_f^\sq)$ on the tower (of towers!) $\{\M_K^\tor\}_{K^\sq\subset G(\A_f^\sq)}$.

In the setting of Section  \ref{unitarygroupPEL}, we 
let $\M_{K,L,\Sigma}^\tor$ be the scheme-theoretic closure of $\M_{K,\CV}$ in $\M_{K,\Sigma}^\tor$. This is a smooth toroidal compactification of the Shimura variety $\M_{K,L}$, as discussed in, e.g., \cite[\S4.1]{Lan-Suh-vanishing} and \cite[\S3-\S4]{lanalgan};
the base change to $F$ is just the usual toroidal compactification of the canonical model. We continue to denote by $s_L$ the induced inclusion $\M_{K,L,\Sigma}^\tor\subset
\M_{K,\Sigma}^\tor$. 
Varying $\Sigma$ and $K$ as above induces maps between the $\M_{K,L,\Sigma}^\tor$. We let $\M_{K,L}^\tor$ be the tower $\{\M_{K,L,\Sigma}^\tor\}_\Sigma$. The action
of $G(\A_f^\sq)$ on $\{\M_K^\tor\}_{K^\sq\subset G(\A_f^\sq)}$ induces an action on $\{\M_{K,L}^\tor\}_{K^\sq\subset G(\A_f^\sq)}$.

Our convention will be to describe constructions over $\M_K^\tor$ as though $\M_K^\tor$ 
were a single scheme. The reader should bear in mind that this means a tower of such constructions over each $\M_{K,\Sigma}^\tor$.
In particular, when we define a sheaf $\CF$ over $\M_K^\tor$ (or some similar tower of schemes), this will be a sheaf $\CF_\Sigma$ on each $\M_{K,\Sigma}^\tor$ such that there is a natural map $\pi_{\Sigma',\Sigma}^*: \pi_{\Sigma',\Sigma}^{-1}\CF_\Sigma\rightarrow \CF_{\Sigma'}$ for any $\Sigma'$ that
refines $\Sigma$. By $H^i(\M_K^\tor,\CF)$ we will mean the direct limit
$\varinjlim_\Sigma H^i(\M_{K,\Sigma}^\tor,\CF_\Sigma)$. In practice, the maps of cohomology groups appearing in such a limit will all be isomorphisms.

\subsection{Level structures at $p$}\label{levelp}

Let $H=\GL_{\O_B\otimes\Zp}(L^+)$.
The identification 
\eqref{G-iso2} determines an isomorphism
\begin{equation}\label{H-iso}
H \isoarrow \prod_{w|p}\prod_{i=1}^m \GL_{a_{w,i}}(\O_w).
\end{equation}
Let $B_H\subset H$ be the $\Zp$-Borel that corresponds via this isomorphism with the product of the upper-triangular Borels and
let $B_H^u$ be its unipotent radical. Let $T_H=B_H/B_H^u$; this is identified by isomorphism \eqref{H-iso} with the diagonal matrices.

Suppose $\sq=\{p\}$.
There exists a semiabelian
scheme $\CA$ over $\M^\tor_{K}$ 
that is part of a degenerating family as in \cite[Thm.~6.4.1.1]{Lan}. In particular,
there exists a dual semiabelian scheme $\CA^\vee$ 
(in the sense of \cite[Thm.~3.4.3.2]{Lan}) together with
a homomorphism $\lambda:\CA\rightarrow \CA^\vee$,
a homomorphism $\iota:\O_B\otimes_{\Z_{(p)}}\rightarrow \End_{\M^\tor_{K}}(\CA)$,
and a $K^{(p)}$-level structure on $\CA_{/\M_{K}}$ such that the restriction
of $(\CA,\lambda,\iota,\alpha)$ over $\M_{K}$ represents the 
universal tuple (that is, the tautological tuple in the sense of \cite[Thm.~6.4.1.1(1)]{Lan}).

We define $\oM_{K_r}$ to be the scheme over $\M^\tor_{K}$ whose $S$-points
classify the $B_H^u(\Zp)$-orbits of $\O_B\otimes\Zp$-injections
$\phi:L^+\otimes\bmu_{p^r}\hookrightarrow \CA^\vee[p^r]_{/S}$ of group schemes with
image an isotropic subgroup scheme.
We write $\M_{K_r}$ for its restriction over $\M_{K}$. The group $B_H(\Zp)$ acts on $\oM_{K_r}$ on the right
through its quotient $T_H(\Zp/p^r\Zp)$. We let $\oM_{K_r,L}$ be the pullback of $\oM_{K_r}$ over $\M_{K,L}^\tor$
and let $\M_{K_r,L}$ be the pullback over $\M_{K,L}$. Generally, the scheme $\oM_{K_r}$ (resp.~$\oM_{K_r,L}$)
is \'etale and quasi-finite but not finite over $\M_K^\tor$ (resp.~$\M_{K_r,L}^\tor$). 
We continue to denote by $s_L$ the inclusions $\M_{K_r,L}\hookrightarrow \M_{K_r}$ and 
$\oM_{K_r,L}\hookrightarrow \oM_{K_r}$ determined by these restrictions.

Let $B^+\subset G_{/\Zp}$ be the Borel that stabilizes $L^+$ and such that
\begin{equation}\label{B+TH}
B^+\twoheadrightarrow \G_m\times B_H\subset \G_m\times H,
\end{equation}
where the map to the first factor is the similitude character $\nu$ and the map to the
second is projection to $H$. Let $B^u\subset B^+$ be the unipotent radical.
Let $I_r^0\subset G(\Zp)$ consist of those $g$ such that $g\mod p^r \in B^+(\Zp/p^r\Zp)$, and
let $I_r\subset I_r^0$ consist of those $g$ projecting under the surjection \eqref{B+TH} to an element in $(\Zp/p^r\Zp)^\times\times B_H^u(\Zp/p^r\Zp)$. Then $I_r^0/I_r\isoarrow T_H(\Zp/p^r\Zp)$.
The choice of a basis of $\Zp(1)$ naturally identifies ${\M_{K_r}}_{/F}$ (resp.~$\M_{K_r,L}$)
with ${\M_{I_rK^p}}_{/F}$ (resp.~${\M_{I_rK^p,L}}_{/F} = S_{I_rK^p}(G,X_P)$), and ${\oM_{K_r}}_{/F}$ (resp.~${\oM_{K_r,L}}_{/F}$) is
the normalization of ${\M^\tor_{K}}_{/F}$ (resp.~${\M^\tor_{K,L}}_{/F}$) in
${\M_{K_r}}_{/F}$ (resp.~$\M_{K_r,L}/F$). Since it should therefore cause no ambiguity, we also put $K_r = I_rK^p$.
We similarly put $K_r^0 = I_r^0K^p$.

Note that under the isomorphisms \eqref{G-iso} and \eqref{G-iso2}, $B^+$ is identified with the group
\begin{equation}\label{B+-iso}
B^+\isoarrow \G_m \times \prod_{w\in\Sigma_p}\prod_{i=1}^m \left\{ \left(\smallmatrix A & B \\ 0 & D \endsmallmatrix\right) \in \GL_{n_i}(\O_w) \ : \ 
\begin{smallmatrix} \text{$A\in\GL_{a_{w,i}}(\O_w)$ is upper-triangular} \\ \text{$D\in\GL_{b_{w,i}}(\O_w)$ is lower-triangular} \end{smallmatrix}\right\}.
\end{equation}

\subsection{Modular forms} We define spaces of modular forms for the groups $G$ and
various Hecke operators acting on them.

\subsubsection{The groups $G_{0}$ and $H_{0}$}\label{G0H0}
Let $V = L\otimes\C$. The homomorphism $h$ defines a pure Hodge structure $V = V^{-1,0}\oplus V^{0,-1}$
of weight $-1$. Let $W = V/V^{0,-1}$. This is defined over the reflex
field $F$. Let $\Lambda_{0}\subset W$ be an $\O_{B}$-stable $S_\sq$-submodule
such that $\Lambda_{0}\otimes_{S_\sq}\C = W$.
Let $\Lambda_{0}^\vee=\Hom_{\Z_{(p)}}(\Lambda_{0},\Z_{(p)}(1))$
with $\O_{B}\otimes S_\sq$-action: $(b\otimes s)f(x) = f(b^{*}sx)$.
Put $\Lambda= \Lambda_{0}\oplus\Lambda_{0}^\vee$, and let $\pair_{can}:\Lambda\times\Lambda
\rightarrow\Z_{(p)}(1)$ be the alternating pairing 
\begin{align*}
\la (x_1,f_1),(x_2,f_2)\ra_{can} = f_2(x_1)-f_1(x_2).
\end{align*}
Note that $\Lambda_{0}$ and $\Lambda_{0}^\vee$ are isotropic submodules
of $\Lambda$.
Note also that the $\O_{B}$-action on $\Lambda$ is such that
$\la bx,y\ra_{can} = \la x, b^{*}y\ra_{can}$. Let $G_{0}$ be the group scheme
over $S_\sq$ such that for any $S_\sq$-algebra $R$
\begin{equation*}
G_{0}(R) = \left\{ (g,\nu)\in \GL_{\O_{B}\otimes R}(\Lambda\otimes_{S_\sq}R)\times R^\times \ : \
\begin{matrix} \la gx,gy\ra_{can} = \nu\la x,y\ra_{can}, \\
 \forall x,y\in \Lambda\otimes_{S_\sq}R \end{matrix} \right\}.
\end{equation*}
Let $H_{0}\subset G_{0}$ be the stabilizer of the polarization
$\Lambda =\Lambda_{0}\oplus \Lambda_{0}^\vee$.
The projection $H_{0}\rightarrow \G_m\times\GL_{\O_{B}\otimes S_\sq}(\Lambda_{0}^\vee)$
is an isomorphism (the projection to $\G_m$ is the similitude factor $\nu)$.
There is a canonical isomorphism
$V\cong \Lambda\otimes_{S_\sq}\C$ of $\O_{B}\otimes\C$-modules that
identifies $V^{-1,0}$ with $\Lambda_{0}\otimes_{S_\sq}\C$ and $V^{0,-1}$ with $\Lambda_{0}^\vee\otimes_{S_\sq}\C$ and
the pairing $\pair$ with $\pair_{can}$, and so identifies ${G}_{/\C}$ with ${G_{0}}_{/\C}$.
Let $C\subset {G}_{/\R}$ be the
centralizer of the homomorphism $h$ and set $U_\infty = U_h := C(\R)$.
The identification of ${G}_{/\C}$ with ${G_{0}}_{/\C}$ identifies $C(\C)$ with $H_{0}(\C)$.

\subsubsection{The canonical bundles}\label{canonicalbun}
Let $\CA$ be the semiabelian scheme over $\M_{K}^\tor$ as in Section  \ref{levelp} and let $\CA^\vee$ be the associated dual semiabelian scheme.
Let $\omega$ be the $\O_{\M_{K}^\tor}$-dual of $\Lie_{\M_{K}^\tor}\CA^\vee$.
The Kottwitz determinant condition implies  that $\omega$ is locally isomorphic
to $\Lambda_{0}^\vee\otimes_{S_\sq}\O_{\M_{K}^\tor}$ as an $\O_{B}\otimes\O_{\M_{K}^\tor}$-module.
Let
$$
\CE = \Isom_{\O_{B}\otimes \O_{\M_{K}^\tor}}
((\omega,\O_{\M_{K}^\tor}(1)),(\Lambda_{0}^\vee\otimes_{S_\sq}\O_{\M_{K}^\tor},\O_{\M_{K}^\tor}(1))).
$$
This is an $H_{0}$-torsor over $\M_{K}^\tor$.
Let $\pi:\CE\rightarrow \M_{K}^\tor$
be the structure map. 
Let $R$ be an $S_\sq$-algebra. An $R$-valued point $f \in \CE$ 
can be viewed as a functorial rule assigning to a pair $(\underline{A},\veps)$ over
an $R$-algebra $S$ an element $f(\underline{A},\veps)\in S$. Here $\underline{A}$
is a tuple classified by $\M_{K}(S)$
and $\veps$ is a corresponding
element of $\CE(S)$.
We let $\CE_{r} = \CE\times_{\M_{K}^\tor}\oM_{K_r}$ and
let $\pi_{r}:\CE_{r}\rightarrow \oM_{K_r}$ be its structure map. Sections of the bundle
$\pi_{r,*}\O_{\CE_r}$ have interpretations as functorial rules of pairs $(\underline{X},\veps)$, where
$\underline{X}=(\underline{A},\phi)$ is a tuple classified by $\M_{K_r}(S)$ and
$\veps$ is a corresponding element in $\CE_r(S)$.

\subsubsection{Representations of $H_{0}$ over $S_0$}\label{H0-OF-reps}
Recall that $\K'$ is the Galois closure of $\K$,
and let $\grp'\subset\O_{\K'}$ be the prime determined by $incl_p$. Let
$$
S_0 = S_\sq\otimes_{\O_{F,(p)}}\O_{\K',(\grp')}.
$$
(So $S_0 = \K'$ if $\sq=\nullset$ and $S_0 = \O_{\K',(\grp')}$ if $\sq=\{p\}$.)
The isomorphism 
$\O\otimes S_0  \isoarrow \prod_{\sigma\in \Sigma_\K} S_0$, $a\otimes s \mapsto (\sigma(a)s)_{\sigma_\in\Sigma_\K},$
induces a decomposition
$$
\O_B \otimes S_0'\isoarrow \O_B\otimes_\O(\O\otimes S_0) \isoarrow \prod_{\sigma\in\Sigma_\K} \O_B\otimes_{\O,\sigma} S_0
= \prod_{\sigma\in\Sigma_\K} \O_{B,\sigma}.
$$
This in turn induces $\O_B\otimes S_0 = \prod_{\sigma\in\Sigma_\K} \O_{B,\sigma}$-decompositions
$\Lambda_0\otimes_{S_0}S_0 = \prod_{\sigma\in\Sigma_\K} \Lambda_{0,\sigma}$ and 
$\Lambda_0^\vee\otimes_{S_0}S_0 = \prod_{\sigma\in\Sigma_\K}\Lambda_{0,\sigma}^\vee$.
The pairing $\pair_{can}$ identifies $\Lambda_{0,\sigma c}^\vee = \Hom_{\Z_{(p)}}(\Lambda_{0,\sigma},\Z_{(p)}(1))$.

Since $S_0$ is a PID, $e_i\Lambda_{0,\sigma}$ and $e_i\Lambda_{0,\sigma}^\vee$ are free $S_0$-modules, of respective 
ranks $a_{\sigma,i}$ and $b_{\sigma,i}$. We fix an $S_0$-basis of $e_i\Lambda_{0,\sigma}$.
By duality, this determines an $S_0$-basis of $e_i\Lambda_{0,\sigma c}^\vee$. This yields an isomorphism
\begin{equation}\label{H0-iso}
{H_0}_{/S_0} \isoarrow \G_m\times \prod_{\sigma\in\Sigma_\K}\prod_{i=1}^m \GL_{\O_i\otimes_{\O,\sigma} S_0'}(e_i\Lambda_{0,\sigma}^\vee)
\cong \G_m\times \prod_{\sigma\in\Sigma_\K}\prod_{i=1}^m \GL_{b_{\sigma,i}}(S_0).
\end{equation}

Let $B_{H_0}\subset {H_{0}}_{/S_0}$ be the $S_0$-Borel that corresponds via the isomorphism \eqref{H0-iso} to
the product of the lower-triangular Borels. Let $T_{H_0}\subset B_{H_0}$ be the diagonal
torus and let $B_{H_0}^u\subset B_{H_0}$ be the unipotent radical. We say that a 
character $\kap$ of $T_{H_0}$ that is defined over an $S_0$-algebra $R$ is a
{\it dominant character of $T_{H_0}$} if it is dominant with respect  
to the opposite (so upper-triangular) Borel $B_{H_0}^{\op}$. Via the isomorphism \eqref{H0-iso}, the characters of $T_{H_0}$
can be identified with the tuples $\kap = (\kap_0,(\kap_{\sigma,i})_{\sigma\in\Sigma_\K,1\leq i\leq m})$, 
$\kap_0\in \Z$ and
$\kap_{\sigma,i} =(\kap_{\sigma,i,j})\in \Z^{b_{\sigma,i}}$, and the dominant characters are those that satisfy
\begin{equation}\label{dom-wt-ineq}
\kap_{\sigma,i,1}\geq \cdots \geq \kap_{\sigma,i,b_{\sigma,i}}, \ \ \forall \sigma\in\Sigma_\K, \ \ i=1,...,m.
\end{equation}
The identification is just
\begin{equation*}\begin{split}
\kap(t) & = t_0^{\kap_0} \cdot
\prod_{\sigma\in\K} \prod_{i=1}^m \prod_{j=1}^{b_{\sigma,i}} t_{\sigma,i,j}^{\kap_{\sigma,i,j}}, \\
 t =(t_0,(\diag&(t_{\sigma,i,1},...,t_{\sigma,i,b_{\sigma,i}}))_{\sigma\in\Sigma_\K,1\leq i\leq m}) \in T_{H_0}.
\end{split}
\end{equation*}

Given a dominant character $\kap$ of $T_{H_0}$ over an $S_0$-algebra $R$, let
$$
W_\kap(R)=\{\phi: {H_{0}}_{/R}\rightarrow\G_a \ : \ \phi(bh) = \kap(b)\phi(h), \ b\in B_{H_0}\},
$$
where $\kap$ is extended trivally to $B_{H_0}^u$.
If $R$ is a flat $S_0$-algebra then this is an $R$-model of the
irreducible algebraic representation of $H_{0}$ of highest weight $\kap$ with respect to $(T_{H_0},B_{H_0}^{\op})$. 
Let $w\in W(T_{H_0},{H_0}_{/S_0})$ be the longest element in the
Weyl group and let $\kap^\vee$ be the dominant character of $T_{H_0}$ defined by
$\kap^\vee(t) = \kap(w^{-1}t^{-1}w)$. The dual
$$
W_\kap^\vee(R) = \Hom_R(W_\kap(R),R)
$$
is, for a flat $S_0$-algebra $R$, an $R$-model of the representation with highest weight $\kap^\vee$.

The submodule $W_\kap(R)^{B_{H_0}^u}$ is a free $R$-module of
rank one spanned by $\phi_\kap$, the function with support containing the big cell $B_{H_0}wB_{H_0}$
(and equal to the big cell if $\kap$ is regular)
and such that $\phi_\kap(wB_{H_0}^u)=1$; $w\phi_\kap$ is a highest-weight vector.
The module $W_\kap^\vee$ is generated over $R$ as an $H_0$-representation by the functional $\ell_{\kap} = (\text{evaluation at $1$})$; $w\ell_\kap$ is a highest-weight vector.  Also,
$$
\Hom_{H_0}\left(W_\kap^\vee(R),W_{\kap^\vee}(R)\right) = R$$
with basis the homomorphism that sends   $\ell_\kap$ to $\phi_{\kap^\vee}$.   (For all this, see \cite{Hida}, Section  8.1.2, and the text of Jantzen cited there.)

For future reference, we also note that via the isomorphism \eqref{H0-iso}
the identification of $C(\C)$ with $H_0(\C)$ identifies
\begin{equation}\label{Uinfty-iso}
U_\infty = C(\R) \isoarrow \{ (h_0,(h_\sigma)_{\sigma\in\Sigma_\K})\in H_0(\C) \ : \ h_0\in \R^\times, h_0{}^t\bar h_\sigma^{-1} = h_\sigma \},
\end{equation}
where the `$\bar{\ }$' denotes complex conjugation on $\C$. That is, $U_\infty$ is identified with the subgroup of the product
$\prod_{\sigma\in\Sigma_\K}GU^+(b_\sigma)$ of {\it full} unitary similitude groups (see Section  \ref{notation-section}) in which  all the similitude factors agree.

\subsubsection{The modular sheaves}\label{aut-sheaves}

Let $R$ be a $S_0$-algebra and $\kap$ a dominant $R$-character of $T_{H_0}$.
Let
$$
\omega_{\kap,M} = \pi_{*}\O_{\CE}[\kap] \ \ \text{and} \ \
\omega_{r,\kap,M} = \pi_{r,*}\O_{\CE_{r}}[\kap]
$$
be the subsheaf of the quasicoherent sheaf $\pi_{*}\O_{\CE}$ on ${\M_{K}^\tor}_{/R}$ and ${\oM_{K_r}}_{/R}$, respectively, on which $B_{H_0}$ acts via $\kap$.  (See \cite{Hida}, Sections 8.1.2 and 8.1.3, for this construction.)
We let 
\begin{equation}\label{byzero}
\omega_{\kap} = s_{L,*}s_L^*\omega_{\kap,M} \ \ \text{and} \ \ \omega_{r,\kap} = s_{L,*}s_L^*\omega_{r,\kap,M}.
\end{equation}
These are the respective restrictions to the toroidal compactifications of the Shimura varieties $\M_{K,L_{/R}}$ and $\M_{K_r,L_{/R}}$ of the sheaves $\omega_{\kap,M}$ and $\omega_{r,\kap,M}$ , 
extended by zero to the full moduli space.  We will use the same notation to denote the restriction of these sheaves over $\M_{K,L}$ and $\M_{K_r,L}$.

\subsubsection{Modular forms over $S_0$ of level $K$}\label{levelK}
Let $R$ be a $S_0$-algebra.
The
$R$-module of modular forms (on $G$) over $R$ of weight $\kap$ and level $K$ is
$$
M_\kap(K;R) = H^0({\M_{K,L}^\tor}_{/R},\omega_{\kap}).
$$
The K\"ocher principle \cite{KLkoecher} and the definition \eqref{byzero}  implies that
\begin{equation}\label{Kocher}
M_\kap(K;R) = H^0\left({\M_{K,L}}_{/R},\omega_{\kap}\right)
\end{equation}
except when $F_0=\Q$ and ${G^{\mathrm{der}}}{/\Q}$ has an irreducible factor isomorphic to $SU(1,1)$. However, in this exceptional case
the toroidal compactifications are the same as the minimal compactification and therefore canonical; we leave it to the reader to make the necessary 
adjustments to our arguments in this case (or to find them in the literature).  We will generally refrain from referring explicitly to this exception.

By \eqref{Kocher}
a modular form $f\in M_\kap(K;R)$ can be viewed as a functorial rule assigning to a pair
$(\underline{A},\veps)$, over an $R$-algebra $S$, that is an $S$-valued point of the Shimura variety $\M_{K,L}$, an element $f(\underline{A},\veps)\in S$ and satisfying
$f(\underline{A},b\veps) = \kap(b)f(\underline{A},\veps)$ for $b\in B_{H_{0}}(S)$.

Let $D_\infty$ be the Cartier divisor $\M_{K}^\tor - \M_{K}$, equipped with its structure of reduced closed subscheme.  The $R$-module of cuspforms (on $G$) over $R$ of weight $\kappa$ and level $K$ is the submodule 
$$
S_\kap(K;R) = H^0\left({\M_{K}^\tor}_{/R},\omega_{\kap}\left(-D_{\infty}\right)\right)
$$
of $M_\kap(K;R)$.  It follows from Proposition 7.5 of \cite{lanIMRN} that $S_\kap(K;R)$ is independent of the choice of toroidal compactification.

\subsubsection{Modular forms over $S_0[\psi]$ with Nebentypus $\psi$}

Let $\psi:T_H(\Zp)\rightarrow \bQ^\times$ be a character factoring through
$T_H(\Zp/p^r\Zp)$. Suppose $R$ is an algebra over $S_0[\psi]$, the ring obtained by adjoining the values of $\psi$ to $S_0$ (we use the analogous notation without comment below).  
We define the $R$-module of modular forms (on $G$) over $R$ of weight $\kappa$, level $K_r$,
and character $\psi$ to be
$$
M_\kap\left(K_r,\psi;R\right)= \left\{f\in H^0\left({\oM_{K_r}}_{/R},\omega_{r,\kap}\right) \ : \ t\cdot f = \psi(t)f \
\forall t\in T_H(\Zp)\right\}.
$$
It follows from Remark 10.2 of \cite{KLkoecher} that the  K\"ocher principle
applies\footnote{We ignore the exceptional case when $F_0=\Q$ and $G^{der}/\Q$ contains a factor isomorphic to $SU(1,1)$.  On the other hand, we are using the fact that the $M_{K_r}$ we have defined here is a special case of the ordinary locus of \cite{Lan2}.}
and we have
\begin{equation}\label{Kocher2}
M_\kap\left(K_r,\psi;R\left[\frac{1}{p}\right]\right) = \left\{f\in H^0\left({\M_{K_r}}_{/R\left[\frac{1}{p}\right]},\omega_{r,\kap}\right) \ :
\ t\cdot f = \psi(t)f \ \forall t\in T_H(\Zp)\right\}.
\end{equation}
A section $f\in M_\kap(K_r,\psi;R)$ can be interpreted as a functorial rule
assigning to a pair $(\underline{X},\veps)$, over an $R$-algebra $S$, that is an $S$-valued point of the Shimura variety  
$\M_{K_r,L}$), an element  
$f(\underline{X},\veps)\in S$,
where $\underline{X}=(\underline{A},\phi)$, satisfying
$f(\underline{A},\phi\circ t,b\veps) = \psi(t)\kap(b)f(\underline{X},\veps)$
for all $t\in T_H(\Zp)$ and $b\in B_{H_0}(S)$.

We similarly define the submodule of cuspforms of character $\psi$ to be
$$
S_\kap\left(K_r,\psi;R\right) = \left\{f\in H^0\left({\oM_{K_r}}_{/R},\omega_{r,\kap}(-D_{\infty,r})\right) \ : \ t\cdot f = \psi(t)f \
\forall t\in T_H(\Zp)\right\}.
$$
Here $D_{\infty,r}$ is the Cartier divisor $D_{\infty,r} : = \oM_{K_r} - M_{K_r}$ with the structure of a reduced 
closed subscheme (cf.~\cite[Thm.~5.2.1.1(3)]{Lan2})

\subsubsection{The actions of $G(\A_f^\sq)$ and $G(\A_f^p)$} 
The action of $G(\A_f^\sq)$ on $\{\M_{K}^\tor\}_{K^\sq}$ gives an action 
of $G(\A_f^\sq)$ on 
$$
\varinjlim_{K^\sq} M_\kap(K;R) \ \ \text{and} \ \ \varinjlim_{K^\sq} S_\kap(K;R).
$$
Similarly, the action of $G(\A_f^p)$ extends to an action on
$\{\oM_{K_r}\}_{K^p}$,  giving
an action of $G(\A_f^p)$ on
$$
\varinjlim_{K^p} M_\kap(K_r,\psi;R)  \ \ \text{and} \ \ \varinjlim_{K^p} S_\kap(K_r,\psi;R).
$$
The submodules fixed by $K^\sq$ (resp.~$K^p$) are just the modular forms and cuspforms of weight $\kap$
and level $K$ (resp.~prime-to-$p$ level $K^p$).  (Here we are using the fact that the toroidal compactifications are normal schemes over $Spec(R)$; see \cite[Proposition ~7.5]{lanIMRN}.)

\subsubsection{Hecke operators away from $p$}\label{Heckeprimetop}

Let $K_j=G(\Zp)K_j^p \subset G(\A_f)$, $j=1,2$,
be neat open compact supgroups. For $g\in G(\A_f^p)$ we define Hecke operators
\begin{equation*}
\begin{split}
\left[K_{2} g K_{1}\right]:M_\kap\left(K_{1};R\right) & \rightarrow M_\kap\left(K_{2};R\right),  \\
[K_{2,r} g K_{1,r}]:M_\kap\left(K_{1,r},\psi;R\right) & \rightarrow M_\kap\left(K_{2,r},\psi;R\right)
\end{split}
\end{equation*}
(in the obvious notation) through the action of $G(\A_f^p)$ on the modules of modular forms:
\begin{equation}
[K_{2,r}gK_{1,r}]f = \sum_{g_j} [g_j]^*f, \ \ K_{2}^p g K_{1}^p = \sqcup_{g_j} g_jK_{1}^p.
\end{equation}
In particular,
\begin{equation}\label{Hecke-eq1}
([K_{2,r} g K_{1,r}]f)(A,\lambda,\iota,\alpha K_2^p,\phi,\veps)
=  \sum_{g_j }f(A,\lambda,\iota,\alpha g_jK_1^p,\phi,\veps).
\end{equation}
These actions map cuspforms to cuspforms.

When $K_2 = K_1$ is understood we write $T(g)$ instead of $[K_1 g K_{1}]$ and $T_r(g) = [K_{1,r} g K_{1,r}]$;
we drop the subscript $r$ when that is also understood.

\subsubsection{Hecke operators at $p$}\label{Heckeatp}
If $p$ is invertible in $R$ (so $R$ is a $\Qp$-algebra) we define Hecke operators
$T(g) = [K g K]$ and $T_r(g) = [K_r g K_r]$ on the spaces of modular forms and cuspforms over $R$ just
as we did in \ref{Heckeprimetop}.
We single out some particular operators: for $w\in \Sigma_p$,
$1\leq i\leq m$, $1\leq j\leq n_i$, we let 
$t_{w,i,j}^+\in B^+(\Qp)$ be the element
identified via \eqref{B+-iso} with $(1,(t_{w',i',j}))$ where
$$
t_{w',i',j} = \begin{cases} \diag(p1_j, 1_{n-j}) & w=w', i=i', j\leq a_{w}  \\
\diag(p1_{a_w}, 1_{n-j}, p1_{j-a_w}) & w=w', i=i', j > a_w \\
1_n & \text{otherwise}.
\end{cases}
$$
Note that $t_{w,i,j}^+$ has the property that
$$
t_{w,i,j}^+ I_r^0 t_{w,i,j}^{+,-1} \subset I_r^0.
$$
Let $t_{w,i,j}^- = (t_{w,i,j}^+)^{-1}$.
We put
\begin{equation}\label{Up-operator}
U_{w,i,j} = K_r t_{w,i,j}^+ K_r, \ \ U^-_{w,i,j} = K_r t_{w,i,j}^- K_r;
\end{equation}

Hida has shown that these Hecke operators can even be defined on $p$-adic modular forms and cuspforms
when $p$ is not a zero-divisor but not necessary invertible (see Section \ref{Heckeatp2}).

\begin{rmk} To define the actions of these Hecke operators on higher coherent cohomology of
automorphic vector bundles it is necessary to use the
class of smooth projective polyhedral cone decompositions used to define toroidal compactifications 
in \cite{Lan,Lan2}.  For holomorphic forms this is generally superfluous because of the Koecher principle \cite{KLkoecher}.
\end{rmk}

\subsubsection{Comparing spaces of modular forms of different weight}\label{compare-weight}
Given an integer $a$, let $\kap_a$ be the weight $\kap_a = (a,(0))$. We define a modular
form $f_a \in M_{\kap_{a}}(K;R)$ by the rule $f_a(\underline{A},\veps) = \lambda^a$, where
$(\underline{A},\veps)$ is a pair over an $R$-algebra $S$ and the isomorphism
from $S(1)$ to itself induced by $\veps$ is multiplication by $\lambda\in S^\times$. 

Let $\kap = (\kap_0,(\kap_{\sigma,i}))$ be a weight, and put $\kap' = (\kap_0+a,(\kap_{\sigma,i}))$.
Then there are isomorphisms
$$
M_\kap(K;R)\stackrel{f\mapsto f_a\cdot f}{\rightarrow} M_{\kap'}(K;R) \ \ \text{and } 
M_\kap(K_r,\psi;R)\stackrel{f\mapsto f_a\cdot f}{\rightarrow} M_{\kap'}(K_r,\psi;R).
$$
These maps induce isomorphisms on spaces of cuspforms, and the Hecke operators $T(g)$ satisfy
$$
f_a\cdot T(g)f = ||\nu(g)||^{a} T(g)(f_a\cdot f).
$$

\subsection{Complex uniformization}\label{cpx}
We relate the objects defined so far to the usual complex analytic description of modular forms on Shimura varieties.  For the identification of the adelic double coset spaces with the sets of complex points we refer to Section 8 of \cite{Kottwitz}

\subsubsection{The spaces}\label{spaces-section}
Let $X$ be the $G(\R)$-orbit under conjugation of the homomorphism $h$.
Recall that the stabilizer of $h$ is the group $U_\infty= C(\R)$, so there is a natural
identification $G(\R)/C(\R)\isoarrow X$, $g\mapsto ghg^{-1}$, which gives
$X$ the structure of a real manifold.
Let $P_0\subset G_0$ be the stabilizer of $\Lambda_0$.
Via the identification of ${G}_{/\C}$
with ${G_{0}}{}_{/\C}$, which identifies $C(\C)$ with $H_0(\C)$, $X$
is identified with an open subspace of $G_{0}(\C)/P_{0}(\C)$,
which gives $X$ a complex structure.   There are natural
complex analytic identifications
\begin{equation}\label{cpx-1}
\begin{split}
\M_{K,L}(\C) = & G(\Q)\backslash X\times G(\A_f) /K \\
\M_{K_r,L}(\C) = & G(\Q)\backslash X\times G(\A_f)/K_r,
\end{split}
\end{equation}
where the class of $(h',g)\in X\times G(\A_f)$, with $g_p \in G(\ZZ_p)$,
corresponds to the equivalence class
of the tuple $\underline{A}_{h',g}=(A_{h'},\lambda_{h'},\iota,\eta_g)$
(or $\underline{X}_{h',g}=(\underline{A}_{h',g},\phi_g)$) consisting of
\begin{itemize}
\item the abelian variety $A_{h'}= (L\otimes\R)/L$ with the complex
structure on $L\otimes\R$ being that determined by $h'$; its dual abelian variety is $A_{h'}^\vee:=(L\otimes\R)/L^\#$, where
again $L\otimes\R$ has the complex structure defined by $h'$ and
where $L^\#=\{x\in L\otimes\R\ : \ \la x,L\ra\subseteq\Z(1)\}$;
\item $\lambda_{h'}:A_{h'}\rightarrow A_{h'}^\vee$ is the isogeny
induced by the identity map on $L\otimes\R$;
\item $\iota$ is induced from the canonical action of $\O_{B}$ on $L$;
\item $\eta_{g}$ is the $K^p$-orbit of the translation by $g$ map
$g^p:L\otimes\A_f^p\isoarrow L\otimes\A_f^p = H_1(A_{h'},\A_f^p)$;
\item in the case of $\M_{K_r,L}$, $\phi_g$ is the $B_H^u(\Zp)$-orbit of the
map $L^+\otimes\mu_{p^r}\hookrightarrow A_{h'}^\vee[p^r] = \frac{1}{p^r}L^\#/L^\#
= (L^\#\otimes\Zp)/\left(p^rL^\#\otimes\Zp\right)$,
$v\otimes e^{2\pi \sqrt{-1}/p^r}\mapsto g_pv \mod\,\left(p^rL^\#\otimes\Zp\right)$.
\end{itemize}
Here we are using that the simple factors of ${G}^{{\rm der}}_{/\R}$ are all of type $A$
(see \cite[\S7-\S8]{Kottwitz} for how this enters into the identifications \eqref{cpx-1}).

\subsubsection{Classical modular forms}\label{classicalmod}
The dual of the Lie algebra of $A_{h'}^\vee$ is
$\omega_{A_{h'}^\vee} = \Hom_\C(L\otimes\R,\C)$ with the complex structure
on $L\otimes\R$ being that determined by $h'$.
Recalling that
$L\otimes\R\isoarrow W = \Lambda_{0}\otimes_{S_0}\C$ is a $\C$-linear
isomorphism for the complex structure on $L\otimes\R$ determined by $h$,
we find that there is a canonical
$\O_{B}\otimes\C$-identification
 $\veps_0:\omega_{A^\vee_{h}}\isoarrow \Lambda_{0}^\vee\otimes_{S_0}\C$.
If $h'=ghg^{-1}$, then $\veps_{h'}(\lambda) = \veps_0(g^{-1}\lambda)$ is an $\O_{B}\otimes\C$-identification of $\omega_{A_{h'}^\vee}$ with $\Lambda_{0}^\vee\otimes_{S_0}\C$.
The complex points of the $H_{0}$-torsors $\CE/\M_{K,L}$ and $\CE_{r}/\M_{K_r,L}$ are then given by
\begin{equation}\label{CE-cpx}
\begin{split}
\CE(\C) = & G(\Q)\backslash G(\R)\times H_{0}(\C)\times G(\A_f)/U_\infty K \\
\CE_{r}(\C) = & G(\Q)\backslash G(\R)\times H_{0}(\C)\times G(\A_f)/U_\infty K_r,
\end{split}
\end{equation}
with the class of $(g,x,g_f)\in G(\R)\times H_{0}(\C)\times G(\A_f)$ corresponding
to the classes of 
$$(\underline{A}_{ghg^{-1},g_f}, (x\veps_0(g^{-1}\cdot),\nu(x)))$$
and 
$$(\underline{X}_{ghg^{-1},g_f}, (x\veps_0(g^{-1}\cdot),\nu(x))),$$
respectively.

As $\C$ is a $\Zp$-algebra via $\iota_p$, a weight $\kappa$ modular form over $\C$ is
therefore identified with a smooth function
$\vphi:G(\A)\times H_{0}(\C)\rightarrow\C$ such that
$\vphi(\gamma g uk, bxu) = \kap(b)\vphi(g,x)$
for $\gamma\in G(\Q)$,
$g\in G(\A)$, $x\in H_{0}(\C)$, $u\in U_\infty$, $b\in B_{H_{0}}(\C)$, and
$k\in K$ or $K_r$. The space
$$
W_\kap(\C) = \{ \phi:H_{0}(\C)\rightarrow\C\ : \ \phi \text{ holomorphic}, \phi(bx)=\kap(b)\phi(x) \
\forall b\in B_{H_{0}}(\C)\}
$$
is the irreducible $\C$-representation of $H_{0}$ of highest weight $\kap$ with respect
to $(T_{H_0},B_{H_0}^{\op})$ (this is the Borel-Weil theorem),
so a weight $\kap$ modular form
is also identified with a smooth function $f:G(\A)\rightarrow W_\kap(\C)$ such that
$f(\gamma guk) = u^{-1}f(g)$ for $\gamma\in G(\Q)$, $u\in U_\infty$, and $k\in K$ or $K_r$.
Here $U_\infty$ acts on  $W_\kap(\C)$ as $u\phi(x) =\phi(xu)$. The connection between $f$ and $\vphi$ is
$f(g)(x)=\vphi(g,x)$.
The condition that the modular form is holomorphic can be interepreted as follows.
Let $\grg= \Lie(G(\R))_\C$, and let $\grg = \grp^-\oplus\grk\oplus\grp^+$
be the Cartan decomposition for the involution $h(\sqrt{-1})$: $ad\,h(\sqrt{-1})$ acts as $\pm \sqrt{-1}$ on
$\grp^\pm$. The identification of $G(\C)$ with $G_{0}(\C)$ identifies
$\Lie(P_{0}(\C))$ with $\grk\oplus\grp^+$, and so $f$ corresponds
to a holomorphic form if and only if $\grp^-*f =0$.

Let $\psi:T_H(\Zp)\rightarrow \bQ^\times$ be a finite character that factors through
$T_H(\Zp/p^r\Zp)$. The condition
that a modular form have character $\psi$ becomes $f(gt) = \psi(t)f(g)$ for all $t\in T_H(\Zp)$,
where the action of $t$ comes via \eqref{B+TH}.

\subsubsection{Hecke operators}\label{Hecke-complex} The actions of the Hecke operators
in \ref{Heckeprimetop} and \ref{Heckeatp}
correspond to the following actions on the functions $f:G(\A)\rightarrow W_\kap(\C)$:
the action of $[K_2gK_1]$ is just
\begin{equation}\label{cpx-hecke1}
f(g) \mapsto \sum_{g_j} f(gg_j), \ \ \ K_2 gK_1  = \sqcup g_j K_1,
\end{equation}
and similarly with $K_i$ replaced by $K_{i,r}$.

\subsection{Igusa towers}\label{Igusatower}
Let $\sq=\{p\}$.
Recall $\CA$ and
$\omega$ from Section \ref{canonicalbun}.
Recall that the hypothesis \eqref{ord-hyp} implies that the completion of $\incl_p(S_\sq)$ is $\Zp$; in this way we
consider $\Zp$ an $S_\sq$-algebra. 
Let $k>0$ be so large that the $k$th-power of the Hasse invariant has a lift
to a section $E\in M_{\det^k}(K;\Zp)$.
(See \cite[Section  6.3.1]{Lan2} for the definition of the Hasse invariant on the toroidal compactification.)  
Put
$$
\SS_{m} = \M_{K,L}^\tor[\frac{1}{E}]_{/\Zp/p^m\Zp}.
$$
Let $\SS_{m}^0 = \M_{K,L}[\frac{1}{E}]_{/\Zp/p^m\Zp}$; this is an open subscheme of
$\SS_{m}$, and $\SS_1$ is dense in the special fiber of $\M_{K,L}^\tor$ 
(this follows from our hypotheses on the moduli problem and the discussion in \cite[Section~6.3.3]{Lan2}).
For $n\geq m$ let $\T_{n,m}/\SS_{m}$ be the finite \'etale scheme over $\SS_{m}$ \cite[Section  8.1.1]{Hida} such that
for any $\SS_{m}$-scheme $S$
$$
\T_{n,m}(S)  = \Isom_S(L^+\otimes \bmu_{p^n},\CA^\vee[p^n]^\circ),
$$
where  the superscript $^\circ$ denotes the identity component and the isomorphisms are of finite flat group schemes over $S$ with
$\O_{B}\otimes\Zp$-actions. The scheme $\T_{n,m}$ is Galois over $\SS_{m}$ with Galois group canonically isomorphic to $H(\Zp/p^n\Zp)$.  
The collection $\{\T_{n,m}\}_{n}$ is called the Igusa tower over $\SS_{m}$.

\subsection{$p$-adic modular forms}\label{padicmforms-section}
Let $\D_{n,m}$ be the preimage of $\D_m=\SS_{m}-\SS_{m}^0$ (with reduced closed subscheme structure)  in
$\T_{n,m}$  (the preimage is also reduced because the morphism is \'etale). For a $p$-adic ring $R$ (that is, $R = \varprojlim_m R/p^mR$),
let
$$
V_{n,m}(R) = H^0(\T_{n,m}/R, \O_{\T_{n,m}})  \ \ \text{and} \ \ V_{n,m}^{\cusp}(R) =
H^0(\T_{n,m}/R,\O_{\T_{n,m}}(-\D_{n,m})).
$$
The group $H(\Zp)$ acts on each through its
quotient $H(\Zp/p^n\Zp)$, the Galois group of $\T_{n,m}/\SS_{m}$.
The $R$-module of $p$-adic modular forms (for $G$) over $R$ of level $K^p$ is
$$
V(K^p,R) = \varprojlim_m\varinjlim_n V_{n,m}(R)^{B^u_H(\Zp)},
$$
and the $R$-module of $p$-adic cuspforms (for $G$) over $R$ of level $K^p$ is
$$
V(K^p,R)^{\cusp} = \varprojlim_m\varinjlim_n V_{n,m}^{\cusp}(R)^{B^u_H(\Zp)}.
$$
The group $T_H(\Zp)=B_H(\Zp)/B_H^u(\Zp)$ acts on these modules.

A $p$-adic modular form over $R$ can be viewed as a functorial rule that assigns
an element of a $p$-adic $R$-algebra $S$ to each tuple $(\underline{A},\phi)$ over $S$, where
$\underline{A}=(\underline{A}_m)\in \varprojlim\SS_m(S)$ and $\phi=(\phi_{n,m})\in\varprojlim_{m}\varprojlim_n \T_{n,m}(S)$ with each $\phi_{n,m}$ over
$\underline{A}_m$.

\subsubsection{$p$-adic modular forms of weight $\kap$ and character $\psi$}
\label{padicweight}
Let $\mathcal{L}'\subset\bQp$ be the extension of $\Qp$  generated by the images of all the embeddings of $\K$
into $\bQp$, and let $\O'$ be its ring of integers.
Let
$$
\kap= (\kap_{\sigma,i})_{\sigma\in\Sigma_\K,1\leq i\leq m}, \  \kap_{\sigma,i}\in \Z^{a_{\sigma,i}}.
$$
We denote also
by $\kap$ the $\O'$-valued character of $T_H(\Zp)$ defined by
\begin{equation*}\begin{split}
\kap(t) & =
\prod_{w|p} \prod_{\sigma\in\Sigma_\K \atop \grp_\sigma=\grp_w} \prod_{i=1}^m
\prod_{j=1}^{a_{\sigma,i}}\sigma(t_{w,i,j})^{\kap_{\sigma,i,j}}, \\
 t =(\diag&(t_{w,i,1},...,t_{w,i,a_{w,i}}))_{w|p,1\leq i\leq m} \in T_H(\Zp).
\end{split}
\end{equation*}
If $\psi:T_H(\Zp)\rightarrow\bQp^\times$ is a finite-order character, then we define
an $\O'[\psi]$-valued character $\kap_\psi$ of $\T_H(\Zp)$ by $\kap_\psi(t) = \psi(t)\kap(t)$.
For $R$ a $p$-adic ring that is also an $\O'[\psi]$-algebra,
the spaces of $p$-adic modular forms and cuspforms of weight $\kap$ and character $\psi$
are
$$
V_\kap(K^p,\psi,R) = \{f\in V(K^p,R) \ : \ t\cdot f = \kap_\psi(t)f \ \forall t\in T_H(\Zp)\}
$$
and
$$
V_\kap^{\cusp}(K^p,\psi,R) = \{f\in V^{\cusp}(K^p,R) \ : \ t\cdot f = \kap_\psi(t)f \ \forall t\in T_H(\Zp)\}.
$$
As a functorial rule, a $p$-adic modular form of weight $\kap$ and character $\psi$
satisfies $f(\underline{A},\phi\circ t) = \kap_\psi(t)f(\underline{A},\phi)$ for
all $t\in T_H(\Zp)$.

\subsubsection{The action of $G(\A_f^p)$}
The action of $G(\A_f^p)$ on $\{\M_{K,L_{/R}}^\tor\}_{K^p}$ induces an action on
$\{\SS_m\}_{K^p}$ and on $\{\T_{n,m}\}_{n,K^p}$, and these actions give
an action of $G(\A_f^p)$ on
$$
\varprojlim_{K^p} V(K^p,R)  \ \ \text{and} \ \ \varprojlim_{K^p} V_\kap(K^p,\psi,R)
$$
and on their submodules of cuspforms.  Indeed, while the lift $E$ is not necessarily a Hecke eigenform, its reduction mod $p$ is invariant under $G(\A_f^p)$.  This implies that the spaces of $p$-adic modular forms don't depend on the choice of lift.   (See \cite{Lan2}, especially sections 8.1.4 and 8.3.6, for a more canonical construction of this action.)

The submodules fixed by $K^p$ are just the $p$-adic modular forms and cuspforms of weight $\kap$ and prime-to-$p$ level $K^p$.
(Here we are using the fact that the Igusa varieties are normal schemes over $Spec(R)$; see Theorem 5.2.1.1 of  \cite{Lan2}.)

\subsubsection{Hecke operators away from $p$}\label{Heckeprimetop2} Let $K_j^p \subset G(\A_f^p)$, $j=1,2$,
be neat open compact supgroups. For $g\in G(\A_f^p)$ we define a Hecke operator
$[K_{2}^p g K_{1}^p]$ on the spaces of $p$-adic modular forms and cuspforms just as in
Section  \ref{Heckeprimetop}.

\subsubsection{Modular forms as $p$-adic modular forms} Let $\sq=\{p\}$.
Under Hypothesis \ref{ord-hyp}, the completion of $\incl_p(S_{(p)})$
is $\Zp$, so $incl_p$ identifies $\Zp$ as an $S_{(p)}$-algebra and $\O'$ as an $S_0$-algebra.
As $\O_B\otimes \Z_{(p)} = \O_{(p)}^m$, we have 
$$
\O_B\otimes\O' = (\O_{(p)}\otimes\O')^m = \prod_{w|p} \prod_{i=1}^m \O_w\otimes\O'  \isoarrow \prod_{w|p}\prod_{\sigma\in\Sigma_\K \atop \grp_\sigma=\grp_w} 
\prod_{i=1}^m\O = \prod_{\sigma\in\Sigma_\K} \prod_{i=1}^m\O' .
$$
The choices in Sections \ref{H0-OF-reps} and \ref{PELunitary} induce 
$\O_B\otimes\O'$-decompositions
$$
\Lambda_0\otimes_{S_\sq}\O' = \prod_{\sigma\in\Sigma_\K}\prod_{i=1}^m e_i\Lambda_{0,\sigma}\otimes_{S_0}\O'= \prod_{\sigma\in\Sigma_\K}\prod_{i=1}^m (\O')^{a_{\sigma,i}}
$$
and
$$
L^+\otimes_\Zp\O' = \prod_{w|p}\prod_{i=1}^m e_iL_w\otimes_\Zp\O' = \prod_{w|p}\prod_{i=1}^m (\O_w\otimes_\Zp\O')^{a_{w,i}}
= \prod_{w|p}\prod_{\sigma\in\Sigma_\K \atop \grp_\sigma = \grp_w} \prod_{i=1}^m (\O')^{a_{\sigma,i}}.
$$
Equating these identifications yields an $\O_B\otimes\O'$-identification $\Lambda_0\otimes_{S_\sq}\O'= L^+\otimes_\Zp~\O'$.
Recalling that $H_0\subset G_0$ is the stabilizer of the polarization $\Lambda=\Lambda_0\oplus\Lambda_0^\vee$ and hence that
${H_0}_{/\O'}\isoarrow \G_m\times\GL_{\O_B\otimes\O'}(\Lambda_0\otimes_{S_\sq}\O')$, this then determines an isomorphism
$$
{H_0}_{/\O'}\isoarrow \G_m\times H_{/\O'}
$$
which is given explicitly in terms of \eqref{H-iso}
and $\eqref{H0-iso}$ by
\begin{equation}\label{H0-H-iso}
{H_0}_{/\O'}\ni (\nu,(g_{\sigma,i})_{\sigma\in\Sigma_\K})\mapsto (\nu,(\prod_{\sigma\in\Sigma_\K\atop \grp_\sigma=\grp w} \nu\cdot {}^tg_{\sigma c,i}^{-1})_{w|p})
\in \G_m\times H_{/\O'},
\end{equation}
where we have used the identification 
$\GL_r(\O_w\otimes_\Zp\O') \isoarrow  \prod_{\sigma\in\Sigma_\K, \grp_\sigma=\grp_w} \GL_r(\O')$.
This identifies ${B_{H_0}}_{/\O'} = \G_m\times {B_H}_{/\O'}$, ${B_{H_0}^u}_{/\O'} = {B_H^u}_{/\O'}$, and ${T_{H_0}}_{/\O'} = \G_m\times {T_H}_{/\O'}$.

To each weight $\kap = (\kap_0, (\kap_{\sigma,i}))$ as in \eqref{H0-OF-reps}, we associate
a $\kap_p$ as in \ref{padicweight}:
$$
\kap_p = (\kap_{\sigma c,i}).
$$
Note that $\kap_{\sigma c,i}\in \Z^{b_{\sigma c,i}} = \Z^{a_{\sigma,i}}$.
If $t\in T_H(\Zp)$, $t=(\diag(t_{w,i,1},...,t_{w,i,a_{w,i}})$, then 
$$
\kap_p(t) = \prod_{w|p}\prod_{\sigma\in\Sigma_\K \atop \grp_\sigma=\grp_w} \prod_{i=1}^m\prod_{j=1}^{a_{w,i}}
\sigma(t_{w,i,j})^{\kap_{\sigma c,i,j}}.
$$
Note that if $x=(t_0,t) \in \Z_p^\times\times T_H(\Zp) \subset T_{H_0}(\O')$, then 
\begin{equation*}\label{kapstar}
\kap(x)= t_0^{c_0}\kap_p(t^{-1}), \ \ \ c_0 = \kap_0+{\sum_{\sigma,i,j}\kap_{\sigma,i,j}}.
\end{equation*}

As we explain in the following, for $\psi:T_H(\Zp)\rightarrow\Q_p$ a finite order character and $R$ a $p$-adic
$\O'[\psi]$-algebra, if $\kap$ satisfies the inequalities \eqref{dom-wt-ineq}, then the modular forms over $R$ of weight $\kap$ and character $\psi$
are $p$-adic modular forms of weight $\kap$ and character $\psi$.

Fixing $\Gm = \mathrm{Spec}(\Z[x,\frac{1}{x}])$ as usual yields an identification
$\mu_{p^n} = \mathrm{Spec}(\Z[x,\frac{1}{x}]/(x^{p^n}-1))$ for each $n\geq 1$, and hence
an identification $\Lie_{\Zp}(\mu_{p^n})=\Z x\frac{d}{dx}$. For any scheme $S$, this identifies
$\Lie_S(\mu_{p^n})$ with $\O_S$, compatibly as $n$ varies.
If $n\geq m$, $S$ is a $\Zp/p^m\Zp$-scheme, and $\phi\in \T_{n,m}(S)$, then this identification
gives an isomorphism
$$
\Lie(\phi):L^+\otimes\O_S = L^+\otimes\Lie_S(\mu_{p^n}) \isoarrow \Lie_S(\CA_{/S}^\vee[p^n]^\circ) = \Lie_S\CA_{/S}^\vee.
$$
The identification $\Lambda_0\otimes\Zp=L^+$ gives
$(\Lie(\phi)^\vee,id) \in \CE_n(S)$. If $f\in M_{\kap}(K_r,\psi;R)$ for $R$ a $p$-adic $\O'[\psi]$-algebra, then
the value of the $p$-adic modular form $f_\padic$ determined by $f$ on
a ($p$-adic) test object $(\underline{A},\phi)$ over a $p$-adic $R$-algebra $S$ is
$$
f_\padic(\underline{A},\phi) = \varprojlim_{m} f(\underline{A}_m,\phi_{m,m,r},(\Lie(\phi_{m,m,r})^\vee,id)) \in \varprojlim_m S/p^mS = S,
$$
where for $n\geq \max\{r,m\}$, $\phi_{n,m,r}$ is the isomorphism
$L^+\otimes\mu_{p^r}\isoarrow \CA_{/S}^\vee[p^r]^\circ$ determined by $\phi_{n,m}$.
If $t\in T_H(\Zp)$ then $\Lie(\phi\circ t)^\vee = t^{-1}\cdot\Lie(\phi)^\vee$, so
$$(t\cdot f_\padic)(\underline{A},\phi) = \varprojlim f(\underline{A}_m,\phi_{m,m,t}\circ t,
(\Lie(\phi_{m,m}\circ t)^\vee,id)) = \psi(t)\kap_p(t) f_\padic(\underline{A},\phi),
$$
hence $f_\padic$ is a $p$-adic modular form of weight $\kap_p$ and character $\psi$.
Clearly, if $f$ is a cuspform, then $f_\padic$ is a $p$-adic cuspform\footnote{A modular form can be a $p$-adic cuspform but not be cuspidal. A simple example is the
the critical $p$-stabilization $E_{2k}^*(z) = E_{2k}(z) - E_{2k}(pz)$ of the
level $1$ weight $2k\geq 4$ Eisenstein series $E_{2k}$.}.
Also, the corresponding $R$-module homomorphisms
\begin{equation}\label{padicclassical}
M_{\kap}(K_r,\psi;R)\hookrightarrow V_{\kap_p}(K^p,\psi,R)  \ \  \text{and} \ \
S_{\kap}(K_r,\psi;R) \hookrightarrow V_{\kap_p}^{\cusp}(K^p,\psi,R)
\end{equation}
-- these are injective because, as already noted, $S_1$ is dense in the special fiber of $\M_{K,L}^\tor$ -- are compatible with Hecke operators
 in the sense that
\begin{equation}\label{p-adic-hecke}
(T(g)\cdot f)_\padic = ||\nu(g)||^{-\kap_0}T(g)\cdot f_\padic
\end{equation}
for $g\in G(\A_f^p)$.

Note that if $\kap' = (\kap_0+a,(\kap_{\sigma,i}))$, then $\kap_p' = \kap_p$. Furthermore,
for $f\in M_\kap(K;R)$ and $f' = f_a f \in M_{\kap'}(K;R)$ (see \ref{compare-weight}),
$$
f_\padic = f'_\padic.
$$

\subsubsection{Hecke operators at $p$}\label{Heckeatp2}  Hida (\cite[8.3.1]{Hida}) has defined an action of
the double cosets $u_{w,i,j} = B^u_H(\Zp)t_{w,i,j}B^u_H(\Zp)$ on the modules of $p$-adic modular forms and
cuspforms; this action is defined via correspondences on the Igusa tower (see also \cite{SkUr-SKlifts}).
Moreover, as Hida shows, if $R$ is a $p$-adic domain in which $p$ is not zero, $\kap$ as in Section  \ref{padicweight},
and $f\in M_{\kap}(K_r,\psi;R)$, then
$u_{w,i,j}\cdot f \in M_{\kap}(K_r,\psi;R)$ and
\begin{equation}\label{Hecke-p-comp}
u_{w,i,j}\cdot f=|\kap_{norm}(t_{w,i,j})|_p^{-1} U_{w,i,j}\cdot f, \ \ \
\kap_{norm}=(\kap_{\sigma,i'}-b_{\sigma,i'}).
\end{equation}
We put
$$
u_p = \prod_{w\in\Sigma_p}\prod_{i=1}^m \prod_{j=1}^{n_i} u_{w,i,j}
$$
and define a projector
\begin{equation}\label{ordprojector}
e = \varinjlim_{n} u_p^{n!}.
\end{equation}

\subsubsection{Ordinary forms}\label{ordinaryforms} Let $R$ be a $p$-adic ring. The submodules of ordinary $p$-adic forms over $R$ are
$$
V^\ord\left(K^p,R\right) = eV\left(K^p,R\right), \ \ \text{and} \ \ V^{\ord,\cusp}\left(K^p,R\right) = eV^{\cusp}\left(K^p,R\right),
$$
and those of weight $\kap$ and character $\psi$ are
$$
M_\kap^\ord\left(K_r,\psi;R\right)  = eM_\kap\left(K_r,\psi;R\right), \ \ S_\kap^\ord\left(K_r,\psi;R\right) = eS_\kap(K_r,\psi;R),
$$
$$
V_{\kap_p}^\ord\left(K^p,\psi,R\right) = eV_{\kap_p}\left(K^p,\psi,R\right), \ \ V_{\kap_p}^{\ord,\cusp}\left(K^p,\psi,R\right) =
eV_{\kap_p}^{\cusp}\left(K^p,\psi,R\right).
$$
Hida's classicality theorem for ordinary forms  establishes that if
$R$ is a finite $\O'[\psi]$-domain (resp. a finite $\O'$-domain) then
\begin{equation}\label{classicalitythm1}
\begin{split}
V_{\kap_p}^{\ord,\cusp}\left(K^p,\psi,R\right) & = S_{\kap}^\ord\left(K_r,\psi;R\right) \\ 
(\text{resp. } V_{\kap_p}^{\ord,\cusp}\left(K^p,R\right) & = S_{\kap}^\ord\left(K_r;R\right) ) \\ \text{if }  \kap_{\sigma,i,a_{\sigma,i}}+\kap_{\sigma c,i,b_{\sigma,i}} & \gg n_ir_i  \ \forall \sigma\in\Sigma_\K, 1\leq i\leq m.
\end{split}
\end{equation}

This theorem (with more precise inequalities on $\kap_{\sigma,i,a_{\sigma,i}}+\kap_{\sigma c,i,b_{\sigma,i}}$) is proved in \cite{Hida, H02} 
assuming conditions denoted (G1)-(G3) (see Section 7 of \cite{H02}), which were subsquently proved by Lan in \cite{Lan}.\footnote{Theorem 6.4.1.1 of \cite{Lan} contains (G2) and part of (G1).  The remaining (projectivity) assertion in (G1) is contained in Theorem 7.3.3.4 of \cite{Lan}.  Condition (G3) is proved in Section  7.2.3 of \cite{Lan}.}
Let $R$ be as in Equation \eqref{classicalitythm1} and let $O^+$ denote the integral closure of $\Z_{(p)}$ in $R$.  
The fraction field
$\Frac(O^+)$ of $O^+$ is a number field over which $S_{\kap}\left(K_r,\psi;R\right)\otimes \Q$ has a rational model,
given by the space of $\Frac(O^+)$-rational cusp forms of type $\kap$ and level $K_r$.  The intersection of this space
with $S_{\kap}^\ord\left(K_r,\psi;R\right)$ is an $O^+$-lattice $S_{\kap}^\ord\left(K_r,\psi;O^+\right)$.  
Given any embedding $\iota:  O^+ \hookrightarrow \C$,
the image of $S_{\kap}^\ord\left(K_r,\psi;O^+\right)$ in the space $S_{\kap}\left(K_r,\psi;\C\right)$ will be
called the space of ordinary complex cusp forms (relative to $\iota$) of type $\kap$ and level $K_r$.

 \subsection{Measures and $\Lambda$-adic families}
 We recall $p$-adic measures and their connections with Hida's theory
 of $\Lambda$-adic modular forms.

 \subsubsection{$p$-adic measures}
 Let $R$  be a $p$-adic ring. The space of $R$-valued measures on $T_H(\Zp)$ is
 $$
 \mathrm{Meas}(T_H(\Zp);R) = \Hom_{\Zp} (C(T_H(\Zp),\Zp),R),
 $$
 where $C(T_H(\Zp),\Zp)$ is the $\Zp$-module of continuous $\Zp$-valued functions on $T_H(\Zp)$.
 Note that 
$C(T_H(\Zp),R) = C(T_H(\Zp),\Zp)\hat\otimes_\Zp R$, so we also have
$$
 \mathrm{Meas}(T_H(\Zp);R) = \Hom_{R} (C(T_H(\Zp),R),R).
 $$
 More generally, if $M$ is a complete $R$-module we can define the $R$-module of $M$-valued measures in the same way.
 The $R$-module of $M$-valued measures is naturally identified with $R[\![T_H(\Zp)]\!]\hat\otimes_R M$; the identification of
 a measure $\mu$ with an element $f$ of the completed group ring is such that for any
 continuous homomorphism $\chi:T_H(\Zp)\rightarrow R^\times$, $\mu(\chi) = \chi(f)$, where
 $\chi(f)$ is the image of $f$ under the homomorphism $R[\![T_H(\Zp)]\!]\rightarrow R$ induced by $\chi$.

 \subsubsection{$\Lambda$-adic forms}
 Let
 $$
 \Lambda_H = \O'[\![T_H(\Zp)]\!].
 $$
 Both $V(K^p,R)$ and $V^{\cusp}(K^p,R)$, $R$ a $p$-adic $\O'$-algebra, are $\Lambda_H$-modules via the actions of $T_H(\Zp)$ on them. A $\Lambda_H$-adic modular form over $R$ is a
 $\mu\in\mathrm{Meas}(T_H(\Zp);V(K^p;R))$ such that $\mu(t\cdot f) = t\cdot\mu(f)$
 for all $t\in\Lambda_H$. In particular, it follows that if $R$ is an $\O'[\psi]$-algebra, then $\mu(\kap_\psi) \in V_\kap(K^p,\psi,R)$. A $\Lambda_H$-adic cuspform is defined in the same way, replacing the $p$-adic modular forms with cuspforms. Similarly, an ordinary $\Lambda_H$-adic modular form or cuspform is also defined in the same way, replacing the
 modular forms and cuspforms with the ordinary forms. Clearly, if $\mu$ is a $\Lambda$-adic modular form, then $e\mu$ (the composition of $\mu$ with the $R$-linear projector
 $V(K^p,R)\rightarrow eV(K^p,R) = V^\ord(K^p,R)$) is an ordinary $\Lambda_H$-adic form.
 Let
 $$
 \CS^\ord(K^p,R) = \{\text{ordinary $\Lambda_H$-adic cuspforms $\mu\in\mathrm{Meas}(T_H(\Zp);V^{\ord,\cusp}(K^p,R))$}\}.
 $$
 The Hecke operators in \ref{Heckeprimetop2} and \ref{Heckeatp2} act on $\CS^\ord(K^p,R)$ through their actions on $V^{\ord,\cusp}(K^p,R)$.

 Let $\Delta\subset T_H(\Zp)$ be the torsion subgroup. Since $p$ is unramified in $\K$ by hypothesis, \eqref{H-iso} induces an identification
 $$
 \Delta\isoarrow \prod_{w|p}\prod_{i=1}^m (k_w^\times)^{a_{w,i}}
 $$
 where $k_w$ is the residue field of $\O_w$. In particular, $\Delta$ has order prime-to-$p$,
 so $\CS^\ord(K^p,R)$ decomposes as a direct sum of isotypical pieces for the $\O'$-characters
 $\omega\in\hat\Delta$ of $\Delta$:
 $$
 \CS^\ord(K^p,R) = \oplus_{\omega\in\hat\Delta} \CS_\omega^\ord(K^p,R).
 $$
 Let $W\subset T_H\left(\Zp\right)$ be a free $\Zp$-submodule such that  $T_H(\Zp)$ is isomorphic to $\Delta\times W$. Then $\Lambda_H = \O'[\![\Delta\times W]\!] = \Lambda^o[\Delta]$, where
 $$
 \Lambda^o = \O'[\![W]\!].
 $$
 Each $\CS_\omega^\ord(K^p,R)$ is a $\Lambda$-module.

 Let $R\subset\bQp$ be a finite $\O'$-algebra and let
 $$
 \Lambda^o_R = \Lambda^o\otimes_{\O'}R = R[\![W]\!].
 $$
 Theorem 7.1 (5) of Hida's paper \cite{H02} asserts (again, under conditions (G1)-(G3), which were proved in \cite{Lan})
 that
 \begin{equation}\label{controlthm2}
 \text{$\CS_\omega^\ord(K^p,R)$ is a free $\Lambda^o_R$-module of finite rank},
 \end{equation}
 and for any finite character $\psi:W\rightarrow \bQ_p^\times$ trivial on
 $W^{p^{r-1}}$ and $\kap$ as in
 \ref{padicweight} satisfying the restriction in \eqref{classicalitythm1},
 \begin{equation}\label{controlthm3}
 (\CS^\ord_\omega(K^p,R)\otimes_RR[\psi])/\grp_{\kap\psi}\CS^\ord_\omega(K^p,R)
 \otimes_RR[\psi] \xrightarrow{\mu\mapsto \mu(\kap\psi)} S_{\kap}^\ord(K_r,\omega\omega_\kap\psi,R[\psi])
 \end{equation}
 is an isomorphism, where $\grp_{\kap\psi}$ is the kernel of the homomorphism
 $\Lambda^o_R\otimes_R R[\psi]\rightarrow R[\psi]$ induced by the character
 $\kap\psi$ and $\omega_\kap\in\hat\Delta$ is $\kap|_\Delta$.

 \begin{rmk} Clearly, one can include types (i.e., irreducible representations $W_S$ of compact open subgroups of $K_S \subset K^p$, for some finite set $S$ of primes) in the definition of $\Lambda$-adic cuspforms, and we write
 the module of $\Lambda$-adic cuspforms of type $W_S$ as $\CS^\ord(K^p,W_S,R)$.  It can be shown that the analogs of the maps
 \eqref{controlthm3} in this context are also isomorphisms, using the fact
 that $\grp_{\kap\psi}$ is generated by a regular sequence.
 \end{rmk}

\section{The PEL data and restriction of forms}\label{doublingsetup}

In this section, we discuss restrictions of modular forms from a larger unitary group to a product of unitary groups, which is important for interpreting the doubling method (first introduced in Section  \ref{doubling16}) geometrically.

\subsection{The PEL data}\label{PELprob}
Let $P=(\K,c,\O,L,\pair,h)$
be a PEL datum of unitary type associated with a
hermitian pair $(V,\pair_V)$ as in Sections \ref{PELdata}, \ref{PELunitary}, and \ref{unitarygroupPEL},
together with all the associated objects, choices, and conventions from Section 
\ref{padic-unitary-section}. In particular, the index $m$ equals $1$.
In what follows we will consider four unitary PEL data $P_i = (B_i,*_i,\O_{B_i},L_i,\pair_i,h_i)$
together with $\O_{B_i}\otimes\Zp$ decompositions $L_i\otimes\Zp = L_i^+\oplus L_i^-$:
\begin{itemize}
\item $P_1 = P=(\K,c,\O,L,\pair,h)$, $L_1^\pm = L^\pm$;
\item $P_2 = (\K,c,\O,L,-\pair,h(\bar{\cdot}))$, $L_2^\pm = L^\mp$;
\item $P_3 = (\K\times\K, c\times c,\O\times\O, L_1\oplus L_2, \pair_1\oplus\pair_2, h_1\oplus h_2)$,
$L_3^\pm = L_1^\pm\oplus L_2^\pm$;
\item $P_4 = (\K, c, \O, L_3, \pair_3, h_3)$, $L_4^\pm = L_3^\pm$.
\end{itemize}
Given the hypotheses, there should be no confusion with the subscript `$i$'
being used in this section for the objects associated to the PEL datum $P_i$.

The reflex fields for $P_1$, $P_2$ and $P_3$ are all equal to the reflex field $F$ of $P$. The reflex field of $P_4$ is $\Q$. We put $G_i = G_{P_i}$ for $i=1,...,4$ and
$H_i = \GL_{\O_{B_i}\otimes\Zp}(L_i^+)$. Then $G_1=G_2$ and
there are obvious, canonical inclusions
$G_3\hookrightarrow G_4$ and $G_3\hookrightarrow G_1\times G_2$ which induce the obvious, canonical inclusions $H_3\hookrightarrow H_4$ and $H_3\hookrightarrow H_1\times H_2$.
For $K\subset G_i(\A_f)$ a neat open compact with $K=G_i(\Zp)K^\sq$ if $\sq=\{p\}$,
let $\M_{i,K} = \M_{K}(P_i)$ be the moduli scheme over $S_\sq$.

The choice of the $\O_{w}$-decomposition of
$L_w^\pm$ determines $\O_{B_i,w}$-decompositions of the modules $L_{i,w}^\pm = L_i\otimes_{\CO\otimes\Zp}\CO_w$ and so determines isomorphisms
\begin{equation}\label{Gi-iso}
{G_i}_{/\Zp} \isoarrow \G_m \times \prod_{w\in\Sigma_p} \begin{cases}
\GL_{n}(\O_w) &  i = 1,2 \\
\GL_{n}(\O_w)\times\GL_{n}(\O_w) & i = 3 \\
\GL_{2n}(\O_w) & i=4
\end{cases}
\end{equation}
and
\begin{equation}\label{Hi-iso}
{H_i}_{/\Zp} \isoarrow \prod_{w|p} \begin{cases}
\GL_{a_w}(\O_w) & i = 1 \\
\GL_{b_w}(\O_w) & i=2 \\
\GL_{a_w}(\O_w)\times\GL_{b_w}(\O_w) & i = 3 \\
\GL_{n}(\O_w) & i = 4.
\end{cases}
\end{equation}
The canonical inclusions in the preceding paragraph just correspond to the identity
map on the similitude factors and the obvious inclusions of the $\GL$-parts (being the
diagonal map in the case of the inclusions $G_3\hookrightarrow G_4$ and $H_3\hookrightarrow H_4$.)

Let $K_i^\sq \subset G_i(\A_f^\sq)$ be neat open compact subgroups.
Let $K_i = K_i^\sq$ if $\sq=\emptyset$ and $K_i = G_i(\Zp)K_i^\sq$ otherwise.
If $K_3^\sq\subset K_4^\sq\cap G_3(\A_f^\sq)$, then there is a natural
$S_\sq$-morphism
\begin{equation}\label{PELmap1}
\M_{3,K_3}\rightarrow \M_{4,K_4}, \ \
\underline{A}=(A,\lambda,\iota,\alpha)\mapsto \underline{A}_4=(A,\lambda,\iota\circ\diag, \alpha K_4^\sq),
\end{equation}
where $\diag:K\hookrightarrow K\oplus K$ is the diagonal embedding.
Let $e_i\in \O\oplus\O$, $i=1,2$, be the idempotent corresponding to the projection
to the $i$th factor.
If $K_3^\sq\subset (K_1^\sq\times K_2^\sq)\cap G_3(\A_f^\sq)$, then there is a natural
$S_\sq$-morphism
\begin{equation}\label{PELmap2}\begin{split}
\M_{3,K_3} &\rightarrow \M_{1,K_1}\times_{S_\sq}\M_{2,K_2}, \\
\underline{A}=(A,\lambda,\iota,\alpha) \mapsto & (\underline{A}_1,\underline{A}_2)= (A_1,\lambda_1,\iota_1,\alpha_1)\times(A_2,\lambda_2,\iota_2,\alpha_2),
\end{split}
\end{equation}
where $A_i = \iota(e_i)A$ (so $A=A_1\times A_2$), $\lambda_i =\iota^\vee(e_i)\circ\lambda\circ\iota(e_i)$,
$\iota_i$ is the restriction of $\iota$ to the $i$th factor, and
$\alpha_{i,s}:L_i\otimes\A_f^\sq\isoarrow H_1(A_{i,s},\A_f^\sq)$ is the restriction of
$\alpha_s$ to $L_i\otimes\A_f^\sq \subset L_3\otimes\A_f^\sq = (L_1\otimes\A_f^\sq)\oplus
(L_2\otimes\A_f^\sq)$  composed with the projection
$H_1(A_s,\A_f^\sq)\rightarrow H_1(A_{i,s},\A_f^\sq)$.

For suitably compatible choices of polyhedral cone decompositions, the morphisms
\eqref{PELmap1} and \eqref{PELmap2} extend to maps of toroidal compactifications 
(\cite{Harris-ToroidalCompactifications}, Proposition 3.4).\footnote{Although \cite{Harris-ToroidalCompactifications} only develops the theory in characteristic zero, the polyhedral cone decompositions used to define the toroidal compactifications in Lan's book \cite{Lan} are independent of characteristic.  Thus the arguments of  \cite{Harris-ToroidalCompactifications} go through without change.}

\subsubsection{Level structures at $p$} The definitions of level structures at $p$ in Section  \ref{levelp} for the
PEL data $P_i$ are compatible, and the morphisms \eqref{PELmap1} and \eqref{PELmap2}
extend to $S_\sq$-morphisms with each $\M_{i,K_i}$ replaced by $\M_{i,K_{i,r}} = \M_{K_{i,r}}(P_i)$.

\subsubsection{The canonical bundles}
To define the groups $G_{0,i}$ and $H_{0,i}$ as in Section  \ref{G0H0} in a compatible manner,
we need to specify the choice of the $\Lambda_{0,i}\subset W_i=V_i/V_i^{0,-1}$, where $V_i = L_i\otimes\C$ with the Hodge structure defined by the complex structure on $L_i\otimes\R$ determined by $h_i$. As $V_1=V$ with the same Hodge structure we
take $\Lambda_{0,1} = \Lambda_0$, but since $V_2= V_1$ with the Hodge indices reversed (so $V_2^{0,-1} = V_1^{-1,0}$)  we take $\Lambda_{0,2}$ to
be the image of $\Lambda_0^\vee$ in $W_2 = V_2/V_2^{0,-1} = V_1/V_1^{-1,0}$ using the
canonical identification $V_1^{0,-1}=V^{0,-1}\cong \Lambda_0^\vee\otimes_{S_0}\C$.
Then $\Lambda_1= \Lambda$ with its canonical pairing,
and $\Lambda_2 = \Lambda_0^\vee\oplus(\Lambda_0^\vee)^\vee = \Lambda$
with its canonical pairing. We then set $\Lambda_{0,3} = \Lambda_{0,4} = \Lambda_{0,1}\oplus\Lambda_{0,2}$ and
$\Lambda_3=\Lambda_4= \Lambda_1\oplus\Lambda_2$.

The fixed decompositions of $\Lambda_0$ and $\Lambda_0^\vee$ as $\O_B\otimes\Zp$-modules then determine compatible isomorphisms
\begin{equation}\label{H0i-iso}
{H_{0,i}}_{/\Zp} \isoarrow \G_m \times \prod_{w|p}
\begin{cases} \GL_{b_w}(\O_w) & i =1 \\ \GL_{a_w}(\O_w) & i=2 \\
\GL_{b_w}(\O_w)\times\GL_{a_w}(\O_w) & i =3 \\ GL_{n}(\O_w) & i=4.
\end{cases}
\end{equation}
There are canonical inclusions $H_{0,3}\hookrightarrow H_{0,4}$ and $H_{0,3}\hookrightarrow H_{0,1}\times H_{0,2}$ which correspond to the obvious inclusions under the isomorphisms
\eqref{H0i-iso}.  They both induce the identity map on the $\G_m$-factor; on the $\GL$-factors they induce the diagonal mapping and identity map, respectively. This gives similar inclusions among the (lower-triangular) Borels $B_{H_{0,i}}$ and the (diagonal) tori $T_{H_{0,i}}$. In particular, a dominant character $\kap$ of
$T_{H_{0,4}}$ or a pair $\kap=(\kap_1,\kap_2)$ consisting of dominant characters $\kap_1$ of $T_{H_{0,1}}$ and $\kap_2$ of $T_{H_{0,2}}$ restricts to
a dominant character of $T_{H_{0,3}}$, which we also denote by $\kap$.

Let $\pi_i:\CE_i\rightarrow \M_{K_i}$ be the canonical bundle.
The maps
\eqref{PELmap1} and \eqref{PELmap2} extend to maps of bundles
\begin{equation}\label{Emap1}
\CE_{3} \rightarrow \CE_4, \ \ (A,\lambda,\iota,\alpha,\veps)\mapsto (A,\lambda,\iota\circ\diag,\alpha K_4^\sq,\veps),
\end{equation}
and
\begin{equation}\begin{split}\label{Emap2}
\CE_{3}\rightarrow & \CE_1\times_{S_\sq}\CE_2, \\
(A,\lambda,\iota,\alpha,\veps) \mapsto (A_1,\lambda_1, & \iota_1,\alpha_1,\veps_1)\times(A_2,\lambda_2,\iota_2,\alpha_2,\veps_2),
\end{split}
\end{equation}
where $\veps_i = e_i\circ \veps \circ \iota(e_i)$.  There are similar maps of the bundles
$\CE_{i,r} = \CE_i\times_{\M^\tor_{i,K_i}}\overline{\M}_{i,K_{i,r}}$ with level structure at $p$.

\subsubsection{The Igusa towers}\label{igusatowers}
Let $\T_{n,m,i}/\SS_{m,i}$, $i=1,..,4$, be the Igusa tower for $\M_{i,K_i}$ as in \ref{Igusatower}.
The maps \eqref{PELmap1} and \eqref{PELmap2} extend to maps of Igusa towers in the obvious ways:
\begin{equation}\label{Igusamap1}
\T_{n,m,3}\rightarrow \T_{n,m,4},
\ \ (\underline{A},\phi)\mapsto (\underline{A}_4,\phi)
\end{equation}
and
\begin{equation}\label{Igusamap2}
\T_{n,m,3}\rightarrow \T_{n,m,1}\times_{\Zp} \T_{n,m,2}, \ \
(\underline{A},\phi)\mapsto ((\underline{A}_1,\phi_1),(\underline{A}_2,\phi_2)),
\end{equation}
where $\phi_i$ is the restriction of $\phi$ to $L^+_i\otimes\mu_{p^n}$ composed with
the projection to $\CA_i^\vee\left[p^n\right]^\circ$.  

\begin{rmk}\label{Igusaembedding} {\rm As explained in \cite[Section  2.1.11]{HLS}, the inclusion
\eqref{Igusamap1} {\it does not} restrict on complex points to the map $i_3$ of Shimura varieties determined by
the inclusion of $G_3$ in $G_4$.  For each prime $w$ of $F^+$ dividing $p$, let 
$$\gamma_{V_w} = \begin{pmatrix}1_{a_w} & 0 & 0 & 0\\
0 & 0 & 0 & 1_{b_w}\\
0 & 0 & 1_{a_w} & 0\\
0& 1_{b_w} & 0 & 0\end{pmatrix} \in G_4(F^+_w);$$
$\gamma_{V_p} = (\gamma_{V_w})_{w \mid p} \in G_4(F^+_p)$.
Then the inclusion \eqref{Igusamap1} is given by $i_3$ composed with right translation by $\gamma_{V_p}$.  (See map \eqref{siegeltranslation}.)}
\end{rmk}

When working with $p$-adic modular forms in subsequent sections, we will consider all the $T_{n,m,i}$ simultaneously,
$i = 1, 2, 3, 4$.  The collection $\{T_{n,m,i}\}$, or equivalently $\varinjlim_m\varprojlim_n T_{n,m,i}$ will be denoted $Ig_i$, $1 \leq i \leq 4$.
Thus, if $K^p_i, i = 1, 2, 3, 4$, are prime-to-$p$ level subgroups of $G_i(\A_f)$,  with $K^p_3 \subset K^p_4$,
$K^p_3 \subset K^p_1 \times K^p_2$, we similarly define 
Igusa varieties ${}_{K^p_i}Ig_i$ and inclusions
\begin{equation}\label{igp}  \gamma_{V_p}\circ i_3:  {}_{K^p_3}Ig_3 \rar {}_{K^p_4}Ig_4; ~~ i_4:  {}_{K^p_3}Ig_3 \rar {}_{K^p_1}Ig_1\times  {}_{K^p_2}Ig_2
\end{equation}

\subsection{Restrictions of forms}\label{restrictions}
The maps between the various moduli spaces and bundles induce
maps between spaces of modular forms.

\subsubsection{Restricting modular forms} Let $R$ be a $\Zp$-algebra and $\kap$ either a dominant character of
$T_{H_{0,4}}$ or a pair $\kap=(\kap_1,\kap_2)$ consisting of dominant characters
$\kap_1$ of $T_{H_{0,1}}$ and $\kap_2$ of $T_{H_{0,2}}$. Then the maps \eqref{Emap1} and
\eqref{Emap2} yield maps of modular forms
\begin{equation*}\label{restriction1}
res_1: M_\kap(K_4;R) \rightarrow M_\kap(K_3;R),
\end{equation*}
and
\begin{equation*}\label{restriction2}
res_2: M_{\kap_1}(K_1;R)\otimes_R M_{\kap_2}(K_2;R)\rightarrow M_{\kap}(K_3;R).
\end{equation*}
Let $\psi$ be either a $\bQ_p^\times$-valued character of $T_{H_4}\left(\Zp/p^r\Zp\right)$ or
a pair $\psi=\left(\psi_1,\psi_2\right)$ consisting of a $\bQ_p^\times$-valued character $\psi_1$ of
$T_{H_1}\left(\Zp/p^r\Zp\right)$ and $\psi_2$ of $T_{H_2}\left(\Zp/p^r\Zp\right)$. Then $\psi$ defines a
character of
$T_{H_3}\left(\Zp/p^r\Zp\right)$ that we continue to denote $\psi$. Let $R$ be a $\Zp[\psi]$-algebra.
The analogs of the maps \eqref{Emap1} and \eqref{Emap2} for
level structures at $p$ yield maps
\begin{equation}\label{restriction3}
res_3: M_\kap\left(K_{4,r},\psi;R\right) \rightarrow M_\kap\left(K_{3,r},\psi;R\right),
\end{equation}
and
\begin{equation*}
res_4: M_{\kap_1}\left(K_{1,r},\psi_1;R\right)\otimes_R M_{\kap_2}\left(K_{2,r},\psi_2;R\right)\rightarrow
M_{\kap}\left(K_{3,r},\psi;R\right).
\end{equation*}

\subsubsection{Restrictions of classical forms} In terms of the complex uniformizations
\eqref{CE-cpx}, the restrictions \eqref{Emap1} and \eqref{Emap2} correspond
to the maps induced by the canonical inclusions of $G_3$ and $H_{0,3}$ into
$G_4$ and $H_{0,4}$ and into $G_1\times G_2$ and $H_{0,1}\times H_{0,2}$, respectively.
In particular, if $\vphi: G_4(\A)\times H_{0,4}(\C)\rightarrow \C$ corresponds to a weight
$\kap$ modular form on $G_4$ of level $K_4$, then the image of $\vphi$ under
$res_1$ or $res_3$ corresponds to the restriction of $\vphi$ to $G_3(\A)\times H_{0,3}(\C)$.
Moreover, if $\vphi$ corresponds to $f:G_4(\A)\rightarrow W_{\kap,4}(\C)$ (we include the subscript `$i$' to indicate that $W_{\kap,i}$ is the
irreducible representation of $H_{0,i}$ of highest weight $\kap$), then its image under
$res_1$ or $res_3$ is just the restriction of $f$ to $G_3(\A)$ composed with the
projection $W_{\kap,4}(\C)\rightarrow W_{\kap,3}(\C)$, $\phi\mapsto \phi|_{H_{0,3}(\C)}$.
The same holds for the maps $res_2$ and $res_4$.

\subsubsection{Restrictions of $p$-adic forms}\label{rest-padic}
The maps \eqref{Igusamap1} and \eqref{Igusamap2} induce the obvious restriction maps on
modules of $p$-adic modular forms - which we also denote by $res_i$ - compatible with
weights $\kap$ and characters $\psi$ in the obvious way, as well as with the inclusion of
spaces of modular forms and with restriction to similitude components. In particular,
the isomorphisms described above extend to isomorphism of spaces of $p$-adic modular forms
(with the tensor product $\otimes_R$ replaced by the completed tensor product $\hat\otimes_R$).

\subsubsection{Base point restrictions}\label{basepointres}

Let $V = V_i$ for $i \in \{1, 2, 3, 4\}$, $G = G_{P_i}$ the corresponding unitary similitude group, so that $(G,X)$ is the Shimura datum associated to the moduli problem $P_i$.  Let $J'_0$ be a torus as in Section  \ref{basepoints}, and let $(J'_0,h_0) \rightarrow (G,X)$ be the morphism of Shimura data defined there.  Say $(J'_0,h_0)$ is {\it ordinary} if the points in the image of the map $S(J'_0,h_0) \rar S(G,X)$  of Shimura varieties reduce to points corresponding to ordinary abelian varieties.  If $(J'_0,h_0)$ is ordinary, then  it has an associated Igusa tower, denoted $T_{n,m}(J'_0,h_0)$ for all $n, m$.   We have $T_{0,m}(J'_0,h_0)  = S_m(J'_0,h_0)$, in the obvious notation, which is the reduction modulo $p^m$ of an integral model of $S(J'_0,h_0)$; each $T_{n,m}(J'_0,h_0)$ is finite over the corresponding $S_m$.

  Moreover, letting $T_{n,m}(G,X)  = T_{n,m}(P_i)$ in the obvious notation, there is a morphism of Igusa towers 
  \begin{equation}\label{CMIgusa}    T_{n,m}(J'_0,h_0) \rar T_{n,m}(G,X).  \end{equation}
 
Thus for any $r$ there are restriction maps $res_{J'_0,h_0}:  M_{\kap}(K_{i,r},R) \rar M_{\kap}((J'_0,h_0),R)$, in the obvious notation; the image is contained in forms of level $r$ on  $S(J'_0,h_0)$, in an appropriate sense, but we don't specify the level.
 The restriction maps behave compatibly with respect to classical, complex, and $p$-adic modular forms; the restriction map for $p$-adic modular forms is denoted $res_{p,J'_0,h_0}$.  In order to formulate a precise statement, we write $V_{\kap_p}((G,X);K^p,R)$ for $p$-adic modular forms of weight $\kap_p$ and level $K^p$ on the Igusa tower for $S(G,X)$ (resp. $V_{\kap_p}^{\ord}((G,X);K^p,R)$ for the image of the ordinary projector) and $V_{\kap_p}((J'_0,h_0),R)$ for the corresponding object for $S(J'_0,h_0)$ (the level away from $p$ is not specified).

\begin{prop}\label{CMrest}  Let $(J'_0,h_0) \rar (G,X)$ be a morphism of Shimura data, with $J'_0$ a torus, and suppose $(J'_0,h_0)$ is ordinary.   Let $G = G_{P_i}$ for $i = 1, 2, 3, 4$, and let $\kappa$ be a dominant weight; let $\kappa_p$ be the corresponding $p$-adic weight, as in  \ref{padicweight}.  Let $R$ be a $p$-adic ring.

(i)  The following diagram is commutative:
$$\begin{CD}
M_{\kap}(K_{i,r},R)  @>R_{\kap,G,X} >>   V_{\kap_p}(K_i^p,R) \\
@Vres_{J'_0,h_0}VV @Vres_{p,J'_0,h_0}VV\\
M_{\kap}((J'_0,h_0),R) @>R_{\kap,J'_0,h_0} >> V_{\kap_p}((J'_0,h_0),R)
\end{CD}$$
Here the horizontal maps are the ones defined in \eqref{padicclassical}

(ii)  Let $f \in V_{\kap_p}^{\ord}((G,X);K^p,R)$.  Suppose for every ordinary CM pair $(J'_0,h_0)$ mapping to $(G,X)$, the restriction $res_{J'_0,h_0}(f) = 0$.  Then $f = 0$.
\end{prop}

\begin{proof}  Point (i) is an immediate consequence of the definitions.  Note that $V_{\kap_p}^{\ord}((G,X);K^p,R)$ is the $R$-dual to a finite rank module over an appropriate Iwasawa algebra, by Hida theory; thus $f$ belongs to a finite rank $R$ module.  Point (ii) follows from the Zariski density of the ordinary locus in the integral model of $S(G,X)$ \cite{wedhorn}, provided we know there are enough CM points in the ordinary locus.  But every point on the ordinary locus over a finite field lifts (with its polarization, endomorphisms, and prime-to-$p$ level structure) to the generic fiber (see for example \cite[Proposition 2.3.12]{moonen}).
\end{proof}

\section{Eisenstein series and zeta integrals}\label{ESeriesZIntegrals-section}

\subsection{Eisenstein series and the doubling method}\label{doubling16}

We begin this section by introducing certain Eisenstein series and (global) zeta functions.  Then we choose specific local data and compute local zeta integrals (whose product gives the global zeta function).  

We assume throughout this section that we are in the setting of Section  \ref{doublingsetup}. In particular,
there is a hermitian pair $\left(V,\pair_V\right)$ over $\K$ such that $V=L_1\otimes\Q$ and
$\pair_1 = \trace_{\K/\Q}\delta\pair_V$. Then $G_1/\Q$ is the unitary similitude group
of the pair $\left(V,\pair_V\right)$. Let $\left(W,\pair_W\right)$ be the hermitian pair with $W=V\oplus V$ and
$\pair_W = \pair_V\oplus -\pair_V$. Then $G_4/\Q$ is the unitary similitude group of
the pair $\left(W,\pair_W\right)$. Most of the constructions to follow take place on the group $G_4/\Q$,
which we denote throughout by $G$ for ease of notation.  We write $Z_i$ to denote the center of $G_i$.

An important observation is that $G_2(\A)=G_1(\A)$, so a function or representation of one of these groups
can be viewed as a function or representation of the other; we use this repeatedly.

In part to aid with the comparison with calculations in the literature, we introduce the unitary
groups $U_i = \ker\left(\nu:G_i\rightarrow \G_m\right)$.

Let $n=\dim_\K V$.
Let $S_0$ be the set of primes dividing either
the discriminant of the pairing $\pair_1$ or the discriminant of $\K$.

\subsubsection*{Plan of this section}
We begin by recalling the general setup for Siegel Eisenstein series on $G$
and the zeta integrals in the context of the doubling method, explaining how the global integral factors as a product over primes of $\K$.  The local factors fall into three classes, which are treated in turn.  The factors at non-archimedean places prime to $p$ are the easiest to address:  in Section  \ref{localvnotp} we recall the unramified factors, which have been known for more than 20 years, and explain how to choose data at ramified places to trivialize the local integrals.   

Factors at primes dividing $p$ are computed in Section  \ref{pchoices-section}.  This is the most elaborate computation in the paper.  The local data defining the Eisenstein series have to be chosen carefully to be compatible with the $p$-adic Eisenstein measure which is recalled in Section   \ref{emeas-section}.   The local data for the test forms on $G_3$ are chosen to be {\it anti-ordinary vectors}, a notion that will be defined explicitly in \ref{aordinarylocal}, and that provide the local expression of the fact, built into Hida theory, that the test forms are naturally {\it dual} to ordinary forms.   The result of the computation is given in Theorem \ref{GJintegrals}:  we obtain $p$-stabilized Euler factors, as predicted by conjectures
of Coates and Perrin-Riou.

Sections \ref{holo-sec} and \ref{archimedeanchoices} are devoted to the local integrals at archimedean places.  Much of the material here is a review of the theory of holomorphic differential operators developed elsewhere, and of classical invariant theory.  We prove in particular (Proposition \ref{archnonvan}) that the archimedean zeta integrals do not vanish; as explained in the introduction, in most cases we do not know explicit formulas for these integrals.

\subsubsection{The Siegel parabolic}\label{SiegelParabolic-section}
Let $V^d = \left\{(x,x)\in W\ : x\in V\right\}$ and $V_d = \left\{(x,-x)\in W\ : \ x\in V\right\}$, so
$W=V_d\oplus V^d$ is a polarization of $\pair_W$.
Projection to the first summand fixes identifications of $V^d$ and $V_d$
with $V$.   Recall that $G$ acts on the right on $W$.  Let $P\subset G$ be the stabilizer
of $V^d$; this is a maximal $\Q$-parabolic, the Siegel parabolic.
Let $M\subset P$ be the stabilizer of the polarization $W=V_d\oplus V^d$
and $N\subset P$ the group fixing both $V^d$ and $W/V^d$, so
$M$ is a Levi subgroup and $N$ the unipotent radical.
Denote by $\Delta$ the canonical map $\Delta:P\rightarrow \GL_\K(V^d)=\GL_\K(V)$.
Then $M\isoarrow \GL_\K(V)\times \G_m$, $m\mapsto (\Delta(m),\nu(m))$; the inverse
map is $(A,\lambda)\mapsto m(A,\lambda) = \diag(\lambda \left(A^*\right)^{-1},A)$, where
$A^* = \t A^{c}$ is the transpose of the conjugate under the action of $c$.  Also, fixing a basis for $V$ gives an identification $\Delta': N\isoarrow \hern \left(\K\right)$, where $\hern$ denotes the space of $n\times n$ hermitian matrices; with respect to this basis and the polarization above, we obtain an identification $N\isoarrow\begin{pmatrix}1_n &\Delta'(N)\\ 0 & 1_n\end{pmatrix}\subseteq \gl_{2n}\left(\K\right)$.

Define $\delta_P(\cdot)=|\det\circ\Delta(\cdot)|^{n}$.

\subsubsection{Induced representations}\label{induced} Let $\chi=\otimes\chi_w$ be a character of
$\K^\times\backslash \A_\K^\times$. For $s\in\C$ let
$$
I\left(\chi,s\right) = \Ind_{P(\A)}^{G(\A)}\left(\chi\left(\det\circ\Delta(\cdot)\right)\cdot\delta_P^{-s/n}\left(\cdot \right)\cdot \left|\nu(\cdot)\right|^{sn/2}\right),
$$
with the induction smooth and unitarily normalized. This factors as a restricted tensor product
$$
I(\chi,s) = \otimes_{v} I_v(\chi_v,s),
$$
with $v$ running over the places of $\Q$, $I_v(\chi_v,s)$ the analogous local induction from
$P(\Q_v)$ to $G(\Q_v)$, and $\chi_v = \otimes_{w\divides v}\chi_w$.

\subsubsection{Eisenstein series} For $f = f_{\chi,s} \in I(\chi,s)$ we form the {\it standard (non-normalized)} Eisenstein series,
$$
E(f,g) = \sum_{\gamma\in P(\Q)\backslash G(\Q)} f(\gamma g).
$$
If $Re(s)$ is sufficiently large, this converges absolutely and uniformily on compact subsets
and defines an automorphic form on $G(\A)$. 
Given a unitary character $\chi$ and a {\it Siegel--Weil section} $f \in I(\chi, s)$, \footnote{Siegel-Weil sections are the functions used to define the Siegel-Weil Eisenstein series, as in \cite{Harris-rationality}} we put
\begin{align*}
f_s := f_{\chi, s} := f\\
E_f(s, g) : = E_{f_s}(g).
\end{align*}
When $f_s\in I(\chi,s)$ is $K$-finite for a maximal compact subgroup $K\subset G$, the Eisenstein
series $E_f(s,g)$ have a meromorphic continuation in $s$.

\subsubsection{Zeta integrals}\label{zetaintegral}
In this section, we briefly summarize key details of the doubling method, which we use to obtain zeta integrals.  The doubling method holds for general classes of cuspidal automorphic representations of $G_1(\adeles)$.  In order to carry out our full $p$-adic interpolation, however, our approach requires us (later in the paper) to place additional conditions on the representations $\pi$ that we consider.  In particular, we will need $\pi$ to be contained in certain induced representations for places dividing $p$ and in certain discrete series representations for places dividing $\infty$.

Denote by $\Oe$ the ring of integers of $\K^+$.  For $i = 1, 2, 3, 4,$ we write 
$U_i(\adeles)=  \prod'_{v}U_{i, v}$, with the (restricted) products over all the places of $\K^+$ and 
$U_{i, v}$ the points of groups defined over $\Oev$.  Similarly, we write $G(\adeles) = G_{\infty}\times \prod'_q G_q$ and $P(\adeles) = P_{\infty}\times \prod'_q P_q$, where the (restricted) products are over 
rational primes $q$.  We can nevertheless write
$$G_p = \Qp^{\times} \times \prod_{w \in \Sigma_p} G_w; P_p = \Qp^{\times} \times \prod_{w \in \Sigma_p} P_w.$$

Let $\pi$ be an irreducible cuspidal automorphic representation of $G_1(\adeles)$, and let $\pi^{\vee}$ be its contragredient; it is a twist of the complex conjugate of $\pi$ and is therefore also a cuspidal automorphic representation.
Let $S_\pi$ be the set of finite primes $v$ in $\Oe$ for which $\pi_v$ is ramified.
Before introducing the zeta integral for $\pi$, we would like to explain what it means for a function in $\pi$ to be factorizable
over places in $\K^+$.  However, $G_1$ is a $\Q$-group that is not the restriction of scalars of a group over $\K^+$.
We therefore choose an irreducible $U_1(\adeles)$-constituent $\underline{\pi} \subset \pi$ that occurs in the space of automorphic forms on
$U_1$, and denote the dual $\underline{\pi}^{\flat}$; note that $\underline{\pi}^{\flat}$ occurs inside the
restriction of $\pi^{\vee}$ to $U_1(\adeles)$.  
We assume $\underline{\pi}$ contains the spherical vectors for $K^{S_\pi}$.  It is well-known (and follows from the unfolding computation recalled below) that {\it the standard $L$-function does not depend
on this choice}.   We fix non-zero unramified vectors $\varphi_{w,0}$ and $\varphi'_{w,0}$ in $\pi_w$ and $\pi_w^\vee$, respectively,
for all 
finite places $w$ outside $S_{\pi}$, and choose factorizations as in  \eqref{factor} compatible with the unramified choices:
\begin{equation}\label{facpi}
\underline{\pi} \isoarrow  \underline{\pi}_{\infty}\otimes \underline{\pi}_f;~~ \underline{\pi}_f \isoarrow \underline{\pi}^{S_{\pi},p}\otimes \underline{\pi}_p\otimes \underline{\pi}_{S_\pi};  
\underline{\pi}_p \isoarrow \otimes_{w \mid p} \underline{\pi}_w; ~  \underline{\pi}_{S_{\pi}} \isoarrow \otimes_{w \in S_\pi} \underline{\pi}_w;
\end{equation}
and analogous factorizations for $\underline{\pi}^{\flat}$.   
We also think of $\underline{\pi}^{\flat}$ as an anti-holomorphic automorphic representation
of $G_2$.  Let $\varphi \in \underline{\pi}^{K^{S_{\pi}}}$,
$\varphi^{\flat} \in \underline{\pi}^{\flat,K^{S_{\pi}}}$; we think of $\varphi$ and $\varphi^{\flat}$ as forms on $G_1$ and $G_2$, respectively.  We suppose they decompose as tensor products with respect to the above
factorizations:
\begin{equation}\label{facphi}  \varphi = \otimes_v \varphi_v;~~ \varphi^{\flat} = \otimes_v \varphi^{\flat}_v \end{equation}
with $\varphi_v$ and $\varphi^{\flat}_v$ equal to the chosen $\varphi_{v,0}$ and $\varphi'_{v,0}$ when $v \notin S_{\pi}$.
We write equalities but the formulas we write below depend on the factorizations in \eqref{facpi} and its counterpart for  $\underline{\pi}^{\flat}$.

In Sections \ref{pchoices-section} (resp. \ref{holo-sec}-\ref{archimedeanchoices}), we will choose specific local components at primes dividing $p$ (resp. at archimedean places).  These will turn out to be anti-ordinary (resp. anti-holomorphic) vectors:  
\begin{equation}\label{facphip}  \varphi_p := \otimes_{w \mid p} ~\varphi_w = \otimes_{w \mid p} ~\phi^{\aord}_{w,r,\pi_w};  \varphi^{\flat}_p := \otimes_{w \mid p} ~\varphi^{\flat}_w = \otimes_{w \mid p} ~\phi^{\aord}_{w,r,\pi^{\flat}_w} 
\end{equation}
and
\begin{equation}\label{facphiinf} \varphi_\infty := \otimes_{\sigma \mid \infty} ~\varphi_\sigma = \varphi_{\kappa_{\sigma},-};  \varphi^{\flat}_\infty := \otimes_{\sigma \mid \infty} ~\varphi^{\flat}_\sigma
= \varphi_{\kappa^{\flat}_{\sigma},-}. \end{equation}

\begin{rmk}\label{conditions} The meaning of the notation in \eqref{facphip} and \eqref{facphiinf} will be explained in Sections \ref{aordinarylocal} and \ref{aordinary-localrep-2} and Section \ref{aholovectors}, respectively.  Here we just note that 
\begin{itemize}
\item[(i)] $\varphi_p$ (resp. $\varphi^{\flat}_p$) is anti-ordinary with respect to a group $I_r$  (resp. $I_r^{\flat}$), in the sense to be described  in Section \ref{aordinarylocal} (resp. \ref{aordinary-localrep-2}), and
\item[(ii)] $r$ is an integer chosen in \eqref{dgeq2t} to be sufficiently large.  
\end{itemize}
\end{rmk}

Having made the choice of irreducible constituent $\underline{\pi}$, we will henceforth forget about the choice.  In order not to make the notation too difficult to read, we will use $\pi$ and $\pi^\flat$ to denote irreducible $U_1(\adeles)$ representations, but we will mean irreducible constituents of the restrictions of representations of $G_1$.   

We also fix local $U_1\left(\K^+_v\right)$-invariant pairings
$\pair_{\pi_v}:\pi_v\times\pi_v^\vee\rightarrow \C$ for all $v$ such that
$\langle \varphi_{v,0}, \varphi_{v,0}^{\flat} \rangle_{\pi_v}=1$ for all $v \notin S_\pi$.

Let $f = f_s(\bullet) \in I(\chi,s)$.
Let $\varphi\in\pi$ and $\varphi^{\flat}\in \pi^{\flat}$ be factorizable vectors as above.  The zeta integral for $f$,$\varphi$, and $\varphi^{\flat}$ is
$$
I(\varphi,\varphi^{\flat},f,s) = \int_{Z_3(\adeles)G_3(\Q)\backslash G_3(\A)} E_f(s,(g_1,g_2))\chi^{-1}(\det g_2)\varphi(g_1)\varphi^{\flat}(g_2)d(g_1,g_2).
$$
By the cuspidality of $\varphi$ and $\varphi^\flat$ this converges absolutely for those values of $s$
at which $E_f(s,g)$ is defined and defines a meromorphic function in $s$ (holomorphic
wherever $E_f(s,g)$ is).   Moreover, it follows from the unfolding in \cite{GPSR} that, for any pair $\pi, \pi'$ of automorphic representations of $G_1$, the map
$$(\varphi,\varphi') \mapsto I(\varphi,\varphi',f,s)$$ defines a $G_1(\adeles)$-invariant pairing
between $\pi$ and $\pi'$.  The multiplicity one hypothesis \ref{multone} implies that the space of such pairings is at most one-dimensional, and is exactly one-dimensional provided $\pi' = \pi^{\vee}$ (upon restriction to $U(V)$.   Thus:

If $\langle \varphi,\varphi^{\flat}\rangle = 0$ then $I(\varphi,\varphi^{\flat},f,s) = 0$ for all $s$.

Here $\langle \bullet, \bullet \rangle$ denotes the standard $L_2$ pairing for the Tamagawa 
measure $dg$.  So we suppose $\langle \varphi,\varphi^{\flat}\rangle \neq 0$.  Then 
$\langle \varphi_{v}\otimes \varphi_{v}^{\flat}\rangle_{\pi_v} \neq 0$ for all $v$.
For $Re(s)$ sufficiently large, `unfolding' the Eisenstein series then yields
$$
I(\varphi,\varphi^{\flat},f,s) = \int_{U_1(\A)} f_s(u,1)\langle \pi(u)\varphi,\varphi^{\flat}\rangle_\pi du.
$$
Denote by $f_U$ the restriction of $f$ to $U_4(\adeles)$.  Henceforward we assume $f_U(g)=\otimes_v f_v(g_v)$ with
$$f_v = f_{v,s} \in I_v(\chi_v,s), \chi_v = \otimes_{w\divides v}\chi_w.$$
Then the last expression for $I(\varphi,\varphi^{\flat},s)$ factors as
\begin{equation}\label{normalizedzeta}
\begin{split}
I(\varphi,\varphi^{\flat},f,s) = \prod_v I_v(\varphi_v,\varphi_v^{\flat},f_v,s) \cdot \langle \varphi, \varphi^{\flat} \rangle, {\rm ~~where} \\
I_v(\varphi_v,\varphi_v^{\flat},f_v,s) = \frac{\int_{U_{1, v}} f_{v,s}(u,1)\langle \pi_v(u)\varphi_v,\varphi^{\flat}_v\rangle_{\pi_v} du}{\langle \varphi_v,\varphi^{\flat}_v \rangle_{\pi_v}}.
\end{split}
\end{equation}
By hypothesis, the denominator of the above fraction equals $1$ whenever $v \notin S_\pi$.
We denote the integral in the numerator by $Z_v(\vphi_v,\vphi_v^\flat,f_v,s)$.

\begin{rmk}\label{chitwist}  In order to define the proper arithmetically normalized period, we want to consider the vectors $\varphi \in \pi$ and 
$$\varphi^\flat_\chi := \varphi^\flat\otimes \chi^{-1}\circ\det \in \pi^\flat \otimes \chi^{-1}\circ\det$$
as the input to the integral.  Thus for $\varphi_1 \in \pi$, $\varphi_2 \in \pi \otimes \chi^{-1}\circ\det$ we define
$\langle \varphi_1,\varphi_2 \rangle_\chi$ to be the $L_2$ pairing of $\varphi_1$ with $\varphi_2\otimes \chi\circ \det$.  Then
$$\langle \varphi,\varphi_\chi^\flat \rangle_\chi = \langle \varphi,\varphi^\flat \rangle.$$
It seems as if nothing has changed, but in the final result we will be taking vectors of the form $\varphi \otimes \varphi_\chi^\flat$ that are integral with respect to a structure native to the space $\pi\bigotimes \pi \otimes \chi^{-1}\circ\det$, and the resulting periods will differ by CM periods.  
\end{rmk}

As in Section  \ref{unitarygroupPEL}, let $\Sigma = \{\sigma\in\Sigma_\K\ : \ \grp_\sigma\in\Sigma_p\}$.
This is a CM type for $\K$.  Throughout the remainder of this section, we take $\chi:\K^\times\backslash \A_\K^\times\rightarrow \C^\times$ to be a unitary character such that 
$\chi_\infty = \otimes_{\sigma\in\Sigma}\chi_\sigma$ is given by
\begin{equation}\label{chi-infinity}
\chi_\infty\left(\left(z_\sigma\right)\right)= \prod_{\sigma\in\Sigma}
z_\sigma^{-\left(k_\sigma+2\nu_\sigma\right)}\left(z_\sigma\overline{z_\sigma}\right)^{\frac{k_\sigma}{2}+\nu_\sigma}, \ \ \
(z_\sigma)\in\prod_{\sigma\in \Sigma}\C^\times,
\end{equation}
where $k = \left(k_\sigma\right)\in\Z^\Sigma_{\geq 0}$, and $(\nu_\sigma)\in \Z^{\Sigma}$.

For the remainder of this section, we choose specific local Siegel--Weil sections $f_v\in I_v(\chi_v,s)$ and
compute the corresponding local zeta integrals (whose product is the Euler product of the global zeta function discussed at the beginning of this section).

\subsection{Local zeta integral calculations at nonarchimedean places $v\ndivides p$}\label{localvnotp}

Let $S_\ram=S_\pi\cup S_\chi\cup S_\K$, where $S_\chi$ denotes the set of finite primes $v$ in $\Oe$ for which $\chi_v= \otimes_{w\divides v}\chi_w$ is ramified and $S_\K$ denotes the set of finite primes in $\Oe$ that ramify in $\K$.  Let $S$ be a finite set of finite primes in $\IQ$ such that $p\nin S$ and such that for all rational primes $\ell$, if a prime in $\K^+$ above $\ell$ is in $S_\ram$, then $\ell\in S$.  Let $S'$ be the set of primes of $\K^+$ lying above the primes of $S$.  In particular, $S'$ contains $S_{\mathrm{ram}}$.

\subsubsection{Unramified case}\label{unrameuler}
For the moment, assume that $\ell\neq p$ is a finite place of $\IQ$ such that $\ell\not\in S$. Then $K_\ell:=G_4\left(\Z_\ell\right)$ is a hyperspecial maximal
compact of $G\left(\Ql\right) = G_4\left(\Ql\right) = \prod_{v\divides \ell}G_{4, v}$, and we choose $f_\ell = \otimes_{v\divides \ell} f_v\in I_\ell(\chi_\ell,s)$ to be the unique $K_\ell$-invariant function
such that $f_\ell(K_\ell)=1$. These sections are used to construct the Eisenstein measure in
\cite{apptoSHL}.
For each prime $v\nin S'$, let $\varphi_v\in\pi_v$ and $\varphi_v'\in\pi_v^\vee$ be the normalized spherical vectors such that $\langle\varphi_v, \varphi_v'\rangle_{\pi_v} = 1$.   The primes $v\nin S'$ fall into two categories: split and inert.  For split places $v\nin S'$, $U_{1, v}\cong \gln\left(\K^+_v\right)$; the zeta integral computations in this case reduce to those in \cite{jacquet} and \cite[Section  6]{GPSR}.  For inert places $v\nin S'$, the computations were completed in \cite[Section  3]{JSLi}.  In either case, we have
\begin{align*}
d_{n,v}\left(s, \chi_v\right)I_v\left(\varphi_v, \varphi_v', f_v, s\right) = L_v\left(s+\frac{1}{2}, \pi_v, \chi_v\right),
\end{align*}
where\footnote{From the formula for $d_{n,v}(s)$ given in \cite[Section  3]{JSLi}, it appears that there is a typographical error in the exponent in the formula for $d_{n, v}$ given in \cite[Equation (3.1.2.5)]{HLS}.  More precisely, according to the final formula in \cite[Theorem 3.1]{JSLi}, the $n-1$ should not appear in the exponent in \cite[Equation (3.1.2.5)]{HLS}.} 
\begin{align}\label{dnv-equ}
d_{n, v}\left(s, \chi_v\right) = d_{n, v}\left(s\right)=\prod_{r=0}^{n-1}L_v\left(2s+n-r, \chi_v\mid_{\K^+}\eta_v^r\right),
\end{align}
$\eta_v$ is the character on $\K^+_v$ attached by local class field theory to the extension $\K_w/\K^+_v$ (where $w$ is a prime of $\K$ lying over $v$),
and $L_v\left(s, \pi_v, \chi_v\right)$ denotes the value at $s$ of the standard local Langlands Euler factor attached to the unramified representation $\pi_v$ of $U_{1, v}$, the unramified character $\chi_v$ of $\K_v$, and the standard representation of the $L$-group of $U_{1, v}$.  As noted on \cite[p. 439]{HLS}\footnote{There is a typographical error on \cite[p. 439]{HLS}.  Although \cite[p. 439]{HLS} gives a base change to $\gl_m$, the base change should actually be to $\gln$.}, for each $v\nin S'$,
\begin{align*}
L_v\left(s, \pi_v, \chi_v\right) = L_v\left(s, \mathrm{BC}\left(\pi_v\right)\otimes\chi_v\circ\det\right),
\end{align*}
where $\mathrm{BC}$ denotes the local base change from $U_{1, v}$ to $\gln\left(\K_v\right)$ and the right hand side is the standard Godement-Jacquet Euler factor.

\subsubsection{Ramified case}\label{nonarchchoices-section}

Now, assume that $\ell\in S$, and let $v\in S'$ be a prime lying over $\ell$.  By \cite[p. 439]{HLS}, $P_v\cdot \left(U_{1, v}\times 1_n\right)\subseteq P_v\cdot U_{3, v}$ is open in $U_{4, v}$.   
Since the big cell $P_vwP_v$ is also open in $U_{4, v}$, we see that $\left(P_v\cdot (U_{1, v}\times 1_n)\right)\cap P_vwP_v$ is open in $U_{4, v}$.  As noted in \cite[Equation (3.2.1.5)]{HLS}, $P_vw = P_v\cdot \left(-1_n, 1_n\right)\subseteq P_v\cdot U_{3, v}$ and $P_v\cap \left(U_{1, v}\times 1_n\right) = \left(1_n, 1_n\right)\in U_{3, v}$.  Therefore $\left(P_v\cdot \left(U(V)\times 1_n\right)\right)\cap P_vwP_v$ is an open neighborhood of $w$ in $P_vwP_v$ and hence is of the form $P_vw \mathfrak{U}$ for some open subset $\mathfrak{U}$ of the unipotent radical $N_v$ of $P_v$.  Let $\varphi_v\in\pi_v$ and $\varphi'_v\in\pi_v'$ be such that $\langle \varphi_v, \varphi_v'\rangle_{\pi_v}=1$.  Let $K_v$ be an open compact subgroup of $G_{1, v}$ that fixes $\varphi_v$.  

For each place $v\in S'$, let $L_v$ be a small enough lattice so that $\mathfrak{U}_v$ contains the open subgroup $N(L_v)$ of $N_v$ defined by
\begin{align*}
N(L_v) = \left\{\begin{pmatrix}1_n & x\\ 0 & 1_n\end{pmatrix} \mid x\in L_v\right\}
\end{align*}
(where we identify $N$ with $\Delta'(N)$ as in Section  \ref{SiegelParabolic-section}) and so that 
\begin{align*}
P_vwN(L_v) \subseteq P_v\cdot \left(-1_n\cdot K_v\times 1_n\right)\subseteq P_v\cdot U_{3, v}.
\end{align*}
Then
\begin{align*}
P_vwN(L_v) = P_v\cdot \left(\mathcal{U}_v\times 1_n\right)
\end{align*}
for some open neighborhood neighborhood $\mathcal{U}_v$ of $-1_n$ contained in the open subset $-1_n\cdot K_v$ of $U_{1, v}$.  Let $\delta_{L_v}$ denote the characteristic function of $N(L_v)$.  As explained on \cite[pp. 449-50]{HLS}, for each finite place $v$ of $\K^+$, there is a Siegel section $f_{L_v}$ supported on $P_vwP_v$ such that
\begin{align*}
f_{L_v}\left(w x\right) = \delta_{L_v}(x)
\end{align*}
for all $x\in N_v$.

For each of the primes $v\in S'$, we define a local Siegel section $f_v\in I(\chi_v, s)$ by
\begin{align*}
f_v& = f_{L_v}^-,
\end{align*}
where
\begin{align*}
f_{L_v}^-(g) = f_{L_v}\left(g\cdot (-1, 1)\right)
\end{align*}
for all $g\in U_{4, v}$.  (Note that $f_{L_v}^-$ is just a translation by $(-1, 1)\in U_{3, v} = U_{1, v}\times U_{2, v} = U_{1, v}\times U_{1, v}$ of local Siegel sections discussed in \cite[Sections (3.3.1)-(3.3.2)]{HLS} and that, where nonzero, the Fourier coefficients associated to $f_{L_v}^-$ are the same as the Fourier coefficients associated to similar Siegel sections discussed in \cite[Section  2.2.9]{apptoSHL} and \cite{sh}.  Therefore, this minor modification of the choice of Siegel sections in \cite{HLS, apptoSHL, sh} will not affect the $p$-adic interpolation of the $q$-expansion coefficients of the Eisenstein series that is necessary to construct an Eisenstein measure.)

\begin{lem}
Let $v\in S'$, and let $f_v = f_{L_v}^-.$  Then
\begin{align*}
I_v\left(\varphi_v, \varphi^\prime_v, f_v, \chi\right) = \mathrm{volume}\left(\mathcal{U}_v\right).
\end{align*}
\end{lem}
\begin{proof}
The support of $f_{L_v}^-$ in $U_{1, v}\times 1_n$ is $-1_n\cdot\mathcal{U}_v\times 1_n$, and for $g\in U_{1, v}\times 1_n$, 
\begin{align*}
f_{L_v}^{-}(g) = \delta_{-1_n\cdot\mathcal{U}_v\times 1_n}(g)
\end{align*}
where $\delta_{-1_n\cdot\mathcal{U}_v\times 1_n}$ denotes the characteristic function of $-1_n\cdot\mathcal{U}_v\times 1_n$.  Since $\pi_v(g)\varphi_v = \varphi_v$ for all $g\in K_v\supseteq -1_n\cdot\mathcal{U}_v$, we therefore see that
\begin{align*}
I_v\left(\varphi_v, \varphi^\prime_v, f_v, \chi\right) &= \frac{\int_{-1_n\cdot\mathcal{U}_v}\langle\varphi_v, \varphi_v^\prime\rangle_{\pi_v} dg}{\langle\varphi_v, \varphi_v^\prime\rangle_{\pi_v} }\\
& = \mathrm{volume}\left(\mathcal{U}_v\right).
\end{align*}
\end{proof}

\subsection{Local zeta integral calculations at places dividing $p$}\label{pchoices-section}

\subsubsection*{Plan of this section}  We begin by choosing local Siegel--Weil sections at the primes $w$ dividing $p$ that are compatible with the Eisenstein measure, and then turn to choosing test vectors (anti-ordinary vectors) in the local representations $\pi_w$ and $\pi^{\flat}_w$.   The last six pages or so contain explicit matrix calculations that reduce the zeta integral to a product of integrals of Godement-Jacquet type, which can then be computed explicitly.    

The reader may observe that the representations $\pi_w$ and $\pi^{\flat}_w$, like the automorphic representations of which they are local components, are logically prior to the local Siegel--Weil sections, inasmuch as our goal is to define $p$-adic $L$-functions of (ordinary) families and the Eisenstein measure is a means to this end.  One of the subtleties of this construction is that a global automorphic representation $\pi$ automatically picks out the function whose integral is the desired value of the Eisenstein measure.  This is unfortunately concealed in the technical details of the construction, but the reader should be able to spot the principle at work in Section  \ref{classpairings}.

The calculations presented here are more general than those needed for our construction of the $p$-adic $L$-functions of ordinary families.  The $p$-adic place $w$ is assigned to an archimedean place $\sigma$ and thus to a signature $(a_w,b_w)$ of the unitary group at $\sigma$; but we also introduce partitions of $a_w$ and $b_w$.  These partitions can be used to study the variation of $p$-adic $L$-functions in {\it $P$-ordinary families}, where $P$ is a parabolic subgroup of $G_1(\Qp)$.  However, this application has been postponed in order not to make the paper any longer than it already is, and we restrict our attention to the usual ordinary families, corresponding to $P = B$ a Borel subgroup.

\subsubsection{Definition of the Siegel--Weil sections}\label{SWsections}

 With a few minor changes,
the description of the Siegel--Weil section at $p$ given below is the same as in \cite{apptoSHL, apptoSHLvv}.
For $w|p$ a place of $\K$ and $U$ a $\K$-space we let $U_w=U\otimes_\K\K_w$.

To describe the section $f_p$ we make use of the isomorphisms \ref{G-iso}. The isomorphism for $G_4$
identifies $G(\Qp)$ with $\Qptimes\times\prod_{w\in\Sigma_p}\GL_{\K_w}(W_w)$ and $P(\Qp)$ with
$\Qptimes\times\prod_{w\in\Sigma_p}P_n(\K_w)$ with $P_n\subset\GL_\K(W)$ the parabolic
stabilizing $V^d$. So $M(\Qp)$ is identified with $\Qptimes\times\prod_{w\in\Sigma_p}
\GL_{\K_w}(V_{d,w})\times \GL_{\K_w}(V^d_w)$ (the factors embedded diagonally in $\GL_{\K_w}(W_w)$), 
and $N(\Qp)$ is identified with $\prod_{w\in\Sigma_p} N_n(\K_w)$ 
with $N_n\subset P_n$ the unipotent radical.

For $w\in\Sigma_p$ let $\chi_{w,1}=\chi_w$ and $\chi_{w,2} = \chi_{\bar w}^{-1}$, where we identify
$\K_w = \K^+_{w^+} = \K_{\bar w}$ and where $w^+ = w|_{\K^+} = \bar w|_{\K^+}$.
The pair $\left(\chi_{w,1}, \chi_{w,2}\right)$ determines a character
$$
\psi_w:P_n(\K_w)\rightarrow \C^\times,  \ \ \
\psi_w(\left(\smallmatrix A & B \\ 0 & D \endsmallmatrix \right))=
\chi_{w,1}\left(\det D\right)\chi_{w,2}\left(\det A\right).
$$
Here we have written an element of $P_n$ with respect to the direct sum decomposition $W = V_d\oplus V^d$.
We put
\begin{align*}
\psi_{w, s} = (\left(\smallmatrix A & B \\ 0 & D \endsmallmatrix \right))=
\chi_{w,1}\left(\det D\right)\chi_{w,2}\left(\det A\right) \left|A^{-1}D\right|_w^{-s}
\end{align*}
Given $\otimes_{w\in\Sigma_p} f_{w, s}\in \otimes_{w\in\Sigma_p}\Ind_{P_n(\K_w)}^{\GL_{\K_w}\left(W_w\right)}\left(\psi_{w, s}\right)$,
we set
\begin{align}\label{fpproductdefinition}
f_{p, s}(g) = \left|\nu\right|_p^{sn/2}\otimes_{w\in\Sigma_p}f_{w, s}(g_w), \ \ \ g=(\nu,(g_w))\in G(\Qp).
\end{align}
Then, as explained in \cite{apptoSHL}, $f_p\in I_p(\chi_p,s)$.

The choice of a level structure at $p$ for the PEL datum $P_1$ amounts to choosing
an $\O_w$-basis of $L_{1,w}$, and hence a $\K_w$-basis of $V_w$, for each $w\in\Sigma_p$.
This then determines a $\K_w$-basis of $V^d_w$ and $V_{d,w}$, via their identifications with $V_w$,
and hence a $\K_w$-basis\footnote{This is not in general the basis 
corresponding to the the level structure for $P_4$ determined by that for $P_1$.}
of $W_w = V_{d,w}\oplus V^d_w$.  This basis identifies
$\Isom_{\K_w}(V^d_w, V_w)$, $\Isom_{\K_w}(V_{d,w}, V_w)$, and an ordered choice of this  basis identifies $\GL_{\K_w}(V_w)$ with $\GL_n(\K_w)$.
This ordered basis also identifies $\GL_{\K_w}(W_w)$ with $\GL_{2n}(\K_w)$, $P_n(\K_w)$ with the subgroup of 
upper-triangular $n\times n$-block matrices and $M_n(\K_w)$ with the subgroup of diagonal $n\times n$-block
matrices.

Let $w\in \Sigma_p$. To each Schwartz function
$\Phi_w: \Hom_{K_w}(V_w, W_w)\rightarrow \C$ (so $\Phi_w$ has compact support),
we attach a Siegel--Weil section
$f^{\Phi_w}\in \Ind_{P_n(\K_{w})}^{\GL_{2n}(\K_{w})}\psi_{w, s}$
as follows.
Consider the decomposition
\begin{align*}
\Hom_{\K_w}(V_w, W_w) = \Hom_{\K_w}(V_w, V_{d, w})\oplus\Hom_{\K_w}(V_w, V^d_w), \ \ \
X = (X_1, X_2).
\end{align*}
Let
\begin{align*}
\mathbf{X}= \left\{X\in\Hom_{\K_w}(V_w, W_w)|X(V_w) = V_w^d\right\} =
\left\{(0, X)|X: V_w\isoarrow V^d_w\right\}.
\end{align*}
For $X\in \mathbf{X}$, the composition
$V_w\xrightarrow{X} V_w^d\isoarrow V_w$, where the last arrow comes from the fixed
identification of $V^d$ with $V$,
is an isomorphism of $V_w$ with itself.
This identifies $\mathbf{X}$ with $\GL_{\K_w}(V_w)$.

We define the section 
$f^{\Phi_w}\in \Ind_{P_n(\K_w)}^{\GL_{2n}(\K_w)}\psi_{w, s}$ 
by\footnote{The minor difference between the definitions of the Siegel section at $p$ in Equation 
\eqref{sectionininducedrepnatp} in this paper and in \cite[Equation (21)]{apptoSHL} is due to the 
fact that we use normalized induction in the present paper, while we did not use normalized induction 
in \cite{apptoSHL}.}
\begin{align}\label{sectionininducedrepnatp}
f^{\Phi_w}(g): = \chi_{2, w}(\det g)\left|\det g\right|_{w}^{\frac{n}{2}+s}\int_{\mathbf{X}}
\Phi_w(Xg)\chi_{1, w}^{-1}\chi_{2, w}(\det X)\left|\det X\right|_{w}^{n+2s}d^{\times}X.
\end{align}
Linear operations are viewed here as acting on the vector space $W_w$ on the right.
We recall that $\bX$ is identified with $\GL_{n}(\K_w)$; $d^\times X$ is the measure
identified with the right Haar measure on the latter.
To define the Siegel sections $f_{w, s}$, we make specific choices of the Schwartz functions $\Phi_w$.

Let $(a_w,b_w)$ be the signature associated to $w|p$ and $L_1,\pair_1$. 
For each $w\in\Sigma_p$, fix partitions
\begin{equation*} \label{rvdefn}
a_w = n_{1, w} + \cdots + n_{t(w), w} \ \ \text{and} \ \
b_w = n_{t(w)+1, w}+ \cdots + n_{r(w), w}.
\end{equation*}
Let $\mu_{1, w} , \ldots, \mu_{r(w), w}$ be characters of $\O_w^\times$, and let 
$\mu_w = (\mu_{1, w}, \ldots, \mu_{r(w), w})$ and $\mu = \prod_{w\in\Sigma_p}\mu_w$.  
We view each character $\mu_{i, w}$ as a character of $\GL_{n_{i, w}}(\O_w)$ via composition with the determinant.
Let
$$
\nu_{i, w} = \chi_{1, w}^{-1}\chi_{2, w}\mu_{i, w}, \ \ \ i=1,...,r(w),
$$
and let $\nu_w = \left(\nu_{1, w}, \ldots, \nu_{r(w), w}\right)$. 

Let $\mathfrak{X}_w\subset M_n(\O_w)$ comprise the matrices 
$\left(\smallmatrix A & B\\ C& D\endsmallmatrix\right)$, with $A\in M_{a_w}(\O_w)$
and $D\in M_{b_w}(\O_w)$, such that the determinant of the leading principal $n_{1, w}+\cdots +n_{i, w}$-th 
minor of $A$ is in $\O_w^\times$ for $i = 1, \ldots, t(w)$ and the determinant of the leading principal 
$n_{t(w)+1, w}+\cdots +n_{i, w}$-th minor of $D$ is in $\O_w^\times$ for $i = t(w)+1, \ldots, r(w)$.
Let $A_i$ be the determinant of the leading principal $i$-th minor of $A$ and $D_i$ the determinant of the
leading principal $i$-th minor of $D$. Define $\phi_{\nu_w}:M_n(\K_w)\rightarrow\C$ to be the function supported on 
$\mathfrak{X}_w$ and defined for $X = \left(\smallmatrix A & B\\ C& D\endsmallmatrix\right)\in\mathfrak{X}_w$ by
\begin{align*}
\phi_{\nu_v}(X) & = \nu_{t(w), w}(A)\cdot 
\prod_{i = 1}^{t(w)-1} (\nu_{i, w}\cdot \nu_{i+1, w}^{-1})(A_{n_{1, w}+\cdots+ n_{i, w}}) \\ 
& \nu_{r(w), w}(D)\cdot
\times \prod_{i=t(w)+1}^{r(w)-1}(\nu_{i, w}\cdot \nu_{i+1, w}^{-1})(D_{n_{t(w)+1, w}+\cdots+ n_{i, w}}).
\end{align*}

Let
\begin{align}\label{tbound}
t\geq \max_{w\in\Sigma_p, 1\leq i\leq r(w)}(1, \ord_w(\cond(\mu_{i, w})), \ord_w(\cond(\chi_w)))),
\end{align}
and let
$\Gamma_w = \Gamma_w(t)\subset\GL_{n}(\O_w)$
be the subgroup of $\GL_n(O_w)$ consisting of matrices whose terms below the $n_{i, w}\times n_{i, w}$-blocks 
along the diagonal are in $\mathfrak{p}_w^t$ and such that the upper right $a_w\times b_w$ block is also in 
$\mathfrak{p}_w^t$.  For each matrix $m\in\Gamma_w$ with $n_{i, w}\times n_{i,w}$-blocks $m_i$ running down the 
diagonal, we define
\begin{align*}
\mu_w(m) = \prod_i \mu_{i, w}(\det(m_i)).
\end{align*}
Let $\Phi_{1, w}$ be the function on $M_{n\times n}(\K_w)$ supported on $\Gamma_w(t)$ 
and such that $\Phi_{1,w}(x) = \mu_w(x)$ for all $x\in\Gamma_w(t)$.  
Let $\Phi_{2, w}$ be the function on $M_{n\times n}(\K_w)$ defined by
\begin{align}\label{Phi2defn}
\Phi_{2, w}(x) = \hat{\phi}_{\nu_w}(x) = \int_{M_{n\times n}(\K_w)} \phi_{\nu_w}(y)e_w(-\trace\, y\t x)dy.
\end{align}  
Note that $\hat{\phi}_{\nu_w}$ is the Fourier transform of $\phi_{\nu_w}$, 
as discussed in \cite[Lemma 10]{apptoSHL}.

For $X=\left(X_1,X_2\right)\in \Hom_{\K_w}\left(V_w,W_w\right) = \Hom_{\K_w}(V_w,V_{d,w})\oplus\Hom_{\K_w}\left(V_w,V^d_w\right)$,
let
\begin{align*}
\Phi_w(X)=\Phi_{\chi, \mu, w} (X_1, X_2) = \vol(\Gamma_w)^{-1}\Phi_{1, w}(-X_1)\cdot \Phi_{2, w}(2X_2).
\end{align*}
Recall that we have identified $X_1$ and $X_2$ with matrices through a choice of basis for $V_w$ (coming
from the level structure at $p$ for $P_1$). Note that $\Phi_{\chi, \mu, w}$ is a partial Fourier transform in 
the second variable in the sense of \cite[Lemma 10]{apptoSHL}.
We then define
\begin{align}\label{fwchimu}
f_{w, s}:=f_w^{\chi, \mu} := f^{\Phi_w}=f^{\Phi_{\chi, \mu, w}}.
\end{align}  
We then define $f_{p, s}\in I_p(\chi_p,s)$ by \eqref{fpproductdefinition}.

The following lemma describes the support of $\Phi_{2,w}$.
\begin{lem}\label{lemmaphisupport} \hfill
\begin{itemize}
\item[(i)] \label{phinuequ}
For $\gamma_1, \gamma_2\in\Gamma_w$,
\begin{align*} 
\phi_{\nu_w}(\t\gamma_1X\gamma_2)=\mu_w(\gamma_1\gamma_2)\chi_{1,w}^{-1}\chi_{2,w}(\det\gamma_1\gamma_2)
\phi_{\nu_w}(X).
\end{align*}
\item[(ii)] \label{phi2xequ} For $X = \left(\smallmatrix 
A & B \\ C & D\endsmallmatrix\right)$ with $A\in M_{a_w\times a_w}(\K_w)$, $B\in M_{a_w\times b_w}(\K_w)$, 
$C\in M_{b_w\times a_w}(\K_w)$, and $D\in M_{b_w\times b_w}(\K_w)$,
\begin{align*}
\Phi_{2,w}(X) = \Phi_w^{(1)}(A)\Phi_w^{(2)}(B)\Phi_w^{(3)}(C)\Phi_w^{(4)}(D),
\end{align*}
with
\begin{align*}
\Phi_w^{(2)} = \mathrm{char}_{M_{a_w\times b_w}(\O_w)}, & \ \ \
\Phi_w^{(3)} = \mathrm{char}_{M_{b_w\times a_w}(\O_w)}\\
\supp(\Phi_w^{(1)})\subseteq \grp_w^{-t}M_{a_w\times a_w}(\O_w), & \ \ 
\supp(\Phi_w^{(4)})\subseteq \grp_w^{-t}M_{b_w\times b_w}(\O_w).
\end{align*}
Here $t$ is as in Inequality \eqref{tbound}.
\end{itemize}
\end{lem}

\begin{proof}
Part (i) follows immediately from the definition of $\phi_{\nu_w}$.
It remains to prove part (ii).  
Let
$$
\grX^{(1)} = \{ \alpha\in M_{a_w\times a_w}(\O_w) \ : \ \left(\smallmatrix \alpha & 0 \\ 0 & 1 \endsmallmatrix\right) \in \grX\},
$$
and
$$
\grX^{(4)} = \{ \delta\in M_{b_w\times b_w}(\O_w) \ : \ \left(\smallmatrix 1 & 0 \\ 0 & \delta \endsmallmatrix\right) \in \grX\}.
$$
We have
\begin{align*}
\Phi_{2,w}(X) & = \int_{M_n(\K_w)}\phi_{\nu_w}(Y)e_w\left(-\trace\, Y\left(\smallmatrix \t A & \t C \\ 
\t B & \t D\endsmallmatrix\right)\right)dY \\
& = \int_{\grX} \phi_{\nu_w}(\left(\smallmatrix\alpha & \beta \\ \gamma &\delta\endsmallmatrix\right))
e_w\left(-\trace\left(\alpha\t A+\beta\t B+\gamma\t C+\delta\t D\right)\right)d\alpha d\beta d\gamma d\delta \\
& = \Phi_w^{(1)}\left(A\right)\Phi_w^{(2)}\left(B\right)\Phi_w^{(3)}\left(C\right)\Phi_w^{(4)}\left(D\right),
\end{align*}
where
\begin{align*}
\Phi_w^{(2)}(B) = \int_{M_{a_w\times b_w}(\O_w)}e_w(-\trace\,\beta\t B)d\beta =
\mathrm{char}_{M_{a_w\times b_w}(\O_w)}(B), \\
\end{align*}
\begin{align*}
\Phi_w^{(3)}(C) = \int_{M_{b_w\times a_w}(\O_w)} e_w(-\trace\, \gamma\t C)d\gamma = 
\mathrm{char}_{M_{b_w\times a_w}(\O_w)}(C),
\end{align*}
\begin{align*}
\Phi_w^{(1)}(A) & = \int_{\grX^{(1)}} \phi_{\nu_w}(\left(\smallmatrix \alpha & 0 \\ 0 & 1\endsmallmatrix\right))
e_w(-\trace\,\alpha\t A) d\alpha  \\
& = \Vol(M_{a_w\times a_w}(\grp_w^{t})) \sum_{x = \left(\smallmatrix \alpha & 0 \\ 0 & 1\endsmallmatrix\right)
\in\grX \mod \grp_w^t}\phi_{\nu_w}(x) e_w(-\trace\,\alpha\t A)\mathrm{char}_{\grp_w^{-t}M_{a_w\times a_w}
(\O_w)}(A),
\end{align*}
and
\begin{align*}
\Phi_w^{(4)}(D) & = \int_{\grX^{(4)}} \phi_{\nu_w}(\left(\smallmatrix 1 & 0 \\ 0 & \delta \endsmallmatrix\right))
e_w(-\trace\,\delta\t A) d\alpha  \\
& = \Vol(M_{b_w\times b_w}(\grp_w^t)) \sum_{x = \left(\smallmatrix 1 & 0 \\ 0 & \delta\endsmallmatrix\right)
\in\grX \mod \grp_w^t}\phi_{\nu_w}(x) e_w(-\trace\,\delta\t D)\mathrm{char}_{\grp_w^{-t}M_{b_w\times b_w}
(\O_w)}(D).
\end{align*}
\end{proof}

\subsubsection{Local induced representations}\label{localinducedreps}
Having chosen a Siegel section for each prime $w\in\Sigma_p$, we move on to the zeta integral calculations for such a $w$.

First, we introduce some additional notation.  Let $B_{a_w}\subseteq\GL_{a_w}$ be the standard parabolic subgroup associated to the partition $a_w = n_{1, w}+\cdots + n_{t(w), w}$.  Let $B_{b_w}\subseteq\GL_{b_w}$ be the standard parabolic subgroup associated to the partition $b_w = n_{t(w)+1, w}+\cdots+ n_{r(w), w}$.  
Let $R_{a_w, b_w}\subseteq \GL_n$ be the standard parabolic subgroup associated to $n = a_w+b_w$.  
Let $L_{a_w, b_w}=\GL_{a_w}\times\GL_{b_w}$ denote the Levi subgroup of $R_{a_w, b_w}$.
Let $R_w \subset R_{a_w,b_w}$ be the 
parabolic 
such that $R_w\cap L_{a_w,b_w}  = B_{a_w}\times B_{b_w}^\op$.

Recall the characters $\mu_{i, w}$ from Section  \ref{SWsections}, which define characters on $\GL_{n_{i, w}}(\O_w)$ via composition with the determinant.  
We define characters $\mu_{i, w}'$ by
\begin{align*}
\mu_{i, w}' = \begin{cases} \chi_{2, w}^{-1}\mu_{i, w}^{-1}, & \mbox{if } 1\leq i\leq t(w) \\ \chi_{1, w}^{-1}\mu_{i, w}, & \mbox{if } t(w)+1\leq i\leq r(w), \end{cases}
\end{align*}
for all $w\in\Sigma_p$.  Let
\begin{gather*}
\mu_{a_w}' = \otimes_{i=1}^{t(w)} \mu_{i,w}, \ \ \mu_{b_w}' = \otimes_{i=t(w)+1}^{r(w)}\mu_{i,w}', \\
\mu_w' = \otimes_{i = 1}^{r(w)}\mu_{i, w}' = \mu_{a_w}' \otimes\mu_{b_w}', \\
\mu' = \otimes_{w\in\Sigma_p}\mu'_w
\end{gather*}
denote characters 
on 
$\prod_{i=1}^{t(w)}\GL_{n_{i, w}}(\O_w)$, $\prod_{i=t(w)+1}^{r(w)}\GL_{n_{i, w}}(\O_w)$, $\prod_{i=1}^{r(w)}\GL_{n_{i, w}}(\O_w)$ and $\prod_w\left(\prod_{i=1}^{r(w)}\GL_{n_{i, w}}(\O_w)\right)$, respectively (again, by composing with determinants).

For all $w\in\Sigma_p$, let
$\tau_{i,w}$, $1\leq i \leq r(w)$, be an unramified irreducible admissible representation of $\GL_{n_{i,w}}(\K_w)$. Let $\beta_{i,w}$, $1\leq i\leq r(w)$, be a character of 
$\GL_{n_{i,w}}(\CK_w)$ such that $\beta_{i,w}|_{\GL_{n_{i,w}}(\O_w)} = \mu_{i,w}'$. Let  $\pi_{i,w} = \tau_{i,w}\otimes\beta_{i,w}$.
Let 
$$
\Ind_{B_{a_w}}^{\GL_{a_w}}(\otimes_{i=1}^{t(w)} \pi_{i,w})\twoheadrightarrow \pi_{a_w} \ \ \text{and} \ \ 
\Ind_{B^\op_{b_w}}^{\GL_{b_w}}(\otimes_{i=t(w)+1}^{r(w)} \pi_{i,w}) \twoheadrightarrow \pi_{b_w}
$$
be irreducible admissible quotients.  Similarly, let 
$$
\Ind_{R_{a_w,b_w}}^{\GL_n} (\pi_{a_w}\otimes\pi_{b_w}) \twoheadrightarrow \pi_w
$$ 
be an irreducible admissible quotient.  By composition of these quotients, $\pi_w$ is realized as a quotient
of $\Ind_{R_w}^{\GL_n} (\otimes_{i=1}^{r(w)}\pi_{i,w})$.

For all $w\in\Sigma_p$, let $\tilde\tau_{i,w}$ be the contragredient of $\tau_{i,w}$, and let $\tilde\beta_{i,w}$ be the contragredient of $\beta_{i,w}$.
Let $\tilde\pi_{i,w} = \tilde\tau_{i,w}\otimes\tilde\beta_{i,w}$; this is the contragredient of $\pi_{i,w}$.
Let $\tilde\pi_{a_w}$, $\tilde\pi_{b_w}$, and $\tilde\pi_w$ be the respective contragredients of $\pi_{a_w}$, $\pi_{b_w}$, and $\pi_w$.

For all $w\in\Sigma$ and $1\leq i\leq r(w)$, let $\pair_{\pi_{i,w}}:\pi_{i,w}\times\tilde\pi_{i,w}$ be the tautological pairing of a representation and its contragredient.
Let $\ppair_{a_w} = \otimes_{i=1}^{t(w)} \pair_{\pi_{i,w}}$ and $\ppair_{b_w}=\otimes_{i=t(w)+1}^{r(w)}\pair_{\pi_{i,w}}$.
Then 
\begin{gather*}
\pair_{a_w}:\Ind_{B_{a_w}}^{\GL_{a_w}}(\otimes_{i=1}^{t(w)} \pi_{i,w})\times\Ind_{B_{a_w}}^{\GL_{a_w}}(\otimes_{i=1}^{t(w)} \tilde\pi_{i,w}) \rightarrow \C, \\
\langle \vphi,\tilde\vphi\rangle_{a_w} = \int_{\GL_{a_w}(\O_w)} (\vphi(k),\tilde\vphi(k))_{a_w} dk, 
\end{gather*}
and
\begin{gather*}
\pair_{b_w}:\Ind_{B^\op_{b_w}}^{\GL_{b_w}}(\otimes_{i=t(w)+1}^{r(w)} \pi_{i,w})\times\Ind_{B^\op_{b_w}}^{\GL_{b_w}}(\otimes_{i=t(w)+1}^{r(w)} \tilde\pi_{i,w})\rightarrow\C, \\
\langle \vphi,\tilde\vphi\rangle_{b_w} = \int_{\GL_{b_w}(\O_w)} (\vphi(k),\tilde\vphi(k))_{b_w} dk
\end{gather*}
are, respectively, perfect $\GL_{a_w}(\K_w)$-invariant and $\GL_{b_w}(\K_w)$-invariant pairings that identify the pairs of representations as 
contragredients (and the pairings with the tautalogical ones). With respect to these identifications, the dual  of the surjections onto $\pi_{a_w}$ and $\pi_{b_w}$ are 
inclusions of irreducible admissible representations
$$
\tilde\pi_{a_w} \hookrightarrow \Ind_{B_{a_w}}^{\GL_{a_w}}(\otimes_{i=1}^{t(w)} \tilde\pi_{i,w})\ \ \text{and} \ \  
\tilde\pi_{b_w} \hookrightarrow \Ind_{B^\op_{b_w}}^{\GL_{b_w}}(\otimes_{i=t(w)+1}^{r(w)} \tilde\pi_{i,w})
$$
such that the tautological pairings
$$
\pair_{\pi_{a_w}}:\pi_{a_w}\times\tilde\pi_{a_w} \rightarrow\C \ \ \text{and} \ \ \pair_{\pi_{b_w}}:\pi_{b_w}\times\tilde\pi_{b_w} \rightarrow\C
$$
are the pairings induced from $\pair_{a_w}$ and $\pair_{b_w}$
by composition with the projections to $\pi_{a_w}$ and $\pi_{b_w}$ and the inclusions of $\tilde\pi_{a_w}$ and $\tilde\pi_{b_w}$.
Similarly, $\ppair_w = \pair_{\pi_{a_w}}\otimes\pair_{\pi_{b_w}}$ determines 
a pairing $\pair_w: \Ind_{R_{a_w,b_w}}^{\GL_n} (\pi_{a_w}\otimes\pi_{b_w})\times\Ind_{R_{a_w,b_w}}^{\GL_n} (\tilde\pi_{a_w}\otimes\tilde\pi_{b_w})\rightarrow \C$
that is identified with the tautological pairing and so induces an inclusion 
$$
\tilde\pi_w \hookrightarrow \Ind_{R_{a_w,b_w}}^{\GL_n} (\tilde\pi_{a_w}\otimes\tilde\pi_{b_w})
$$
(and hence by composition an inclusion $\tilde\pi_w\hookrightarrow
\Ind_{R_w}^{\GL_n}(\otimes_{i=1}^{r(w)}\tilde\pi_{i,w})$)
such that the tautological pairing $\pair_{\pi_w}:\pi_w\times\tilde\pi_w\rightarrow\C$ is the pairing induced from $\pair_w$ via
the projection to $\pi_w$ and the inclusion of $\tilde\pi_w$.
In particular, for $\phi\in \pi_w$ and 
$\tilde\phi\in \tilde\pi_w$, let $\varphi \in \Ind_{R_{a_w,b_w}}^{\GL_n} (\pi_{a_w}\otimes\pi_{b_w})$ project to $\phi$ and let $\tilde\varphi\in \Ind_{R_{a_w,b_w}}^{\GL_n} (\tilde\pi_{a_w}\otimes\tilde\pi_{b_w})$ be the image of $\tilde\phi$. Then 
\begin{equation}\label{pairing1}
\langle\phi,\tilde\phi\rangle_{\pi_w} = \int_{\GL_n(\O_w)} \langle \vphi(k),\tilde\vphi(k)\rangle_w dk.
\end{equation}

\subsubsection{Local congruence subgroups and (anti-ordinary) test vectors}\label{localtestvectors}

Let $t$ satisfy the inequality \eqref{tbound}, and let
\begin{equation}\label{dgeq2t}
r\geq 2t.
\end{equation}
Consider the following groups:
\begin{align*}
\Gamma_{R, w} & = \left\{\gamma\in\gln\left({\O}_w\right)\mid \gamma\mod\mathfrak{p}_w^r\in R_w\left(\O/\mathfrak{p}^r_w\O\right)\right\}, \ \
\Gamma_R = \prod_{w\in\Sigma_p}\Gamma_{R, w},\\
\Gamma_{a_w, w} & = \left\{\gamma\in\gl_{a_w}\left({\O}_w\right)\mid \gamma\mod\mathfrak{p}_w^r\in B_{a_w}\left(\O/\mathfrak{p}^r_w\O\right)\right\}, \ \
\Gamma_a  = \prod_{w\in\Sigma_p}\Gamma_{a_w, w},\\
\Gamma_{b_w, w} & = \left\{\gamma\in\gl_{b_w}\left({\O}_w\right)\mid \gamma\mod\mathfrak{p}_w^r\in B_{b_w}\left(\O/\mathfrak{p}^r_w\O\right)\right\}, \ \
\Gamma_b = \prod_{w\in\Sigma_p}\Gamma_{b_w, w}, \\
\Gamma_{a_w,b_w} & = \left\{\gamma\in\GL_n\left(\O_w\right) \mid \gamma\mod\mathfrak{p}_w^r \in R_{a_w,b_w}\left(\O/\mathfrak{p}^r\O\right)\right\}, 
\end{align*}
By the choice of $r$, the character $\mu_w$ extends to a character of both $\Gamma_R$ and its transpose $\t\Gamma_R$ such that 
for $\gamma$ in $\Gamma_R$ or $\t\Gamma_R$, $\mu'_w(\gamma) = \prod_{i=1}^{r_w} \mu'_{i,w}(\gamma_{ii})$, where $\gamma =(\gamma_{ij})$ is the block matrix form corresponding
to the partition $n = n_{1,w} + \cdots + n_{r(w),w}$. Similarly, $\mu'_{a_w}$ (resp.~$\mu'_{b_w}$) extend to characters of $\Gamma_{a_w,w}$ and $\t\Gamma_{a_w,w}$
(resp.~$\Gamma_{b_w,w}$ and $\t\Gamma_{b_w,w}$). The same holds for $\tilde\mu_w'$, $\tilde\mu_{a_w}'$, and $\tilde\mu_{b_w}'$.

For all $w\in\Sigma_p$ and $1\leq i\leq r(w)$, let $0\neq \phi_{i,w} \in \pi_{i,w}$ such that $\pi_{i,w}(k)\phi_{i,w} = \mu_{i,w}'(k)\phi_{i,w}$ for all $k\in\GL_{n_{i,w}}(\O_w)$.
Such a $\phi_{i,w}$ exists (and is unique up to non-zero scalar) since $\pi_{i,w} = \tau_{i,w}\otimes\beta_{i,w}$ with $\tau_{i,w}$ unramified.
Let $\phi_{a_w}^0 = \otimes_{i=1}^{t(w)}\phi_{i,w}$ and $\phi_{b_w}^0 = \otimes_{i=t(w)+1}^{r(w)}\phi_{i,w}$.
Let $\varphi_{a_w} \in \Ind_{B_{a_w}}^{\GL_{a_w}}(\otimes_{i=1}^{t(w)} \pi_{i,w})$ be the unique function such that: 
\begin{equation}\label{phiinvariancea}
\text{the support of $\vphi_{a_w}$ is $B_{a_w}(\K_w)\Gamma_{a_w,w}$ and $\varphi_{a_w}(\gamma) = \mu_{a_w}'(\gamma) \phi_{a_w}^0$ $\forall \gamma\in\Gamma_{a_w,w}$.} 
\end{equation}
Similarly, let $\varphi_{b_w} \in \Ind_{B^\op_{b_w}}^{\GL_{b_w}}(\otimes_{i=t(w)+1}^{r(w)} \pi_{i,w})$ be the unique function such that: 
\begin{equation}\label{phiinvarianceb}
\text{the support of $\vphi_{b_w}$ is $B^\op_{b_w}(\K_w)\t\Gamma_{b_w,w}$ and $\varphi_{b_w}(\gamma) = \mu_{b_w}'(\gamma) \phi_{b_w}^0$ $\forall \gamma\in\t\Gamma_{b_w,w}$.} 
\end{equation}
We assume that 
\begin{equation}\label{phiinvarianceab}
\text{the image $\phi_{a_w}$ (resp. $\phi_{b_w}$) of $\vphi_{a_w}$ in $\pi_{a_w}$ (resp.~of $\vphi_{b_w}$ in $\pi_{b_w}$) is nonzero.}
\end{equation}
Let $\phi_w^0 = \phi_{a_w}\otimes\phi_{b_w} \in \pi_{a_w}\otimes\pi_{b_w}$. 
Let $\vphi_w\in \Ind_{R_{a_w,b_w}}^{\GL_n}\pi_{a_w}\otimes\pi_{b_w}$ be the unique function such that
\begin{equation}\label{phiinvariancew}
\text{the support of $\vphi_{w}$ is $R_{a_w,b_w}(\K_w)\Gamma_{R,w}$ and $\varphi_{w}(\gamma) = \mu_{w}'(\gamma) \phi_{w}^0$ $\forall \gamma\in\Gamma_{R,w}$.} 
\end{equation}
We also assume that 
\begin{equation}\label{phinvariancew2}
\text{the image $\phi_{w}$ of $\vphi_{w}$ in $\pi_{w}$ is nonzero.}
\end{equation}
Note that 
\begin{equation}\label{phiinvariancew3}
\pi_w(\gamma) \phi_w = \mu'_w(\gamma)\phi_w \ \forall \gamma\in \Gamma_{R,w}.
\end{equation}

Let $0\neq \tilde\phi_{i,w}  \in \tilde\pi_{i,w}$ be such that $\tilde\pi_{i,w}(k)\tilde\phi_{i,w}= \tilde\mu'_{i,w}(k)\tilde\phi_{i,w}$ for all $k\in \GL_{n_i}(\O_w)$.
Let $\tilde\phi_{a_w}^{0} = \otimes_{i=1}^{t(w)} \tilde\phi_{i,w}$ and $\tilde\phi_{b_w}^{0} = \otimes_{i=t(w)+1}^{r(w)} \tilde\phi_{i,w}$.
We suppose that $\tilde\phi_{a_w} \in \tilde\pi_{a_w}$, $\tilde\phi_{b_w} \in \tilde\pi_{b_w}$, and $\tilde\phi_{w} \in \tilde\pi_{w}$
are such that 
\begin{equation}\begin{split}\label{tphiinvarianceab}
\text{the image $\tilde\vphi_{a_w}$ of $\tilde\phi_{a_w}$ in $\Ind_{B_{a_w}}^{\GL_{a_w}} (\otimes_{i=1}^{t(w)}\tilde\pi_{i,w})$ satisfies $\tilde\vphi_{a_w}(1) = \tilde\phi_{a_w}^0$,} \\
\text{the image $\tilde\vphi_{b_w}$ of $\tilde\phi_{b_w}$ in $\Ind_{B_{b_w}^\op}^{\GL_{b_w}} (\otimes_{i=t(w)+1}^{r(w)}\tilde\pi_{i,w})$ satisfies $\tilde\vphi_{b_w}(1) = \tilde\phi_{b_w}^0$,}
\end{split}\end{equation}
and
\begin{equation}\label{tphiinvariancew}
\text{the image $\tilde\vphi_{w}$ of $\tilde\phi_{w}$ in $\Ind_{R_{a_w,b_w}}^{\GL_{n}} (\tilde\pi_{a_w}\otimes\tilde\pi_{b_w})$ satisfies $\tilde\vphi_{w}(1) = \tilde\phi_{a_w}\otimes\tilde\phi_{b_w}$.}
\end{equation}
We also suppose that 
\begin{equation}\label{tphiinvariancew2}
\tilde\pi_w(\gamma)\tilde\phi_w = \tilde\mu_w'(\gamma)\tilde\phi_w \ \forall\gamma\in\t\Gamma_{R,w}.
\end{equation}
One consequence of \eqref{tphiinvariancew} and \eqref{tphiinvariancew2} is that 
\begin{equation}\label{phi'invariancew2} \text{the support of $\tilde\vphi_w$  contains $R_{a_w,b_w}(\K_w)\t\Gamma_{R,w}$.}
\end{equation}

All of the above conditions imposed on $\varphi_w$ and $\tilde{\varphi}_w$ will be used in our computations of the local zeta integrals later in this section.  We note in particular that \eqref{phiinvariancew} and \eqref{phi'invariancew2} (resp. \eqref{dgeq2t}) correspond to condition (i) (resp. condition (ii)) of Remark \ref{conditions}.

By \eqref{pairing1},
$$
\langle \phi_w,\tilde\phi_w\rangle_{\pi_w} = \int_{\GL_n(\O_w)} \langle\varphi_w(k),\tilde\varphi_w(k)\rangle_w dk.
$$
By the choice of the support of $\varphi_w$ in \eqref{phiinvariancew}, 
the integrand is zero outside of $\Gamma_{a_w,b_w} = \GL_n(\O_w) \cap R_{a_w,b_w}(\K_w)\Gamma_{R,w}$.
Let $k \in \Gamma_{a_w,b_w}$. Then $k$ can be expressed as a product
$$
k = \left(\smallmatrix 1 & B \\ 0 & 1\endsmallmatrix\right) \left(\smallmatrix A & 0 \\ 0 & D\endsmallmatrix\right) \left(\smallmatrix 1 & 0 \\ C & 1\endsmallmatrix\right)
$$
with $B\in M_{a_w,b_w}(\O_w)$, $A\in \GL_{a_w}(\O_w)$, $D\in \GL_{b_w}(\O_w)$, and $C\in \M_{b_w,a_w}(\grp_w^t)$.
Then
$$
\vphi_w(k) = \vphi_w( \left(\smallmatrix A & 0 \\ 0 & D\endsmallmatrix\right)) = \pi_{a_w}(A)\phi_{a_w} \otimes \pi_{b_w}(D)\phi_{b_w},
$$
and
$$
\tilde\vphi_w(k) = \tilde\vphi_w( \left(\smallmatrix A & 0 \\ 0 & D\endsmallmatrix\right)) = \tilde\pi_{a_w}(A)\tilde\phi_{a_w} \otimes \tilde\pi_{b_w}(D)\tilde\phi_{b_w}.
$$
In particular, 
$$
\langle \vphi_w(k),\tilde\vphi_w(k)\rangle_w = \langle\phi_{a_w},\tilde\phi_{a_w}\rangle_{\pi_{a_w}} \cdot \langle \phi_{b_w},\tilde\phi_{b_w}\rangle_{\pi_{b_w}}
$$
for all $k\in \Gamma_{a_w,b_w}$.
It follows that
\begin{equation}\label{pair=ab}
\langle \phi_w,\tilde\phi_w\rangle_{\pi_w} = \Vol(\Gamma_{a_w,b_w})\cdot 
\langle\phi_{a_w},\tilde\phi_{a_w}\rangle_{\pi_{a_w}} \cdot \langle \phi_{b_w},\tilde\phi_{b_w}\rangle_{\pi_{b_w}}.
\end{equation}
Here the volume $\Vol(\Gamma_{a_w,b_w})$ is with respect to the chosen Haar measure on $\GL_n(\O_w)$.

Similar considerations show that 
$$
\langle\phi_{a_w},\tilde\phi_{a_w}\rangle_{\pi_{a_w}} = \Vol(\Gamma_{a_w,w})\prod_{i=1}^{t(w)} \langle \phi_{i,w}, \tilde\phi_{i,w}\rangle_{\pi_{i,w}}
$$
and
$$
\langle\phi_{b_w},\tilde\phi_{b_w}\rangle_{\pi_{b_w}}  = \Vol(\Gamma_{b_w,w})\prod_{i=t(w)+1}^{r(w)} \langle \phi_{i,w}, \tilde\phi_{i,w}\rangle_{\pi_{i,w}}.
$$
As $\Vol(\Gamma_{R,w}) = \Vol(\Gamma_{a_w,b_w}) \Vol(\Gamma_{a_w,w}) \Vol(\Gamma_{b_w,w})$, it follows that
\begin{equation}\label{pairneq0}
\langle \phi_w,\tilde\phi_w\rangle_{\pi_w} = \Vol(\Gamma_{R,w}) \prod_{i=1}^{t(w)} \langle \phi_{i,w}, \tilde\phi_{i,w}\rangle_{\pi_{i,w}} \neq 0.
\end{equation}
The nonvanishing of each $\langle \phi_{i,w},\tilde\phi_{i,w}\rangle_{\pi_{i,w}}$ is an easy consequence of the choice of $\phi_{i,w}$ and $\tilde\phi_{i,w}$.

\begin{rmk}\label{testvectors-remark}
In Sections \ref{ordvec} and \ref{ordvec-2} we identify
specific vectors in certain local representations (constituents of global representations of interest) that satisfy the conditions imposed in this section (see especially Remarks \ref{testvectors-remark} and \ref{testvector-remarkII}). These sections, denoted
$\phi_{w,r}^\aord$ and $\phi_{w,r}^{\flat,\aord}$ are natural choices from the perspective of Hida theory and (anti-)ordinary automorphic forms. The main result of the following section therefore shows that the local zeta integral for these natural choices of test vectors contributes a factor at $p$ of the expected form for a $p$-adic $L$-function.
\end{rmk}

\subsubsection{The main calculation}\label{maincalculation}

The ordered $\K_w$-basis for $V_w$ chosen above (that comes from the choice of a level structure for $P_1$) determines a $\K_w$-basis for $W_w= V_w\oplus V_w$.  This ordered basis for $W_w = V_w\oplus V_w$ identifies $\gl_{\K_w}\left(W_w\right)$ with $\gl_{2n}\left(\K_w\right)$ and identifies $\gl_{\K_w}\left(V_w\right)\times \gl_{\K_w}\left(V_w\right)\subseteq\gl_{\K_w}\left(V_w\oplus V_w\right)$ with $\gln\left(V_w\right)\times \gln\left(\K_w\right)\subseteq \gl_{2n}\left(\K_w\right)$.  {\em Note that this is a different identification of $\gl_{\K_w}\left(W_w\right)$ with $\gl_{2n}\left(\K_w\right)$ from the identification coming from the decomposition $W_w = V_{d, w}\oplus V_w^d$.}  With respect to this new decomposition, $\mathbf{X}$ no longer consists of elements $(0, X)$ but instead elements $(X, X)$.  (The switch between these two decompositions is often convenient in similar computations, e.g. in the computations in the doubling method introduced in \cite{GPSR}.)
Recall the Siegel section $f_w^{\chi, \mu}$ defined in Equation \eqref{fwchimu}.
In the computation of the zeta integrals, we replace $f_w^{\chi, \mu}$ with the translation $\tilde{f}_{w^+}$, defined by:
\begin{align}\label{siegeltranslation}
g\mapsto  g\cdot  \begin{pmatrix}1_{a_w} & 0 & 0 & 0\\
0 & 0 & 0 & 1_{b_w}\\
0 & 0 & 1_{a_w} & 0\\
0& 1_{b_w} & 0 & 0\end{pmatrix}.
\end{align}
(This is the translation from Remark \ref{Igusaembedding}.)  The matrices in Equation \eqref{siegeltranslation} are given with respect to the identification of $\gl_{\K_w}\left(W_w\right)$ with $\gl_{2n}\left(\K_w\right)$ introduced at the beginning of this paragraph.

To avoid cumbersome notation, we will denote $\Phi_{\chi, \mu, w}$ by $\Phi$ for the remainder of this section.  
The identification $\K_w = \K_{w^+}^+$ identifies the representation $\pi_w$ with a representation $\pi_{w^+}$
(and hence $\tilde\pi_w$ with $\tilde\pi_{w^+}$).
The sections $\phi_w\in\pi_w$ and $\tilde\phi_w\in\tilde\pi_w$ are then identified with sections $\phi_{w^+}\in\pi_{w^+}$ and $\tilde\phi_{w^+}\in \tilde\pi_{w^+}$, respectively.
The local zeta integral $Z_{w^+}\left(\phi_{w^+}, \tilde\phi_{w^+}, \tilde{f}_{w^+}, s\right)$,  the numerator of the local factor
$I_{w^+}\left(\phi_{w^+}, \tilde\phi_{w^+}, \tilde{f}_{w^+}, s\right)$ defined in \eqref{normalizedzeta}, then equals
\begin{align*}
Z_w:=
 \int_{\gln\left(\K_w\right)} & \chi_{2, w}\left(g\right)\left|\det g\right|_w^{s+\frac{n}{2}}\int_{\gln\left(\K_w\right)}\Phi\left(\left(Xg, X\right)\begin{pmatrix}1_{a_w} & 0 & 0 & 0\\
0 & 0 & 0 & 1_{b_w}\\
0 & 0 & 1_{a_w} & 0\\
0& 1_{b_w} & 0 & 0\end{pmatrix}\right)\\
&\times\chi_{1, w}^{-1}\chi_{2, w}\left(\det X\right)\left|\det X\right|_w^{2s+n}\langle\pi_w(g)\phi_w, \tilde\phi_w\rangle_{\pi_w} d^\times X d^\times g.\nonumber
\end{align*}
We put
\begin{align*}
Z_p & := 
\prod_{w\in\Sigma_p} Z_w = 
\prod_{w\in\Sigma_p}Z_{w^+}\left(\phi_{w^+}, \tilde\phi_{w^+}, f_{w^+}, \chi\right).
\end{align*}

Given $g, X\in\gln\left(\K_w\right)$, we denote by $Z_1 = \left( Z_1', Z_1''\right)$ and $Z_2=\left( Z_2', Z_2''\right)$ the matrices in $M_{n\times n}\left(\K_w\right) = M_{n\times a_w}\left(\K_w\right)\times M_{n\times b_w}\left(\K_w\right)$ given by
\begin{align*}
Z_1 & = Xg  = \left[ Z_1', Z_1''\right]\\
Z_2 & = X  = \left[ Z_2', Z_2''\right],
\end{align*}
with $Z_1', Z_2'\in M_{n\times a_w}\left(\K_w\right)$ and $Z_ 1'', Z_2''\in M_{n\times b_w}\left(\K_w\right)$.  So
\begin{align*}
\Phi\left(\left(Xg, X\right)\right) = \Vol\left(\Gamma_w\right)^{-1}\Phi_{1, w}\left(Z_1', Z_2''\right)\Phi_{2, w}\left(Z_2', Z_1'' \right),
\end{align*}
and
\begin{align*}
\left\langle\pi_w\left(g\right)\phi_w, \tilde\phi_w\right\rangle_{\pi_w} = \left\langle\pi_w \left(Xg\right)\phi_w, \tilde{\pi}_w\left(X\right)\tilde\phi_w\right\rangle_{\pi_w}
= \left\langle \pi_w\left(Z_1\right)\phi_w, \tilde{\pi}_w\left(Z_2\right)\tilde\phi_w\right\rangle_{\pi_w}.
\end{align*}
Therefore,
\begin{align}\label{IZ1Z2}
Z_w = \Vol\left(\Gamma_w\right)^{-1}\int_{\gln\left(\K_w\right)}&\int_{\gln\left(\K_w\right)}\chi_{2, w}\left(\det Z_1\right)\chi_{1, w}^{-1}\left(\det Z_2\right)\left|\det(Z_1Z_2)\right|_w^{s+\frac{n}{2}}&\\
&\times\Phi_{1, w}\left(Z_1', Z_2''\right)\Phi_{2, w}\left(Z_2', Z_1''\right)
\left\langle\pi_w\left(Z_1\right)\phi_w, \tilde{\pi}_w\left(Z_2\right)\tilde\phi_w\right\rangle_{\pi_w} d^\times Z_1 d^\times Z_2.\nonumber
\end{align}

We take the integrals over the following open subsets of full measure.  We take the integral in $Z_1$ over
\begin{align*}
\left\{\begin{pmatrix}1 & 0\\
C_1 & 1\end{pmatrix}\begin{pmatrix}A_1 & 0\\
0 & D_1\end{pmatrix}\begin{pmatrix}1 & B_1 \\ 0 & 1\end{pmatrix}\mid C_1, { }^tB_1\in M_{b_w\times a_w}\left(\K_w\right), A_1\in \gl_{a_w}(\K_w), D_1\in \gl_{b_w}(\K_w) \right\},
\end{align*}
with the measure
\begin{align*}
\left|\det A_1^{b_w}\det D_1^{-a_w}\right|_wdC_1d^\times A_1d^\times D_1dB_1,
\end{align*}
and we take the integral in $Z_2$ over
\begin{align*}
\left\{\begin{pmatrix}1 & B_2 \\ 0 & 1\end{pmatrix}\begin{pmatrix}A_2 & 0\\
0 & D_2\end{pmatrix}\begin{pmatrix}1 & 0\\
C_2 & 1\end{pmatrix}\mid C_2, { }^tB_2\in M_{b_w\times a_w}\left(\K_w\right), A_2\in \gl_{a_w}(\K_w), D_2\in \gl_{b_w}(\K_w)\right\},
\end{align*}
with the measure
\begin{align*}
\left|\det A_2^{-b_w}\det D_2^{a_w}\right|_wdC_2d^\times A_2d^\times D_2dB_2.
\end{align*}
So
\begin{align}
\Phi_{1, w}\left(Z_1', Z_2''\right) & = \Phi_{1, w}\left(\begin{pmatrix}A_1 & B_2D_2\\C_1A_1 &D_2\end{pmatrix}\right)\label{Phi1abcd}, \\
\Phi_{2, w}\left(Z_2', Z_1''\right) & = \Phi_{2, w}\left(\begin{pmatrix} A_2+ B_2D_2C_2& A_1B_1\\ D_2C_2&C_1A_1B_1+D_1\end{pmatrix}\right).\label{Phi2abcd}
\end{align}

\begin{prop}\label{Jfactors}
The product $\Phi_{1, w}\left(\begin{pmatrix}A_1 & B_2D_2\\C_1A_1 & D_2\end{pmatrix}\right)\Phi_{2, w}\left(\begin{pmatrix}A_2+B_2D_2C_2 & A_1 B_1\\
D_2C_2 & C_1A_1B_1+D_1\end{pmatrix}\right)$ is zero unless all of the following conditions are met:
\begin{align*}
A_1&\in \Gamma_{a_w, w}(t)\\
C_1&\in \mathfrak{p}_w^tM_{b_w\times a_w}\left(\O_w\right)\\
D_2&\in \Gamma_{b_w, w}(t)\\
B_2&\in \mathfrak{p}_w^tM_{a_w\times b_w}\left(\O_w\right)\\
C_2&\in M_{b_w\times a_w}\left(\O_w\right)\\
A_2&\in \mathfrak{p}_w^{-t}M_{a_w\times a_w}\left(\O_w\right)\\
B_1&\in M_{a_w\times b_w}\left(\O_w\right)\\
D_1&\in \mathfrak{p}_w^{-t}M_{b_w\times b_w}\left(\O_w\right).
\end{align*}
When all of the above conditions are met, we have the following factorization at each prime $w\in\Sigma_p$:
\begin{align}\label{phivarphi}
\Phi_{1, w}\left(Z_1', Z_2''\right)\Phi_{2, w}\left(Z_2', Z_1''\right)\langle \pi_w\left(Z_1\right)\phi_w, \tilde{\pi}_w\left(Z_2\right)\tilde\phi_w\rangle_{\pi_w} = 
\Vol(\Gamma_{a_w,b_w})\cdot
J_1\cdot J_2,
\end{align}
where
\begin{align}\label{J1equ}
J_1 & = \chi_{2, w}\left(\det A_1\right)^{-1}\Phi_w^{(4)}\left(D_1\right)\left|\det D_1^{a_w}\right|_w^{1/2}\left\langle \phi_{b_w}, \tilde\pi_{b_w}\left(D^{-1}_1\right)\tilde\phi_{b_w}\right\rangle_{\pi_{b_w}}\\
J_2 & = \chi_{1, w}\left(\det D_2\right)\Phi_w^{(1)}\left(A_2\right)\left|\det A_2^{b_w}\right|_w^{1/2}\left\langle \phi_{a_w}, \tilde\pi_{a_w}\left(A_2\right)\tilde\phi_{a_w}\right\rangle_{\pi_{a_w}}\label{J2equ}.
\end{align}
\end{prop}

\begin{proof}
By Lemma \ref{lemmaphisupport} and the definition of $\Phi_{1, w}$, the product
\begin{align*}
\Phi_{1, w}\left(\begin{pmatrix}A_1 & B_2D_2\\C_1A_1 & D_2\end{pmatrix}\right)\Phi_{2, w}\left(\begin{pmatrix}A_2+B_2D_2C_2 & A_1 B_1\\
D_2C_2 & C_1A_1B_1+D_1\end{pmatrix}\right)
\end{align*}
 is zero unless all of the above conditions are met.  For the remainder of the proof, we will work only with matrices meeting the above conditions.  We now prove the second statement of the proposition.  Note that when the above conditions are met,
\begin{align*}
\pi_w\left(Z_1\right)\phi_w &= \pi_w\left(\begin{pmatrix}1 & 0\\ C_1 & D_1\end{pmatrix}\right)\mu_w'\left(\begin{pmatrix}A_1 & 0\\ 0&1\end{pmatrix}\right)\phi_w,\\
\tilde{\pi}_w\left(Z_2\right)\tilde\phi_w& = \tilde{\pi}_w\left(\begin{pmatrix}A_2 & B_2\\ 0 & 1\end{pmatrix}\right)\left(\mu_w'\right)^{-1}\left(\begin{pmatrix}1 & 0\\ 0 & D_2\end{pmatrix}\right)\tilde\phi_w.
\end{align*}
So
\begin{align*}
\Phi_{1, w}\left(Z_1', Z_2''\right)&\Phi_{2,w}\left(Z_2', Z_1''\right)\langle \pi_w\left(Z_1\right)\phi_w, \tilde{\pi}\left(Z_2\right)\tilde\phi_w\rangle_{\pi_w}\\
&  = \chi_{2, w}^{-1}\left(\det A_1\right)\chi_{1, w}\left(\det D_2\right)\left\langle\pi_w\left(\begin{pmatrix}1 & 0\\ C_1& D_1\end{pmatrix}\right)\phi_w, \tilde{\pi}_w\left(\begin{pmatrix}A_2 & B_2\\ 0 & 1\end{pmatrix}\right)\tilde\phi_w\right\rangle_{\pi_w}.
\end{align*}

Let $A\in M_{a_w}\left(\K_w\right)$, $D\in M_{b_w}\left(\K_w\right)$, $C\in M_{b_w\times a_w}\left(\K_w\right)$, and $B\in M_{a_w\times b_w}\left(\K_w\right)$ be matrices such that
\begin{align*}
\begin{pmatrix}
1 & -B_2\\
0& 1
\end{pmatrix}\begin{pmatrix}
1 & 0\\
C_1 & D_1
\end{pmatrix} =
\begin{pmatrix}
1 & 0\\
C & 1
\end{pmatrix}\begin{pmatrix}
A & 0\\
0 & D
\end{pmatrix}
\begin{pmatrix}
1 & B\\
0 & 1
\end{pmatrix}.
\end{align*}
Then
\begin{align*}
A & = 1-B_2C_1 \in 1+ \mathfrak{p}_w^{2t} M_{a_w}\left(\O_w\right)\\
CA &  = C_1\in \mathfrak{p}_w^t M_{b_w\times a_w}\left(\O_w\right)\\
AB & = -B_2 D_1\in M_{a_w\times b_w}\left(\O_w\right).
\end{align*}
So
\begin{align*}
C&\in \mathfrak{p}_w^tM_{b_w\times a_w}\left(\O_w\right)\\
B&\in M_{a_w\times b_w}\left(\O_w\right)\\
D& = D_1-CAB = \left(1+CB_2\right)D_1\in \left(1+\mathfrak{p}^{2t}M_{b_w}\left(\O_w\right)\right)D_1.
\end{align*}
Therefore, applying the invariance conditions \eqref{phiinvariancew3} and \eqref{tphiinvariancew2}
we obtain
\begin{align*}
\left\langle\pi_w\left(\begin{pmatrix}1 & 0 \\ C_1 & D_1\end{pmatrix}\right)\phi_w, \tilde{\pi}_w\left(\begin{pmatrix}A_2 & B_2\\ 0 & 1\end{pmatrix}\right)\tilde\phi_w\right\rangle_{\pi_w} & = \left\langle \pi_w\left(\begin{pmatrix}1 & -B_2\\ 0 & 1\end{pmatrix}\begin{pmatrix}1 & 0 \\ C_1 & D_1\end{pmatrix}\right)\phi_w, \tilde{\pi}_w\left(\begin{pmatrix}A_2 & 0 \\
0 & 1\end{pmatrix}\right)\tilde\phi_w \right\rangle_{\pi_w}\\
& = \left\langle \pi_w\left(\begin{pmatrix}1& 0 \\ C_1& 1 \end{pmatrix}\begin{pmatrix}1 & 0\\ 0 & D_1\end{pmatrix}\right)\phi_w, \tilde{\pi}_w\begin{pmatrix}A_2 & 0\\ 0 & 1\end{pmatrix}\tilde\phi_w  \right\rangle_{\pi_w}\\
 & = \left\langle \pi_w\left(\begin{pmatrix}1 & 0\\ 0 & D_1\end{pmatrix}\right)\phi_w, \tilde{\pi}_w\left(\begin{pmatrix}A_2 & 0 \\ 0& 1\end{pmatrix}\begin{pmatrix}1 & 0 \\-C_1 A_2 &1\end{pmatrix}\right)\tilde\phi_w\right\rangle_{\pi_w}\\
& = \left\langle\pi_w\left(\begin{pmatrix}1 & 0\\ 0 & D_1\end{pmatrix}\right)\phi_w, \tilde{\pi}_w\left(\begin{pmatrix}A_2 & 0 \\ 0& 1\end{pmatrix}\right)\tilde\phi_w\right\rangle_{\pi_w}\\
& = \left\langle\phi_w, \tilde{\pi}_w\left(\begin{pmatrix}A_2 & 0\\ 0 & D_1^{-1}\end{pmatrix}\right)\tilde\phi_w\right\rangle_{\pi_w}.
\end{align*}
Note that since $r\geq 2t$,  $A_2\in \mathfrak{p}_w^{-\lfloor{\frac{r}{2}}\rfloor}M_{a_w}\left(\O_w\right)$ and $D_1\in \mathfrak{p}_w^{-\lfloor{\frac{r}{2}}\rfloor}M_{b_w}\left(\O_w\right)$.

From the definition of $\phi_w$ we find that the support of $\varphi_w \in \Ind_{R_{a_w,b_w}}^{\GL_n}(\pi_{a_w}\otimes\pi_{b_w})$ inside of $\GL_n\left(\CO_w\right)$ is 
$\Gamma_{a_w,b_w}$. It follows that 
\begin{align*}
&\left\langle\phi_w, \tilde{\pi}_w\left(\begin{pmatrix}A_2 & 0 \\ 0& D_1^{-1}\end{pmatrix}\right)\tilde\phi_w\right\rangle_{\pi_w} \\
& = \Vol(\Gamma_{a_w,b_w})\cdot
\left\langle\varphi_w(1), \left|\det A_2^{b_w}\det D_1^{a_w}\right|_w^{1/2}\tilde\pi_{a_w}\left(A_2\right)\otimes\tilde\pi_{b_w}\left(D_1^{-1}\right)\tilde\varphi_w(1)\right\rangle_{\pi_{a_w}\otimes \pi_{b_w}}.
\end{align*}
As $\varphi_w(1) = \phi_{a_w}\otimes\phi_{b_w}$ and $\tilde\varphi(1) = \tilde\phi_{a_w}\otimes\tilde\phi_{b_w}$ by definition,
\begin{align*}
\langle\varphi_w(1), \tilde\pi_{a_w}\left(A_2\right) & \otimes\tilde\pi_{b_w}\left(D_1^{-1}\right)\tilde\varphi_w(1)\rangle_{\pi_{a_w}\otimes \pi_{b_w}} \\
& = \langle \phi_{a_w}, \tilde\pi_{a_w}\left(A_2\right)\tilde\phi_{a_w}\rangle_{\pi_{a_w}}\cdot\langle \phi_{b_w}, \tilde\pi_{b_w}\left(D_1^{-1}\right)\tilde\phi_{b_w}\rangle_{\pi_{b_w}}.
\end{align*}
Consequently,
\begin{align*}
\Phi_{1, w}&\left(Z_1', Z_2''\right)\Phi_{2, w}\left(Z_2', Z_1''\right)\langle \pi_w\left(Z_1\right)\phi_w, \tilde{\pi}_w\left(Z_2\right)\tilde{\phi}_w\rangle_{\pi_w} = 
\Vol(\Gamma_{a_w,b_w})\cdot J_1 J_2,
\end{align*}
where
\begin{align}
J_1 & = \chi_{2, w}\left(\det A_1\right)^{-1}\Phi_w^{(4)}\left(D_1\right)\left|\det D_1^{a_w}\right|_w^{1/2}\left\langle \phi_{b_w}, \tilde\pi_{b_w}\left(D^{-1}_1\right)\tilde\phi_{b_w}\right\rangle_{\pi_{b_w}}\label{j1def}\\
J_2 & = \chi_{1, w}\left(\det D_2\right)\Phi_w^{(1)}\left(A_2\right)\left|\det A_2^{b_w}\right|_w^{1/2}\left\langle \phi_{a_w}, \tilde\pi_{a_w}\left(A_2\right)\tilde\phi_{a_w}\right\rangle_{\pi_{a_w}}.\label{j2def}
\end{align}
\end{proof}

\begin{cor}\label{Ifactorsatp}
The integral $Z_w$
factors as
\begin{align*}
Z_w &=  \Vol(\Gamma_{a_w,b_w})\cdot I_1\cdot I_2,\\
I_1 &= \int_{\gl_{b_w}\left(\K_w\right)}\chi_{2, w}\left(\det D\right)\Phi_w^{(4)}(D)\left|\det D\right|_w^{s+\frac{b_w}{2}}\langle\pi_{b_w}(D)\phi_{b_w}, \tilde{\phi}_{b_w}\rangle_{\pi_{b_w}} d^\times D\\
I_2 &= \int_{\gl_{a_w}\left(\K_w\right)}\chi_{1, w}^{-1}\left(\det A\right)\Phi_w^{(1)}(A)\left|\det A\right|_w^{s+\frac{a_w}{2}}\langle\phi_{a_w}, \tilde{\pi}_{a_w}\left(A\right)\tilde{\phi}_{a_w}\rangle_{\pi_{a_w}} d^\times A.
\end{align*}
\end{cor}
\begin{proof}
By Equation \eqref{IZ1Z2} and Proposition \ref{Jfactors},
\begin{align}\label{J1J2Integrand}
Z_w= \Vol\left(\Gamma_w\right)^{-1}&\int_{A_1, A_2, B_1, B_2,C_1, C_2, D_1, D_2}\chi_{2, w}\left(\det \left(A_1\right)\det\left(D_1\right)\right)\chi_{1, w}^{-1}\left(\det \left(A_2\right)\det\left(D_2\right)\right)\\
&\times \left|\det\left(A_1\right)\det\left(D_1\right)\det\left(A_2\right)\det\left(D_2\right)\right|_w^{s+\frac{n}{2}} \Vol(\Gamma_{a_w,b_w}) J_1 J_2\nonumber \\
&\times \left|\det A_1^{b_w}\det D_1^{-a_w}\right|_w\left|\det A_2^{-b_w} \det D_2^{a_w}\right|_wd^\times A_1d^\times A_2dB_1dB_2dC_1dC_2d^\times D_1d^\times D_2,\nonumber
\end{align}
where $J_1$ and $J_2$ are defined as in Equations \eqref{j1def} and \eqref{j2def}, respectively, and
\begin{align*}
A_1&\in \Gamma_{a_w}(t)\\
C_1&\in \mathfrak{p}_w^t\prod_{w\divides p}M_{b_w\times a_w}\left(\O_w\right)\\
D_2&\in \Gamma_{b_w}(t)\\
B_2&\in \mathfrak{p}_w^t\prod_{w\divides p}M_{a_w\times b_w}\left(\O_w\right)\\
C_2&\in\prod_{w\divides p} M_{b_w\times a_w}\left(\O_w\right)\\
A_2&\in \mathfrak{p}_w^{-t}\prod_{w\divides p}M_{a_w\times a_w}\left(\O_w\right)\\
B_1&\in \prod_{w\divides p}M_{a_w\times b_w}\left(\O_w\right)\\
D_1&\in \mathfrak{p}_w^{-t}\prod_{w\divides p}M_{b_w\times b_w}\left(\O_w\right).
\end{align*}
Note that for such $A_1$ and $D_2$, $\left|\det A_1\right|_w = \left|\det D_2\right|_w = 1.$
Applying Equations \eqref{j1def} and \eqref{j2def}, we therefore see that the integrand in Equation \eqref{J1J2Integrand} equals
\begin{align*}
\chi_{2, w}\left(\det D_1\right)\Phi_w^{(4)}&\left(D_1\right)\langle\pi_{b_w}(D_1)\phi_{b_w}, \tilde{\phi}_{b_w}\rangle_{\pi_{b_w}}\left| D_1\right|_w^{s+\frac{b_w}{2}}\\
&\times\chi_{1, w}^{-1}\left(\det A_2\right)\Phi_w^{(1)}\left(A_2\right)\langle\phi_{a_w}, \tilde{\pi}_{a_w}\left(A_2\right)\tilde{\phi}_{a_w}\rangle_{\pi_{a_w}}\left|A_2\right|_w^{s+\frac{a_w}{2}}.
\end{align*}
Therefore,
\begin{align*}
Z_w = &\Vol\left(\Gamma_w\right)^{-1}\Vol\left(\Gamma_w\right)\left(\Vol\left(M_{a_w\times b_w}\left(\O_w\right)\right)\right)^2 \Vol(\Gamma_{a_w,b_w})\\
&\times \int_{\gl_{b_w}\left(\K_w\right)}\chi_{2, w}\left(\det D_1\right)\Phi_w^{(4)}\left(D_1\right)\langle\pi_{b_w}(D_1)\phi_{b_w}, \tilde{\phi}_{b_w}\rangle_{\pi_{b_w}}\left| D_1\right|_w^{s+\frac{b_w}{2}} d^\times D_1\\
&\times\int_{\gl_{a_w}\left(\K_w\right)}\chi_{1, w}^{-1}\left(\det A_2\right)\Phi_w^{(1)}\left(A_2\right)\langle\phi_{a_w}, \tilde{\pi}_{a_w}\left(A_2\right)\tilde{\phi}_{a_w}\rangle_{\pi_{a_w}}\left|A_2\right|_w^{s+\frac{a_w}{2}}d^\times A_2\\
=  & \Vol(\Gamma_{a_w,b_w}) \times \int_{\gl_{b_w}\left(\K_w\right)}\chi_{2, w}\left(\det D\right)\Phi_w^{(4)}\left(D\right)\langle\pi_{b_w}(D)\phi_{b_w}, \tilde{\phi}_{b_w}\rangle_{\pi_{b_w}}\left| D\right|_w^{s+\frac{b_w}{2}} d^\times D\\
&\times\int_{\gl_{a_w}\left(\K_w\right)}\chi_{1, w}^{-1}\left(\det A\right)\Phi_w^{(1)}\left(A\right)\langle\phi_{a_w}, \tilde{\pi}_{a_w}\left(A\right)\tilde{\phi}_{a_w}\rangle_{\pi_{a_w}}\left|A\right|_w^{s+\frac{a_w}{2}}d^\times A.
\end{align*}
\end{proof}

\subsubsection{The main local theorem}\label{mainlocaltheorem}

In Theorem \ref{GJintegrals}, we calculate the integrals $I_1$ and $I_2$ from Corollary \ref{Ifactorsatp}.

\begin{thm}\label{GJintegrals}  
The integrals $I_1$ and $I_2$ are related to familiar $L$-functions as follows.
\begin{align*}
I_1  &= \frac{L\left(s+\frac{1}{2}, \pi_{b_w}\otimes\chi_{2, w}\right)}{\varepsilon\left(s+\frac{1}{2}, \pi_{b_w}\otimes\chi_{2, w}\right)L\left(-s+\frac{1}{2}, \tilde\pi_{b_w}\otimes\chi^{-1}_{2, w}\right)}\cdot \Vol(\grX^{(4)})\cdot \langle \phi_{b_w},\tilde\phi_{b_w}\rangle_{\pi_{b_w}} \\
I_2 &= \frac{\varepsilon\left(-s+\frac{1}{2}, \pi_{a_w}\otimes\chi_{1, w}\right)L\left(\frac{1}{2}+s, \tilde\pi_{a_w}\otimes\chi_{1, w}^{-1}\right)}{L\left(-s+\frac{1}{2}, \pi_{a_w}\otimes\chi_{1, w}\right)} \cdot \Vol(\grX^{(1)})\cdot\langle \phi_{a_w},\tilde\phi_{a_w}\rangle_{\pi_{a_w}}.
\end{align*}
Consequently, upon setting $\mathfrak{V}_w := \Vol(\grX^{(1)})\Vol(\grX^{(4)})$, 
\begin{equation*}
Z_w  = \frac{L\left(s+\frac{1}{2}, \pi_{b_w}\otimes\chi_{2, w}\right) \cdot \varepsilon\left(-s+\frac{1}{2}, \pi_{a_w}\otimes\chi_{1, w}\right)L\left(\frac{1}{2}+s, \tilde\pi_{a_w}\otimes\chi_{1, w}^{-1}\right)} {\varepsilon\left(s+\frac{1}{2}, \pi_{b_w}\otimes\chi_{2, w}\right)L\left(-s+\frac{1}{2}, \tilde\pi_{b_w}\otimes\chi^{-1}_{2, w}\right)\cdot L\left(-s+\frac{1}{2}, \pi_{a_w}\otimes\chi_{1, w}\right)} \cdot \grV_w \cdot \langle \phi_w, \tilde\phi_w\rangle_{\pi_w},
\end{equation*}
and thus
\begin{equation}\label{zetapfinalform}
I_w = L(s+ \frac{1}{2},\ord,\pi_w,\chi_w) \cdot \grV_w\cdot \langle \phi_w, \tilde\phi_w\rangle_{\pi_w},
\end{equation}
where $L(s+ \frac{1}{2},\ord,\pi_w,\chi_w)$ is the ratio of $L$-factors and $\varepsilon$-factors that appear in the formula for $Z_w$.
\end{thm}
\begin{proof}
The integrals $I_1$ and $I_2$ are of the same form as the ``Godement-Jacquet" integral defined in \cite[Equation (1.1.3)]{jacquet}.  Applying the ``Godement-Jacquet functional equation" in \cite[Equation (1.3.7)]{jacquet}, we obtain
\begin{align}
I_1  = &\frac{L\left(s+\frac{1}{2}, \pi_{b_w}\otimes\chi_{2, w}\right)}{\varepsilon\left(s+\frac{1}{2}, \pi_{b_w}\otimes\chi_{2, w}\right)L\left(-s+\frac{1}{2}, \tilde\pi_{b_w}\otimes\chi^{-1}_{2, w}\right)}\nonumber\\
&\times\int_{\gl_{b_w}\left(\K_w\right)}\left(\Phi_w^{(4)}\right)^{\wedge}\left(D\right)\left|\det D\right|_w^{-s+\frac{b_w}{2}}\chi_{2, w}^{-1}\left(D\right)\langle\phi_{b_w}, \tilde{\pi}_{b_w}(D)\tilde\phi_{b_w}\rangle_{\pi_{b_w}} d^\times D,\label{I1integral}
\end{align}
where $\left(\Phi_w^{(4)}\right)^{\wedge}$ denotes the Fourier transform of $\Phi_w^{(4)}$.  

From its definition, 
\begin{align}\label{Phi4hat}
\left(\Phi_w^{(4)}\right)^{\wedge}\left(D\right) = \phi_{\nu_w}\left(\begin{pmatrix}1_{a_w}& 0\\ 0 &D\end{pmatrix}\right).
\end{align}
In particular, the support of $\left(\Phi_w^{(4)}\right)^{\wedge}$ is 
$\grX^{(4)} = {}^t\Gamma_{b_w}\Gamma_{b_w}$.
Let $D = \gamma_1 \gamma_2$, $\gamma_1 \in {}^t\Gamma_{b_w}$ and $\gamma_2 \in \Gamma_{b_w}$.
Applying Equations \eqref{phiinvarianceb} and \eqref{tphiinvariancew2} and the definition of $\phi_{\nu_w}$ we see that 
\begin{align*}
\langle \phi_{b_w}, \tilde{\pi}_{b_w}(D)\tilde\phi_{b_w}\rangle_{\pi_{b_w}} 
& = \langle \pi_{b_w}(\gamma_1^{-1})\phi_{b_w}, \tilde{\pi}_{b_w}(\gamma_2)\tilde\phi_{b_w}\rangle_{\pi_{b_w}} \\
& = \mu'_{b_w}(\gamma_1^{-1})\mu_{b_w}(\gamma_2^{-1}) \langle \phi_{b_w}, \tilde\phi_{b_w}\rangle_{\pi_{b_w}}
& = \chi_{2,w}(D)\phi_{\nu_w}\left(\begin{pmatrix}1_{a_w}& 0\\ 0 &D\end{pmatrix}\right)^{-1} \langle \phi_{b_w}, \tilde\phi_{b_w}\rangle_{\pi_{b_w}}.
\end{align*}
Plugging this into Equation \eqref{I1integral} and applying Equation \eqref{Phi4hat}, we obtain
\begin{align*}
I_1  = \frac{L\left(s+\frac{1}{2}, \pi_{b_w}\otimes\chi_{2, w}\right)}{\varepsilon\left(s+\frac{1}{2}, \pi_{b_w}\otimes\chi_{2, w}\right)L\left(-s+\frac{1}{2}, \tilde\pi_{b_w}\otimes\chi^{-1}_{2, w}\right)} \cdot \Vol(\grX^{(4)}) \cdot \langle \phi_{b_w}, \tilde\phi_{b_w}\rangle_{\pi_{b_w}}.
\end{align*}
The computation of $I_2$ is similar.  
The consequence for $I_w$ then follows from Corollary \ref{Ifactorsatp} and Equations \eqref{pair=ab}
and \eqref{normalizedzeta}.
\end{proof}

\begin{rmk}
Let $\omega_{a_w}$ denote the central quasi-character of $\pi_{a_w}$.  Then
\begin{align*}
\varepsilon\left(-s+\frac{1}{2}, \pi_{a_w}\otimes\chi_{1, w}\right) = \frac{\omega_{a_w}(-1)}{\varepsilon\left(s+\frac{1}{2}, \tilde\pi_{a_w}\otimes\chi_{1, w}^{-1}\right)}.
\end{align*}
So we may rewrite $I_2$ as
$$
I_2 = \frac{\omega_{a_w}(-1)L\left(\frac{1}{2}+s, \tilde\pi_{a_w}\otimes\chi_{1, w}^{-1}\right)}{\varepsilon\left(s+\frac{1}{2}, \tilde\pi_{a_w}\otimes\chi_{1, w}^{-1}\right)L\left(-s+\frac{1}{2}, \pi_{a_w}\otimes\chi_{1, w}\right)}\cdot\Vol(\mathfrak{X}^{(1)}),
$$
and hence $I_w/\grV$ as  
$$
\frac{I_w}{\grV} = \frac{\omega_{a_w}(-1) L\left(\frac{1}{2}+s, \tilde\pi_{a_w}\otimes\chi_{1, w}^{-1}\right)}{\varepsilon\left(s+\frac{1}{2}, \tilde\pi_{a_w}\otimes\chi_{1, w}^{-1}\right)L\left(-s+\frac{1}{2}, \pi_{a_w}\otimes\chi_{1,w}\right)}\frac{L\left(s+\frac{1}{2}, \pi_{b_w}\otimes\chi_{2, w}\right)}{\varepsilon\left(s+\frac{1}{2}, \pi_{b_w}\otimes\chi_{2, w}\right)L\left(-s+\frac{1}{2}, \tilde\pi_{b_w}\otimes\chi^{-1}_{2, w}\right)}.
$$
Therefore, the Euler factor at $p$, which is the product 
$$\prod_{w \mid p} L(s+ \frac{1}{2},\ord,\pi_w,\chi_w)$$
of the factors defined in equation \eqref{zetapfinalform},
can also be written 
\begin{equation}\label{Ireexpressed}
\prod_{w\divides p}\frac{\omega_{a_w}(-1) L\left(\frac{1}{2}+s, \tilde\pi_{a_w}\otimes\chi_{1, w}^{-1}\right)}{\varepsilon\left(s+\frac{1}{2}, \tilde\pi_{a_w}\otimes\chi_{1, w}^{-1}\right)L\left(-s+\frac{1}{2}, \pi_{a_w}\otimes\chi_{1,w}\right)}\frac{L\left(s+\frac{1}{2}, \pi_{b_w}\otimes\chi_{2, w}\right)}{\varepsilon\left(s+\frac{1}{2}, \pi_{b_w}\otimes\chi_{2, w}\right)L\left(-s+\frac{1}{2}, \tilde\pi_{b_w}\otimes\chi^{-1}_{2, w}\right)}.
\end{equation}

Note the similarity of the form of the zeta integral at $p$ in Equation \eqref{Ireexpressed} with the form of the modified Euler factor at $p$ for the $p$-adic $L$-functions predicted by Coates in \cite[Section  2, Equation 18b]{coatesII}.
\end{rmk}

\begin{rmk} \label{volume-rmk}
The factor $\grV_w$ can be written as      
\begin{equation}\label{grV-formula}
\grV_w = \frac{\Vol(\Gamma_{R,w})\cdot \Vol(\t\Gamma_{R,w})}{\Vol(\Gamma_{R,w}\cap\t\Gamma_{R,w})}
\end{equation}
for any $r\geq 1$.
From this we conclude that
\begin{equation*}
\frac{I_w}{\Vol(\Gamma_{R,w})\cdot \Vol(\t\Gamma_{R,w})} 
= L(s+ \frac{1}{2},\ord,\pi_w,\chi_w) \frac{\langle \phi_w,\tilde\phi_w\rangle_{\pi_w}}{\Vol(\Gamma_{R,w}\cap\t\Gamma_{R,w})}.
\end{equation*}
The right-hand side (and hence the left-hand side) is easily seen to be independent of $r$.

In order to explain the cancellation of various intermediate volume factors appearing along the way
to the final expression for the values of our $p$-adic $L$-function in Theorem \ref{axiomaticmainthm2}, 
we identify the volumes in this last expression (in a special case) with volumes of groups defined elsewhere.

Suppose that $n_{i,w} = 1$ for all $i$. Then $R_w$ is just the Borel $B_w$ and $\Gamma_{R,w} = I_{w,r}^0$,
where $I_{w,r}^0$ is the $w$ factor of the image of $I_{r,V}^0 = I_r^0$ under the isomorphism \eqref{B+-iso}.
Then $\t\Gamma_{R,w} = \t I_{w,r}^0$ can be identified with the corresponding factor of $I_{r,-V}^0$.
\end{rmk}

\subsection{Holomorphic representations of enveloping algebras and anti-holomorphic vectors}\label{holo-sec}

\subsubsection{Holomorphic and anti-holomorphic modules}\label{holoreps}
Throughout this section, we identify $\Sigma$ with $\Sigma_{\K^+}$, and we identify each element $\sigma\in\Sigma$ with the restriction $\sigma|_{\K^+}$. 
To simplify notation, we let $G^* = GU^+_1 = R_{\K/\Q}GU^+(V)$ where $GU^+(V)$ denotes the full unitary similitude group
of $V$.  Thus $G^*(\R) = \prod_{\sigma \in \Sigma_{\K^+}} G_{\sigma}$, with $G_{\sigma} = GU^+(V)_{\K^+_{\sigma}} \simeq GU^+(a_{\sigma},b_{\sigma})$. 
For any $h: R_{\C/\R}(\mathbb{G}_{m,\C}) \rar G^*_{\R}$ as in Section  \ref{PELdata}, the image of $h$ is contained in the subgroup $G$
of $(g_{\sigma}, \sigma \in \Sigma_{\K^+})$ for which the similitude factor $\nu(g_{\sigma})$ is independent of $\sigma$, and it is to this
latter subgroup that the Shimura variety is attached.

 Let $U_\infty = C(\IR)\subseteq G(\IR)$ and $X$ be as in Section  \ref{spaces-section}.  
We assume $U_{\infty}$ is the centralizer of  
a {\it standard} $h$ as in Section  \ref{basepoints}; let $U_{\sigma} \subset G_{\sigma}$ denote its intersection with $G_{\sigma}$
and let $K^o_{\sigma} \subset U_{\sigma}$ denote its maximal compact subgroup; $K^o_{\sigma}$ is isomorphic to the product of compact unitary
groups $U(a_{\sigma}) \times U(b_{\sigma})$.   We have
$$\grk_{\sigma} := Lie(U_{\sigma}) = \grz_{\sigma}\oplus Lie(K^o_{\sigma})$$
where $\grz_{\sigma}$ is the $\R$-split center of $\grg_{\sigma} := Lie(GU^+(a_{\sigma},b_{\sigma}))$.  
We let $U(\grg_{\sigma})$ denote the enveloping algebra of $\grg_{\sigma}$.

For $\sigma \in \Sigma_{\K^+}$, we write the Harish-Chandra decomposition
$$\grg_{\sigma}  = \grk_{\sigma} \oplus \grp^-_{\sigma} \oplus \grp^+_{\sigma}.$$
Because $h$ was chosen to be standard, this decomposition is naturally defined over $\sigma(\K) \subset \C$. 
For any irreducible representation $(\tau_{\sigma},W_{\tau_{\sigma}})$ of $U_{\sigma}$
of $G_{\sigma} := G(\K^+_{\sigma})$, we let
\begin{equation}\label{verma}
\DD(\tau_{\sigma}) = U(\grg_{\sigma})\otimes_{U(\grk_{\sigma} \oplus \grp^-_{\sigma})} W^{\vee}_{\tau_{\sigma}}
\end{equation}

We have assumed that our chosen $h$ takes values in a rational torus $T (= J_{0,n}) \subset G$ (so that $(T,h)$ is a CM Shimura datum), and let $T_\s \subset G_\s$
be the $\s$-component of $T(\R)$, $\grt_\s$ its Lie algebra.  We choose a positive root system $R^+_{\s}$ for $T_{\s}$ so that the roots on
$\grp^+_{\sigma}$ are positive, and let $\grb^+_{\s}$ be the corresponding Borel subalgebra.

Let $R^{+,c}_{\s} \subset R^+_{\s}$ be the set of positive compact roots.
The highest weight of $\tau_{\sigma}$ relative to $R^{+,c}_{\s}$ can be denoted 
$\kap_{\sigma} = (c_{\s}; \kap_{1,\s} \geq \dots  \geq \kap_{a_{\sigma},\s}; \lap_{1,\s} \geq \dots \geq \lap_{b_{\sigma},\s}) \in \Z \times  \Z^{a_\sigma} \times \Z^{b_{\sigma}}$,
where $c_\s$ is the character of $\grz_\s$.  
We call $(\tau_{\sigma},W_{\tau_{\sigma}})$ {\it strongly positive} if there exists an irreducible representation $\mathcal{W}_\s$ of $G_{\sigma}$,
with highest weight $\mu = (-c_{\s}; a_1 \geq \dots \geq a_n) \in \Z \times \Z^n$ relative to $R^+_{\s}$, such that, setting $a = a_\s$ and $b = b_\s$,
\begin{equation}\label{coeffs}   (a_1,\dots, a_n) = (-\lap_{b,\s}-a, \dots, -\lap_{1,\s}-a; -\kap_{a,\s}+b, \dots, -\kap_{1,\s}+b);
\end{equation}
in other words, if and only if  $-\lap_{1,\s} - a \geq -\kap_{a,\s} + b$.
The contragredient of $\DD(\tau_{\sigma})$ is denoted
\begin{equation}\label{vermadual}
\Dc(\tau_{\sigma}) = \DD(\tau_{\sigma})^{\vee} \cong U(\grg_{\sigma})\otimes_{U(\grk_{\sigma} \oplus \grp^+_{\sigma})} W_{\tau_{\sigma}}.
\end{equation}
It is the complex conjugate representation of $\DD(\tau_{\sigma})$ with respect to the $\R$-structure on $\grg_{\sigma}$; we call this
the {\it anti-holomorphic representation} of type $\tau_{\sigma}$.

In what follows, we usually write $\DD(\kap_\s)$ instead of $\DD(\tau_\s)$.   It is well known that
if $\tau_{\sigma}$ is strongly positive then $\DD(\kap_{\sigma})$ (resp. $\Dc(\kap_{\sigma})$) is the $(U(\grg_{\sigma}),U_{\sigma})$-module
of a {\it holomorphic } (resp. {\it anti-holomorphic}) {\it discrete series} representation of $G_{\sigma}$, and moreover that
$$\dim H^{ab}(\grg_{\s},U_\s; \DD(\kap_{\sigma})\otimes \mathcal{W}_\s) =  \dim H^{ab}(\grg_{\s},U_\s; \Dc(\kap_{\sigma})\otimes \mathcal{W}^{\vee}_\s) = 1$$
with $\mathcal{W}_\s$ the representation with highest weight given by \eqref{coeffs}, and
$\mathcal{W}^{\vee}_\s$ its dual, with highest weight 
\begin{equation}\label{coeffsc} (c_\s; -a_n,\dots, -a_1) = (c_\s; \kap_{1,\s} - b; \dots, \kap_{a,\s} - b, \lap_{1,\s} + a,\dots, \lap_{b,\s} + a).
\end{equation}

The {\it minimal $U_{\sigma}$-type} of $\DD(\kap_{\sigma})$ (resp. of $\Dc(\kap_{\sigma}$))
is the 
subspace 
$$1\otimes W^\vee_{\tau_{\sigma}} \subset U(\grg_{\sigma})\otimes_{U(\grk_{\sigma} \oplus \grp^-_{\sigma})} W^\vee_{\tau_{\sigma}}
~~\text{(resp. $1\otimes W_{\tau_{\sigma}} \subset U(\grg_{\sigma})\otimes_{U(\grk_{\sigma} \oplus \grp^+_{\sigma})} W_{\tau_{\sigma}}$)}.$$
The minimal $U_{\sigma}$-type of $\DD(\kap_{\sigma})$ (resp. of $\Dc(\kap_{\sigma})$) is also called the space
of {\it holomorphic vectors} (resp. {\it anti-holomorphic vectors}).   

\subsubsection{Canonical automorphy factors and representations}\label{automorphy}

The $(U(\grg_\s),U_{\s})$ module $\DD(\kap_{\sigma})$ can be realized as a subrepresentation
of the right regular representation on $C^{\infty}(G_{\sigma})$ generated by a canonical automorphy factor.
We recall this construction below when $G_{\s} = G_{4,\s} \simeq GU(n,n)$ and $\tau_\s$ is a scalar representation. 

Let $M_n$ be the affine group scheme of $n \times n$-matrices over $Spec(\Z)$, $M_n = Spec(\CP(n))$.   For $\s \in \Sigma$,
let $\CP(n)_\s$ denote the base change of $\CP(n)$ to $\CO_\s = \sigma(\CO_{\CK})$.  
Corresponding to the factorization $G^*(\R) = \prod_{\sigma} G_{\sigma}$, we write
$X = \prod_{\sigma\in\Sigma} X_\sigma$.    The maximal parabolic $P_n$, together with $U_\s$, defines
an unbounded realization of a connected component $X_\s^+  \subset X_\s$ as a tube domain in $\grp^+_{4,\s}$ (\cite{harris86} (5.3.2)).  A choice of 
basis
for $L_1$, together with the identification of $V$ with $V_d$ and $V^d$ introduced in Section  \ref{SiegelParabolic-section}, 
identifies $\grp^+_{4,\s}$ with $M_n(\C)$ and therefore
identifies $X^+_\sigma$ with a  tube domain in $M_n(\C)$.  Let $\gimel_\s \in X^+_\s$ be the fixed point of $U_\s$.
Without loss of generality, we may assume $\gimel_\s$ to be a diagonal matrix with values in $\s(\CK)$ whose entries 
have trace zero down to $\CK^+$.  Then $X^+_{\sigma}$ is identified with
the standard tube domain 
$$X_{n,n} := \left\{z\in M_n\left(\IC\right)\mid \gimel_\s\left({ }^t\bar{z}-z\right)>0\right\}.$$
With respect to this identification,  any $g_\sigma = \begin{pmatrix}a_\sigma & b_\sigma\\ c_\sigma & d_\sigma\end{pmatrix} \in G_\sigma$ acts by 
$g_\sigma(z) = \left(a_\sigma z +b_\sigma\right)\left(c_\sigma z+ d_\sigma\right)^{-1}$.  (Here $a_\sigma, b_\sigma, c_\sigma,$ and $ d_\sigma$ are $n\times n$ matrices.)

For $z = (z_\s)_{\s \in\Sigma}\in X = \prod_{\s \in\Sigma}X_\s$ and $g = (g_\s)_{\s \in\Sigma} \in \prod_{\s \in\Sigma}G(\IR)$, let
\begin{align*}
J'(g_\s, z_\s) &= \overline{c_\s}\cdot{ }^tz_\s + \overline{d_\s} ~~~ &J'(g, z) &= \prod_{\s\in\Sigma}J'(g_\s, z_\s)\\
J(g_\s, z_\s) &= c_\s z_\s + d_\s ~~~
&J(g, z)  &= \prod_{\s\in\Sigma}J(g_\s, z_\s)
\end{align*}
Let
\begin{align*}
j_{g_\s}(z_\s) & = j(g_\s, z_\s) = \det J(g_\s, z_\s)\\
 &(= \nu(g_\s)^{-n}\det(g_\s)\det(J'(g_\s, z_\s)) = \nu(g_\s)^n\det(\overline{g_\s})^{-1}\det(J'(g_\s, z_\s))\\
j_g(z)&=j(g, z) = \prod_{\s\in\Sigma}j_{g_\s}(z_\s).
\end{align*}

Fix $\s \in \Sigma$.  For $g \in G_{\sigma}$, let
$$J(g) = J(g,\gimel_\sigma); J'(g) = J'(g,\gimel_\sigma).$$
These are $C^{\infty}$-functions on $G_{\sigma}$ with values in $\GL(n,\C)$, and any polynomial function
of $J$ and $J'$ is annihilated by $\grp^-_{\sigma}$ and is contained in a finite-dimensional $\grk_{\sigma}$ subrepresentation of 
$C^{\infty}(G_{\sigma})$.  Similarly, let
$$j(g) = \det(J(g));~~ j'(g) = \det(J'(g)),$$
viewed as $C^{\infty}$-functions on $G_{\sigma}$ with values in $\C^{\times}$.

Let $\chi = ||\bullet||^m\cdot\chi_0$ be an algebraic Hecke character of $\K$, where
 $m\in \ZZ$ and 
 $$\chi_{0,\s}(z) = z^{-a\left(\chi_\sigma\right)}\bar{z}^{-b\left(\chi_\sigma\right)}$$
 for any archimedean place $\s$.  
Define
$\DD^{2}(\chi_\s) = \DD^{2}(m,\chi_{0,\sigma})$ 
to be the holomorphic $(Lie(G_{4,\s}),U_{\s})$-module with highest $U_\sigma$-type
$$\Lambda(\chi_{\sigma}) = \Lambda(m,\chi_{0,\sigma}) = (m-b(\chi_\s),m-b(\chi_\s),\dots, m-b(\chi_\s); -m + a(\chi_\s),\dots, -m + a(\chi_\s);\bullet)$$
in the notation of \cite[(3.3.2)]{harriscrelle}.  Here $\bullet$ is the character of the $\R$-split center of $U_{\s}$ (denoted $c$ in \cite{harriscrelle}), which we omit
to specify because it has no bearing on the integral representation of the $L$-function.
We define a map of $(U(\grg_\s),U_{\s})$-modules
\begin{equation}\label{archautomorphy} \iota(\chi_{\sigma}):  \DD^{2}(\chi_{\sigma})  \rar C^{\infty}(G_{\sigma}) \end{equation}
as follows.  Let $v(\chi_\s)$ be the tautological generator of the $\Lambda(m,\chi_{0,\sigma})$-isotypic subspace (highest $U_\sigma$-type subspace)
of $\DD^{2}(m,\chi_{\sigma})$.  Let
$$\iota(\chi_{\sigma})(v(\chi_{\sigma})) = J_{\chi_{\sigma}}(g) 
:= j(g)^{-m + a(\chi_\s)}\cdot  j'(g)^{-m + b(\chi_\s)}\nu(g)^{n(m+a(\chi_\s)+b(\chi_\s))} $$
and extend this to a map of $U(\grg_\s)$-modules.    Let $C(G_{\sigma},\chi_{\sigma})$ denote the image of $\iota(\chi_{\sigma})$.

\begin{rmk}\label{independence}  Note that $J_{\chi_{\sigma}}$ depends only on the archimedean character 
$\chi_{\s} = ||\bullet||_\s^m\chi_{0,\s}$.  
\end{rmk}

We will only take $m$ in the closed right half-plane bounded by the center of symmetry of the functional equation of
the Eisenstein series, as in \cite{Harris-rationality}.   
For such $m$, the restriction of  $\DD^{2}(m,\chi_{0,\sigma})$ to $U_{3,\sigma} = U(a_{\sigma},b_{\sigma})\times U(b_{\sigma},a_{\sigma})$
decomposes as an infinite direct sum of irreducible holomorphic discrete series representations of the kind introduced in \ref{holoreps}:
\begin{equation}\label{decomp}
\DD^{2}(m,\chi_{0,\sigma}) = \bigoplus_{\kap_\s \in C_3(m,\chi_{0,
\sigma})} \DD(\kap_\s) \otimes \DD(\kap^{\flat}_\s\otimes \chi_{0,\s}) = 
\bigoplus_{\kap_\s \in C_3(\chi_\sigma)} \DD(\kap_\s) \otimes \DD(\kap^{\flat}_\s\otimes \chi_{\s})
\end{equation}
where $C_3(\chi_{\sigma}) = C_3(m,\chi_{0,\sigma})$ is a countable set of highest weights:
\begin{equation}\label{C3m} C_3(\chi_{\sigma}) = \{(-m+b(\chi_\s)-r_{a_\s},\dots,-m+b(\chi_\s) - r_1; m-a(\chi_\s) + s_1,\dots, m - a(\chi_\s) + s_{b_\s})\}
\footnote{Note the change of sign relative to $\Lambda(\chi_{\sigma})$! This is due to the duality in the definition \eqref{verma}.} 
\end{equation} 
where 
\begin{equation}\label{parameters} r_1 \geq r_2 \geq \dots \geq r_{a_\s} \geq 0; s_1  \geq s_2 \geq \dots \geq s_{b_\s}  \geq 0.\end{equation}
(Compare \cite[Lemma 3.3.7]{harriscrelle} when $a(\chi_s) = 0$.)   There is an explicit formula for $\kappa^{\flat}$ in \eqref{kapflat}, but the simplest explanation is probably that, if we identify holomorphic representations of $U(b_{\sigma},a_{\sigma})$ with anti-holomorphic representations of $U(a_{\sigma},b_{\sigma})$, then 
$$\DD(\kappa^{\flat}_\s) \isoarrow \DD(\kappa_\s)^{\vee}$$
as representations of $U(a_{\sigma},b_{\sigma})$.

For each $\sigma \in \Sigma$, we define 
$$(\alpha(\chi_{\sigma}),\beta(\chi_{\sigma})) =  (-m+b(\chi_{\sigma}),\dots,-m+b(\chi_{\sigma}); m-a(\chi_\s),\dots,m-a(\chi_\s)) \in\ZZ^{a_{\sigma} + b_{\sigma}}$$
and let
\begin{equation}\label{shift} (\alpha(\chi),\beta(\chi)) = (\alpha(\chi_{\sigma}),\beta(\chi_{\sigma}))_{\sigma \in \Sigma}
\end{equation}
For $\kappa = (\kappa_{\sigma})_{\sigma \in \Sigma}$, with $\kappa_{\sigma} \in C_3(\chi_{\sigma})$, we define 
\begin{equation}\begin{split}
\label{iota} \rho_\s = \kap_\s - (\alpha(\chi_{\sigma}),\beta(\chi_{\sigma})) = (-r_{a_\s},\dots, - r_1;  s_1,\dots,  s_{b_\s}); \\
\rho^{\upsilon}_\s = ( r_1, \dots,r_{a_\s};  s_1,\dots,  s_{b_\s});  \rho = (\rho_\s)_{\s \in \Sigma}; \rho^{\upsilon} = (\rho^{\upsilon}_\s)_{\s \in \Sigma}
\end{split}
\end{equation}
The involution $\upsilon$  on the parameters $(r_i,s_j)$ corresponds to an algebraic involution, also denoted $\upsilon$, of the torus $T$.

The algebraic characters $\rho$, $\rho^{\upsilon}$, and $\kappa$ all determine one another and will be used in the characterization of the
Eisenstein measure in subsequent sections.

Note that the twist by $\chi_{0,\s}$ coincides with the twist by $\chi_\s$ because
the norm of the determinant is trivial on $U(b_{\sigma},a_{\sigma})$.  We prefer to write the twist by $\chi_\s$, which is more appropriate for
parametrizing automorphic representations of unitary similitude groups.
  
\begin{lem}\label{d2ms}   For such $\chi$, the map $\iota(\chi_{\sigma})$ of \eqref{archautomorphy} is injective for all $\sigma$.   In particular, the image
$C(G_{\sigma},\chi_{\sigma})$ of $\iota(\chi_{\sigma})$ is a free $U(\grp^+_{\sigma}) \isoarrow S(\grp^+_{\sigma})$-module of rank $1$.
\end{lem}
\begin{proof}  Indeed,  $\DD^{2}(\chi_{\sigma})$ is always a free rank one $U(\grp^+_{\sigma})$-module,
and for $m$ in the indicated range is irreducible as $U(\grg_{\sigma})$-module.  Since $\iota(\chi_{\sigma})$ is not the zero homomorphism,
it is therefore injective.
\end{proof}

\begin{defi}  Let $\kap = (\kap_\s, \s \in \Sigma)$, where for each $\s$, $\kap_\s$ is the highest weight of an irreducible
representation $\tau_\s$ of $U_\s$.  Let $(\chi_\s, \s \in \Sigma)$ be the archimedean parameter of an algebraic Hecke character $\chi$ of
$\K$.  The pair $(\kap,\chi)$ (or the triple $(\kappa,m,\chi_0)$) is {\rm critical} if
$\kap_{\sigma} \in C_3(\chi_\s)$ for all $\s \in \Sigma$.

If $\pi$ is an anti-holomorphic automorphic representation of $G_1$ of type $\kap$, we say $(\pi,\chi)$ is {\rm critical} if $(\kap,\chi)$ is critical.
\end{defi}

\begin{rmk}  When $\K$ is imaginary quadratic, the discussion in \cite[Section  3]{harriscrelle} shows that, for fixed $\pi$ and $\chi$, the set of $m$
such that $(\pi,m,\chi_0)$ is critical is exactly the set of critical values of $L(s+ \frac{1}{2},\pi,\chi)$ greater than or equal to the center of symmetry of the functional equation.
The same considerations show that this is true for an arbitrary CM field.  The verification is simple but superfluous unless one wants to compare the
results of the present paper to conjectures on critical values of $L$-functions.
\end{rmk}

Let $v_{\kap_\s}\otimes v_{\kap^{\flat}_\s\otimes \chi_{\s}}$ denote a highest weight vector in the minimal $K_3$-type of $\DD(\kap_\s) \otimes \DD(\kap^{\flat}_\s\otimes \chi_{\s})$, relative to a choice of compact maximal tori in $U_{3,\sigma}$ as in \ref{holoreps}.   
The holomorphic module $\DD^{2}(\chi_{\sigma})$  is a free rank one module
over $U(\grp_4^+)$, generated by $v(\chi_\s) \in \Lambda(\chi_{\sigma})$.  There is therefore a unique element 
$\delta_{\chi_\s,\kap_\s} \in U(\grp_4^+)$
such that
\begin{equation}\label{kappadiffops}  \delta_{\chi_\s,\kap_\s}\cdot v(\chi_\s) = v_{\kap_\s}\otimes v_{\kap^{\flat}_\s\otimes \chi_{\s}}.\end{equation}
The differential operator $\delta_{\chi_\s,\kap_\s}$ depends on the choice of basis vectors but is otherwise well-defined up to scalar multiples.
The module $\DD(\kap_\s) \otimes \DD(\kap^{\flat}_\s\otimes \chi_{\s})$ has a natural rational structure over the field of definition $E(\tau_\s,\chi_\s)$ of
$\tau_\s \boxtimes \tau^{\flat}_\s\otimes \chi_\s$.
Let $span(v_{\kap_\s}\otimes v_{\kap^{\flat}_\s\otimes \chi_{\s}})$ denote the $E(\tau_\s,\chi_\s)$-line in  $\DD(\kap_\s) \otimes \DD(\kap^{\flat}_\s\otimes \chi_{\s})$
spanned by the indicated vector.  We always choose $v_{\kap_\s}\otimes v_{\kap^{\flat}_\s\otimes \chi_{\s}}$ to be rational over
$E(\tau_\s,\chi_\s)$.

\subsubsection{Holomorphic projection}\label{holoprojection-sec}

We let $pr_{\kap,\s}:  \DD^{2}(\chi_{\sigma}) \rar \DD(\kap_\s) \otimes \DD(\kap^{\flat}_\s\otimes \chi_{\s})$
denote the natural projection and 
$$pr_{\kap,\s}^{hol} = pr_{\kap,\s}^{hol;a_\s,b_\s}:  \DD^{2}(\chi_{\sigma}) \rar span(v_{\kap_\s}\otimes v_{\kap^{\flat}_\s\otimes \chi_{\s}})$$
denote $pr_{\kap,\s}$ followed by orthogonal projection on the highest weight component of the holomorphic subspace.  
Let 
$$\DD^{2}(\chi_{\sigma})^{hol;a_\s,b_\s} = \bigoplus_{\kap_\s \in C_3(m,\chi_\s)} im(pr_{\kap,\s}^{hol})$$
and let 
$$pr^{hol} = \bigoplus pr_{\kap,\s}^{hol}:  \DD^{2}(\chi_{\sigma})  \rar \DD^{2}(\chi_{\sigma})^{hol;a_\s,b_\s}.$$

Because we have
chosen $h$ standard, the enveloping algebra $U(\grg_\s)$ and its subalgebra
$U(\grp^+_{4,\s}) \simeq S(\grp^+_{4,\s})$ have
models over $\CO_\s$.
We define an isomorphism of $\CO_\s$ algebras
 \begin{equation}\label{pluridiff} S(\grp_{4,_\s}^+) \isoarrow \CP(n)_\s \end{equation} 
using the identification of section \ref{automorphy}.

Let $n = a_\s + b_\s$ be a signature at $\s$.  We write $X \in M_n$ in the form
$X = \begin{pmatrix} A(X) & B(X) \\ C(X) & D(X) \end{pmatrix}$ with $A(X) \in M_{a_\s}$ (an $a_\s \times a_\s$) block, $D(X) \in M_{b_\s}$,
and $B(X)$ and $C(X)$ rectangular matrices.  With respect to this decomposition and the isomorphism \eqref{pluridiff} we obtain a natural
map
$$j(a_\s,b_\s):  \CP(a_\s)_\s \otimes \CP(b_\s)_\s \hookrightarrow \CP(n)_\s \isoarrow U(\grp_{4,\s}^+).$$
  For $i = 1,\dots, a_\s$ (resp. $j = 1,\dots, b_\s$) let 
$\Delta_i(X)$ (resp. $\Delta'_j(X)$) be the element of $\CP(a_\s)_\s$ (resp. $\CP(b_\s)_\s$)
given by the $i$th minor of $A$ (resp. the $j$th minor of $D$) starting from the upper
left corner.  
Let $r_{1,\s} \geq \dots \geq r_{a_\s,\s} \geq r_{a_\s + 1,\s} = 0$, $s_{1,\s} \geq \dots \geq s_{b_\s,\s} \geq s_{b_\s + 1,\s} = 0$ be
descending sequences of integers as in Inequalities \eqref{parameters}.   Let 
$$\tb_{i,\s} = r_{i,\s} - r_{i+1,\s}, i = 1, \dots, a_\s;  \tc_{j,\s} = s_{j,\s} - s_{j+1,\s}, j = 1, \dots, b_\s .$$
and define 
$p(\ub_\s,\uc_\s) \in \CP(n)_\s$ by
\begin{equation}\label{plurihar} p(\ub_\s,\uc_\s)(X)  = j(a_\s,b_\s)(\prod_{i = 1}^{a_\s} \Delta_i(X)^{\tb_{i,\s}}\cdot  \prod_{j = 1}^{b_\s} \Delta'_j(X)^{\tc_{j,\s}}) 
\end{equation}
Let $\delta(\ub_\s,\uc_\s) \in U(\grp^+_\s)$ be the differential operator corresponding to $p(\ub_\s,\uc_\s)$ under the isomorphism \eqref{pluridiff}.

The group $GL(a_\s)\times GL(a_\s)$ (resp. $GL(b_\s)\times GL(b_\s)$) acts on $\CP(a_\s)_\s$ (resp. $\CP(b_\s)_\s$) by the map
$(g_1,g_2)(X) = {}^tg_1^{-1}Xg_2$, and the action preserves the grading by degree.  
With respect to the standard upper-triangular Borel subgroups, we can index representations
of $GL(a_\s)$ (resp. $GL(b_\s)$) by their highest weights, which are $a_\s$-tuples  of integers $r_1  \geq r_2 \geq \dots \geq r_{a_\s}$ (resp.
$b_\s$-tuples $s_1 \geq s_2 \geq \dots \geq s_{b_\s}$).   The following is a statement of classical Schur-Weyl duality:

\begin{lem}\label{schurweyl}  Let $u = a_\s$ or $b_\s$.   As a representation of $GL(u)\times GL(u)$, the degree $d$-subspace $\CP(u)^d_\s \subset
\CP(u)_\s$ decomposes
as the direct sum
$$\CP(u)^d_\s \isoarrow \bigoplus_{\mu} [F^{\mu,\vee}\otimes F^{\mu}]$$
where $\mu$ runs over $r$-tuples $c_1  \geq c_2 \geq \dots \geq c_u \geq 0$ such that $\sum_i c_i = d$.  Moreover, if
$\mu = c_1  \geq c_2 \geq \dots \geq c_u \geq c_{u+1} = 0$, the highest weight space $\CF^{\mu,+} \subset [F^{\mu,\vee}\otimes F^{\mu}]$ is spanned by the polynomial
$\Delta^{\mu} = \prod_{i = 1}^r \Delta_i^{c_i - c_{i+1}}$.  
\end{lem}
\begin{proof}  This is the case $n = k = r$ of Theorem 5.6.7 of \cite{goodmanwallach}.
\end{proof}

Define the (one-dimensional) highest weight space $\CF^{\mu,+}$ as in the statement of the lemma, and write
$$\CP(u)_\s^+ = \bigoplus_{\mu} \CF^{\mu,+}.$$
Recall the notation of \eqref{iota}.  

\begin{cor}\label{pluridecomp}  Let $(\kap,\chi)$ be critical.  For each $\s \in \Sigma$, there is a unique $a_\s + b_\s$-tuple
$$\rho^{\upsilon}_{\sigma} = (r_{1,\s} \geq \dots \geq r_{a_\s,\s} \geq  0; s_{1,\s} \geq \dots \geq s_{b_\s,\s} \geq 0)$$ 
as above such that 
$$pr^{hol}_{\kap,\s}(\delta(\ub_\s,\uc_s)\cdot v(\chi_\s)) = P_{\kap_\s,\chi,\s}\cdot  v_{\kap_\s}\otimes v_{\kap^{\flat}_\s\otimes \chi_{\s}}$$
with $P_{\kap_\s,\chi,\s}$ a non-zero scalar in $E(\tau_\s,\chi_\s)^{\times}$.

We write 
$$D(\rho^{\upsilon}_\s) = D(\kap_\s,\chi_\s) = \delta(\ub_\s,\uc_\s),  D(\rho^{\upsilon}) = D(\kap,\chi) = \prod_{\s} D(\kap_\s,\chi_\s)$$ 
and
$$D^{hol}(\rho^{\upsilon}_\s) = D^{hol}(\kap_\s,\chi_\s) = pr^{hol}_{\kap,\s}\delta(\ub_\s,\uc_\s),  D^{hol}(\rho^{\upsilon}) = D^{hol}(\kap,\chi) = \prod_{\s} D^{hol}(\kap_\s,\chi_\s)$$ 
for these choices of $(r_{i,\s};s_{j,\s})$.
Then for all $\kap^{\dag} \leq \kap$ there exist unique elements $\delta(\kap,\kap^{\dag}) \in U(\grp_3^+)$, defined over algebraic number fields, such that 
$$D(\kap,\chi) = \sum_{\kap^{\dag} \leq \kap} \delta(\kap,\kap^{\dag})\circ D^{hol}(\kap^\dag,\chi);$$
$ \delta(\kap,\kap)$ is the scalar $\prod_{\s} P_{\kap_\s,\chi,\s}$.
\end{cor}
\begin{proof}   Consider $j(a_\s,b_\s)(\CP(a_\s)_\s^+ \otimes \CP(b_\s)_\s^+) \subset \CP(n)_\s$.  This is the space spanned by the
$p(\ub_\s,\uc_\s)$ defined in \eqref{plurihar}.  Let $\delta(a_\s,b_\s)^+ \subset U(\grp^+_\s)$ be the 
subspace identified with  $j(a_\s,b_\s)(\CP(a_\s)_\s^+ \otimes \CP(b_\s)_\s^+)$ by the isomorphism \eqref{pluridiff}.
The decomposition \eqref{decomp} is based on the fact that the composition
$$\delta(a_\s,b_\s)^+\otimes v(\chi_{\sigma}) \hookrightarrow \DD^{2}(\chi_{\sigma})
\overset{pr^{hol}}\to   \DD^{2}(\chi_{\sigma})^{hol;a_\s,b_\s}$$
is an isomorphism.     See the disscussion in Section  7.11 of \cite{harris86}.

This {\it does not say} that $\delta(\ub_\s,\uc_\s)\otimes v(\chi_{\sigma})$ lies in the highest weight
space of the holomorphic subspace of the direct factor $\DD(\kap_\s) \otimes \DD(\kap^{\flat}_\s\otimes \chi_{\s})$ corresponding
to the $a_\s + b_\s$-tuple $(\ub_\s,\uc_\s)$; but it does say that its projection on that highest weight space is non-trivial.  
This is equivalent to the first statement of the corollary.  The remaining statements are formal consequences of the decomposition \eqref{decomp}
and the fact that the decomposition is rational over an appropriate reflex field, cf. Lemma 7.3.2 of \cite{harris86}.
\end{proof}

\subsubsection{Differential operators on $C^{\infty}$-modular forms}\label{ci-do}

Let $\chi = ||\bullet||^m\chi_0$ be an algebraic Hecke character of $\CK$, as before.  We view $G_4$ as the  rational similitude group of
a maximally isotropic hermitian space $V_4$; this allows us to write $Sh(V_4)$ for the corresponding Shimura variety.
Let $\Lambda(\chi) = (\Lambda(\chi_{\sigma}),\sigma \in \Sigma)$ be the character of
$U_{\infty}$ whose restriction to $U_{\s}$ is $\Lambda(\chi_\s)$.  Let $L(\chi)$ be the $1$-dimensional space on which $U_\infty$ acts
by $\Lambda(\chi)$; it can be realized over a number field $E(\chi_{\infty})$ which depends only on $\chi_{\infty}$.
The dual of the highest $U_{\infty}$-type $\Lambda(\chi)$, restricted to the intersection of $U_{\infty}$
with $G_4(\R)$, defines an automorphic line bundle $\CL(\chi)$ on $Sh(V_4)$ with fiber at the fixed point $h$ of
$U_{\infty}$ isomorphic to $L(\chi)$.  If $\pi = \pi_{\infty}\otimes \pi_f$ is an automorphic representation
of $G_4$, with $\pi_{\infty}$ a $(Lie(G_4),U_{\infty})$ module isomorphic to $\DD^2(\chi) = \otimes_{\s \in \Sigma}(\DD^2(\chi_{\sigma}))$ and
$\pi_f$ an irreducible smooth representation of the finite adeles of $G_4$, then there is a canonical embedding
\begin{equation}\label{trivializechi}
\pi_f \isoarrow \pi_f\otimes H^0(\grP_h,U_{\infty};\DD^2(\chi)\otimes L(\chi)) \hookrightarrow H^0(Sh(V_4)^{tor},\CL(\chi)^{can}).
\end{equation}

Write $\Omega = \Omega_{Sh(V_4)}$ for the cotangent bundle.  For any integer $d \geq 0$, and for any ring $\CO$, let $\CP(n)^d(\CO)$
denote the $\CO$-module of $\CO$-valued polynomials of degree $d$ on the matrix space $M_n$, and let
$\CP(n)^{d,*}(\CO) = Hom_{\CO}(\CP(n)^d,\CO)$ denote the dual $\CO$-module.  There is a canonical action of $U_{\infty}$ on $\CP(n)^d$,
for every $d$, defined over the field of definition $E(h)$ of the standard CM point $h$ stabilized by $U_{\sigma}$, and even over its integer ring.  The Maass operator
of degree $d$, as defined in Section  7.9 of \cite{harris86} is a $C^{\infty}$-differential operator 
\begin{equation}\label{maass}  \delta^d_{\chi}:  \CL(\chi) \rightarrow \CL(\chi)\otimes Sym^d\Omega. \end{equation}
We can view the target of $\delta^d_{\chi}$ as the automorphic vector bundle attached to the representation $L(\chi)\otimes \CP(n)^{d,*}$ of $U_{\infty}$,
using the identification of section \ref{automorphy} as in \eqref{pluridiff}.  We use the same notation to denote the action on the space
$\CA(G_4)$ of (not necessarily cuspidal) $L(\chi)_h$-valued automorphic forms on $G_4$:
\begin{equation}\label{maass2} \delta^d_{\chi}:  \CA(G_4,L(\chi)_h) \rightarrow \CA(G_4,L(\chi)_h\otimes \CP(n)^{d,*})  \end{equation}
where the notation denotes automorphic forms with values in the indicated vector space.   For any polynomial $\phi \in \CP(n)^d \equiv Sym^d\grp_4^+$ we thus obtain
a differential operator
\begin{equation}\label{maass3}  \delta^d_{\chi}(\phi):  \CA(G_4,L(\chi)_h) \rightarrow \CA(G_4(\Q)\backslash G_4(\adeles),L(\chi)_h);  \delta^d_{\chi}(\phi)(f) = [\delta^d_{\chi}(f)\otimes \phi] \end{equation}
where the bracket denotes contraction $\CP(n)^{d,*}\otimes \CP(n)^{d} \rightarrow E(h)$.  

Finally, for each $\sigma$ define sequences $\ub_\s$ and $\uc_\s$ as in Section  \ref{holoprojection-sec}; let $\ub = (\ub_\s)$, $\uc = (\uc_\s)$.  Suppose
$\sum_\s [\sum_i \tb_{i,\s} + \sum_j \tc_{j,\s}] = d$.  Then we define $p(\ub,\uc) = \prod_{\s} p(\ub_\s,\uc_\s)$ where the factors are as in \eqref{plurihar}, and let
\begin{equation}\label{maass4}
\delta^d_{\chi}(\ub,\uc) = \delta^d_{\chi}(p(\ub,\uc)): H^0(Sh(V_4)^{tor},\CL(\chi)^{can}) \rightarrow \CA(G_,\CL(\chi)_h).
\end{equation}
Under the isomorphisms \eqref{trivializechi}, $\delta^d_{\chi}(\ub,\uc)$ is identified with the
operator on the left hand side deduced from multiplying by the element $p(\ub,\uc)$, viewed as an element of $Sym^d\grp_4^+$, which maps
$H^0(\grP_h,U_{\infty};\DD^2(\chi)\otimes L(\chi)_h) = \otimes_{\s}\C v_{\chi_\s}\otimes L(\chi)_h$ to 
$p(\ub,\uc)\otimes \otimes_{\s}\C v_{\chi_\s} \otimes L(\chi)_h \in \DD^2(\chi)\otimes L(\chi)_h$.

The holomorphic differential operators of Corollary \ref{pluridecomp} define operators on automorphic forms, as follows.  Let
$S_{\kap,V}^{\infty}(K_1,\C)$ denote the space of $C^{\infty}$ modular forms of type $\kap$ on $Sh(V_1)$, of level $K_1$, and define 
$S^{\infty}_{\kap^{\flat},-V}(K_2,\C)$ analogously.   The following Proposition restates Proposition 7.11.11 of \cite{harris86}:

\begin{prop}\label{holodiffops}  Let $(\kap,\chi)$ be critical as in Corollary \ref{pluridecomp}.  Fix a level subgroup $K_4 \subset G_4(\A_f)$ and a subgroup
$K_1 \times K_2 \subset G_3(\A_f)\cap K_4$.  There are differential operators 
$$D(\kap,\chi):  H^0(_{K_4}Sh(V_4)^{tor},\CL(\chi)^{can}) \rar S_{\kap,V}^{\infty}(K_1,\C)\otimes S^{\infty}_{\kap^{\flat},-V}(K_2,\C)\otimes \chi\circ \det;$$
$$D^{hol}(\kap,\chi):  H^0(_{K_4}Sh(V_4)^{tor},\CL(\chi)^{can}) \rar S_{\kap,V}(K_1,\C)\otimes S_{\kap^{\flat},-V}(K_2,\C)\otimes \chi\circ \det$$
which give the operators $\delta^d_{\chi}(\ub,\uc)$ and $pr^{hol}\circ\delta^d_{\chi}(\ub,\uc)$ upon pullback to functions on $G_4(\ad)$ and
restriction to $G_3(\ad)$. 
\end{prop}

\subsubsection{The Hodge polygon}\label{hodgepolygon}
If $\pi$ is a cuspidal automorphic representation
of $GU(V)$ whose component at $\sigma$ is an anti-holomorphic discrete series representation of the form $\Dc(\tau_{\sigma})$,
then its base change $\Pi$ to an automorphic representation of $\GL(n)_{\K}$ (ignoring the split center) is cuspidal, cohomological,
and satisfies $\Pi^{\vee} \isoarrow \Pi^c$, and therefore the associated $\ell$-adic Galois representations have associated motives
(in most cases), realized in the cohomology of Shimura varieties attached to unitary groups, with specified Hodge structures.
In what follows, we fix $\sigma$ and attach a Hodge structure to the anti-holomorphic representation $\Dc(\tau)$, according to the
rule used to assign a motive to $\Pi$.    The Hodge structure is pure of weight $n-1$ and has the following Hodge types, each with
multiplicity one:
\begin{equation}\label{Hodge}
\begin{split}
(\kap_1 - b + n-1,b - \kap_1), \dots, (\kap_a, n - 1 - \kap_a), (n-1 -a - \lap_b, \lap_b + a), \dots, (-\lap_1,\lap_1 + n  - 1), \\
(\lap_1 + n  - 1, - \lap_1), \dots, (\lap_b + a, n-1 -a - \lap_b),  (n - 1 - \kap_a, \kap_a), \dots,  (b - \kap_1,\kap_1 - b + n-1).
\end{split}
\end{equation}
Label the pairs in \eqref{Hodge} $(p_i,q_i)$, $i = 1, \dots, 2n$, in order of appearance; thus $(p_i,q_i)$ is in the top row if and only if $i \leq n$.
\begin{hyp}[Critical interval hypothesis] \label{critint}  We assume that the weights $(\kap,\lap)$ are adapted to the signature $(a,b)$
in the sense that, for every pair $(p_i,q_i)$ in the collection \eqref{Hodge}, $p_i \neq q_i$ and $p_i > q_i$ if and only if $i \leq n$.
\end{hyp}

One checks that Hypothesis \ref{critint} holds if and only if $2\kap_a > n-1$ and $-2\lap_1 >  n-1$.   We define the {\it Hodge polygon} 
$Hodge(\kap,\lap) = Hodge(\Dc(\tau))$, to be the polygon in the right half-plane connecting the vertices $(i,p_i)$ with 
$(p_i,q_i)$ the pairs in \eqref{Hodge}.

\subsubsection{Specific anti-holomorphic vectors}\label{aholovectors}  When $\tau_{\sigma}$ is strongly positive with highest weight
$\kap = \kappa_{\sigma}$, we write $\DD(\kap) = \DD(\tau_{\sigma})$, $\DD_c(\kap) = \DD_c(\tau_{\sigma})$ when it's
clear that $\kap$ is a weight and $\tau_{\sigma}$ is an irreducible representation.   Let $\pi$ be a cuspidal automorphic representation
of $G$ with $\pi_{\sigma} = \DD_c(\kap)$ as above.  In the computation of the zeta integral, we use a factorizable automorphic form $\varphi = \otimes_v \varphi_v \in \pi$, with
$\varphi_v$ a vector in the minimal $U_{\sigma}$-type $1\otimes W_{\tau_{\sigma}}^{\vee}$ of $\DD_c(\kap)$.  In practice, we choose $\varphi_v$ to be either the highest weight vector $\varphi_{\kappa,+}$ or the lowest weight vector 
$\varphi_{\kappa,-}$ in $1\otimes W_{\tau_{\sigma}}^{\vee}$.  If $w_0$ is the longest element of the Weyl group of  $T_{\s}$ relative to $R^+_{\s}$,
then $\varphi_{\kappa,+}$ (resp. $\varphi_{\kappa,-}$) is an eigenvector for $T_{\s}$ of weight $-w_0(\kappa)$ 
(resp. of weight $-\kappa$).

\subsection{Local zeta integrals at archimedean places}\label{archimedeanchoices}

\subsubsection{Choices of local data}  This material has been covered at length
in \cite{harriscrelle} and \cite{Harris-rationality}, so we can afford to be brief.   Notation for induced representations is
as in Section  \ref{induced} above.    The notation for holomorphic representations is as in Section  \ref{automorphy}.
An easy computation, similar to that in \cite{harriscrelle}, yields
\begin{lem}  As subspaces of $C^{\infty}(G_{\s})$,
$\iota(m,\chi_{\s})(\DD^{2}(m,\chi_{\sigma})) \subset I_{\sigma}(m - \frac{n}{2},\chi)$.
\end{lem}

\begin{rmk}
Note that we have omitted similitude factors here.  Strictly speaking, these should be included; but they do not change the theory in any significant way.
\end{rmk}

\subsubsection{Non-vanishing of $I_{\infty}$}  Let $\sigma$ be an archimedean place, $f_\s = f_{\sigma}(\chi_{\sigma},c) \in I(\chi_{u,\sigma},m)$ the local section at $\sigma$.
We assume $f_\s$ is of the form
\begin{equation}\label{ddd} f_{\sigma}(\chi_{\sigma},c,g) = B(\chi_\s,\kap_\s) D(\kap_\s,m,\chi_{u,\s}) J_{m,\chi_{u,\sigma}}(g), g \in G_{4,\sigma} \end{equation}
where $J_{m,\chi_{u,\sigma}} \in 
C^{\infty}(G_4)$ is the canonical automorphy factor introduced in Section  \ref{automorphy} and $B(\chi_\s,\kap_\s)$
is a non-zero algebraic scalar. 
 Let $\varphi_\s\otimes \varphi^{\flat}_\s$ be an anti-holomorphic vector in the highest weight subspace
of the minimal $K_{\sigma}$-type of $\pi_\s \otimes \pi^{\flat}_\s$.
\begin{prop}\label{archnonvan}  The local factor $I_\s(\varphi_\s,\varphi^{\flat}_\s,f_\s,m)$ is not equal to $0$.
\end{prop}

\begin{proof}  If $D(\kap_\s,\chi_\s)$ is replaced by $D^{hol}(\kap_\s,\chi_\s)$ in \eqref{ddd}, this follows from Remark (4.4)(iv) of \cite{Harris-rationality}.  Since 
$\varphi_\s\otimes \varphi^{\flat}_\s$  is an anti-holomorphic vector, the pairing of (the Eisenstein section)
$D(\kap_\s,\chi_\s)J_{m,\chi_\s,\s}$ with (the highest weight vector) $\varphi_\s\otimes \varphi^{\flat}_\s$ factors through the projection of $D(\kap_\s,\chi_\s)J_{m,\chi_\s,\s}$ onto
$D^{hol}(\kap_\s,\chi_\s)J_{m,\chi_\s,\s}$.  The Proposition is thus a consequence of Corollary \ref{pluridecomp}.
\end{proof}

When the extreme $K$-type $\tau_\sigma= \tau_{a_\sigma, \sigma}\otimes\tau_{b_\sigma, \sigma}$ in $\pi_\sigma$ is one-dimensional, the archimedean zeta integrals have been computed in \cite{sh, shar}.  Garrett has shown in \cite{ga06} that the archimedean zeta integrals are algebraic up to a predictable power of the transcendental number $\pi$, which can be normalized away.  The zeta integrals at $\sigma$ depend only upon the local data at $\sigma$.  When at least one of the two factors ($\tau_{a_\sigma, \sigma}$, $\tau_{b_\sigma, \sigma}$) of the extreme $K$-type is one-dimensional, the archimedean zeta integrals are given precisely on \cite[p. 12]{ga06}; and furthermore, Garrett showed in \cite{ga06} that when both factors are scalars, the archimedean zeta integrals are non-zero algebraic numbers.  
They  have not been computed in the more general case (i.e. when neither $\tau_{a_\sigma, \sigma}$ nor $\tau_{b_\sigma, \sigma}$ is one-dimensional).  However, 
the  analogous computation for the doubling method for symplectic groups has been carried out in complete generality by Zheng Liu in \cite{liu19}; the result matches the factor predicted by Coates and Perrin-Riou in \cite{CPR}.  One of us (E.E.) plans with Liu to adapt her method to the current situation.  

In the meantime, we will be satisfied with the following result, due to Garrett \cite{ga06}.
\begin{prop} \label{archfactors} Let 
$I_\s(\chi_\s,\kap_\s)$ be the local zeta integral
$$I_\s(\chi_\s,\kap_\s) =   I_\s(\varphi_\s,\varphi^{\flat}_\s,f_\s,m),$$
where $\varphi_\s = \varphi_{\kap_\s,-}$, $\varphi^{\flat}_\s = \varphi_{\kap^{\flat}_\s,-}$
and $f_\s$ is given by \eqref{ddd}.
Then $I_\s(\chi_\s,\kap_\s)$ is a non-zero algebraic number.
\end{prop}

\begin{rmk}  When $\kap_\s$ is a scalar representation, Shimura obtains an explicit formula for
the local zeta integral.  In general, as explained at the end of \cite[Section  5]{Harris-rationality}, Garrett's calculation
actually determines the value of the integral up to an element of a specific complex embedding 
of the CM field $F$.  In that paper $F$ is imaginary quadratic, but the same reasoning applies in general.
Undoubtedly the calculation actually gives a rational number, but the method is based on the choice of
rational structures on $U_{\s}$ and the aforementioned differential operators.
We do not need to use this more precise information here.
\end{rmk}

\subsection{The global formula}\label{globalzeta}

We have now computed all the local factors of the Euler product \eqref{normalizedzeta}.   The Proposition below summarizes the result of our computation.  Bear in mind that, although we write $\varphi \in \pi$, we actually mean that $\varphi \in \underline{\pi}$, where the latter is the irreducible $U_1(\adeles)$ constituent of $\pi$ chosen as in \eqref{facpi}.   

First, write $\chi = ||\bullet||^{m}\cdot\chi_u$ with $\chi_u$ a unitary Hecke character of $\CK$.   Denote by $\chi^+$ the restriction of $\chi_u$
to the id\`eles of $\CK^+$; it is a character of finite order.  Let $\eta = \eta_{\CK/\CK^+}$ denote the quadratic id\`ele class character of $\CK^+$
attached to the quadratic extension $\CK/\CK^+$.  For any finite place $v$ of $\CK^+$, define the Euler factor
$$D_v(\chi) = \prod_{r = 0}^{n-1} L_v(2m + n - r,\chi^+\cdot \eta^r).$$
(In the notation of Equation \eqref{dnv-equ}, we have $D_v(\chi) = d_{n, v}\left(m, \chi_v\right)$.)  For any finite set $S$ of finite places, let
\begin{equation}\label{normalizingDS}  D^S(\chi) = \prod_{v \notin S} D_v(\chi);~~ D(\chi) = D^{\emptyset}(\chi), \end{equation}
where the product is taken over finite places.

\begin{thm}\label{globaleuler}  Let the test vectors $\varphi \in \pi$ and $\varphi^{\flat} \in \pi^{\flat}$ be chosen to be factorizable vectors as in \eqref{facphi}, with the local components at $p$ and $\infty$ given as in \eqref{facphip} and \eqref{facphiinf}, respectively.  Assume the local components at finite places outside $S = S_{\pi}$ are unramified vectors, and the local choices at ramified places are as in \ref{nonarchchoices-section}.  Moreover, assume the Siegel--Weil section $f_s \in I(\chi,s)$ is chosen as in the preceding sections.  
Write $\chi = ||\bullet||^m\chi_u$.  Then
we have the equality
$$D(\chi)\cdot I(\varphi,\varphi^{\flat},f,s) =  \langle \varphi, \varphi^{\flat} \rangle\cdot I_p(\chi,\kappa)I_{\infty}(\chi,\rho^{\upsilon})I_SL^S(s + \frac{1}{2},\pi,\chi_u)$$
where
$$I_S = \prod_{v \in S} D_v(\chi)\cdot \mathrm{volume}(\mathcal{U}_v),$$
$$I_{\infty}(\chi,\kappa) = \prod_{\s} I_\s(\chi_\s,\kap_\s)$$ is the product of factors described in Proposition \ref{archfactors},
$$I_p = L_p(s,\ord,\pi,\chi) \times \prod_w [\grV_w\cdot \langle \phi_w, \tilde\phi_w\rangle_{\pi_w}],$$
where $\grV_w$ is the factor that appears in \eqref{zetapfinalform}, and we define
$$L_p(s,\ord,\pi,\chi) : = \prod_{w \mid p} L(s,\ord,\pi_w,\chi_w).$$
Finally, $\langle \bullet, \bullet \rangle$ is the $L^2$ inner product on cusp forms.
\end{thm}  

\begin{rmk}\label{Ip-formula}
In light of \eqref{grV-formula} and the identification of volumes in Remark \ref{volume-rmk},
the expression for $I_p$ can be rewritten as 
$$
\frac{I_p}{\Vol(I_{r,V}^0)\Vol(I_{r,-V}^0)} =  L_p(s,\ord,\pi,\chi) \frac{ \prod_w\langle \phi_w, \tilde\phi_w\rangle_{\pi_w}}
{\Vol(I_{r,V}^0\cap I_{r,-V}^0)}.
$$
\end{rmk}

\vfill
\pagebreak

\part*{Part III:  Ordinary families and $p$-adic $L$-functions}

\section{Measures and restrictions}\label{EMeasure-padicLfunctions-section}
This section focuses on measures and restrictions.  In particular, Section  \ref{emeas-section} gives a measure whose values at certain specified characters are the Eisenstein series associated to the local data chosen when we calculated the zeta-integrals above.

\subsection{Measures: generalities} Let $X$ be a compact and totally disconnected
topological space. For a $p$-adic ring $R$ we let $C(X,R)$ be the $R$-module
of continuous maps from $X$ to $R$ (continuous with respect of the $p$-adic
topology on $R$). Note that $C(X,\Zp)\hat\otimes_\Zp R \isoarrow C(X,R)$. Let
$M$ be a $p$-adically complete $R$-module. Then by an $M$-valued measure
on $X$ we mean an element of the $R$-module
$$
\Meas(X,M) = \Hom_\Zp(C(X,\Zp), M)  = \Hom_R(C(X,R),M).
$$

Suppose $X$ is a profinite abelian group. Then $\Meas(X,R)$ is identified with the
completed group ring $R[\![X]\!]$.   In particular $\Meas(X,R)$ is itself a ring.  The following lemma is immediate:
\begin{lem}\label{measprod}   Suppose $X = X_1 \times X_2$ is a product of profinite abelian groups.  Then
there is a natural isomorphism
$$\Meas(X_1\times X_2,R) \isoarrow \Meas(X_1,Meas(X_2,R))$$
\end{lem} 

If we write
$X = \varprojlim_i X/X_i$, where $X = X_0 \supset X_1 \supset X_2 \supset \dots $ is a neighborhood basis
of the identity consisting of open subgroups of $X$ of finite index, then 
$$\Lambda_{X,R} = \varprojlim_i R[X/X_i].$$
This is a compact topological ring.  The following dictionary is well-known and due to Mazur:
 
 \begin{fact}\label{Mazurdictionary}  The identification of
 a measure $\mu$ on $X$ with an element $f$ of the $\Lambda_{X,R}$ has the property that, for any
 continuous homomorphism $\chi:X \rightarrow R_1^\times$, with
$R_1$ a $p$-adic $R$-algebra, 
 $$\int_{X} \chi d\mu :=  \mu(\chi) = \chi(f)$$
where $\chi(f)$ is the image of $f$ under the homomorphism $\Lambda_{X,R} \rightarrow R_1$ induced by $\chi$.
\end{fact}

We let $\chi$ denote the homomorphism $\Lambda_{X,R} \rightarrow R$ of Fact \ref{Mazurdictionary}; in this way
$\chi$ defines an $R$-valued point of $\Lambda_{X,R}$.

In what follows, characters of $X_1$ will be Hecke characters, $X_2$ will be the group of integral
points of a $p$-adic torus, whose characters parametrize
weights of $p$-adic modular forms, and $M$ will be the ring of $p$-adic modular forms.  
When $X_2$ is a point, the measure on $X = X_1$ will be an Eisenstein measure that pairs
with modular forms of fixed weight, and in particular can be used to construct what we will call,
loosely and somewhat abusively, a {\it $p$-adic $L$-function of one variable}, the variable Hecke character,
attached to a fixed holomorphic automorphic representation.  When $X_2$ is the group of
points of a non-trivial torus, we will be constructing {\it $p$-adic $L$-function of two variables},
the second variable running through the points of a Hida family.  

The following is a version of a well-known lemma (see \cite[Proposition 4.1.2]{kaCM} for the formulation below):

\begin{lem}\label{measures-characters}  Suppose $X=\varprojlim_mX_m$ is a profinite abelian group.  Suppose $R$ is a ring that is flat over $Z_p$ and that contains a primitive $n$th root of unity for each $n$ dividing the order of $X_m$ for some $m$.  Each $R$-valued measure on X is completely determined by its values on locally constant continuous homomorphisms $\chi: X\rightarrow R$, and any function $\alpha$ from the continuous characters to $R$ determines an $R$-valued measure on $X$ whenever the values of $\alpha$ on the space of $R$-valued locally constant characters on $X$ satisfy the usual Kummer congruences (as in \cite[\S4.0]{kaCM}).
\end{lem}

\subsection{The space $X_p$}\label{Xp-section}
 For each integer $r>0$, let
$$
U_r = (\O\otimes\widehat\Z^{p})^\times\times (1+p^r\O\otimes\Zp)\subset (\K\otimes\wZ)^\times
$$
and
$$
X_p = \varprojlim_r\K^\times\backslash (\K\otimes\wZ)^\times/U_r.
$$
This is the projective limit of the ray class groups of $\K$ of conductor $(p^r)$. In particular,
it is a profinite abelian group.  More generally, if $N_0$ is prime to $p$, we let 
$$X_{p,N_0} = \varprojlim_r\K^\times\backslash (\K\otimes\wZ)^\times/U_{r,N_0}$$
where $U_{r,N_0} = (1+N_0\O\otimes\widehat\Z^{p})^\times\times (1+p^r\O\otimes\Zp) \subset U_r$.

\subsubsection{Admissible measures on $X_p$}
We suppose now that we are in the situation of
Section  \ref{doublingsetup}, and we freely use the notation and conventions introduced therein.
Using the isomorphism \eqref{Hi-iso} we identify $H_1(\Zp)$ with
$H_2(\Zp)$ via $h_1=(h_{1,w})_{w|p}\mapsto h_2 = (h_{2,w})_{w|p}$ with $h_{2,w}=h_{1,\bar w}$.
This then identifies $T_{H_2}(\Zp)$ with $T_{H_1}(\Zp)$ and
$T_{H_4}(\Zp) = T_{H_3}(\Zp) = T_{H_1}(\Zp)\times T_{H_2}(\Zp)$ with $T_{H_1}(\Zp)\times T_{H_1}(\Zp)$.
In particular, the characters $\psi$ of $T_{H_3}(\Zp)$ are identified with pairs of characters
$(\psi_1,\psi_2)$ of $T=T_{H_1}(\Zp)$.

Let
\begin{itemize}
\item $\kap = (\kap_\sigma)$ be an $\O'$-character of $T$ as in
Section  \ref{padicweight} and let $\kap'$ be the $\O'$-character of $T_{H_3}(\Zp)$ identified
with the pair $(\kap,\kap^\vee)$;
\item $\psi$ be a finite-order $\bQptimes$-valued character
of $T(\Zp)$;
\item $K_i^p\subset G_i\left(\A_f^{p}\right)$, $i=1,2$,  be open compact subgroups such that
$\nu\left(K_1\right)=\nu\left(K_2\right)$;
\item $R$ be a $p$-adic $\O'[\psi]$-algebra.
\end{itemize}
For any finite-order $\bQptimes$-valued character $\chi$ of $X_p$, let
$\psi^{-1}_\chi = \psi^{-1}\cdot\chi\circ\det$,
where by $\det$ we mean the map
$\det:H_1(\Zp)\rightarrow (\O\otimes\Zp)^\times = \prod_{w|p}\O_w^\times$
that is the composition of the isomorphism \eqref{Hi-iso} with the products of
the determinants of
each of the $\GL$-factors,
and let $\psi_\chi'$ be the character of $T_{H_3}(\Zp)$ identified with the pair
$(\psi,\psi^{-1}_\chi)$, .
By an admissible $R$-measure on $X_p$
of weight $\kap$, character $\psi$, and level $K_3^p=(K_1^p\times K_2^p)\cap G_3(\A_f^p)$,
we mean a measure $\mu(\cdot)=\mu(\kap,\psi,\cdot)\in \Meas(X_p; V_{\kap'}^{\ord}(K_3^p,R))$ such that
for any finite-order $\bQptimes$-valued character $\chi$ of $X_p$,
\begin{equation*}
\mu\left(\chi\right)=\mu\left(\kap,\psi,\chi\right) \in V_{\kap'}^{\ord}\left(K^p_3,\psi'_\chi,R\left[\chi\right]\right).
\end{equation*}

Let $R'$ be any $p$-adic $R$-algebra and $\ell$ an $R$-linear functional
$\ell: V^{\ord}_{\kap'}\left(K_3^p,R\right)\rightarrow R'$.
Then
$$
\mu_\ell(\cdot)=\mu_\ell(\kap,\psi,\cdot) :=\ell\circ\mu(\kap,\psi,\cdot)
$$
is an $R'$-valued measure on $X_p$.

We will need a slight generalization of the above definition.  Let $\rho$, $\psi$, $K_i^p$,  $R$,  $\psi^{-1}_\chi$, and $\psi'_\chi$ be as above.
\begin{itemize}
\item $\rho = (\rho_\sigma)$ be an $\O'$-character of $T = T_{H_1}$ as in
Section  \ref{padicweight} and let $\rho^{\triangle}$ be the $\O'$-character of $T_{H_3}(\Zp)$ identified
with the pair $(\rho,\rho^\vee)$;
\item $\psi$ be a finite-order $\bQptimes$-valued character
of $T(\Zp)$;
\item $K_i^p\subset G_i\left(\A_f^{p}\right)$, $i=1,2$,  be open compact subgroups such that
$\nu\left(K_1\right)=\nu\left(K_2\right)$;
\item $R$ be a $p$-adic $\O'[\psi]$-algebra.
\end{itemize}
Note that $\rho^{\triangle}$ and $(\rho,\rho^{\flat})$ coincide as characters of $T_{H_3}(\Zp)$, where $\rho^\flat$ is defined by
analogy with \eqref{kapflat}.  
For any finite-order $\bQptimes$-valued character $\chi$ of $X_p$, let
$\psi^{-1}_\chi= \psi^{-1}\cdot\chi\circ\det$,
where by $\det$ we mean the map
$\det:H_1(\Zp)\rightarrow (\O\otimes\Zp)^\times = \prod_{w|p}\O_w^\times$
that is the composition of the isomorphism \eqref{Hi-iso} with the products of
the determinants of
each of the $\GL$-factors,
and let $\psi_\chi^{\triangle}$ be the character of $T_{H_3}(\Zp)$ identified with the pair
$(\psi,\psi^{-1}_\chi)$.  
Let $(\alpha,\beta)$ be a character of $T_{H_3}(\Zp)$, written as a pair
of characters of $T_{H_1}(\Zp) \equiv T_{H_2}(\Zp)$.
By an admissible $R$-measure on $X_p$
of weight $\rho$, character $\psi$, {\it shift} $(\alpha,\beta)$, and level $K_3^p$,
we mean a measure $\mu(\cdot)=\mu(\rho,\psi,\cdot)\in \Meas(X_p; V_{\rho^{\triangle}\cdot(\alpha,\beta)}^{\ord}(K_3^p,R))$ such that
for any finite-order $\bQptimes$-valued character $\chi$ of $X_p$,
\begin{equation}\label{shifts}
\int_{X_p} \chi d\mu := \mu\left(\chi\right) = \mu\left(\rho,\psi,\chi\right) \in V_{\rho^{\triangle}\cdot(\alpha,\beta)}^{\ord}\left(K_3^p,\psi'_\chi,R\left[\chi\right]\right).
\end{equation}

\subsubsection{Admissible measures on $X_p \times T_{H}$:  two variables}\label{2varmeas}
In this section we fix $H = H_1$ and consider admissible measures of weight $\rho$ and shift $(\alpha,\beta)$ where $\rho$ and $(\alpha,\beta)$
are allowed to vary.  This requires a slight adjustment to the notation of the previous
section.   More precisely, suppose we are given a homomorphism
$sh:  T_{H_3}(\Zp) \rar X_p$ as before.  By duality this gives a map $sh^*: C(X_p,R) \rar C(T_H(\Zp),R)$ for any ring $R$;
$sh^*$ takes characters to characters.  

We also suppose we are given an algebraic automorphism $\upsilon:  T_{H} \rar T_H$.  If $\rho$ is a function on $T_H$,
we let $\rho^{\upsilon}(t) = \rho(\upsilon(t))$.

We fix a tame level $N_0$ as in Section  \ref{Xp-section} and define $X_p = X_{p,N_0}$ as before.
By an admissible $R$-measure on $X_p \times T_H$
of character $\psi$, shift $sh$, twist $\upsilon$, and level $K^p=K_3^p=(K_1^p\times K_2^p)\cap G_3(\A_f^p)$,
we mean a measure 
$$\mu(\cdot)=\mu(\psi,sh,\cdot)\in \Meas(X_p, \Meas(T_H,V^{\ord}(K^p,R)))$$ such that 
for any finite-order $\bQptimes$-valued character $\chi$ of $X_p$ and any character $\rho$ of $T_H$,
\begin{equation*}\label{admissmeas2}
\mu\left(\chi\right)(\rho^{\upsilon})=\mu\left(\psi,sh,\chi\right)(\rho^{\upsilon}) \in V_{\rho^{\triangle}\cdot sh^*(\chi)}^{\ord}\left(K^p,\psi^{\triangle}_\chi,R\left[\chi\right]\right).
\end{equation*}

\subsection{Eisenstein measures on $X_p\times T$}\label{emeas-section}

Now, we recall the Eisenstein measures on $X_p\times T$.  
We briefly summarize their basic properties, as developed in \cite{apptoSHL, apptoSHLvv, EFMV} -- with special attention to the fact that they $p$-adically interpolate values of the Eisenstein series associated to the local data chosen above for the zeta integral calculations.
As in Section  \ref{unitarygroupPEL}, let $\Sigma = \{\sigma\in\Sigma_\K\ : \ \grp_\sigma\in\Sigma_p\}$.
This is a CM type for $\K$.  Throughout this section, we take $\chi:\K^\times\backslash \A_\K^\times\rightarrow \C^\times$ to be a unitary Hecke character.

\subsubsection{Axiomatics of the Eisenstein measure}
The Eisenstein measures of \cite{apptoSHL, apptoSHLvv, emeasurenondefinite}, as well as the local components of ordinary vectors in Hida families,
have been reverse-engineered in order to meet the requirements of the construction of the $p$-adic $L$-functions.  In this section we first present the axioms the Eisenstein measure is required to satisfy, and then explain how they are satisfied by the ones constructed in the references just cited.  We write $T_H = T_{H_1}$ in this section.

The Eisenstein measure is, in the first place, a $p$-adic measure on the space $X_p \times T_H(\Zp)$ with values
in the space of $p$-adic modular forms on $G_4$.  It is characterized by its specializations at classical points.  Let $Y_H$ be the formal scheme over $\Zp$ whose 
points with values in a complete $\Zp$-algebra $R$ are given by $\Hom(X_p\times T_H(\Zp), R^{\times})$.  Let $Y_H^{\alg} \subset Y_H(\C_p)$ be the set of
pairs $(\chi,c)$, where $\chi:  X_p \rar R^{\times}$, for some $R \subset \C_p$, is the $p$-adic character associated to an algebraic Hecke character, denoted $\chi^{class}$, and
$c = \psi\rho^{\upsilon}$ is a locally algebraic character of $T_H(\Zp)$: $\rho$ is an algebraic character, $\upsilon$ is an involution of $T_H$, as in \eqref{iota}
and $\psi$ is a character of finite order.     In other words, $c \in C_r(T_H(\Zp),R)$ for some $r \geq 0$, in the notation of Lemma \ref{meas}.

Note that we are not requiring $\chi^{class}$ to be unitary here; rather, the variable ``$s$'' is included in the infinity type of $\chi$; we fix an integer
$\mu$ such that, for each
$\s \in \Sigma$ we have $\chi_\s = ||\bullet||_\s^\mu \chi_{0,\s}$, where $\chi_{0,\s} = (z_{\s}^{-a(\chi_\s)}\bar{z}_\s^{-b(\chi_\s)})$.    This 
factorization is not unique; however, recall the set
$C_3(\mu,\chi_\s)$ of \eqref{decomp}.    We assume we are given a subset $Y_H^{class} \subset Y_H^{\alg}$, determined by the following positivity
condition:  
\begin{equation}\label{C3m2} (\chi,c) \in Y_H^{class} \Leftrightarrow \kap_\s \in C_3(\mu,z_{\s}^{-a(\chi_\s)}\bar{z}_\s^{-b(\chi_\s)}) \forall \s \in \Sigma
\end{equation}
This condition is independent of the choice of $m$ as above, in other words is independent of the choice of factorization.

Now we return to the notation of Section  \ref{globalzeta}:  write $\chi = ||\bullet||^{m}\cdot\chi_u$  and define the finite order id\`ele class character $\chi^+$ of  $\CK^+$.  
We omit the expression of $\mu$ and $\chi_0$
in terms of $m$ and $\chi_u$, and vice versa.   Define the normalizing factors $D^S(\chi)$ and $D(\chi)$ as in \ref{normalizingDS}.
Let $U_{p,\kap} = \prod_{w\in\Sigma_p}\prod_{j=1}^{n} U_{w,j,\kap}$, with notation as in \eqref{Up-operator};  here and below, the index $i$ of \eqref{Up-operator} is superfluous
because $G$ is the unitary similitude group of a single hermitian space.  Let 
\begin{equation}\label{ordinaryprojector} e_{\kap} = \varinjlim_N U_{p,\kap}^{N!}
\end{equation}
 (as an operator). We call this the {\it ordinary projector}.

\begin{defi}\label{axiomeis}  Let $K_i^p$ be an open compact subgroup of $G_i(\adeles_f^p)$, $i = 3, 4$, with $K_3^p \subset K_4^p \cap G_3(\adeles_f^p)$.   Let $S$ be the set of primes at which $K_4^p$ and $K_3^p$ do not contain a hyperspecial maximal compact subgroup.
An axiomatic Eisenstein measure on $X_p \times T_H(\Zp)$ of level $S$, relative to the set $Y_H^{class}$, of level $K_4^p$ and with coefficients in $R$,
is a measure $dEis$ with values in $V(K_3^p,R)$ such that, for every pair $(\chi = ||\bullet||^m\cdot\chi_u, c = \psi\rho^{\upsilon}) \in Y_H^{class}$, there is a factorizable Siegel section 
$$f(\chi,c) = \otimes'_v f_v(\chi_v,c) ~~\in~~ \bigotimes_v{}^{\prime}~ I_v(\chi_{u,v},m)$$
and
such that 
\begin{itemize}
\item If $v$ is a finite place outside $S$ --  so in particular $\chi_v$ is unramified for all $\chi \in Y_H^{class}$ -- then $f_v(\chi_v,c)$ is the unramified vector
in $I_v(\chi_{u,v},m)$ with $f_v(\chi_v,c)(1) = 1$.
\item If $v \in S$ then $f_v(\chi_v,c)$ is independent of the pair $(\chi,c)$.
\item For any prime $w$ dividing $p$ and for any real prime $\sigma \in \Sigma_w$, the local section $f_{\sigma}(\chi_{\sigma},c)$ depends only on $\chi^{class}_{\sigma}$ and
$\kappa_w$ (and on the choice of signature), and is of the form
$$f_{\sigma}(\chi_{\sigma},c,g) = B(\chi_\s,\kap_\s) D(\kap_\s,m,\chi_{u,\s}) J_{m,\chi_{u,\sigma}}(g), g \in G_{4,\sigma}$$
where $J_{m,\chi_{0,\sigma}} \in 
C^{\infty}(G_4)$ is the canonical automorphy factor introduced in Section  \ref{automorphy} and $B(\chi_\s,\kap_\s) \in \bar{\Q}^\times$.   In particular, $f_{\sigma}(\chi_{\sigma},c,g)$ does not depend on the factorization of $\chi_\s$.  (This follows
from Remark \ref{independence}. ) 

\item  For any prime $w$ dividing $p$, the local section $f_w(\chi_w,c)$ depends only on $\chi_w$ and $\psi_w$ (and on the choice of signature).  
\item $$e_\kappa\circ\int_{X_p \times T_H(\Zp)} (\chi,c) dEis = D^S(\chi)\cdot e_\kappa\circ res_3 E_{f(\chi,c)}$$ 
for all $(\chi,c) \in Y_H^{class}$,
where $D^S(\chi)$ is the normalizing factor defined in \eqref{normalizingDS}, $res_3$ is as in \eqref{restriction3}, and $e_\kappa$ is the ordinary projector
of \eqref{ordinaryprojector}.
\end{itemize}

The measure $dEis$ is said to be {\em normalized} at $S$ if instead of the last relation one has
 $\int_{X_p \times T_H(\Zp)} (\chi,c) dEis = D(\chi)\cdot res_3 E_{f(\chi,c)}$ for all $(\chi,c) \in Y_H^{class}$.  The measure is said to have {\em shift} $(\alpha,\beta)$ if it satisfies \eqref{shifts}.  

\end{defi}

One obtains a measure normalized at $S$ from an unnormalized measure by multiplying by the appropriate product of local Euler factors at $S$.
We write $D^?(\chi)$ for $? = S$ or empty if we haven't specified whether or not $dEis$ is taken to be normalized.  

 Definition \ref{axiomeis} makes no mention of whether or not the measure $dEis$ contains a shift.  The Eisenstein measure whose construction is recalled in Section  \ref{existence-Eisenstein} comes with a shift that will be specified in Corollary \ref{eismeasureG3}.

In previous sections we have chosen $f(\chi, c)$ meeting the conditions of Definition \ref{axiomeis} in Sections \ref{unrameuler} (local choices for $v\nin S$), \ref{nonarchchoices-section} (local choices for $v\in S$), \ref{archimedeanchoices} (local choices for archimedean places) and \ref{pchoices-section} (local choices for $v\divides p$).  Note that the choices at $p$ and $\infty$ depend on the signature of the unitary group $G_1$.  The existence of the Eisenstein measure itself that corresponds to these choices is proved in \cite{apptoSHL, EDiffOps}; see Theorem \ref{eismeasure-thm} below.

In the applications, the integrals of elements of $Y_H^{class}$ against $dEis$ suffice to determine $dEis$ completely.
We write 
\begin{equation*}  f^{holo}(\chi,c) = \otimes_{\sigma \in \Sigma_F} J_{m,\chi_{u,\sigma}} \otimes \otimes_{v \nmid \infty} f_v(\chi_v,c); \end{equation*}
\begin{equation}\label{holo} E^{holo}_{\chi_u,c}(m) = E_{f^{holo}(\chi_u,c)}(m). \end{equation}
Then the last condition of Definition \ref{axiomeis} can be rewritten 
\begin{equation}\label{holoint}\begin{aligned} 
e_\kap\circ \int_{X_p \times T_H(\Zp)} (\chi,c) dEis &= D^?(\chi)\cdot e_\kap \circ res_3 D(\kappa,m,\chi_u)E^{holo}_{\chi_u,c}(m), \\  &\forall (\chi = ||\bullet||^m\cdot\chi_u,c) \in Y_H^{class},
\end{aligned}
\end{equation}
where $D(\kappa,m,\chi_u)$ is as defined in Corollary \ref{pluridecomp}.

\section{Serre duality, complex conjugation, and anti-holomorphic forms}\label{serreduality-section}

\subsection{The Shimura variety $Sh(V)$}

Let $P=(\K,c,\O,L,\pair,h)$
be a PEL datum of unitary type associated with a
hermitian pair $(V,\pair_V)$ as in Sections \ref{PELdata}, \ref{PELunitary}, and \ref{unitarygroupPEL},
together with all the associated objects, choices, and conventions from Section 
\ref{padic-unitary-section}. However, since the number of factors $m$ equals $1$, the 
indexing subscript `$i$' will disappear from our notation. 
Let $G=G_P$ be the group scheme over $\Z$ associated with $P$ and let $X=X_P$ be the $G(\R)$ conjugacy
class of $h$. Let $Z_G$ be the center of $G$.
In this section we take $\sq = \nullset$, so the moduli problems are all considered
over the reflex field $F$.

Given $K\subset G(\A_f) = GU(V)(\A_f)$ (notation as in Section  \ref{notation-section}), we now write ${}_KSh(V)$ for the Shimura variety
associated with 
the Shimura datum\footnote{If $a_\sigma b_\sigma =0$ for all $\sigma\in\Sigma_\K$, then, properly speaking,
the datum $(G,X)$ does not satisfy the axioms of a Shimura variety as set out in \cite{deshimura}. Nevertheless, in this case, as the datum arises from a PEL datum
$P$, the notion of the associated `Shimura variety' still makes sense, following the conventions in \cite{lanalgan}.} $(G,X)$. So ${}_KSh(V)$ 
is just the $F$-scheme $\M_{K,L}$. We set 
$$
Sh(V) = \varprojlim_K {}_K Sh(V) = \varprojlim_{K} \M_{K,L}.
$$
The dimension of each ${}_KSh(V)$ is just the $\C$-dimension of $X$, which is
$$
d = \frac{1}{2}\sum_{\sigma\in\Sigma_\K} a_\sigma b_\sigma.
$$

At times we will be comparing constructions for both $Sh(V)$ and the Shimura variety
$Sh(-V)$ for the pair $(V,-\pair_V)$ (and the PEL datum
$P^c=(\K,c,\O,L,-\pair,h^c)$, where $h^c(z) = h(\bar z)$). 
When it is important to distinguish which hermitian space an object is associated with, 
we will generally add a subscript `$V$' (for the pair $(V,\pair_V)$) or `$-V$' (for
the pair $(V,-\pair_V)$, if the notation does not already distinguish the space (such as
is done by $Sh(V)$ and $Sh(-V)$).   We will also be using the notation $G_1 = GU(V)$, $G_2 = GU(-V)$ as in \ref{PELprob}.

\subsubsection{Automorphic vector bundles}
Recall that automorphic vector bundles on $Sh(V)=Sh(G,X)$ are defined by a $\otimes$-functor
$$
G-\mathrm{Bun}(\hat{X}) ~~ \longrightarrow ~~\mathrm{Bun}(Sh(V)),
$$
where $\hat{X}$ is the compact dual of $X$, so a flag variety for $G$, and $G-\mathrm{Bun}$ is the 
$\otimes$-category of $G$-equivariant vector bundles.  The base point $h\in X$ determines a point
$P_h\in\hat{X}$; this is just the stabilizer of the Hodge filtration on $L\otimes\R$ determined
by $h$. There is then a fibre functor $G-\mathrm{Bun}(\hat{X})\rar \mathrm{Rep}_\C(P_h) \cong 
\mathrm{Rep}_\C(P_0)$, where the last equivalence comes from the fixed identifications
in \ref{G0H0}.  Given an irreducible representation $W$ of $P_0$ that factors
through the Levi quotient $H_0$ of $P_0$, we let $\omega_W$ be the corresponding
automorphic vector bundle. Each such bundle has a canonical model over
a number field $F(W)/F$ contained in $\K'$. For $W=W_\kap$ as in \ref{H0-OF-reps} (here and in the following we write $W_\kap$ for $W_\kap(\C)$), 
the vector bundle $\omega_{\kap}$ defined in \ref{aut-sheaves} is the base change to $\K'$ of the canonical
model of $\omega_{W_\kap}$.  In fact, the $\omega_\kap$, which are defined over the toroidal 
compactifications, are the canonical extensions of the automorphic vector bundles, and
their twists by the ideal sheaves of the boundaries are the subcanonical bundles.

\subsubsection{Coherent cohomology and $(\grP_h,K_h)$-cohomology.}
We will write $H^i(Sh(V),\omega_{\kap})$ instead of 
$H^0(Sh(V)^{tor},\omega_{\kap})$,
which is imperfect shorthand for 
$$
\varinjlim_{K,\Sigma} H^i({}_{K}Sh(V)_{\Sigma},\omega_{\kap})
$$
where the limit is taken over toroidal compactifications (indexed by $\Sigma$) at finite level
(indexed by $K$).  For $i=0$, this is superfluous, by Koecher's principle, except possibly when $n = 2$ and $F = \Q$, and the reader
can be trusted to supply the missing indices in this case.  
Likewise we write $H^i(Sh(V),\omega_{\kap}^{\mathrm{sub}})$ for 
$$
\varinjlim_{K,\Sigma} H^i({}_{K}Sh(V)_{\Sigma},\omega_{\kap}(-D_\Sigma))
$$
where $D_\Sigma = {}_KSh(V)_\Sigma - {}_KSh(V)$. We let
$$
H^i_!(Sh(V),\omega_\kap) = im\{H^i(Sh(V),\omega_{\kap}^{\mathrm{sub}})\rightarrow H^i(Sh(V),\omega_{\kap})\}.
$$
Note that the ground field here can be taken to be any extension of $\K'$. Moreover, these definitions make sense over the ring $\O_{\K',(\grp')}$, provided we restrict to those
$K$ of the form $K=G(\Zp)K^p$ or $K=I_rK^p$.

Over $\C$ the coherent cohomology can be computed in terms of Lie algebra cohomology. 
Let  $\grg = \Lie(G(\R))_{\C}$, and let $\grg = \grp_h^-\oplus \grk_h \oplus \grp^+_h$ be the 
Harish-Chandra decomposition associated with $h$ (the eigenvalue decomposition for the involution
$ad\,h(\sqrt{-1})$). Let $\grP_h = \grp_h^-\oplus \grk_h$; this is just $\Lie(P_h(\R))_\C$ (so the Lie algebra
of $P_h(\C)$). 
We put
$$
K_h = U_\infty = C(\R).
$$
Let $\CA_0(G)$ be the space of cuspforms on $G(\A)$. Then over $\C$ there is a natural identification
of $G(\A_f)$-modules:
\begin{equation}\label{dbar}
H^i_!(Sh(V),\omega_\kap) = H^i(\grP_h,K_h; \CA_0(G)\otimes W_\kap).
\end{equation}
Here we use the identifications of $P_h(\C)$ with $P_0(\C)$ and $C(\C)$ with $H_0(\C)$ to
realize $W_\kap$ as a $(\grP_h,K_h)$-module. For $i=0$ this just restates the identification, 
recalled in \ref{classicalmod}, of $S_\kap(K,\C)$ with the space of $U_\infty\times K$-invariant
smooth functions $f:G(\A)\rightarrow W_\kap$ that are annihilated by $\grp^-_h$.

\subsubsection{The $\star$ involution}

There is an anti-holomorphic involution $\st$ of  $G-\mathrm{Bun}(\hat{X})$ that takes a $G$-equivariant bundle to
the complex conjugate bundle; on representations of $P_0$ factoring through the Levi quotient $H_0$
(which has been identified over $\C$ with the stabilizer $C$ in $G_{/\R}$ of $h$) it takes the irreducible representation
$W_\kap$ to a representation $W_{\kap^\star}$ whose restriction to the maximal compact subgroup of $U_\infty = C(\R)\subset H_0(\C)$
is dual to the restriction of $W_\kap$ but whose restriction to $\R^\times\subset G(\R)$ coincides with that of $\kap$. 
Concretely, if $\kap$ is identified with the tuple $\kap = (\kap_0,(\kap_\sigma))$, $\kap_\sigma = (\kap_{\sigma,1},\ldots,\kap_{\sigma, b_\sigma})$,
then $\kap^\star$ is the weight
\begin{equation}\label{kap-star}
\kap^\star = (\kap^\star_0, (\kap_\sigma^\star)), \ \ \kap^\star_0 =-\kap_0+a(\kap), \ \kap_\sigma^\star = (-\kap_{\sigma, b_\sigma},\ldots,-\kap_{\sigma,1})
\end{equation}
and 
$$W_{\kap^\star}\cong W_\kap^\vee\otimes\nu^{a(\kap)},
$$ 
where
\begin{equation}\label{a-kap}
a(\kap) = 2\kap_0+\sum_{\sigma\in\Sigma_\K}\sum_{j=1}^{b_\sigma} \kap_{\sigma, j}.
\end{equation}

There is a unique, up to scalar multiple, $c$-semilinear, $K_h$-equivariant isomorphism $W_\kap \isoarrow W_{\kap^\star}$.
Such an isomorphism is given explicitly by the map 
 that sends $\phi\in W_{\kap}$
to $\phi^\star\in W_{\kap^\star}$, where if $h\in H_0(\C)$ is identified with 
$(h_0,(h_\sigma)) \in \C^\times\times\prod_{\sigma\in\Sigma_\K} \GL_{b_\sigma}(\C)$ via
\eqref{H0-iso}, then 
\begin{equation*}\label{star-iso}
\phi^\star(h) =h_0^{a(\kap)}\cdot \overline{\phi((\bar h_0^{-1},(w_\sigma{}^t\bar h_\sigma^{-1}))}.
\end{equation*}
Here $w_\sigma \in \GL_{b_\sigma}(\C)$ is the longest element of the Weyl group of the standard pair
and the overline $\bar{}$ denotes complex conjugation.
The $K_h$-invariance follows easily from \eqref{Uinfty-iso}. 

The identification of $G(\C)$ with $G_0(\C)$ in
\ref{G0H0} identifies $\Lie(P_0(\C))$ with $\grP_h$
and $\Lie(H_0(\C))$ with $\grk_h$. 
It then follows that the map 
$\phi\mapsto \phi^\star$
is $\grP_h$-equivariant, up to $c$-semilinearity.

The action of $h = (h_0,(h_\sigma))\in H_0(\C)$ on $\Hom_\C(\wedge^d \grp_h^\pm,\C)$ is
just multiplication by $h_0^{\mp d}\prod_{\sigma\in\Sigma_\K} \det(h_\sigma)^{\pm 2a_\sigma}$;
this is just the character 
$$
\kap_h^\pm = (\mp d, (\kap_{h,\sigma}^\pm)), \ \ \kap_{h,\sigma}^\pm = (\pm 2a_\sigma, \ldots, \pm 2a_\sigma).
$$
Then the $H_0(\C)$-representation
$$
\Hom_\C(\wedge^d\grp_h^+,W_{\kap^\star}) = \Hom_\C(\wedge^d\grp_h^+,\C)\otimes_\C W_{\kap^\star}
$$
is naturally identified with $W_{\kap^D}$ (the identification depends on a choice of basis of the one-dimensional space 
$\wedge^d\grp_h^-$), where
$$
\kap^D = \kap^\star +\kap_h^+.
$$
The Killing form on $\grg$ defines an $H_0(\C)$-equivariant contraction map
$$
\wedge^d\grp_h^-\otimes_\C\wedge^d\grp_h^+ \rightarrow \C,
$$
and so defines an $H_0(\C)$-equivariant inclusion
$$
i_{\kap^\star}:W_{\kap^\star} \hookrightarrow \Hom_\C(\wedge^d\grp_h^-\otimes_\C\wedge^d\grp_h^+, W_{\kap^\star}) = \Hom_\C(\wedge^d\grp_h^-,W_{\kap^D}).
$$

\subsection{Complex conjugation and automorphic forms}
In this section we describe 
three 
actions of complex conjugation on spaces of modular forms. Each 
has an interpretation 
in Deligne's formalism for motives of absolute Hodge cycles, though we do not emphasize this here. We describe these actions in
terms of $(\grP_h,K_h)$-cohomology as well 
in terms of coherent cohomology.

\subsubsection{Complex conjugation on automorphic forms}\label{CConaut}
Let $\pi$ be a $(\grg,K_h)\times G(\A_f)$-representation occurring in the space $\CA_0(G)$ of cuspforms 
on $G(\A)$. We define $\bar\pi$ to be the complex conjugate representation; that is, $\bar\pi$ consists
of the functions $\bar\vphi(g) =\overline{\vphi(g)}$ for $\vphi\in\pi$. The map $\pi\rightarrow\bar\pi$, $\vphi\mapsto \bar\vphi$, is $c$-semilinear and $K_h \times G(\A_f)$-equivariant,
and even $\grg$-equivariant up to $c$-semilinearity.  We then obtain a $c$-semilinear $G(\A_f)$-equivariant map
\begin{equation}\label{cB-map}
(\pi\otimes_\C W_\kap)^{K_h} \stackrel{\vphi\otimes\phi\mapsto \bar\vphi\otimes\phi^\star}{\longrightarrow} (\bar\pi\otimes_\C W_{\kap^\star})^{K_h}
\stackrel{id\otimes i_{\kap^\star}}{\longrightarrow}
\Hom_\C(\wedge^d\grp_h^+,\otimes\bar\pi\otimes_\C W_{\kap^D})^{K_h}
\end{equation}
that is also $\grP_h$-equivariant, up to $c$-semilinearity.
This induces  a $c$-semilinear  $G(\A_f)$-equivariant isomorphism
\begin{equation}\label{C_B-pi}
c_B:H^0(\grP_h,K_h;\pi\otimes_\C W_\kap) \rightarrow H^d(\grP_h,K_h;\bar\pi\otimes_\C W_{\kap^D}).
\end{equation}
Taking $\pi$ to be the space of cuspforms $\CA_0(G)$ of $G(\A)$ (so, in particular, $\bar\pi = \pi)$, we obtain
a $c$-semilinear $G(\A_f)$-equivariant isomorphism
\begin{equation}\label{C_B-coh}
c_B:H^0_!(Sh(V),\omega_\kap) \isoarrow H^d_!(Sh(V),\omega_{\kap^D}).
\end{equation}

\subsubsection{Complex conjugation on $Sh(V)$}\label{CConSh}
Recall that
$$
P^c  = (\K,c,\O,L,-\pair,h^c), \ \ \ h^c(z) = h(\bar z),
$$ 
is just the PEL datum of unitary type associated with the 
hermitian pair $(V,-\pair_V)$. The corresponding reflex field is $F_{-V}=cF_V=cF$, the complex conjugate of $F$. 
There is a canonical identification $G_{P^c}=G_P=G$. The respective stabilizers in $G(\R)$ of $h$ and $h^c$ (action by conjugation) are the same: they both
equal $U_\infty$ (that is, $K_h = U_\infty = K_{h^c}$).  Let $X = G(\R)/U_\infty$. We then have identifications $X \isoarrow X_h = X_P$, $g\mapsto ghg^{-1}$, and
$X\isoarrow X_{h^c}=X_{P^c}$, $g\mapsto gh^cg^{-1}$. Each of $X_h$ and $X_{h^c}$ have a complex structure, and the pullbacks of these complex structures to $X$
are complex conjugates. In particular, the composition $X_h \isoarrow X \isoarrow X_{h^c}$ is an anti-holomorphic map. So a holomorphic function on $X_{h^c}$ defines
an anti-holomorphic function on $X_h$, and {\it vice versa}. This explains the map $F_\infty$ in \eqref{Finfty-coh} below. 

The automorphic sheaves on $Sh(-V)$ are associated to representations of the group $H_{0,-V}$, which is canonically identified
with $H_{0,V} = H_0$ by switching the roles of $\Lambda_0$ and $\Lambda_0^\vee$).  The analog of \eqref{H0-iso} for $H_{0,-V}$ is
the isomorphism
\begin{equation}\label{H0-iso-V}
{H_{0,-V}}_{/S_0} \isoarrow \G_m\times\prod_{\sigma\in\Sigma_\K} \GL_{\O\otimes_{\O,\sigma}}(\Lambda_{0,\sigma}) 
\cong \G_m\times\prod_{\sigma\in\Sigma_\K} \GL_{a_\sigma}(S_0).
\end{equation}
The identification $H_{0,V}=H_{0,-V}$ is given in terms of \eqref{H0-iso} and \eqref{H0-iso-V} by 
$(h_0,(h_\sigma)) \mapsto (h_0,(h_0{}^th_{\sigma c}^{-1}))$.  We have associated to each dominant character $\kap$ of the diagonal torus
$T_{H_{0,-V}}$ of $H_{0,-V}$ a representation $W_{\kap,-V}$ of $H_{0,-V}$ and hence a vector bundle $\omega_{\kap,-V}$ on $Sh(-V)$.
Given a dominant character $\kap = (\kap_0,(\kap_\sigma))$ of $T_{H_{0,V}}$, we define a dominant character 
$\kap^\flat= (\kap_0,(\kap_{\sigma c}))$ of $T_{H_{0,-V}}$. With respect to the canonical identification $H_{0,-V} = H_{0,V}$ described above, there is an
 explicit identification of $H_0$-representations
\begin{equation}\label{kapflat}
W_{\kap^\flat,-V} \isoarrow W_{\kap^\star,V}, \ \ \phi \mapsto ((h,(h_\sigma))\mapsto \phi(h_0,(w_\sigma h_0 {}^th_{\sigma c}^{-1}))). 
\end{equation}

The Harish-Chandra decompositions $\grg = \grp_h^-\oplus\grk_h\oplus\grp_h^+ = \grp_{h^c}^-\oplus\grk_{h^c}\oplus\grp_{h^c}^+$ 
satisfy $\grp_{h}^\pm = \grp_{h^c}^\mp$ and $\grk_h = \grk_{h^c}$. 
Let $\pi$ be a $(\grg,K_h) \times G(\A_f)$-representation occurring in the automorphic forms on $G(\A)$. Then the natural map
\begin{equation}\label{Finfty-map}
(\pi^{\grp_{h^c}^-}\otimes_\C W_{\kap^\flat,-V})^{K_{h^c}} =
(\pi^{\grp_{h}^+}\otimes_\C W_{\kap^\star, V})^{K_h} \stackrel{id\otimes i_{\kap^*}}{\rightarrow} \Hom_\C(\wedge^d\grp_h^-,\pi\otimes W_{\kap^D,V})^{K_h}
\end{equation}
induces a $\C$-linear $G(\A_f)$-equivariant isomorphism
\begin{equation}\label{Finfty-pi}
F_\infty: H^0(\grP_{h^c},K_{h^c};\pi\otimes_\C W_{\kap^\flat,-V})\rightarrow H^d(\grP_h,K_h;\pi\otimes_\C W_{\kap^D,V}).
\end{equation}
Taking $\pi$ to be $\CA_0(G)$ we then obtain a $\C$-linear $G(\A_f)$-equivariant isomorphism
\begin{equation}\label{Finfty-coh}
F_\infty: H^0_!(Sh(-V),\omega_{\kap^\flat,-V})\isoarrow H^d_!(Sh(V),\omega_{\kap^D,V}).
\end{equation}
Note that {\bf no complex conjugation is involved in this isomorphism}:  $F_\infty$ identifies a cohomology class on $G_2$ represented by a holomorphic modular form with a cohomology class on $G_1$ represented by an anti-holomorphic modular form, simply because the groups $G_1$ and $G_2$ are canonically equal but the hermitian symmetric domains have opposite complex structure.

\subsubsection{The involution `$\dagger$' and the isomorphisms ${}_K Sh(V) \cong {}_{K^\dagger}Sh(-V)$}\label{invo}

Recall that we have assumed that $h$ is standard (see~\ref{basepoints}). This means that there is a $\K$-basis of $V$ with respect to which the hermitian pairing $\pair_V$
is given by a diagonal matrix $D = \diag(d_1,...,d_n)$, $d_1,...,d_n\in\K^+$, and such that the image of $h$ is diagonal with respect to the induced basis on each of the
spaces $V_\sigma=V\otimes_{\K,\sigma}\C$. Under the hypothesis that each prime above $p$ in $\K^+$ splits in $\K$, it is always possible to 
choose such a $\K$-basis and the lattice $L$ so that $D$ is a diagonalization of the perfect hermitian pairing on $L\otimes\Z_{(p)}$ induced
by $\pair_V$; we fix such a choice of $\K$-basis and a lattice $L$. Let $I:V\rightarrow V$ be the $\K^+$-involution of $V$ that is just 
the action of $c$ on the coordinates with respect to this fixed $\K$-basis. Note that $L\otimes\Z_{(p)}$ is $I$-stable, and the map induced by $I$ on $L\otimes\Zp$ 
interchanges $L^+$ and $L^-$.

With respect to the fixed $\K$-basis, $G_{/\Q}$ is identified with a subgroup of $\Res_{\K/\Q}\GL_n(\K)$, 
and the action of $c$ on $\K$ induces
an automorphism $g\mapsto \bar g$ of $G_{/\Q}$ (note that $g^c = I g I$). This automorphism takes $h$ to $h^c$ and so maps $U_\infty$ to itself. 
In particular, it induces an automorphism of $X$.
The composition $X_h \isoarrow X \stackrel{g\mapsto \bar g}{\longrightarrow} X \isoarrow X_{h^c}$ (which is just
$ghg^{-1} \mapsto \bar g h^c \bar g^{-1}$) is holomorphic. In particular, the induced map $Sh(V)(\C)\rightarrow Sh(-V)(\C)$ is holomorphic and so a morphism 
of Shimura varieties over $\C$.

We modify this map at $p$, to more easily compare level structures. Recall that for each prime $w|p$ 
we fixed decompositions $L_w = L_w^+\oplus L_w^-$ (see \ref{PELunitary}). We also fixed an $\O_w$-basis of
each $L_w^\pm$, which gives an $\O_w$-basis of each $L_w$. 
We define level structures at $p$ for $P^c$ by taking $L_w^{c,\pm} = L_w^\pm$. Then $I_{w,-V}^0 = {}^t I_{w,V}^0 = {}^t (I_{w,V}^{0})^{-1}$
with respect to this $\O_w$-basis of $L_w$. 
This chosen $\O_w$-basis of $L_w$ may not be the $\K_w$-basis of $V\otimes_\K\K_w$ induced by the fixed
$\K$-basis of $V$; let $\beta_w \in\GL_{\K_w}(V\otimes_\K\K_w) \cong \GL_n(\K_w)$ be an element taking the latter to former.
Let $\delta_p = (1, D{}^t\beta_w^{-1}\beta_w^{-1})_{w\in\Sigma_p})\in \Q_p^\times \times \prod_{w\in\Sigma_p} \GL_n(\K_w) \cong G(\Qp)$,
where the isomorphism is determined by the fixed $\K$-basis of $V$.
Then 
\begin{equation}\label{Iwa-p-relations}
\bar \delta_p = \delta_p^{-1}, \ \ \delta_p^{-1}\overline{G(\Zp)} \delta_p  = G(\Zp), \ \ \text{and} \ \ 
\delta_p^{-1} \bar I_{r,V}^0 \delta_p = I_{r,-V}^0.
\end{equation}
We then define an automorphism $g\mapsto g^\dagger$ of $G(\A)\rightarrow G(\A)$ by $g^\dagger = \nu(g)^{-1} \delta_p^{-1} \bar g\delta_p$.
 Given $K\subset G(\A_f)$ we let $K^\dagger$ be the image of $K$ under $\dagger$. As a consequence of \eqref{Iwa-p-relations}, if $K=G(\Zp)K^p$, then 
$K^\dagger = G(\Zp)\bar K^p$ and
\begin{equation}\label{level-p-relations}
(K_{r,V})^\dagger = K^\dagger_{r,-V}.
\end{equation}
Consequently, the map $Sh(V)(\C) \rightarrow Sh(-V)(\C)$ induced by $g\mapsto \bar g \delta_p$ identifies ${}_{K_{r,V}}Sh(V)$ with 
${}_{K^\dagger_{r,-V}}Sh(-V)$.   The following Proposition is then obvious.

\begin{prop}\label{moduliflat} The isomorphism ${}_{K_{r,V}}Sh(V)\isoarrow {}_{K^\dagger_{r,-V}}Sh(-V)$ is defined over $\O_{\K',(\grp')}$. On moduli problems
it is given by the map that sends a tuple $(A,\lambda,\iota,\alpha,\phi)$ classified by $M_{P,K_{r},L}(R)$ to the tuple 
$(A,\lambda,\iota\circ c, \alpha\circ I, \phi\circ I)$ classified by $M_{P^c,K^\dagger_{r},L}(R)$ for any $\O_{\K',(\grp')}$-algebra $R$.
\end{prop}

The automorphism $g\mapsto g^\dagger$ takes
$\grp_h^\pm$ to $\grp_{h^c}^\pm$ and $P_h$ to $P_{h^c}$. 
The action of
$g\mapsto g^\dagger$ on $K_h$ is identified via \eqref{Uinfty-iso} as $(h_0,(h_\sigma)) \mapsto (h_0^{-1},({}^th_\sigma^{-1}))$.
Let 
$$
\kap^\dagger = \kap^\flat\cdot||\nu||^{-a(\kap)}, 
$$ 
so
$$
W_{\kap^\dagger,-V}\cong W_{\kap^\vee,V} \cong W_{\kap,V}^\vee.
$$
The map
$$
W_{\kap,V} \stackrel{\phi\mapsto\phi^\dagger}{\longrightarrow} W_{\kap^\dagger,-V} = W_{\kap^\flat,-V}\otimes||\nu||^{-a(\kap)}, \ \ 
\phi^\dagger((h_0,(h_\sigma))) = \phi((h_0,(h_{\sigma c}))) h_0^{-a(\kap)},
$$
satisfies $( k^\dagger \cdot \phi)^\dagger = k\cdot \phi^\dagger$ for all $k\in K_{h^c}= K_h$.  It follows that under the isomorphism 
$Sh(V)\isoarrow Sh(-V)$ defined by $g\mapsto g^\dagger$, $\omega_{\kap^\dagger,-V}$ pulls back to $\omega_{\kap,V}$,
and so there are $\C$-linear isomorphisms
\begin{equation}\label{Fdagger-iso}
F^\dagger:H^i_!(Sh(V),\omega_{\kap,V})\isoarrow H^i_!(Sh(-V),\omega_{\kap^\dagger,-V})
\end{equation}
that are $G(\A_f)$-equivariant up to the action of the automorphism `$\dagger$.'
In particular, these induce isomorphisms
\begin{equation}\label{Fdagger-level}
F^\dagger:H^i_!({}_{K_r}Sh(V),\omega_{\kap,V})\isoarrow H^i_!({}_{K_r^\dagger}Sh(-V),\omega_{\kap^\dagger,-V}),
\end{equation}
even over $\O_{\K',(\grp')}$-algebras $R\subset \C$. In particular, $F^\dagger$ restricts to an isomorphism
\begin{equation}\label{Fdagger-nebentypus}
F^\dagger: S_{\kap,V}(K_{r,V},\psi;R)\isoarrow S_{\kap^\dagger,-V}(K^\dagger_{r,-V},\psi^\dagger;R)
\end{equation}
for $R\subset\C$ any $\O_{\K',(\grp')}[\psi]$-algebra, where $\psi^\dagger=\psi^{-1}$ if both are viewed as characters of the diagonal torus of the right side of \eqref{G-iso} via the isomorphisms \eqref{G-iso2}. 

The action of $F^\dagger$ is described in terms of automorphic forms as follows.
Let $\pi$ be a $(\grg,K_h)\times G(\A_f)$ representation occurring in the space of automorphic forms on $G(\A)$. We define
$\pi^\dagger$ to be the space of functions $\varphi^\dagger(g) = \varphi(g^\dagger)$ for $\varphi\in \pi$. The map $\pi\rightarrow \pi^\dagger$, $\varphi\mapsto\varphi^\dagger$,
is $\C$-linear and is both $\grg$- and $K_h$-equivariant up to the action of the automorphism `$\dagger$'.
The map
$$
(\pi\otimes_\C W_{\kap,V})^{K_h} \stackrel{\varphi\otimes\phi\mapsto \varphi^\dagger\otimes\phi^\dagger}{\longrightarrow}(\pi^\dagger\otimes_\C W_{\kap^\dagger,-V})^{K_{h^c}}
$$
is then a $\C$-linear isomorphism that intertwines the actions of $g$ and $g^\dagger$ for all $g\in G(\A_f)$. This induces a corresponding isomorphism
\begin{equation}\label{Fdagger-pi}
F^\dagger:H^i(\grP_h,K_h;\pi\otimes_\C W_{\kap,V}) \isoarrow H^i(\grP_{h^c},K_{h^c};\pi^\dagger\otimes_\C W_{\kap^\dagger,-V}).
\end{equation}
Taking $\pi = \CA_0(G)$ (and so $\pi^\dagger = \pi$), we get $F^\dagger$ from before.

\subsection{Serre duality and pairing of automorphic forms}\label{Petersson}
Since 
$$
W_{\kap^D} = \Hom_\C(\wedge^d \grp_h^+, W_{\kap^\star}) \cong \Hom_\C(\wedge^d\grp_h^+,\C)\otimes_\C W_\kap^\vee\otimes\nu^{a(\kap)},
$$
the natural contraction $W_\kap\otimes_\C W_\kap^\vee \rightarrow \C$ gives a homomorphism of $H_0(\C)$-representations
$$
W_\kap\otimes_\C W_{\kap^D} \rightarrow \Hom_\C(\wedge^d\grp_h^+,\C)\otimes\nu^{a(\kap)}.
$$
This induces a natural map
$$
\omega_{\kap} \otimes \omega_{\kap^D} ~~\rar ~~ \Omega^d_{Sh(V)}\otimes L(\kap),
$$
where $L(\kap)$ is the automorphic line bundle attached to the character $\nu^{a(\kap)}$.  
Since the character is trivial on $G^{\mathrm{der}}$, $L(\kap)$ is topologically the $\CO_{Sh(V)}$-bundle
attached to the constant (trivial) sheaf, but the action of $G(\A_f)$ on $L(\kap)$ is non-trivial.
Fixing a level subgroup $K$ and a toroidal compactification $_{K}Sh(V) \hookrightarrow _{K}Sh(V)_{\Sigma}$,
we can extend this to a natural pairing
$$
\omega_{\kap} \otimes \omega_{\kap^D}^{\mathrm{sub}} ~~\rar ~~ \Omega^d_{ _{K}Sh(V)_{\Sigma}}\otimes L(\kap)$$
and the analogous pairing on $\omega_\kap^{\mathrm{can}}\otimes\omega_{\kap^D}$.
As in \cite[Cor.~2.3]{H90}, Serre duality therefore defines a perfect pairing

\begin{equation}\label{SerreD} 
H^0_!(Sh(V),\omega_{\kap}) \otimes H^d_!(Sh(V),\omega_{\kap^D}) 
~ \rar ~ \varinjlim_{K,\Sigma}H^d(_{K}Sh(V)_{\Sigma},\Omega^d_{Sh(V)_{\Sigma}}\otimes L(\kap)) 
 \end{equation}    
The function $g \mapsto ||\nu(g)||^{-a(\kap)}$ defines a global
section of $L(\kap)^{\vee}$ and therefore an isomorphism
$$
\varinjlim_{K,\Sigma}H^d({}_{K}Sh(V)_{\Sigma},\Omega^d_{Sh(V)_{\Sigma}}\otimes L(\kap))  \isomto 
\varinjlim_{K,\Sigma}H^d({}_{K}Sh(V)_{\Sigma},\Omega^d_{Sh(V)_{\Sigma}}).
$$
The right-hand side is isomorphic under the trace map to the space of functions 
$C(\pi_0(V))$ 
on the compact space $\pi_0(V)$ of similitude components of $Sh(V)$.   
Composing with the projection of $C(\pi_0(V))$ onto the invariant line $C(\pi_0(V))^{G(\A)}$ -- in other words, integration over $\pi_0(V)$ with respect to an invariant measure
with rational total mass -- we thus obtain a canonical perfect pairing:
\begin{equation}\label{SerreC} 
\pairS_\kap:  H^0_!(Sh(V),\omega_{\kap}) \otimes H^d_!(Sh(V),\omega_{\kap^D}) 
~~\rar ~   \C
\end{equation}

 \begin{rmk}  In what follows, we will be using the Tamagawa number to normalize the Serre duality pairings. 
This is likely to introduce a factor of 
 a power of $2$ in a comparison
 of our results with those predicted by motivic conjectures.
\end{rmk}

The pairing $\pairS_\kap$ can be described in terms of automorphic forms as follows.
Let $\grp = \grp_h^+\oplus \grp_h^-$. Then $\pairS_\kap$ is just the pairing 
\begin{equation*}\label{SerreL2}
H^0(\grP_h,K_h;\CA_0(G)\otimes W_{\kap})\otimes H^d(\grP_h,K_h;\CA_0(G)\otimes W_{\kap^{D}}) 
~\rar \C
\end{equation*}
defined by multiplication of cuspforms, contraction of the coefficients, and integration. More precisely, let $dg$ denote Tamagawa measure,
and factor $dg = dg_\infty\cdot dg_f$, where $dg_\infty$ (resp. $dg_f$) is an invariant measure on $G(\R)$ (resp. on $G(\adeles_f)$).  We assume
the measure $dg_f$ takes rational values on open subgroups of $G(\adeles_f)$, and we write $dg_\infty = dk_\infty \times dx \times dt/t$, where $dk_h$ is the measure
that gives $K_h$ measure $1$, $dt$ is Lebesgue measure on the center $Z_{G(\R)} \isoarrow \R^\times$ -- which disappears in the integral -- and $dx$ is a differential form on $\grp_h$; this will inevitably be rational over the reflex field of $h$.  
We denote the contraction
$$
W_\kap\otimes_\C\Hom_\C(\wedge^d\grp_h^-,\Hom_\C(\wedge^d\grp_h^+,W_\kap^\vee\otimes\nu^{a(\kap)}))\rightarrow \Hom_\C(\wedge^{2d}\grp,\C(\nu^{a(\kap)})),
$$
by
$$ 
\phi\otimes \phi' \mapsto  [\phi,\phi'].
$$
Then for 
$$
\varphi \in (\CA_0(G)^{\grp_h^-}\otimes_\C W_\kap)^{K_h} \ \ \text{and} \ \ \varphi' \in \Hom_\C(\wedge^d \grp_h^-, \CA_0(G)\otimes_\C W_{\kap^D})^{K_h},
$$
we normalize $\pairS_\kap$ so that
\begin{equation}\label{L2}  
\vpair{\varphi,\varphi'}^\Serre_\kap = \int_{G(\Q)Z_G(\R)\backslash G(\A)} [\varphi(g),\varphi'(g)]\cdot||\nu(g)||^{-a(\kap)} dg_f.
\end{equation}
It we use the basis $dx$ of $(\wedge^{2d}\grp)^{\vee}$ to identify $(\wedge^{2d}\grp)^{\vee}$ with $\C$, then we write
$[\varphi(g),\varphi'(g)]_x$ for the element of $\C(\nu^{a(\kap)})$ corresponding to $[\varphi(g),\varphi'(g)]$, and \eqref{L2}
can be rewritten 
$$\vpair{\varphi,\varphi'}^\Serre_\kap = \int_{G(\Q)Z_G(\R)\backslash G(\A)} [\varphi(g),\varphi'(g)]_{dx}\cdot||\nu(g)||^{-a(\kap)} dg. $$
In what follows we fix the basis $dx$ and omit the subscript $_x$ from $[\bullet,\bullet]$.

From $\pairS_\kap$ we obtain the hermitian Petersson pairing: 
\begin{equation}\label{pet}
\pairP_\kap:H^0_!(Sh(V),\omega_\kap)\times H^0_!(Sh(V),\omega_\kap)\rightarrow \C , \ \ \pairP_\kap = \vpair{\cdot,c_B(\cdot)}^\Serre_\kap.
\end{equation}

\subsection{A normalized pairing, trace maps, and integral structures}\label{normal pairing}

While the pairing $\pairS_\kap$ is canonical it is not compatible with traces of automorphic forms with respect to levels (this is 
manifestly a failure of the right-hand side of \eqref{L2}). We correct this by introducing a normalization of the pairing $\pairS$ that
depends on the level. We then use this normalized pairing to define integral stuctures on {\em top} degree cohomology. One
consequence of these definitions is that the integral structures are compatible with respect to the trace maps.

\subsubsection{A normalized Serre duality} For an open compact subgroup $K^p\subset G(\A_f^p)$ we define 
\begin{equation}\label{pairC-def}
\pair_{\kap,K_r}: H^0_!({}_{K_r}Sh(V),\omega_{\kap}) \otimes H^d_!({}_{K_r}Sh(V),\omega_{\kap^D}) 
~~\rar ~   \C
\end{equation}
as
\begin{equation}
\pair_{\kap,K_r}  = \frac{1}{\Vol(I_{r,V}^0)} \pairS_{\kap},
\end{equation}
where $I_r = I_{r,V}$ and the volume $\Vol(I_r^0)$ of $K_r^0$ is taken with respect to the Tamagawa measure $dg$ also appearing in \eqref{L2}.
In particular, if $\vphi$ and $\vphi'$ in the left-hand side of \eqref{L2} are both invariant by $K_r$, then
\begin{equation}\label{L2-C}
\vpair{\varphi,\varphi'}_{\kap,K} = \frac{1}{\Vol(I_r^0)} \int_{G(\Q)Z_G(\R)\backslash G(\A)} [\varphi(g),\varphi'(g)]\cdot||\nu(g)||^{-a(\kap)} dg.
\end{equation}
From \eqref{L2-C} it is easily seen that if $r'\geq r$ and if $\vphi$ is invariant 
by $K_r$ and $\vphi'$ by $K'_r$,
then 
\begin{equation}\label{pairC-trace}
\vpair{\vphi,\vphi'}_{\kap,K_{r'}} = \vpair{\vphi,\trace_{K_r/K_{r'}}(\vphi')}_{\kap,K_r}, \ \ \ 
\trace_{K_r/K_{r'}}(\vphi') = \frac{\#(I_r^0/I_r)}{\#(I_{r'}^0/I_{r'})}\sum_{\gamma\in K_r/K_{r'}} \gamma\cdot \vphi'.
\end{equation}
The analogous relation also holds for $\vphi$ invariant under $K_{r'}$ and $\vphi'$ invariant under $K_r$.
The key point here is, of course, that $I_r/I_{r'} \isoarrow K_r/K_{r'}$ and 
$$
\Vol(I_r^0) = \Vol(I_{r'}^0) \cdot \# (I_r^0/I_{r'}^0) = \Vol(I_{r'}^0) \cdot \# (I_r/I_{r'}) \frac{\#(I_r^0/I_r)}{\#(I_{r'}^0/I_{r'})}.
$$

 \subsubsection{Integral structures on top cohomology}\label{integraltop}
 
 Let $\CO_{\CK',(\mathfrak{p}')}$ be as in Section \ref{H0-OF-reps}.
 Fix $V$ and write $\omega_\kap = \omega_{\kap,V}$.  The spaces $H^i_!({}_{K_r}Sh(V),\omega_{\kap})$ have natural integral structures over $\CO_{\CK',(\mathfrak{p}')}$ with respect to $\CO_{\CK',(\mathfrak{p}')}$-integral structures on the underlying schemes, for any $i$.  However, because the special fibers become progressively more singular as $r$ increases, we {\it do not choose} integral structures on the schemes.  For cohomology in degree $i = 0$, we define the $\CO_{\CK',(\mathfrak{p}')}$-structure on $H^0_!({}_{K_r}Sh(V),\omega_{\kap,V})$ by $S_{\kap}(K_r,\CO_{\CK',(\mathfrak{p}')})$ as in Section 2, specifically in sections \ref{levelp}, \ref{levelK}, and especially \ref{padicmforms-section}.   We then define the 
 $\CO_{\CK',(\mathfrak{p}')}$-structure on $H^d_!(_{K_r}Sh(V),\omega_{\kap^D})$ to be {\it dual} to the integral structure on $H^0_!({}_{K_r}Sh(V),\omega_{\kap})$ with respect to the pairing $\pair_{\kap,K_r}$ defined in \eqref{pairC-def}.  In other words, for any $\CO_{\CK',(\mathfrak{p}')}$-algebra $R$, we have an identification
 \begin{equation}\label{integraltop1}
 H^d_!(_{K_r}Sh(V),\omega_{\kap^D},R) = \Hom_{\CO_{\CK',(\mathfrak{p}')}}(S_{\kap}(K_r,\CO_{\CK',(\mathfrak{p}')}),R)
 \end{equation}
induced by $\pair_{\kap,K_r}$. 
It follows from \eqref{pairC-trace} that
these integral structures are respected by the trace maps:
\begin{equation}\label{integraltop2}
\vphi \in H^d_!(_{K_{r+1}}Sh(V),\omega_{\kap^D},R) \implies \trace_{K_r/K_{r+1}} \vphi \in H^d_!(_{K_{r}}Sh(V),\omega_{\kap^D},R).
\end{equation}

 \subsection{(anti-)holomorphic automorphic representations}

By an automorphic representation of $G$ we will always mean an {\it irreducible} $(\grg,K_h)\times G(\A_f)$-representation ocurring in the space
of automorphic forms on $G(\A)$. This convention allows us to distinguish holomorphic representations from anti-holomorphic representations.
(Note that $K_h$, which is the stabilizer of $h$ in $G(\R)$, need not project to the maximal compact in $G(\R)/Z_G(\R)$.)

\subsubsection{Holomorphic and anti-holomorphic cuspidal representations of type $(\kap,K)$}\label{autoreps}
Let $\pi$ be a cuspidal automorphic representation of $G$ (always assumed irreducible). 
Write $\pi = \pi_{\infty} \otimes \pi_f$, where $\pi_f$
is an irreducible admissible representation of $G(\A_f)$ and $\pi_{\infty}$ is an irreducible $(\grg,K_h)$-module.  
Let $K \subset G(\A_f)$ be an open compact.
We say $\pi$ is {\it holomorphic} (resp. {\it anti-holomorphic}) of type $(\kap,K)$ if $H^0(\grP_h,K_h;\pi_{\infty}\otimes_\C W_{\kap}) \neq 0$
(resp. $H^d(\grP_h,K_h;\pi_{\infty}\otimes_\C W_{\kap^{D}}) \neq 0$) and if $\pi_f^K \neq 0$.  
In this paper, we will only be concerned with $\pi$
that are either holomorphic or anti-holomorphic. If $\pi$ is holomorphic (resp.~anti-holomorphic) of type $(\kap,K)$, then by our conventions $\bar\pi$ is anti-holomorphic (resp.~holomorphic)
of type $(\kap,K)$.

Note that, with $G$ fixed, $\pi$ can be either holomorphic or anti-holomorphic, but not both (unless $G$ is definite); however, the isomorphism $F_\infty$ of \eqref{Finfty-pi} identifies anti-holomorphic representations of $G_2$ with holomorphic representations of $G_1$, and vice versa.    Although Hida theory is generally understood to be a theory of $p$-adic variation of (ordinary) holomorphic modular forms, the nature of the doubling method makes it more natural for us to take our basic object $\pi$ to be an {\it anti-holomorphic} (and anti-ordinary, see \ref{ord-hecke} below) cuspidal automorphic representation of $G_1$.    Thus $\pi$ is a {\it holomorphic} automorphic representation of $G_2$ but the natural object there is $\pi^{\flat}$, or $\bar{\pi}$, which is again anti-holomorphic.  Because this is inevitably a source of confusion,  reminders of these conventions have been inserted at strategic locations in the text.

\begin{rmk} \label{Epi-rmk}
If $\pi$ is holomorphic or anti-holomorphic, then, by the considerations in \cite{BHR}, $\pi_f$ is always defined
over a number field, say $E(\pi)$.   
We will always take $E(\pi)$ to contain $\K'$.
\end{rmk}

\subsubsection{The $\flat$ involution and the MVW involution $\dagger$}\label{MVWdag}
Let $\pi$ be a cuspidal automorphic representation of $G$. Let $\xi_\pi$ be the central character $\pi$. 
If $(\pi_\infty\otimes_\C W_\kap)^{K_h}\neq 0$ (for example, if $\pi$ is holomorphic of type $(\kap,K)$), 
then $\xi_{\pi,\infty}(t) = t^{a(\kap)}$ for $t\in\R^\times$.  Let 
\begin{equation}
\pi^\flat = \pi^{\vee}\otimes |\xi_{\pi}\circ \nu| = \pi^\vee\otimes ||\nu||^{a(\kap)}.
\end{equation}
Because $\pi\otimes |\xi_{\pi}\circ\nu|^{-\frac{1}{2}}$
is unitary, 
\begin{equation}
\pi^\flat \cong \bar{\pi},
\end{equation}
and when $\pi$ occurs with multiplicity one, as we will
generally assume, $\pi^\flat$ and $\bar{\pi}$ are the same spaces of automorphic forms.   
In particular, the operation $\pi \mapsto \pi^\flat$ is an involution of the set of cuspidal automorphic representations of $G$.  
If $\pi$ is holomorphic, then $\pi^\flat$ is anti-holomorphic, and vice versa.
  
The involution $g\mapsto g^\dagger$ of $G$ that was fixed in \ref{invo} is of the type considered by Moeglin, Vigneras, and Waldspurger in \cite[Chapitre 4]{MVW}.
In particular, there is an element $h_0\in \GL_{\K^+}(V)$ such that $h_0$ is $c$-semilinear for the $\K$-action on $V$ and 
$\vpair{h_0v,h_0w}_V = \vpair{w,v}_V$ and such that $\bar g = h_0 g h_0^{-1}$; with respect to the fixed $\K$-basis of $V$, $h_0$ is
just `act-by-$c$ on the coordinates'. Let $\pi = \otimes_{\ell\leq \infty}\pi_\ell$ be an automorphic representation of $G$.
If the hermitian pair $(V,\pair_V)$ is unramified at $\ell$, then it is a deep result proved in \cite[Chapitre 4]{MVW} (cf.~\cite{HKS}) that
\begin{equation}\label{dagger=flat}
\pi_\ell^\dagger \cong (\pi_\ell\circ\mathrm{Ad}(h_0))\otimes(\xi_\pi^{-1}\circ\nu) \cong \pi^\vee.
\end{equation}
In particular, if $\pi$ satisfies strong multiplicity one -- which we expect if the places at which $\pi$ or the group $G$ is ramified all 
split in $\K/\K^+$ and its base change to ${\GL_n}_{/\K}$ is cuspidal -- then 
$\pi^\dagger \cong \pi^\vee$ and so
$\pi^\dagger \otimes ||\nu||^{a(\kap)} = \pi^\flat = \bar\pi$. 
In any event, \eqref{dagger=flat} permits the Hecke actions on $\pi^\flat$ to be expressed in terms of the Hecke actions on $\pi^\dagger$, at least at the unramified primes $\ell$.
As will be explained later, the doubling method will pair $\pi$ and (a twist) of $\pi^\flat$, but we will use the involution `$\dagger$' to compare level structures and Hecke algebras.
This partly motivates our putting
\begin{equation}\label{flat-notation}
K^\flat = K^\dagger, \ \ \psi^\flat = \psi^\dagger, \ \ \text{and} \ \ \kap^\flat = \kap^\dagger\cdot \nu^{a(\kap)}.
\end{equation}

\subsubsection{Relating $\pair_\pi$ to $\pairS_\kap$}
Let $\pi$ be a holomorphic cuspidal automorphic representation of $G$ of type $(\kap,K)$.
Recall that the canonical pairing $\pair_\pi:\pi\otimes\pi^\vee\rightarrow \C$ can be expressed as
\begin{equation}\label{pair-pi}
\vpair{\vphi,\vphi'}_\pi = \int_{G(\Q)Z_G(\R)\backslash G(\A)} \vphi(g)\vphi'(g) dg, \ \ \ \vphi\in\pi,\ \vphi'\in\pi^\vee.
\end{equation}
The pairing $\pairS_\kap$ can be expressed in terms of $\pair_\pi$ as follows.

Let $w_1,...,w_m$ be a basis of $W_\kap$ and let $w_i^\vee,...,w_m^\vee$ be the dual basis of $W_\kap^\vee$.
As $W_{\kap^D}$ is the twist of $W_\kap^\vee$ by a character, the $\omega_i^\vee$ also defined a basis of $W_{\kap^D}$.
Let $\vphi \in (\pi^{\grp_h^{-}}\otimes_\C W_\kap)^{K_h}$ and 
$\vphi' \in \Hom(\wedge^d{\grp_h^-},\pi^\flat \otimes_\C W_{\kap^D})^{K_h}$. Write $\vphi = \sum_i \vphi_i\otimes w_i$
and $\vphi' = \sum_j \vphi'_j\otimes w_j^\vee$. Then it follows from \eqref{L2} that
\begin{equation}\label{pair-compare}
\vpair{\vphi,\vphi'}_\kap^\Serre = \sum_i \vpair{\vphi_i,\vphi_i'\cdot ||\nu||^{-a(\kap)}}_\pi.
\end{equation}

Let $\chi$ be a Hecke character of type $A_0$.  Recall that we have defined a twisted Petersson norm in \eqref{chitwist}, pairing vectors in $\pi$ with vectors in $\pi^\flat\otimes \chi^{-1}\circ\det$.  We may analogously define a pairing 
$$\pairS_{\kap,\chi}:  \pi \bigotimes \pi^{\flat}\otimes \chi^{-1}\circ\det$$ 
with the property that
\begin{equation}\label{pair-compare_chi}
\vpair{\vphi,\vphi'}_{\kap,\chi}^\Serre = \sum_i \vpair{\vphi_i,\vphi_i' ||\nu||^{-a(\kap)}}_{\pi,\chi}.
\end{equation}
Here the subscript $\chi$ at the end of \eqref{pair-compare_chi} has the same meaning as in \eqref{chitwist}.

\subsection{Hecke algebras}\label{heck}

We continue to let $G = G_1 = GU(V)$ and we return to the notation of Section  \ref{kapstar}; thus classical modular forms are of weight $\kap$.  
Fix a positive integer $r$ as in \ref{levelp} and a level subgroup $K = K_r = K^p\cdot I_r \subset G(\A_f)$.  Henceforth we will
write $T(g) = T_r(g)$ for the Hecke operators $[K^p_rgK^p_r]$
for $g \in G(\A_f^p)$; we have also introduced $U$-operators $U_{w,j}$ in \eqref{Up-operator}.  

For any $S_0$-algebra $R\subset\C$, we let $\bT_{K_r,\kap,R}$ be the $R$-subalgebra of 
$\End_\C(S_\kap(K_r;\C)) = \End_\C(H^0({}_{K_r}Sh(V),\omega_\kap))$
generated by the $U_{w,j,\kap} = |\kap'(t_{w,j})|_p^{-1} U_{w,j}$, where $\kap'$
is related to $\kap$ as in \eqref{Hecke-p-comp},
and by the $T(g)=T_r(g)$ for $g \in G(\A_f^S)$, where $S = S(K^p)$ is the set of places of $\K^+$
at which $K^p$ does not contain a hyperspecial maximal subgroup.\footnote{Since $G$ is a $\Q$-group, this needs to be clarified.  
We let $G(\A_f^S) = \prod_p G(\Q_p)^S$, where the product is taken over rational primes.  
If $p$ is not divisible by a prime in $S$ then $G(\Q_p)^S = G(\Q_p)$.
In general, let $\CP$ denote the set of primes of $\K^+$ dividing $p$, and write $\CP = \CP_1 \coprod \CP_2$, where primes in $\CP_1$ split in $\K$ and
those in $\CP_2$ do not split.  Then
$$G(\Qp) = \prod_{v \in \CP_1} GL(n,\K^+_v) \times G_{\CP_2},$$
where $G_{\CP_2}$ is the subgroup of elements $((x_w),t) \in \prod_{w \in \CP_2}GL(n,\K_w) \times \Q_p^\times$ such that each
$x_w$ preserves the hermitian form on $V(\K_w)$ with similitude factor $t$.  Write $S\cap \CP_i = S_{p,i}, i = 1, 2$.  Then
$$G(\Q_p)^S = \prod_{v \in \CP_1 \setminus S_{p,1}} GL(\K^+_v) \times G_{\CP_2}$$
if $S_{p,2}$ is empty, and $G(\Q_p)^S = \prod_{v \in \CP_1 \setminus S_{p,1}} GL(\K^+_v)$ otherwise.   We could also ignore all the divisors of $p$ if even
one of them belongs to $S$, since dropping finitely many generators from the unramified Hecke algebra doesn't change anything.}

We similarly define $\bT^-_{K_r,\kap,R}$ and $\bT^d_{K_r,\kap,R}$ 
by replacing $U_{w,j,\kap}$ with $U_{w,j,\kap}^- = |\kap'(t_{w,j})|_pU_{w,j}^-$
and, in the second case, also replacing $S_\kap(K_r;\C)$ with $H^d({}_{K_r}Sh(V),\omega_\kap) = H^d(Sh(V),\omega_\kap)^{K_r}$.
We will follow the convention 
of adding a subscript `$V$' (reps.~`$-V$') to notation if it is needed
to indicate that it relates to the hermitian pair $(V,\pair_V)$ (resp.~$(V,-\pair_V)$). 

\begin{lem}\label{Hecke-isoms} Let $R\subset\C$ be a subring.\hfill
\begin{itemize}
\item[\rm (i)] There exists a unique $R$-algebra isomorphism
$\bT_{K_r,\kap,R} \isoarrow \bT^d_{K_r,\kap^D,R}$, $T\mapsto T^d$,
such that $U_{w,j,\kap}^d = U_{w,j,\kap^D}^-$ and $T(g)^d= ||\nu(g)||^{a(\kap)}\cdot T(g^{-1})$.
\item[\rm (ii)] 
There exists a unique $R$-algebra isomorphism $\bT_{K_r,\kap,V,R} \isoarrow \bT_{K_r^\flat,\kap^\flat,-V,R}$,
$T\mapsto T^\flat$, such that 
$U_{w,j,\kap}^\flat  = U_{w,n,\kap^\flat}^{-1} U_{w,n-j,\kap^\flat}$ and
$T(g)^\flat = T(g^\dagger)=T(\bar g)$.
\item[\rm(iii)] There exists a unique isomorphism $\bT^-_{K_r,\kap,R}\isoarrow \bT^d_{K_r,\kap^D,cR}$ that 
maps $r\in R$ to $c(r)$, $U_{w,j,\kap^D}^-$ to $U_{w,j,\kap^D}^-$, and $T(g)$ to $T(g)$.
\end{itemize}
\end{lem}

\begin{proof} Part (i) follows from Serre duality, part (ii) from the isomorphism $F^\dagger$, and part (iii) from the isomorphism $c_{B}$.
\end{proof}
 
For a nebentypus $\psi$ of level $r$ (a character of $T_H(\Zp)$ that factors through $T_H(\Zp/p^r\Zp)$),
we let $\bT_{K_r,\kap,\psi,R}$ and $\bT^d_{K_r,\kap,\psi,R}$ be the quotients of $\bT_{K_r,\kap,R}$ 
and $\bT^d_{K_r,\kap,R}$ upon restriction to
the (invariant) subspaces $S_{\kap}(K_r,\psi,\C) \subset S_{\kap}(K_r,\C)$
and $H^d({}_{K_r}Sh(V),\omega_\kap)^\psi$, the subspace of $H^d({}_{K_r}Sh(V),\omega_\kap)$
on which the action of $K_r^0$ (which factors through $T_H(\Zp/p^r\Zp)$ as in Section   \ref{levelp}) acts via $\psi$.   

\begin{lem}\label{Hecke-char-isom} The isomorphisms in Lemma \ref{Hecke-isoms}(i)-(ii) induce 
$R$-algebra isomorphisms 
$$
\bT_{K_r,\kap,\psi,R}\isoarrow \bT^d_{K_r,\psi^{-1},\omega_{\kap^D},R} \ \ \text{and} \ \
\bT_{K_r,\kap,\psi,V,R} \isoarrow \bT_{K_r^\flat,\kap^\flat,\psi^\flat, -V, R}.
$$
\end{lem}

\noindent This is clear from the definitions.

The $R$-modules $S_\kap(K_r;R)$ and $S_\kap(K_r,\psi;R)$
are stable under the action of the Hecke operators $U_{w,j,\kap}$ and $T(g)$, $g\in G(\A_f^S)$. In particular, 
the cuspforms over $\C$ can be replaced by those over $R$ in the definition of $\bT_{K,\kap,R}$ and $\bT_{K,\kap,\psi,R}$.

For any of these Hecke algebras $\bT^{?}_\bullet$, we write $\bT^{?,p}_\bullet$ for the subalgebra generated over the ring $R$ 
by the $T(g)$, $g\in G(\A_f^S)$ (so omitting the $U_{w,j,\kap}$ and $U_{w,j,\kap}^-$). The isomorphisms of Lemmas \ref{Hecke-isoms}
and \ref{Hecke-char-isom} restrict to corresponding isomorphisms of these ($p$-depleted) Hecke algebras.

If $R=S_0$, then we omit the subscript `$R$' from our notation.

\subsubsection{The homomorphism $\lambda_\pi^p$, isotypical subspaces, and the multiplicity one hypothesis}
Let $\pi$ be a holomorphic cuspidal representation of $G$ of type $(\kap,K_r)$.  Then the natural action of $\bT^p_{K_r,\kap}$ on $\pi^{K_r}$
is given by a character that we denote $\lambda^p_{\pi}$; these homomorphisms are compatible under the natural projections 
$\bT^p_{K_r,\kap}\twoheadrightarrow \bT^p_{K_{r'},\kap}$, $r\geq r'$, so we do not include the $r$ in our notation.
Via the isomorphism $\bT_{K_r,\kap,V}\isoarrow \bT_{K_r^\flat,\kap^\flat,-V}$ of Lemma \ref{Hecke-isoms}(ii),
$\lambda_{\pi,V}^p=\lambda_\pi^p$ determines a homomorphism $\lambda^{p,\flat}_{\pi,V}$ of $\bT^p_{K_r^\flat,\kap^\flat,-V,R}$, which,
by \eqref{dagger=flat}, 
satisfies
\begin{equation}\label{flatflat}  
\lambda_{\pi,V}^{p,\flat} =   \lambda^p_{\pi^{\flat},-V}. 
\end{equation}

For an $S_0$-algebra $R\subset\C$, the homomorphism $\lambda_\pi^p$ extends $R$-linearly to a homomorphism of the Hecke
algebras over $R$; we use the same notation for this homomorphism.

We say that $\pi$ satisfies the {\it multiplicity one hypothesis for $\pi$} if: 
\begin{hyp}[Multiplicity one hypothesis] \label{multone-pi}
For any holomorphic cuspidal $\pi' \neq \pi$ of type $(\kap,K_r)$, $\lambda^p_{\pi'} \neq \lambda^p_{\pi}$.
\end{hyp}

This multiplicity one hypothesis for $\pi$ is expected to hold 
if $S=S(K^p)$ consists only of places that are split in $\CK/\CK^+$ (so all local $L$-packets are singletons) and if the base change
of $\pi$ to ${\GL_n}_{/\CK}$ is cuspidal (so $\pi$ is not obtained by endoscopic transfer from a non-trivial elliptic endoscopic group of $G$).
When $G$ is quasi-split multiplicity one has been established for the {\it unitary group} by Mok \cite{mok}, and the general case has been proved under certain
restrictive hypotheses by Kaletha, M\'inguez, Shin, and White \cite{gangof4}.  All this work is based in part on results that have been announced by Arthur but
that have not yet appeared.
  
We will generally assume that $\pi$ satisfies this multiplicity one hypothesis.  This is not indispensable, but it simplifies the notation.  However, one of the referees pointed out that this hypothesis may be restrictive for certain applications.  For example, a $\pi$ that is ramified at a prime $v$ that is inert in $\CK/\CK^+$ cannot in general  be distinguished from a $\pi'$ that is isomorphic to $\pi$ at all unramified places; indeed, $\pi_v$ can  belong to an $L$-packet that contains several non-isomorphic
representations.  Such $\pi$ arise naturally by automorphic induction from representations of unitary groups over cyclic extensions of $\CK^+$.   In Remark \ref{globalmult1} we sketch a construction without this hypothesis.

We fix a basis of the one-dimensional space $H^0(\grP_h,K_h;\pi_\infty\otimes_\C W_\kap)$ and a choice of $E(\pi)$-rational spherical vector in $\pi_f^S$.
Let $S_{\kap}(K_r,\C)(\pi)$ be the $\lambda^p_{\pi}$-isotypic subspace of $S_{\kap}(K_r,\C)$ for the action
of $\bT^p_{K_r,\kap}$.    
There is then an embedding 
$$
j_{\pi}:  H^0(\grP_h,K_h; \pi^{K_r}\otimes_\C W_\kap) \cong \pi^{K_r}_f \hookrightarrow S_{\kap}(K_r,\C)(\pi)
$$ 
of  $\bT^p_{K_r,\kap}$-modules.

\begin{lem}\label{mult1Q}  
Let $\pi$ be a holomorphic cuspidal automorphic representation of type $(\kap,K_r)$, and suppose $\pi$ satisfies Hypothesis \ref{multone-pi}.
\begin{itemize}
\item[\rm (i)] The injection $j_{\pi}$ defines an isomorphism
\begin{equation*}\label{facto} 
j_\pi:   \pi^{K_S}_S \otimes \pi_p^{I_r}\isoarrow S_{\kap}(K_r,\C)(\pi).
\end{equation*}  
\item[\rm(ii)] Let $\lambda$ be any extension of $\lambda_{\pi}^p$ to a character of $\Tb_{K_r,\kap,R}$.
Let $R \subset \C$ be a finite extension of $E(\pi)$ containing the values of $\lambda$, and let $S_{\kap}(K_r,R)[\lambda]$
be the localization of the $\bT_{K_r,\kap,R}$-module $S_{\kap}(K_r,R)$ at the prime ideal $\grp_{\lambda} \subset \bT_{K_r,\kap,R}$
that is the kernel of the character $\lambda$; in other words, $S_{\kap}(K_r,R)[\lambda]$ is the $\lambda$-isotypic
component of $S_{\kap}(K_r,R)$.  Then $j_{\pi}$ defines an isomorphism
$$
j_\pi: \pi_S^{K_S}\otimes \pi_p^{I_r}[\lambda]\isoarrow S_{\kap}(K_r,R)[\lambda]\otimes_R\C = S_\kap(K_r,\C)[\lambda].
$$
Here $\pi_p^{I_r}[\lambda]$ is the subspace
of $\pi_p^{I_r}$ on which each $U_{w,j,\kap}$ acts as $\lambda(U_{w,j,\kap})$.
\end{itemize}
\end{lem}

\subsubsection{The (anti-)ordinary projector and (anti-)ordinary Hecke algebra}\label{ord-hecke}
Suppose $R\subset \C$ is the localization of a finite $S_0$-algebra
at the maximal prime determined by $incl_p$ or a $p$-adic algebra in the sense that $\iota_p(R)$ is $p$-adically complete.

Recall the definition \eqref{ordinaryprojector} of the ordinary projector $e_\kappa$.
Set $\bT^\ord_{K_r,\kap,R} = e_\kap\bT_{K_r,\kap,R}$ and $\bT^\ord_{K,\kap,\psi,R}=e_\kap
\bT_{K_r,\kap,\psi,R}$. Then $\bT^\ord_{K_r,\kap,R}$ and  $\bT^\ord_{K,\kap,\psi,R}$ are just
the rings obtained by restricting the Hecke operators to the (stable) subspaces
$S_\kap^\ord(K_r;R)$ and $S_\kap^\ord(K_r,\psi;R)$. For $R$ not $p$-adic
we define the latter modules to be the respective intersections of $S_\kap(K_r;R)$ and
$S_\kap(K_r,\psi;R)$ with the ordinary spaces over the $p$-adic completion of $R$ 
(that is, the completion of $\incl_p(R)$).

Similarly,  let $U_{p,\kap}^- = \prod_{w\in\Sigma_p}\prod_{j=1}^{n} U_{w,j,\kap}^-$ and 
let $e_{\kap}^- = \varinjlim_N (U_{p,\kap}^-)^{N!}$ (as an operator, when it exists).
We call this the {\it anti-ordinary projector}, and put
$\bT^\aord_{K_r,\kap,R} = e_\kap^-\bT^d_{K_r,\kap,R}$  and $\bT^\aord_{K,\kap,\psi,R}=e_\kap^-
\bT^d_{K,\kap,\psi,R}$.

\begin{lem}\label{Hecke-ord-iso}
Suppose $R$ is as above. The isomorphisms of Lemmas \ref{Hecke-isoms}(i)-(ii)
and \ref{Hecke-char-isom} restrict to $R$-algebra isomorphisms:
\begin{itemize}
\item[\rm (i)] $\bT^\ord_{K_r,\kap,R} \isoarrow \bT^\aord_{K_r,\kap^D,R}$ and $\bT^\ord_{K_r,\kap,\psi,R} \isoarrow \bT^\aord_{K_r,\kap,\psi^{-1},R}$,
\item[\rm(ii)] $\bT^\ord_{K_r,\kap,V,R} \isoarrow \bT^\ord_{K_r^\flat,\kap^\flat-V,R}$ and 
$\bT^\ord_{K_r,\kap,\psi,V,R}\isoarrow \bT^\ord_{K_r^\flat,\kap^\flat,\psi^{-1},-V,R}$.
\end{itemize}
\end{lem}

\noindent This is immediate from the definitions.

\subsubsection{Spaces of ordinary forms and the character $\lambda_\pi$}\label{lambdapi}
Let $\pi$ be a holomorphic cuspidal automorphic representation of $G$ of type $(\kap,K_r)$.
Let 
$$
\pi_p^\ord = e_\kap\pi_p^{I_r}.
$$
This space has dimension at most one and it does not depend on $r$, in the sense that $e_\kap\pi_p^{I_r} = e_\kap\pi_p^{I_{r'}}$ for all $r'\geq r$. This is
a consequence of the following:
\begin{thm}[Hida]\label{dim1} 
For any representation $\pi_p$ of $G(\Qp)$, the ordinary eigenspace $e_\kap\pi_p^{I_r}\subset \pi_p^{I_r}$ is of dimension $\leq 1$, for any $r$.
\end{thm}
This theorem is a variant of \cite[Corollary 8.3]{H98} (we thank Hida for this reference).
The proof, an adaptation of Hida's, is given in Section  \ref{ordvec} below. 

We will say that $\pi$ is {\it ordinary} if $\pi_p^\ord\neq 0$. 
Note that $\pi_p^\ord$ is stable under the action of $I_r^0$, and so $I_r^0$ will act on $\pi_p^\ord$ (when it is non-zero) through a 
well-defined character $\psi$; we call its identification with a character of $T_H(\Zp)$ 
the {\it ordinary nebentypus} of $\pi$. 

The space
$$
\pi^{\flat,\aord}_{p,r} = e_{\kap^D}^-\pi^{\flat,I_r}_p \subset \pi_p^\flat
$$
is at most one-dimensional, and is non-zero (and so has dimension one) if and only if $\pi_p^\ord$ is non-zero. This follows from Lemma \ref{aord-vect-lem-w} below.
While it is not generally true that $\pi^{\flat,\aord}_{p,r}$ is independent of $r$, if $r'\geq r$ then Lemma \ref{aord-vect-lem-w-2} asserts that
$$
\trace_{K_r/K_{r'}} \pi_{p,r'}^{\flat,\aord} = \pi^{\flat,\aord}_{p,r}.
$$

Suppose that $\pi$ is ordinary. We let $\lambda_\pi$ be the (unique) extension of $\lambda^p_\pi$ to the Hecke character giving the action of $\bT_{K_r,\kap}$ 
on $\pi_p^\ord\otimes\pi^{p,K^p}$. For $R$ as in \ref{ord-hecke}, this character factors through $\bT^\ord_{K_r,\kap,\psi,R}$ for $\psi$ the ordinary nebentypus of $\pi$.
Let $E(\lambda_\pi)$ be the finite extension of $E(\pi)$ generated by the values of $\lambda_\pi$, and let $R(\lambda_\pi)$ be the localization of the ring
of integers of $E(\lambda_\pi)$ at the maximal ideal determined by $\incl_p$; then $\lambda_\pi$ is $R(\lambda_\pi)$-valued.
Let $\bar\lambda_\pi$ be the reduction of $\lambda_\pi$ modulo the maximal ideal of $R(\lambda_\pi)$; this can be viewed as taking values in the
residue field of $\overline\Z_{(p)}$. We let 
\begin{equation}\label{SKrkap}
\CS(K_r,\kap,\pi) = \{ \text{ordinary holomorphic $\pi'$ of type $(\kap,K_r)$ such that $\bar\lambda_{\pi'} = \bar\lambda_\pi$}\}.
\end{equation}
Starting in \ref{pcontinued} we will also write $\pi' \in \CS(K_r,\kap,\pi)$ when both $\pi$ and $\pi'$ are {\it anti-holomorphic} (and anti-ordinary); this means that
$\pi^{\prime,\flat} \in \CS(K_r,\kap,\pi^\flat)$ where $\pi^\flat$ is holomorphic of type $(\kap,K_r)$ and the notation is used  in the sense of \eqref{SKrkap}.

\begin{lem}\label{ordveclem}
Let $\pi$ be a holomorphic cuspidal automorphic representation of type $(\kap,K_r)$. Suppose $\pi$ is ordinary.
Suppose also that $\pi$ satisfies Hypothesis \ref{multone-pi}.
Let $R\subset \C$ be the localization of a finite extension of $R(\lambda_\pi)$ at the prime determined by $\incl_p$ or the $p$-adic
completion of such a ring. Let $E=R[\frac{1}{p}]$. 
\begin{itemize}
\item[\rm (i)] $S_\kap^\ord(K_r;E)[\lambda_\pi] = e_\kap S_\kap(K_r;E)[\lambda_\pi]$ and $j_\pi$ restricts to an isomorphism
$$
j_\pi: \pi_p^\ord\otimes \pi_S^{K_S} \cong \pi_S^{K_S}\isoarrow S_\kap^\ord(K_r;E)[\lambda_\pi]\otimes_E\C.
$$
\item[\rm (ii)]  Let $\grm_\pi$ be the maximal ideal of $\bT_{K_r,\kap,R}$
that is the kernel of the reduction of $\lambda_\pi$ modulo the maximal ideal of $R$. 
Let $S_\kap^\ord(K_r;R)_\pi$ be the localization of $S_\kap^\ord(K_r;R)$ at $\grm_\pi$. Then 
$$
S_\kap^\ord(K_r;R)[\pi] := S_\kap^\ord(K_r;R)_\pi \cap S_\kap^\ord(K_r;E)[\lambda_\pi]
$$
is identified by $j_\pi$ with an $R$-lattice in $\pi_p^\ord\otimes\pi_S^{K_S}\cong \pi_S^{K_S}$,
and
$S_\kap^\ord(K_r,R)_\pi$ is identified with an $R$-lattice in 
$$
\bigoplus_{\pi' \in \CS(K_s,\kap,\pi)} \pi_p^{\prime,\ord}\otimes (\pi'_{S})^{K_{S}}.
$$
This last identification is via $\oplus_{\pi'} \lambda_{\pi'}$.
\end{itemize}
\end{lem}
 
 We also need a dual picture.
Let 
$$
\hat S_\kap(K_r; R) = \Hom_R(S_\kap(K_r;R),R) \ \ \text{and} \ \ \hat S_\kap^\ord(K_r;R) = \Hom_R(S_\kap^\ord(K_r;R),R).
$$
These are $\bT_{K_r,\kap,R}$-modules through the Hecke action on $S_\kap(K_r;R)$, so $\hat S_\kap^\ord(K_r,R)$
is a $\bT^\ord_{K_r,\kap,R}$-module. The normalized Serre duality of Section \ref{normal pairing} identifies $\hat S_\kap(K_r:R)$ with
$$
H^{d}_{\kap^D}(K_r,R) = \{ \vphi\in H^d({}_{K_r}Sh(V),\omega_{\kap^D}) \ : \ \vpair{S_\kap(K_r;R),\vphi}_{\kap,K_r} \subseteq R\}.\footnote{The 
bundle $\omega_{\kap^D}$ is the subcanonical extension of its restriction to the open Shimura variety; thus the space $H^d({}_{K_r}Sh(V),\omega_{\kap^D})$ 
is represented by cusp forms and the pairing above is well-defined.}
$$
Let $S_\kap^{\ord,\perp}(K_r;R) \subset H^{d}_{\kap^D}(K_r,R)$ denote the annihilator of $S_\kap^\ord(K_r;R)$ with respect to
the pairing $\pair_{\kap,K_r}$. 
Then (the normalized) Serre duality identifies $\hat S_\kap^\ord(K_r;R)$ with
$$
H^{d,\ord}_{\kap^D}(K_r,R) =
 \{ \vphi\in H^d_!({}_{K_r}Sh(V),\omega_{\kap^D})/S_\kap^{\ord,\perp}(K_r;R) \ : \ \vpair{S_\kap^\ord(K_r;R),\vphi}_{\kap,K_r} \subseteq R\}.
$$
Each of these is a $\bT_{K_r,\kap,R}$-module through its action on $S_\kap(K_r;R)$ or, equivalently, the isomorphism of
Lemma \ref{Hecke-isoms}(i), so $H^{d,\ord}_{\kap^D}(K_r;R)$ is a $\bT^\ord_{K^r,\kap,R}$-module.  

\begin{lem}\label{ortho}
The natural map $H^d_{\kap^D}(K_r;R)\rightarrow H^{d,\ord}_{\kap^D}(K_r;R)$,
which is just restriction to $S_\kap^\ord(K_r;R)$, induces an isomorphism
\begin{equation}\label{Hecke-duals}
e_{\kap^D}^- H^d_{\kap^D}(K_r;R) 
\isoarrow H^{d,\ord}_{\kap^D}(K_r;R).
\end{equation}
\end{lem}
\begin{proof}  This is an immediate consequence of Lemma \ref{ord-dual-lem-w}, (iii).
\end{proof}

Let $\pi$ be a holomorphic cuspidal automorphic representation of $G$ of type $(\kap,K_r)$. 
Then $\pi^\flat$ is anti-holomorphic of type $(\kap^\flat, K_r)$.  The choice of a basis of 
the one-dimensional space
$H^d(\grP_h,K_h;\pi^\flat_\infty\otimes_\C W_{\kap^D})$ determines an injection
$$
j_{\pi^\flat}^\vee: H^d(\grP_h,K_h;\pi^{\flat,K_r}\otimes_\C W_{\kap^D}) \cong \pi^{\flat,K_r} \hookrightarrow H^d_{\kap^D}(K_r;\C) = H^d({}_{K_r}Sh(V),\omega_{\kap^D}).
$$

\begin{lem}\label{aordveclem}  
Let $\pi$, $R$, and $E$ be as in Lemma \ref{ordveclem}. 
Let $H^{d,\ord}_{\kap^D}(K_r,R)_{\pi}$ be the localization of $H^{d,\ord}_{\kap^D,V}(K_r,R)$ at 
$\grm_{\pi^{\flat}}$, and let  
$$
H^{d,\ord}_{\kap^D}(K_r,R)[\pi]  = H^{d,\ord}_{\kap^D}(K_r;R)_{\pi}\cap H^{d,\ord}_{\kap^D}(K_r;E)[\lambda_\pi]
$$
where the notation $[\lambda_\pi]$ again denotes the 
$\lambda_\pi$-isotypic component.
\begin{itemize}
\item[\rm (i)] The inclusion $j_{\pi^\flat}^\vee$ restricts to an isomorphism
$$
j_{\pi^\flat}^\vee: \pi_{p,r}^{\flat,\aord}\otimes \pi_S^{\flat,K_S}\cong \pi_S^{K_S}  \isoarrow  H^{d,\ord}_{\kap^D}(K_r,E)[\lambda_\pi]\otimes_E\C.
$$
\item[\rm (ii)] The map $j_{\pi^\flat}^\vee$ identifies $H^{d,\ord}(K_r;R)[\pi]$ with an $R$-lattice in 
$\pi_{p,r}^{\flat,\aord}\otimes \pi_S^{\flat,K_S}$,  
and $H^{d,\ord}(K_r;R)_\pi$ is identified with an $R$-lattice in 
$$
\oplus_{\pi'\in\CS(K_r,\kap,\pi)} \pi_{p,r}^{\prime,\flat,\ord}\otimes\pi_S^{\prime,\flat,K_S}.
$$
This last identification is by $\oplus j_{\pi^{',\flat}}^\vee$.
\item[\rm (iii)] Normalized Serre duality induces perfect $\bT^\ord_{K_r,\kap,R}$-equivariant pairings (with respect to the isomorphisms of Lemma \ref{Hecke-ord-iso})
$$
S_\kap^\ord(K_r;R)[\pi] \otimes _R H^{d,\ord}_{\kap^D}(K_r;R)[\pi] \rightarrow R  \ \ \text{and} \ \ 
S_\kap^\ord(K_r;R)_\pi\otimes _R H^{d,\ord}_{\kap^D}(K_r;R)_\pi \rightarrow R 
$$
\end{itemize}
\end{lem}

For any $r \geq 0$, we say $\pi$ is {\it ordinary of type $(\kap,K_r)$} if 
$\pi$ is holomorphic of type $(\kap,K_r)$ and if the image
of $j_{\pi}$ has non-trivial intersection with  $S_{\kap}^{\ord}(K_r,R)$ for $R$ as in Lemma \ref{ordveclem} (this is independent of $R$).   In that case, $\lambda_{\pi}$, defined as above, takes values
in a $p$-adic integer ring, say $\CO_{\pi}$, with residue field $k(\pi)$, and we let $\bar{\lambda}_{\pi}:  \Tb_{K_r,\kap} \rar k(\pi)$
denote the reduction of $\lambda_{\pi}$ modulo the maximal ideal of $\CO_{\pi}$.
\subsubsection{Change of level}    
For fixed $\kap$
we consider the inclusion 
\begin{equation}\label{raise} S^{\ord}_{\kap,V}(K_r,R) \rar S^{\ord}_{\kap,V}(K_{r'},R) \end{equation}
with $r' \geq r$ and the dual map
\begin{equation}\label{raisedual} \hat{S}^{\ord}_{\kap,V}(K_{r'},R) \rar \hat{S}^{\ord}_{\kap,V}(K_r,R) \end{equation}

\begin{lem}\label{levelraise}  Let $R$ be either a local $\Z_{(p)}[\lambda_{\pi}]$-algebra or a finite flat $\Zp[\lambda_{\pi}]$-algebra.
Then the image of the map \eqref{raise} is an $R$-direct factor of $S^{\ord}_{\kap,V}(K_{r'},R)$, identified with the submodule of $I_r/I_{r'}$-invariants
of the latter.  Moreover, the morphism
\eqref{raisedual} is surjective.
\end{lem}

\begin{proof}  The first assertion is obvious; the second is an immediate consequence of the first.
\end{proof}

\subsection{Normalized periods}\label{normper}

Fix the group $G=G_1$ as above. 
We assume $\pi$ to be an anti-holomorphic representation of $G$ of type $(\kappa,K)$ and anti-ordinary at $p$
with character $\psi$.

\begin{lem}\label{period1}  
Let $R$ be as in Lemma \ref{ordveclem}.
The images 
\begin{align*}\\
L[\pi] &= \frac{1}{\Vol(I_{V,r}\cap I_{-V,r})}\la H^{d,\ord}_{\kap,V}(K,\psi,R)[\pi], H^{d,\ord}_{\kap^\flat,-V}(K_r^\flat,\psi^{-1},R)[\pi^\flat]\ra^\Serre_\kap \\
L_{\pi} &=\frac{1}{\Vol(I_{V,r}\cap I_{-V,r})}\la H^{d,\ord}_{\kap,V}(K,\psi,R)[\pi], H^{d,\ord}_{\kap^\flat,-V}(K_r^\flat,\psi^{-1},R)_{\pi^\flat}\ra^\Serre_\kap \end{align*}
are rank one
$R$-submodules of $\C$, generated by positive real numbers $Q[\pi]$ and $Q_{\pi}$, respectively.
\end{lem}

Here we have used the identification \eqref{Finfty-coh} to evaluate
the pairings.

\begin{proof} 
Recall that $\pi_f$ and $\pi^\flat_f\cong \bar\pi_f$ are defined over the finite extension $E(\pi)\subset \C$ of $\Q$. Then Schur's Lemma 
together with the irreducibility and admissibility of the representations
$\pi_f$ and $\bar\pi_f$ implies that the pairing $\pairS$ is a 
$\C^\times$-multiple of a pairing taking values in $E(\pi)$
on given $E(\pi)$-structures, and the hermitian nature of the
Petersson pairing (and its relation \ref{pet} with Serre duality)
shows that this multiple is a positive real number. This is 
essentially explained
in \cite{HarrisIMRN}.
Since $R$ is a discrete valuation ring and 
$H^{d,\ord}_{\kap,V}(K,\psi,R)$ and $H^{d,\ord}_{\kap^\flat,-V}(K_r^\flat,\psi^{-1},R)$
are finite $R$-modules, the result follows immediately from this.
\end{proof}

The numbers $Q[\pi]$ and $Q_{\pi}$ are well-defined up to multiples by $R^{\times}$; they are respectively unnormalized and normalized periods for $\pi$.   We can also write 
$Q[\pi]_V$ and $Q_{\pi,V}$ 
to emphasize the dependence on $G=G_1$.
Note that $Q[\pi^\flat]_{-V} = Q[\pi]_{V}$ and $Q_{\pi^\flat,-V} = Q_{\pi,V}$.
Furthermore, these periods are independent of $r\gg 0$. This is essentially an easy consequence
of the properties of anti-ordinary forms (see also Lemma \ref{local-pairing-lem})

Let
$$H^{d,\ord}_{\kap,\psi,V}(K_r,\R)[\pi]^{\perp} \subset H^{d,\ord}_{\kap^\flat,-V}(K_r^\flat,\psi^{-1},R)_{\pi^\flat}$$
be the orthogonal complement to $H^{d,\ord}_{\kap,V}(K_r,\psi,R)[\pi]$ with respect to 
$\frac{1}{\Vol(I_{V,r}\cap I_{-V,r})}\pairS$.  This is the intersection of 
$H^{d,\ord}_{\kap^\flat,-V}(K_r^\flat,\psi^{-1},R)_{\pi^\flat}$ with $\oplus_{\pi' \neq \pi^\flat}  H^{d,\ord}_{\kap^\flat,-V}(K_r^\flat,\psi^{-1},R[\frac{1}{p}])[\pi']$.

\begin{defi}
Define the {\it congruence ideal} $C(\pi)=C_V(\pi) \subset R$ to be the annihilator of 
$$H^{d,\ord}_{\kap^\flat,-V}(K_r^\flat,\psi^{-1},R)_{\pi}/H^{d,\ord}_{\kap^\flat,-V}(K^\flat_r,\psi^{-1},R)[\pi^\flat] + H^{d,\ord}_{\kap,V}(K_r,\psi,R)[\pi]^{\perp}.$$
\end{defi}

\begin{lem}\label{period2}  Let $c(\pi) \in R$ be such that $c(\pi)Q_{\pi} = Q[\pi]$. Then $c(\pi)$ is a generator of 
$C(\pi)$.   
\end{lem}

\begin{proof}  This is an elementary consequence of the definitions.
\end{proof}

More generally, the congruence ideal $C(\pi,M)$ can be defined for any $\bT_{K^\flat,\kap^\flat,R}^{\flat,\ord}$-module $M$ as
the annihilator of $M_{\pi^\flat}/(M[\pi^\flat] + M[\pi]^{\perp})$, where the notation has the same meaning as above.
In particular, we can define $C(\pi,\TT)$ to be the congruence ideal for $\bT_{K,\kap,R}^{\ord}\isoarrow \bT_{K^\flat,\kap^\flat,R}^{\flat,\ord}$ considered
as a free module over itself.

\begin{rmk}   The congruence ideal $C(\pi)$ has a local component, due to possible congruences between
the representation $\pi_p^{\flat,\ord}\otimes \pi_S^{\flat,K_S}$ and the $\pi_p^{\flat,\ord}\otimes(\pi'_S)^{\flat,K_S}$ for 
$\pi'$ such that $\bar{\lambda}_{\pi} = \bar{\lambda}_{\pi'}$.  Here
if $S$ has the property that, for every rational prime $q$, either all the primes of $\CK^+$ dividing $q$ split in $\CK$ or none of them does, we can view the latter
as representations of the (integral) Hecke algebra of $K_S$-biinvariant functions on $GU(V)(\A_{f,S})$.
The separation of global and local
components of $C(\pi)$ will need to be understood for applications, but it is not addressed here.
\end{rmk}

\subsubsection{Normalized periods twisted by Hecke characters}\label{periodtwist} 

Let $\chi$ be a Hecke character of type $A_0$.  The twisted pairings \eqref{chitwist} and \eqref{pair-compare_chi} give rise to period invariants that account for the Hecke character twist.  We reformulate Lemma \ref{period1} in this framework.  If the (anti-holomorphic) representation $\pi^\flat$ contributes to the space
denoted $H^{d,\ord}_{\kap^\flat,-V}$, we let $\kap^\flat\star\chi$ denote the coherent cohomology space to which $\pi^\flat\otimes \chi^{-1}\circ\det$ contributes.  Then we can define the $R$-module
$$H^{d,\ord}_{\kap^\flat,-V}(K_r^\flat,\psi^{-1},R)[\pi^\flat]_\chi := H^{d,\ord}_{\kap^\flat\star\chi,-V}(K_r^\flat,\psi^{-1},R)[\pi^\flat\otimes \chi^{-1}\circ\det].$$

Note:  This is not a simple algebraic twist of the original $H^{d,\ord}_{\kap^\flat,-V}(K_r^\flat,\psi^{-1},R)[\pi^\flat]$!  Even the rational structure is modified by a CM period corresponding to $\chi$.  See \S 2.9 of \cite{harriscrelle} for examples of this when $F^+ = \Q$. 

The extension of Lemma \ref{period1} is then

\begin{lem}\label{period1chi}  
Let $R$ be as in Lemma \ref{ordveclem}.
The images 
\begin{align*}\\
L[\pi] &= \frac{1}{\Vol(I_{V,r}\cap I_{-V,r})}\la H^{d,\ord}_{\kap,V}(K,\psi,R)[\pi], H^{d,\ord}_{\kap^\flat,-V}(K_r^\flat,\psi^{-1},R)[\pi^\flat]\ra^\Serre_\kap \\
L_{\pi} &=\frac{1}{\Vol(I_{V,r}\cap I_{-V,r})}\la H^{d,\ord}_{\kap,V}(K,\psi,R)[\pi], H^{d,\ord}_{\kap^\flat,-V}(K_r^\flat,\psi^{-1},R)_{\pi^\flat}\ra^\Serre_\kap \end{align*}
are rank one
$R$-submodules of $\C$, generated by positive real numbers $Q[\pi,\chi]$ and $Q_{\pi,\chi}$, respectively.
\end{lem}

We may define $C(\pi,\chi)$ by analogy with $C(\pi)$.  Then the relation between the two period invariants is again determined by a congruence number:

\begin{lem}\label{period2chi}  Let $c(\pi,\chi) \in R$ be such that $c(\pi,\chi)Q_{\pi,\chi} = Q[\pi,\chi]$. Then $c(\pi,\chi)$ is a generator of 
$C(\pi,\chi)$.   
\end{lem}

However, since the family of $p$-adic Hecke characters is smooth, it is easy to see that $C(\pi,\chi)$ and $c(\pi,\chi)$ do not depend on $\chi$.

\subsubsection{The Gorenstein hypothesis and the congruence module}
In what follows, $R$ is a sufficiently large finite flat $p$-adic integer ring.

\begin{defi}\label{Gorfin}  Write $\TT = \TT_\pi := (\bT^{\ord}_{K,\kap,R})_{\pi}$. The $\TT$-module $S^{\ord}_{\kap}(K,R)_{\pi}$ is said to satisfy
the {\it Gorenstein hypothesis} if the following conditions hold.
\begin{itemize}
\item  There exists an isomorphism
$$G = G_{K,\kap,R}:  \TT \isomto \Hom_R(\TT,R)$$
 as $\TT$-modules.
\item  $S^{\ord}_{\kap}(K,R)_{\pi}$ is free over $\TT$.
\end{itemize}
The $\bT^{\ord}_{K,\kap,R}$-module $S^{\ord}_{\kap}(K,R)$ is said to satisfy
the Gorenstein hypothesis if all its localizations at maximal ideals
of $\bT_{K,\kap,R}$ satisfy the two conditions above.
\end{defi}

Note that we are using the notation $\TT$ rather than $\bT$ for the localization at a maximal ideal of the ordinary Hecke algebra.  The following is then obvious.

\begin{lem}\label{congduality} Assume $S^{\ord}_{\kap}(K,R)_{\pi}$ satisfies the Gorenstein hypothesis.  Then
we have
$$C_V(\pi) = C(\pi,\TT) = C_{-V}(\pi^\flat).$$
\end{lem}

The congruence ideal for $\pi$ can be calculated as follows.  We assume the multiplicity one
hypothesis, so that the localization of $\TT$ at the kernel of $\lambda_{\pi}$ is of rank $1$
over $R$.   Let $e_1,\dots, e_n$ be
an $R$-basis for $\TT$, and let $e_1^*,\dots, e_n^*$ be the dual basis of
$\Hom_R(\TT,R)$.   Write $E = \Frac(R)$, and write
$$\TT_E = \TT\otimes_R E = \oplus E_{i},$$
indexed by the maximal ideals $\lambda_{\pi_i}$ of $\TT$, with $\pi = \pi_1$.  We assume
$R$ is sufficiently large that $E_1 = E$.  Choose $d_1,\dots, d_n \in \TT$ that form a basis of $\TT_E$,
with $d_1$ an $R$-generator of $\TT\cap E_1$ and $d_2,\dots, d_n$ an $R$-basis  of $\TT\cap \oplus_{i > 1} E_{i}$.
Write $e_i = \sum c_{ij}d_j$, with $c_{ij} \in E$.  Then 
\begin{equation}\label{cong1}  C(\pi,\TT) = \sup_{c_{i1} \neq 0} ~ -v(c_{i1}) \end{equation}
where $v$ is the valuation on $R$.  

The following lemma is then clear:

\begin{lem}\label{cflat2}  The second isomorphism of Lemma \ref{Hecke-char-isom} takes $C(\pi)$ isomorphically to $C(\pi^{\flat})$.
\end{lem}

We omit the statement of the analogous assertion for the ideal $C(\pi,\chi)$ which, as noted above, does not actually depend on $\chi$.

\begin{rmk}\label{ambiguity}  The normalized period $Q_{\pi}$ and the generator $c(\pi)$ of the congruence ideal are well defined up to units in
$R$.  However, this ambiguity is unsatisfactory; one expects there is a natural choice of global function $c$ in $\TT$ which is not a zero divisor and 
whose value at the classical point $\pi$ generates $C(\pi)$.   This would allow a uniform choice of periods $Q_{\pi,\chi}$.   

Let $\mathcal{G}_n$ be the algebraic group introduced in \cite{CHT} as the target of the compatible family of $\ell$-adic representations attached to $\pi$; it is the semidirect product of $GL(n) \times GL(1)$ with the Galois group of $\CK/\CK^+$.   It is natural to expect that $c$ can be taken to be a $p$-adic $L$-function attached to the adjoint representation on the Lie algebra of $\mathcal{G}_n$.   The corresponding complex $L$-function has a single pair of critical values, interchanged by the functional equation, so the hypothetical $p$-adic $L$-function would be an element of $\TT$, without any additional variation for twists by characters.
\end{rmk}

\section{Families of ordinary $p$-adic modular forms and duality}\label{families}

\subsection{Big Hecke algebras}\label{BIG}
We return to the notation of Sections \ref{heck} and \ref{ord-hecke}.
In particular $G=G_1$, and we let $T_H = T_{H_1}(\Zp)$ be the torus introduced in Section \ref{levelp}.

Let $R$ be a $p$-adic ring.  The inclusion
$S^{\ord}_{\kap}(K_r,R) \subset S^{\ord}_{\kap}(K_{r'},R)$, $r'\geq r$, defines by restriction a map of ordinary 
Hecke algebras $\bT^{\ord}_{K_{r'},\kap,R} \rar \bT^{\ord}_{K_r,\kap,R}$.  Let 
$$
\bT^{\ord}_{K^p,\kap,R} = \varprojlim_r \bT^{\ord}_{K^p_r, \kap,R}.
$$
The following theorem is due to Hida:
\begin{thm}\label{bigHecke}  For any pair of characters $\kap_1$, $\kap_2$ of $T_H$, there is a canonical isomorphism
$$\bT^{\ord}_{K^p,\kap_1,R} \isoarrow \bT^{\ord}_{K^p,\kap_2,R}.$$
\end{thm}
Thus we drop the superscript `${\ord}$' and write $\bT_{K^p,R}$ to designate any $\bT^{\ord}_{K^p,\kap,R}$ without fear of ambiguity.  We will even write $\bT$ for $\bT_{K^p,R}$ when there is no danger of ambiguity.

\begin{rmk}  In the application to unitary groups this  theorem and 
Theorem \ref{controlkap0} are special cases of 
\cite[Theorem 7.1]{H02} and the results of \cite[Chapter 8]{Hida}.  
Hida's Theorems \ref{bigHecke} and \ref{controlkap} are proved assuming the conditions (G1)-(G3) mentioned in connection with \eqref{classicalitythm1}.
\end{rmk}

As noted in \eqref{integraltop2}, the trace map 
$\trace_{K_r/K_{r'}}$ maps $H^1_{\kap^D}(K_{r'};R)$ to 
$H^1_{\kap^D}(K_r;R)$ for all $r'\geq r>0$. It follows easily from
the definition of the anti-ordinary projectors $e_{\kap^D}^-$
that this trace map also maps $e_{\kap^D}^-H^1_{\kap^D}(K_{r'};R)$
to $e_{\kap^D}^-H^1_{\kap^D}(K_{r'};R)$, yielding a natural homomorphism $\bT^{\aord}_{K_{r'},\kap^D,R} \rar \bT^{\aord}_{K_r,\kap^D,R}$ that is compatible with the isomorphisms of Lemma \ref{Hecke-ord-iso}(i) and the maps $\bT^{\ord}_{K_{r'},\kap,R} \rar \bT^{\ord}_{K_r,\kap,R}$. In particular, putting
$$
\bT^{\aord}_{K^p,\kap^D,R} = \varprojlim_r \bT^{\aord}_{K^p_r, \kap^D,R},
$$
the isomorphisms of Lemma \ref{Hecke-ord-iso}(i) induce an isomorphism
\begin{equation}\label{big-hecke-iso1}
\bT^{\ord}_{K^p,\kap,R}\isoarrow \bT^\aord_{K^p,\kap^D,R}.
\end{equation}
Note that it then follows from Theorem \ref{bigHecke} that $\bT^\aord_{K^p,\kap^D,R}$ is
also independent of the weight $\kap^D$, and we write
$\bT^\aord_{K^p,r}$ for $\bT^\aord_{K^p,\kap^D,R}$.

We similarly define
$$
\bT^\flat_{K^p,R}:=\bT^{\ord}_{K^{p,\flat},\kap^\flat,R} = \varprojlim_r \bT^{\ord}_{K^{p,\flat}_r, \kap^\flat,R} \ \
\text{and}  \ \ 
\bT^{\flat,\ord}_{K^p,R} = \bT^{\aord}_{K^{p,\flat},\kap^{\flat,D},R} = \varprojlim_r \bT^{\aord}_{K^{p,\flat}_r, \kap^{\flat,D},R}.
$$
We the have isomorphisms
\begin{equation}\label{big-hecke-iso2}
\bT_{K^p,R} \isoarrow \bT^{\flat}_{K^{p,\flat},R}
\isoarrow \bT^{\flat,\aord}_{K^{p,\flat},R},
\end{equation}
where the first isomorphism is induced by those of Lemma \ref{Hecke-ord-iso}(ii), and the second isomorphism is induced by the corresponding version of the isomorphisms in part (i) of the same lemma.

Via the isomorphisms in \eqref{big-hecke-iso1} and \ref{big-hecke-iso2} we will view 
$H^{d,\ord}_{\kap^D}(K_r;R)$, $S_{\kap^\flat}(K_r^\flat;R)$,
and $H^{d,\ord}_{\kap^{\flat,D}}(K^\flat_r;R)$ as $\bT_{K^p,R}$-modules.

\subsection{The control theorem}\label{CONTROL}
Let $\Lambda_R = R[\![T]\!]$, the completed group algebra of the integer points $T = T_H(\Z_p)$of the torus $T_H$. Then $\bT_{K^p,R}$
is a $\Lambda_R$-module such that $t\in T$ is
identified with the the Hecke operator $K_r t K_r$.   For any tame character $\epsilon$ of $T$ we let $\Lambda_{R,\epsilon}$ denote the localization
of $\Lambda_R$ at the maximal ideal defined by $\epsilon$.   Let $\Lambda^o_R \subset \Lambda_R$ be the completed group algebra of the
maximal pro-$p$-subgroup of $T$, and define $\Lambda^o_{R,\epsilon}$ analogously.
As noted above (see also \eqref{controlthm2}) the following Theorem is a special case of \cite[Theorem 7.1]{H02}.

\begin{thm}\label{controlkap0} \hfill
\begin{itemize}
\item[(i)]  For each tame character $\epsilon$, the Hecke algebra $\bT_{K^p,R,\epsilon}$ is a finite, free
$\Lambda^o_{R,\epsilon}$-algebra.  
\item[(ii)] (Control theorem)  Let $I_{\kap} \subset \Lambda^o_R$ be the kernel of the map $\Lambda^o_R \rar R\subset \C_p$ defined
by the character $\kap$.  Suppose $\kap$ is sufficiently regular.  Then the natural homomorphism
$$
\bT_{K^p,R}\otimes_{\Lambda^o_R} \Lambda^o_R/I_{\kap} \rar \bT_{K^p_r,\kap,R}
$$
is an isomorphism.
\end{itemize}
\end{thm}

\subsection{The Gorenstein and multiplicity one hypotheses}\label{GOR-MULT1}

Fix a cuspidal holomorphic\footnote{Be warned, however, that in the main theorem $\pi$ denotes an {\it anti-holomorphic} representation.} 
automorphic representation $\pi$ of $G=G_1$ which is ordinary of
type $(\kap,K)$ as in Section \ref{heck}.  
We let $R = \CO_{\pi}$, and let $\Lambda_\pi  = 
\Lambda_{\CO_\pi}:= \CO_\pi[\![T]\!]$, $\Lambda^o_\pi = \Lambda^o_{\CO_\pi}$.
The homomorphisms 
$\lambda_{\pi}: \bT_{K,\kap,\CO_\pi} \rar \CO_\pi$ and $\bar{\lambda}_{\pi}:  \bT_{K,\kap,\CO_\pi}  \rar k(\pi)$ induce
homomorphisms $L_{\pi}:  \bT_{K^p,\CO_\pi} \rar \CO_\pi$ and $\bar{L}_{\pi}:  \bT_{K^p,\CO_\pi}  \rar k(\pi)$ of $\Lambda_{\pi}$-algebras.
Let $\grm_{\pi} = \ker \bar{L}_{\pi}$, and let 
\begin{align*}
\TT = \TT_{\pi} := \bT_{K^p,\CO_\pi,\grm_{\pi}}
\end{align*}
 denote the localization, with notation as in Section  \ref{normper}.   {\it Here and below we use $\TT$ to designate an ordinary Hecke algebra (at fixed level $I_r$ or not) localized at $\pi$, and $\bT$ to designate a non-localized ordinary Hecke algebra.}  
 The intersection $\grm_{\pi} \cap \Lambda_{\pi}$ is the maximal ideal defined by some tame character
of $T$.  

The following theorem is immediate from Theorem \ref{controlkap0}.

\begin{thm}\label{controlkap}\hfill
\begin{itemize}
\item[(i)]  The Hecke algebra $\TT_{\pi}$ is a finite, free
$\Lambda^o_{\pi}$-algebra.  
\item[(ii)] (Control theorem)  Let $I_{\kap} \subset \Lambda^o_{\pi}$ be the kernel of the map $\Lambda^o_{\pi} \rar \CO_\pi\subset \C_p$ defined by the character $\kap$.  Suppose $\kap$ is sufficiently regular.  Then the natural map
$$\TT_{\pi} \otimes_{\Lambda^o_\pi} \Lambda^o_\pi/I_{\kap} \rar \TT_{K^p_r,\kap,\CO_\pi}$$
is an isomorphism.
\end{itemize}
\end{thm}

Here and in what follows, for any $r$ and $K^p$ we will let $\TT_{\pi}$ act on 
$\Hom_{\CO_\pi}(S^{\ord}_{\kap}(K^p_r,\CO_\pi),\CO_\pi)_{\grm_{\pi}}$ by the natural action.
We consider the following hypotheses:

\begin{hyp}\label{gor}{(Gorenstein Hypothesis)}  Let $\hat{\TT}_{\pi} = \Hom_{\Lambda^o_{\pi}}(\TT_{\pi},\Lambda^o_{\pi})$.
Then 
\begin{itemize}
\item[(i)] There exists an isomorphism
$$G = G_\pi: \hat{\TT}_{\pi} \isomto \TT_\pi$$
of free rank-one $\TT_{\pi}$-modules. 
\item[(ii)] $\Hom_{\CO_\pi}(\varinjlim_r S^{\ord}_{\kap}(K^p_r,\CO_\pi),\CO_\pi)_{\grm_{\pi}}$ is a free
$\TT_{\pi}$-module.
\end{itemize}
\end{hyp}

This is, of course, a variant of the condition \ref{Gorfin} of the previous section.  The isomorphism $G_\pi$ induces compatible isomorphisms at finite level $K_r$, as in Definition \ref{Gorfin}, for $r \geq 1$; we denote these isomorphisms $G_r$ in the following sections.   Recall the set \eqref{SKrkap}, and let
$$\CS(K^p,\pi) = \cup_{r \geq 1}\cup_{\kap^1} \CS(K_r,\kap^1,\pi);$$
these are automorphic representations $\pi'$ of varying level at $p$ and weight but with $\bar\lambda_{\pi'} = \bar\lambda_\pi$.

\begin{hyp}\label{multone}{(Global Multiplicity One)}  Let $\pi' \in \CS(K^p,\pi).$  Then the representation $\pi'$ occurs with multiplicity one in the cuspidal spectrum of $G$.
\end{hyp}

This is the extension of Hypothesis \ref{multone-pi} to all $\pi' \in \CS(K^p,\pi)$, but it is weaker because it is refers to the global automorphic representation
and not only to the corresponding character of the global unramified Hecke character.  As noted above, it is implied by \cite{mok,gangof4}.  By Theorem \ref{bigHecke} we may identify $\TT_\pi$ with  $\TT_{\pi'}$ for   $\pi' \in \CS(K^p,\pi)$.

 \subsubsection{Local representation theory}\label{locrep}
 
 Henceforth, we abuse notation and write $\CO$ for $\CO_\pi$.  (The ring of integers of $\CK$ does not appear in the context in which we do this; so we will only be using $\CO$ for $\CO_\pi$ here.)  We usually include the subscript `$\pi$' for clarity.    
  
By hypothesis, $\Hom_{\CO}(\varinjlim_r S_\kap^\ord(K_r,\O)_{\grm_\pi},\O)$ is a free $\TT_\pi$-module of finite rank. We fix a finite, free $\O$-module $\hat{I}_\pi$
together with a $\TT_\pi$-isomorphism 
$$
\TT_\pi\otimes_\CO \hat{I}_\pi \isoarrow \Hom_{\CO}(\varinjlim_r S_\kap^\ord(K_r,\O)_{\grm_\pi},\O) = \varprojlim_r H^{d,\ord}_\kap(K_r,\CO)_\pi.
$$
By part (ii) of Theorem \ref{controlkap}, tensoring over $\Lambda_\pi$ with $\Lambda_\pi/I_\kap$ gives an isomorphism
$$
\TT_{K_r^p,\kap,\CO}\otimes_\CO \hat{I}_\pi \isoarrow \Hom_{\CO}( S_\kap^\ord(K_r,\O)_{\grm_\pi},\O) = H^{d,\ord}_{\kap^D}(K_r,\CO)_\pi.
$$
Restricting to the $\lambda_\pi$-isotypical parts of both sides yields isomorphisms
$$
\hat{I}_\pi \isoarrow \bT_{K_r^p,\kap,\O}[\lambda_\pi]\otimes_\O \hat{I}_\pi = (\bT_{K_r^p,\kap,\CO,\pi}\otimes_\CO \hat{I}_\pi)[\lambda_\pi] 
\isoarrow H^{d,\ord}_{\kap^D}(K_r,\O)[\lambda_\pi],
$$
and the map $j_{\pi^\flat}^\vee$ yields an injection
$$
H^{d,\ord}_{\kap^D}(K_r,\O)[\lambda_\pi] \stackrel{(j_{\pi^\flat}^\vee)^{-1}}{\hookrightarrow} \pi_{p,r}^{\flat,\aord}\otimes\pi_S^{\flat,K_S} \cong \pi_S^{\flat,K_S},
$$
where the last isomorphism comes from fixing a basis element $f_{\pi_p^{\flat},r}^{\aord}$ of the $1$-dimensional $\Frac(\O)$-space $\pi_p^{\flat,\aord}$.
In particular, $\hat{I}_\pi$ is identified with a free $\O$-lattice in $\pi_S^{\flat,K_S}$.
(The anti-ordinary subspace $\pi_p^{\flat,\aord}\subset \pi_p^{\flat, I_r}$ is the tensor product over $w\divides p$ of the local anti-ordinary subspaces $\pi_w^{\aord,I_{w,r}}$, which will be defined in Lemma \ref{aord-vect-lem-w}.)

Tensoring with $\TT_{K_r^p,\kap,\CO}/\ker(\lambda_\pi)$ yields isomorphisms
$$
\hat{I}_\pi \isoarrow \TT_{K_r^p,\kap,\CO}/\ker(\lambda_\pi) \otimes_\CO \hat{I}_\pi \isoarrow \Hom_{\CO}( S_\kap^\ord(K_r,\O)[\lambda_\pi],\O).
$$
Letting 
$$
I_{\pi} = \Hom(\hat{I}_{\pi},\CO),
$$
we then get an $\O$-module injection
$$
I_\pi \isoarrow  S_\kap^\ord(K_r,\O)[\lambda_\pi] \hookrightarrow \pi_p^\ord\otimes \pi_S^{K_S} \cong \pi_S^{K_S},
$$
where the last isomorphism comes from fixing a basis $f_{\pi_p}^\ord$ of the one-dimensional $\Frac(\O)$-space $\pi_p^\ord$.

In view of the comment following Hypothesis \ref{multone},
the following consequence of Hypotheses \ref{gor} and \ref{multone} follows from the preceding discussion and Theorem \ref{bigHecke}.    

\begin{prop}\label{badp}{(Minimality Hypothesis)}  For every pair  $(\kappa^1,r^1)$, there is an isomorphism of $\TT_{K^p_{r^1},\kap^1,\CO,\pi}$-modules
$$\TT_{K^p_{r^1},\kap^1,\CO,\pi}\otimes \hat{I}_{\pi} \isoarrow \Hom(S^{\ord}_{\kap^1}(K^p_{r^1},\CO)_{\grm_{\pi}},\CO).$$
such that the following diagrams commute when $r^1 > r$:
$$\begin{CD}
\TT_{K^p_{r^1},\kap^1,\CO,\pi}\otimes \hat{I}_{\pi} @>\sim >>\Hom_\CO(S^{\ord}_{\kap^1}(K^p_{r^1},\CO)_{\grm_{\pi}},\CO) \\
@VVV @VVV\\
\TT_{K^p_{r},\kap^1,\CO,\pi}\otimes \hat{I}_{\pi} @>\sim >>\Hom_\CO(S^{\ord}_{\kap^1}(K^p_{r},\CO)_{\grm_{\pi}},\CO).
\end{CD}$$
Moreover, the specialization of this isomorphism at the $\CO$-valued point $\lambda_{\pi}$ (tensoring both sides over $\TT_{K^p_{r^1},\kap^1,\CO,\pi}$ with $\TT_{K^p_{r^1},\kap^1,\CO,\pi}/\ker(\lambda_{\pi}) \simeq \CO$)
gives rise to 
a
commutative diagram
$$\begin{CD}\TT_{K^p_{r^1},\kap^1,\CO,\pi}/\ker(\lambda_{\pi})\otimes_{\CO} \hat{I}_{\pi} @>\sim >> \Hom_\CO(S^{\ord}_{\kap^1}(K^p_{r^1},\CO)_{\grm_{\pi}},\CO)
\otimes_{\TT_{K^p_r,\kap^1,\CO,\pi}}\TT_{K^p_{r^1},\kap,\CO,\pi}/\ker(\lambda_{\pi})\\
@VV=V @VVV\\
\CO \otimes_{\CO} \hat{I}_{\pi}  @>>> \Hom_\CO(I_{\pi},\CO)
\end{CD}$$
where the bottom arrow 
is just the tautological isomorphism $\hat{I}_{\pi} \isoarrow \Hom(I_{\pi},\CO)$.

\end{prop}

Assuming that $\Hom_{\CO}(\varinjlim_r S_{\kap^\flat}^\ord(K^\flat_r,\O)_{\grm_{\pi}^\flat},\O)$ is a free $\TT_{\pi^\flat}$-module of finite rank
and using the isomorphism $\TT_{\pi}\cong\TT_{\pi^\flat}$ arising from Lemma \ref{Hecke-ord-iso}(ii), we can similarly fix a $\TT_\pi$-module isomorphism
$$
\TT_\pi\otimes_\CO {I}_\pi \isoarrow \Hom_{\CO}(\varinjlim_r S_{\kap^\flat}^\ord(K^\flat_r,\O)_{\grm_{\pi}^\flat},\O) = \varprojlim_r H^{d,\ord}_{\kap^\flat}(K_r^\flat,\CO)_\pi
$$
with corresponding properties.
Since $\pi\rightarrow \pi^\flat$, $g\mapsto g^\flat$, maps holomorphic forms to anti-holomorphic forms, is natural to identifiy the free $\O$-module $\hat{I}_{\pi^\flat}$ with
$I_\pi$.

\subsection{Equivariant measures}\label{equivariantmeasures-section}

In this section we consider measures with values in $p$-adic modular forms on $G_3$.   
We fix a prime-to-$p$ level subgroup $K^p_1 \subset G_1(\adeles_f^p)$ 
and let $K_{1,r} = I_r K_1^p \subset G_1(\A_f)$. Under the canonical idenitfication $G_2(\A_f) = G_1(\A_f)$, let $K^p_2 = (K^p_1)^\flat \subset G_2(\A_f)$ 
and $K_{2,r} = K_{1,r}^\flat \subset G_2(\A_f)$. 
Let $K^p = (K^p_1\times K^p_2)\cap G_3(\adeles_f^p)$ and
$K_{3,r} = (K_{1,r}\times K_{2,r}) \cap G_3(\A_f)$. 
For $\CO=\CO_\pi$ as above, we write $\CV = V^{\ord,\cusp}(K^p,\CO)$ for the corresponding space of ordinary $p$-adic cusp forms on $G_3$ with values in $\CO$.

\begin{rmk}\label{noncusp}
{\rm Although the Eisenstein measure does not generally take values in the space of cusp forms, even after ordinary projection, we will be localizing at a non-Eisenstein
maximal ideal of the Hecke algebra.  Much of the discussion below applies without change to measures with values in the space of ordinary $p$-adic forms.}
\end{rmk}

\subsubsection*{Guide to this section}  The aim of this section is to explain how to obtain a $p$-adic analytic function $L$ from a $\CV$-valued measure on the group $X_p \times T$ as in \ref{2varmeas}, under the  Gorenstein and Global Multiplicity One Hypotheses of the previous section.   In the application the factor $X_p$ contributes a CM Hecke character $\chi$, which parametrizes the Eisenstein measure, and $T$ parametrizes the weights of a Hida family; by this we mean that the localization $\TT = \TT_\pi$ of the big ordinary Hecke algebra at the maximal ideal corresponding to a cuspidal holomorphic representation $\pi$ of $G_1 = G_2$ is a finite flat $\Lambda$-algebra, where $\Lambda$ is the Iwasawa algebra of the torus $T$.  The function $L$ should be viewed as an element of the completed tensor product 
$\Lambda_{X_p}\hat{\otimes} \TT$.  However, because we have not specified canonical local test vectors at ramified places, our methods only provide a canonical element of  $\Lambda_{X_p}\hat{\otimes} \TT \otimes End_{\CO}(I_{\pi})$.  Let us suppose for the remainder of this paragraph that $End_{\CO}(I_{\pi})$ is a free rank one module over the coefficient ring $\CO$; in other words, that we can ignore ramified places.  
We also simplify the situation by assuming that $\TT = \Lambda$ (so that the Gorenstein Hypothesis is automatic).    The result of the construction is then a measure on $X_p\times T$.  This is defined by interpolating locally constant functions on $T$ -- functions that are constant modulo the open subgroups $T_r$ to be introduced below -- obtained from the Eisenstein measure, which is itself a $\CV$-valued measure on $X_p$.  The fundamental property of the Eisenstein measure is summarized in Assumption \ref{doub} below:  that at every finite level $K_r$ it defines a Hecke-equivariant map from the $\CO$-dual of the holomorphic forms of fixed weight $\rho$ to the $\CO$-module of holomorphic forms.  (Later the $\CO$-dual of the holomorphic forms will be identified by Serre duality with an $\CO$-lattice in a coherent cohomology space of top degree, and then with something roughly equivalent to an $\CO$-lattice in the space of anti-holomorphic forms.)  As $r$ varies, these maps satisfy the distribution relations,
proved as Lemma \ref{trace2}, 
that guarantees that they patch together into a $p$-adic measure.

\bigskip
\bigskip

We choose a sequence of congruence subgroups $T \supset \cdots\supset T_r \supset T_{r+1} \cdots$ such that
$\cap_r T_r = \{1\}$.   
Recall that $\Lambda = \Lambda_\pi = \CO[\![T]\!]$.
Let $\CI_r \subset \Lambda$ be the augmentation ideal of $T_r$, and let $\Lambda_r = \Lambda/\CI_r$.
For $\CO$ as above, let $C_r(T,\CO) = C(T/T_r,\CO)$ be the (free) $\CO$-module of $T_r$-invariant functions on $T$.  Then there is a
natural identification
$\Lambda_r = \Hom_{\CO}(C_r(T,\CO),\CO)$; alternatively, viewing $\Lambda$ as the algebra of distributions on $T$ with coefficients in $\CO$, and $C(T,\CO)$ the
module of continuous $\CO$-valued functions on $T$, the canonical pairing
$\Lambda_\pi\otimes C(T,\CO) \rar \CO$ restricts to a pairing $\Lambda \otimes C_r(T,\CO) \rar \CO$ which factors through a perfect
pairing $\Lambda_r \otimes C_r(T,\CO) \rar \CO$.  

Let $\eta_r:  C_r(T,\CO) \hookrightarrow C_{r+1}(T,\CO)$ be the canonical inclusion.  
The next lemma follows from the definitions.  Note that $\CV = V^{\ord,\cusp}_3(K^p,\CO)$ is a $\Lambda_{\CO}$-module
by  the action on the first factor.   We fix an involution $\upsilon:  T \rar T$ and for any
function $\rho \in C(T,\CO)$
define $\rho^{\upsilon} = \rho\circ\upsilon$.

\begin{lem}\label{meas}  Fix a character $\rho: T \rar \CO^{\times}$ and let
$C_r(T,\CO)\cdot \rho^{\upsilon} \subset C(T,\CO)$ denote multiples of $\rho^{\upsilon}$ by
elements of $C_r(T,\CO)$.  There is an equivalence between 
\begin{enumerate}
\item{$\CV$-valued measures $\phi$ on $T$ satisfying
$$\phi(t\cdot f) = \rho^\upsilon(t)\cdot \phi(f), f \in C(T,\CO), ~t \in T;$$}\label{meas1ee}

\item{Collections $\phi_{\rho} = (\phi_{r,\rho})$ with
$$\phi_{r,\rho} \in \Hom_{\Lambda}(C_r(T,\CO)\cdot \rho^{\upsilon},\CV),$$
satisfying $\eta^*_r(\phi_{r+1, \rho}) = \phi_{r, \rho}$, where $\eta^*_r$ is induced by the dual to $\eta_r$.}\label{meas2ee}
\end{enumerate}
\noindent The equvalence is such that $\phi(f) = \phi_{r,\rho}(f\cdot \rho^\upsilon)$ for $f\in C_r(T,\CO)$.
\end{lem}

We let $\CI_{r,\rho} \subset \Lambda$ be the annihilator of $C_r(T,\CO)\cdot \rho^{\upsilon}$, and let
$\Lambda_{r,\rho} = \Lambda_\pi/\CI_{r,\rho}$.  Thus Lemma \ref{meas} identifies equivariant measures on $T$ with twist $\rho^\upsilon$
with collections of linear forms on $\Lambda_{r,\rho}$ that are compatible with the natural
projection maps as $r$ varies.

Let $\hat{\CV} = \Hom_{\CO}(\CV,\CO)$ and let $\phi$ be an equivariant measure on $T$ with twist $\rho^\upsilon$ as above.
We assume $\phi$ to be the specialization
at a character $\chi$ of $X_p$ of an admissible measure in two variables with shift $sh^*(\chi) = (\alpha(\chi),\beta(\chi))$ and twist $\upsilon$
as in Section  \ref{2varmeas}.
So $\phi$ is equivalent to some $\phi_{\chi,\rho} = (\phi_{\chi,r,\rho})$ as in the preceding lemma (we write $\phi_{\chi,r,\rho}$ to indicate dependence on $\chi$).  

\begin{rmk}\label{restriction123} We write 
$$M_{\kap_1}\left(K_{1,r},\psi_1;R\right)[\otimes]_R M_{\kap_2}\left(K_{2,r},\psi_2;R\right)$$
for the image of $res_4$ in $M_{\kap}\left(K_{3,r},\psi;R\right)$, and use the notation $[\otimes]$
more generally for restrictions of this kind from (classical or $p$-adic) modular forms on $G_1 \times G_2$
to forms on $G_3$.
\end{rmk}
For $\kappa = \rho\cdot \alpha(\chi)$ sufficiently regular, 
\begin{equation}\label{phirkap} \Im(\phi_{\chi,r,\rho}) \subset S^{\ord}_{(\rho\cdot \alpha(\chi)),V}(K_{1,r},\CO) [\otimes] S^{\ord}_{(\rho^{\flat}\cdot \beta(\chi)),-V}(K_{2,r},\CO)\end{equation}
where the notation $[\otimes]$ is as in Remark \ref{restriction123}.

We also have
$$\hat{\CV} = \varprojlim_r \Hom_{\CO}(S^{\ord}_{\rho,V}(K_{1,r},\CO) [\otimes] S^{\ord}_{\rho^{\flat},-V}(K_{2,r},\CO),\CO).$$
In the situation of \eqref{phirkap}, assuming $\kappa = \rho \cdot \alpha(\chi)$ is sufficiently regular, we thus have
$$\Im(\phi_{\chi,r,\rho}) \subset \Hom_{\CO}(\hat{S}^{\ord}_{(\rho\cdot \alpha(\chi)),V}(K_{1,r},\CO),S^{\ord}_{(\rho^{\flat}\cdot \beta(\chi)),-V}(K_{2,r},\CO)).$$

The following hypothesis expresses a basic property of the Garrett map that is the basis of the doubling method for
studying standard $L$-functions of classical groups.
\begin{ass}\label{doub} 
$$\Im(\phi_{\chi,r,\rho}) \subset \Hom_{\bT_{r,\rho\cdot \alpha(\chi)}}(\hat{S}^{\ord}_{(\rho\cdot \alpha(\chi)),V}(K_{1,r},\CO),S^{\ord}_{(\rho\cdot \alpha(\chi))^{\flat},-V}(K_{2,r},\CO)\otimes \chi\circ\det).
$$
\end{ass}
Recall that $\bT_{r,\rho\cdot \alpha(\chi)}$ acts on $S^{\ord}_{(\rho\cdot \alpha(\chi))^{\flat},-V}(K_{2,r},\CO)$ via the map
in Lemma \ref{Hecke-ord-iso}(ii).
Also bear in mind that $\phi_{\chi,r,\rho}$ designates integration of functions locally equal to $\rho^{\upsilon}$ -- not $\rho$ -- against the specialization at $\chi$
of a two-variable measure.

\begin{rmk}  We sometimes write $\kappa = \rho\cdot \alpha(\chi)$ when we want to emphasize
the weight of the specialized Hecke algebra rather than the weight of the 
character of $T$.  Here and below the algebra $\bT_{\kap} = \bT_{r,\rho\cdot \alpha(\chi)}$ ignores the twist by $\chi\circ \det$ at the end.  One checks that incorporating
the $\chi\circ\det$ into the subscript of the second $S^{\ord}$ replaces $\alpha(\chi)^{\flat}$ by the $\beta(\chi)$ of \eqref{phirkap}.
\end{rmk}

Now let $\pi$ be an anti-holomorphic representation of $G_1$ of type $(\kappa = \rho\cdot \alpha(\chi),K_r)$.  
Let $\phi_{\chi,r,\rho,\pi}$ denote the composition of $\phi_{\chi,r,\rho}$ with projection on the localization at the ideal $\grm_{\pi}$ in the first variable.
Bearing in mind our conventions for the subscripts $_{\pi}$ and $_{\pi^{\flat}}$, it then follows from Assumption \ref{doub} that
\begin{equation}\label{locpi} \Im(\phi_{\chi,r,\rho,\pi}) \subset 
\Hom_{\bT_{r,\kap,V,\pi}^\ord}(\hat{S}^{\ord}_{\kap,V,\pi}(K_{1,r},\CO),S^{\ord}_{\kap^{\flat},-V,\pi^{\flat}}(K_{2,r},\CO)\otimes \chi\circ\det).
\end{equation}

Now both $\hat{S}^{\ord}_{\kappa,V,\pi}(K_{1,r},\CO)$ and $S^{\ord}_{\kap^{\flat},-V,\pi^{\flat}}(K_{2,r},\CO)$ are $\TT_{\pi}$-modules, and indeed the Gorenstein hypothesis guarantees that they are free 
$\TT_{r,\kappa,\pi}$-modules (in the obvious notation) of the same rank.     In the next few paragraphs they are
denoted $\hat{S}^{\ord}_{r,V,\pi}$ and $S^{\ord}_{r,-V,\pi^{\flat}}$ to save space, the character $(\kappa = \rho\cdot \alpha(\chi))$ being understood.

Recall that from by discussion in Section \ref{locrep} 
there are $\T_\pi$-module isomorphisms
$$\TT_{r,\kappa,\pi}\otimes \hat{I}_{\pi} \isoarrow \hat{S}^{\ord}_{r,V,\pi}
$$
and
$$
\hat{\TT}_{r,\kappa,\pi}\otimes \hat{I}_{\pi} \isoarrow S^{\ord}_{r,-V,\pi^\flat}.
$$
(These isomorphisms depended on certain choices, but the final results will not depend on these).

Thus Assumption \ref{doub} yields (with $\kappa = \rho\cdot \alpha(\chi)$ as above)
\begin{ass}\label{doub2} 
\begin{align*} \Im(\phi_{\chi,r,\rho,\pi}) &\subset \Hom_{\TT_{r,\kappa,\pi}}(\hat{S}^{\ord}_{r,V,\pi},S^{\ord}_{r,-V,\pi^{\flat}}\otimes \chi\circ \det) \\ &\isoarrow 
\Hom_{\TT_{r,\kappa\cdot a(\chi)}}(\TT_{r,\kappa,\pi}\otimes \hat{I}_{\pi},\hat{\TT}_{r,\kappa,\pi}\otimes I_{\pi^{\flat}}) \\
& \isoarrow \Hom_{\TT_{r,\kappa}}(\TT_{r,\kappa,\pi},\hat{\TT}_{r,\kappa,\pi})\otimes_\O \End_{\O}(I_{\pi^{\flat}}).
\end{align*}
\end{ass}
\noindent Here we have tensored with $\chi^{-1}\circ\det$ in the first line.

\medskip
In the remainder of this subsection we no longer need to localize at $\grm_{\pi}$.  
We write $C_r = C_r(T,\CO)$ and drop the $\CO$'s from the
notation for modules of ordinary cusp forms, and ignore the twists by $\chi\circ\det$ where relevant.
  The natural inclusion $C_r \hookrightarrow C_{r+1}$, together with the map $\iota^*_r:  \hat{S}^{\ord}_{r+1,V} \rar \hat{S}^{\ord}_{r,V}$
(dual to the tautological inclusion $\iota_r:  S^{\ord}_{r,V} \hookrightarrow S^{\ord}_{r+1,V}$)
defines a diagram

$$\begin{CD}
\Hom_{\bT_{r+1,\kappa}}(C_{r+1}\otimes \hat{S}^{\ord}_{r+1,V}, S^{\ord}_{r+1,-V}) @> \eta^*_r\otimes id_{r+1}^*>> 
\Hom_{\bT_{r+1,\kappa}}(C_{r}\otimes \hat{S}^{\ord}_{r+1,V}, S^{\ord}_{r+1,-V}) \\
@. @AA\mathbf{i}^*_r A\\
@. \Hom_{\bT_{r,\kappa}}(C_{r}\otimes \hat{S}^{\ord}_{r,V}, S^{\ord}_{r,-V}) 
\end{CD}
$$
Here $id_{r+1}^*:  \hat{S}^{\ord}_{r+1,V} \rar \hat{S}^{\ord}_{r+1,V}$ is the identity map and $\mathbf{i}_r^*$ is the dual to $\iota_r$
(applied in the contravariant variable) composed with $\iota_r$ (in the covariant variable).   It follows from the equivariance hypothesis that the tensor products ($C_{r+1}\otimes \hat{S}^{\ord}_{r+1,V}$ and the other two)
can be taken over $\Lambda_{\pi}$, and then $\Hom_{\bT_{r+1,\kappa}}$ is relative to the action of the Hecke algebra on 
$ \hat{S}^{\ord}_{r+1,V}$ and $S^{\ord}_{r+1,-V}$.  Then

\begin{fact} Under Assumption \ref{doub} we have that, for all $r$, the image of $\phi_{r+1,\kappa}$ under $ \eta^*_r\otimes id^*_{r+1}$ lies in  $\Im(\iota^*_r)$.
More precisely,
\begin{equation}\label{trace} ( \eta^*_r\otimes id^*_{r+1})(\phi_{r+1,\kappa}) = \iota_r \circ \phi_{r,\kap}\circ (id_{C_{r}}\otimes \iota_r^*)
\end{equation}
as maps from $C_{r}\otimes \hat{S}^{\ord}_{r+1,V}$ to $S^{\ord}_{r+1,-V}$,
where $id_{C_{r}}$ is the identity map on $C_{r}$.
\end{fact}

\subsubsection{Serre duality and change of level}

We can interpret the map $\iota_r^*$ with respect to the Serre duality pairing \ref{SerreC} as follows.  In this section we 
let $R_0$ be a finite $\ZZ_{(p)}$-algebra with $p$-adic completion $R_0 \hookrightarrow \CO_\pi$.  Identify
$S_{\kap,V}(K_r,R_0)$ with an $R_0$-lattice in $H^{0}_{!}({}_{K_r}Sh(V),\omega_{\kap})$.  Let $H^{0,ord}_{!}({}_{K_r}Sh(V),\omega_{\kap})$
be the $\C$-linear span of $S^{\ord}_{\kap,V}(K_r,R_0)$ and let $H^{d,\ord}_{!}({}_{K_r}Sh(V),\omega^D_{\kap})$ be the corresponding
quotient of $H^{d}_{!}({}_{K_r}Sh(V),\omega^D_{\kap})$; then the action of the Hecke algebra identifies $H^{d,\ord}_{!}({}_{K_r}Sh(V),\omega^D_{\kap})$
as a direct summand of $H^{d}_{!}({}_{K_r}Sh(V),\omega^D_{\kap})$ that is in a perfect pairing with $H^{0,ord}_{!}({}_{K_r}Sh(V),\omega_{\kap})$.
We can thus identify $\hat{S}^{\ord}_{\kap,V}(K_r,R_0)$ with an $R_0$-lattice in $H^{d,\ord}_{!}({}_{K_r}Sh(V),\omega^D_{\kap})$ in such
a way that  
\begin{equation}\label{Rdual}
\begin{split}
\hat{S}^{\ord}_{\kap,V}(K_r,R_0) &=  H^{d,\ord}_{!}(_{K_r}Sh(V),\omega^D_{\kap})(R_0) \\
&:= \{h \in H^{d,\ord}_{!}(_{K_r}Sh(V),\omega^D_{\kap}) ~~|~ \la f,h \ra_{\kap,K_r} \in R_0 ~ \forall f \in S^{\ord}_{\kap,V}(K_r,R_0) \}
\end{split}
\end{equation}

The following statements (Lemma \ref{trace2}, Proposition \ref{Heckemeasure}, and Definition \ref{class}) are written in terms of $\kappa$
rather than $\rho\cdot \alpha(\chi)$, for simplicity.

\begin{lem}\label{trace2} With respect to the identification \eqref{Rdual}, the map 
$$\iota_r^*:  \hat{S}^{\ord}_{\kap,V}(K_{r+1},R_0) \rar \hat{S}^{\ord}_{\kap,V}(K_r,R_0)$$
is given by the trace map:
$$t_r(h) =  \frac{\#(I_r^0/I_r)}{\#(I_{r+1}^0/I_{r+1})}\sum_{\gamma \in K_r/K_{r+1}} \gamma(h).$$

In particular, the trace map $t_r$ defines a surjective homomorphism
$$H^{d,\ord}_{!}(_{K_{r+1}}Sh(V),\omega^D_{\kap})(R_0)  \rar H^{d,\ord}_{!}(_{K_r}Sh(V),\omega^D_{\kap})(R_0)$$
\end{lem} 
\begin{proof}          
This is essentially a rewording of Section \ref{normal pairing}.   
\end{proof}

Now we complete at $\grm_{\pi}$ but omit the subscript $\pi$ from the notation; so $\CO = \CO_\pi$.   As in Assumption \ref{doub2} we can identify
\begin{equation*}\begin{split}  \Hom_{\TT_{r,\kappa}}(C_{r}\otimes \hat{S}^{\ord}_{r,V},S^{\ord}_{r,-V})  \simeq
~& \Hom_{\TT_{r,\kappa}}(C_{r}\otimes \TT_{r,\kappa}\otimes \hat{I}_{\pi},\hat{\TT}_{r,\kappa}\otimes I_{\pi^{\flat}}) \\  = ~&\Hom_{\TT_{r,\kappa}}(C_{r}\otimes \TT_{r,\kappa} , \hat{\TT}_{r,\kappa})\otimes  End_{\CO}(I_{\pi^{\flat}})\\  \simeq ~&  \hat{\TT}_{r,\kappa}\otimes End_{\CO}(I_{\pi^{\flat}})
\end{split}\end{equation*}
with appropriate modifications to accomodate a function $\Phi_X$ as above.  We are using Proposition \ref{badp} 
systematically.  Be advised that $\TT = \TT_{\pi}$ in the following Proposition.

\begin{prop}\label{Heckemeasure}  With respect to the identifications 
$$\Hom_{\TT_{r,\kappa},\flat}(C_{r}\otimes \hat{S}^{\ord}_{r,V},S^{\ord}_{r,-V})  \simeq \hat{\TT}_{r,\kappa}\otimes End_{\CO}(I_{\pi^{\flat}}),$$
\eqref{trace}, and 
the isomorphism $G_r: \hat{\TT}_{r,\kappa} \isoarrow \TT_{r,\kappa}$ of the Gorenstein Hypothesis \ref{gor},
the measure $\{\phi_{r,\kap}\}$
defines an element 
$$ L(\phi_{\kap}) \in \varprojlim_r \TT_{r,\kappa}\otimes I_{\pi^{\flat}}\otimes I_{\pi}   \isoarrow \TT\otimes End_{\CO}(I_{\pi^{\flat}}). $$
Moreover, if $\kap^1$ is a second sufficiently regular character, then 
$L(\phi_{\kap})$ and $L(\phi_{\kap^1})$ are identified with respect to the identifications $\TT \isoarrow \TT_{K^p,\kap,\CO_\pi} \isoarrow \TT_{K^p,\kap^1,\CO_\pi}$ of 
Theorem \ref{bigHecke} (after localization at $\pi$).  Thus the measures $\{\phi_{r,\kap}\}$ and $\{\phi_{r,\kap^1}\}$ define the same element 
$L(\phi) \in  \TT\otimes End_{\CO}(I_{\pi^{\flat}})$.  Conversely, any such $L(\phi)$ defines a measure $\{\phi_{r,\kap}\}$ for
any sufficiently regular $\kap$.
\end{prop}
\begin{proof}  This is a consequence of Lemma \ref{trace2} and follows by unwinding the definitions.

\end{proof}

The above construction adapts easily to accommodate the compact $p$-adic Lie group $X_p$.  We have
seen that a
$\CV$-valued measure on $X_p\times T$ is the same thing as a measure on $X_p$ with values
in $\CV$-valued measures on $T$.  In particular,  one obtains a $\CV$-valued measure
on $X_p\times T$ from a collection,
for all characters $\alpha$ of $X_p$, of $\CV$-valued measures $\phi_{\alpha}$ of type $\alpha$ on $T$ satisfying 
the congruence properties of Lemma \ref{measures-characters}.   

In what follows we identify $\bT$ with the ordinary Hecke algebra for the group $G = G_1$; the same definition holds, with appropriate modifications, when $G = G_2$.
\begin{defi}\label{class}  Fix a level $r$, a character $\kap$, and an $\CO$-algebra $R$. Let $\lambda:   \bT  \rar R$ be a
continuous homomorphism.  Say $\lambda$ is  {\it classical}
of level $p^r$ and weight $\kap$ if it
factors through a homomorphism (still denoted) $\lambda:  \bT_{r,\kap} \rar R$, which is of
the form $\lambda_{\pi}$ for some anti-holomorphic automorphic representation $\pi$ of type $(\kap,K_{1,r})$ 
with $K_{1,r} = K^pI_{r}$ for some open compact $K^p \subset G(\A_f^p)$, as before

Let $X(\kap,r,R)$ denote
the set of classical homomorphisms of level $p^r$ and weight $\kap$ with values in $R$; 
let $X^{class}(R) = \cup_{\kap,r} X(\kap,r,R)$.    
Any $\lambda \in X^{class}(R)$ is called {\it classical} (with values in $R$).

When $R = \bT_{r,\kap}$, we let $\lambda_{taut}:  \bT_{r,\kap} \rar \bT_{r,\kap}$ be the identity homomorphism.
When $\pi$ is a cuspidal anti-holomorphic representation of weight $\kap$ as above, let
$\lambda_{taut,\pi}:  \bT_{r,\kap} \rar \TT_{r,\kap,\pi}$ be $\lambda_{taut}$ followed by localization at $\grm_{\pi}$.
\end{defi} 

When $\kappa$ is sufficiently regular, the character $\lambda_{taut}$ deserves to be called classical because its composition with any
map from $\bT_{r,\kap}$ to a $p$-adic field is attached to a classical modular form of weight $\kappa$.
The relationship between $L(\phi)$ and the elements $\phi_{\chi,r,\kap}$ is given by the following proposition. 

\begin{prop}\label{specialize} ÊLet $\chi$ be a character of $X_p$. Ê
Let $\phi = d\phi(x,t)$ be a measure on $X_p \times T$ as in Section  \ref{2varmeas}, with shift $sh$: $sh^*(\chi) = (\alpha(\chi),\beta(\chi))$
Let $\rho$ be an algebraic character
of $T$, $\kappa = \rho\cdot \alpha(\chi)$. ÊFix a cuspidal anti-holomorphic representation $\pi$ of weight $\kappa$ satisfying the hypotheses above.
We consider $L(\phi_{\chi}) = \int_{X_p} \chi(x) d\phi(x,t)$, localized at $\grm_{\pi}$, as an element of $\TT \otimes End_{R}(I_{\pi^{\flat}})$.
Let $L(\phi_{\chi},\kap,r)$ denote the image of $L(\phi_{\chi})$ in $ \TT_{r,\kap}\otimes End_{R}(I_{\pi^{\flat}})$. ÊEquivalently,
$$L(\phi_{\chi},\kap,r) Ê= \int_{X_p\times T} Ê\chi \times \lambda_{taut,\pi} d\phi(x,t),$$
where integration against $\lambda_{taut,\pi}$ amounts to the projection
$$\bT \twoheadrightarrow \bT\otimes_{\Lambda} \Lambda_{r,\kap} = \bT_{r,\kap}$$
followed by localization at Ê$\grm_{\pi}$.

Then $L(\phi_{\chi},\kap,r)$ corresponds to the element
$$\phi_{\chi,r,\kappa} \in \Hom_{\TT_{r,\kappa}}(C_{r}\otimes \hat{S}^{\ord}_{r,V,\pi},S^{\ord}_{r,-V,\pi^{\flat}}\otimes \chi\circ\det) $$
under the identifications in Proposition \ref{badp}, 
compatible with the Gorenstein isomorphism $G_r$ (from Proposition \ref{Heckemeasure}).
\end{prop}

\begin{proof}  This is just a restatement of the definition of the element $L(\phi) \in \TT \otimes End_{R}(I_{\pi^{\flat}}) = \TT_\pi \otimes End_{R}(I_{\pi^{\flat}})$
introduced in Proposition \ref{Heckemeasure}.
\end{proof}

The following is now an elementary consequence
of Proposition \ref{Heckemeasure}.  Again, recall that $\TT = \TT_{\pi} = \bT_{\mathfrak{m}_\pi}$.   We say that an anti-holomorphic cuspidal representation
satisfies Hypotheses \ref{gor} and \ref{multone} if its complex conjugate (or contragredient) does.  

\begin{prop}\label{Heckemeasure-2variables}{\bf (Abstract $p$-adic $L$-functions of families)}  
  Let $\phi = d\phi(x,t)$ be a measure on $X_p \times T$ such that, for
each character $\chi$ of $X_p$, $\int_{X_p} \chi(x) d\phi(x,t)$ is a $\CV$-valued measure $\phi_{\chi}$
of type $\chi$ satisfying Assumption \ref{doub}.   
Fix a cuspidal anti-holomorphic representation $\pi$ satisfying 
Hypotheses \ref{gor} and \ref{multone}.  Then there is an element 
$L(\phi) \in \Lambda_{X_p}\hat{\otimes} \TT \otimes End_{R}(I_{\pi^{\flat}})$
such that, for every $R$-valued character $\chi$ of $X_p$, the image of $L(\phi)$ under the map
$$\chi\otimes Id: \Lambda_{X_p}\hat{\otimes} \TT \otimes End_{R}(I_{\pi^{\flat}}) \rar  \TT \otimes End_{R}(I_{\pi^{\flat}})$$
given by contraction in the first factor, or equivalently integration against $\chi$ with respect
to the first variable, is the element $L(\phi_{\chi})$ of Proposition \ref{specialize}.
\end{prop}

The following standard fact (see, for example, \cite[Lemma 3.3]{Hida1988}) shows that the specializations of Proposition \ref{specialize}
determine the abstract $L$-function $L(\phi)$:

\begin{lem}\label{density}  The $\CV$-valued measure $\phi = \phi_{\chi}$ of type $\chi$ and the abstract $L$-function $L(\phi)$ are completely determined
by their integrals against elements of the sets $X(\kap,r,\OCp)$ for any fixed sufficiently regular $\kap$ and all $r$.
\end{lem}

We write
\begin{equation}\label{dualagain}
End_{R}(I_{\pi^{\flat}}) = \hat{I}_{\pi^{\flat}}\otimes I_{\pi^{\flat}} \simeq \Hom(\hat{I}_{\pi}\otimes \hat{I}_{\pi^{\flat}},R). \end{equation}

Then for any 
$\varphi \otimes \varphi^{\flat} \in \hat{I}_{\pi}\otimes \hat{I}_{\pi^{\flat}}$ we have a tautological pairing
\begin{equation}\label{basicpairing}
L(\chi,\phi,r,\kap, \varphi \otimes \varphi^{\flat}) =  [L(\phi_{\chi},\kap,r),\varphi \otimes \varphi^{\flat}]_{loc}  \in \TT_{r,\kap}.
\end{equation}
where  $[\bullet,\bullet]_{loc}$ is the tautological pairing 
$$\Hom(\hat{I}_{\pi}\otimes \hat{I}_{\pi^{\flat}},R)\otimes \hat{I}_{\pi}\otimes \hat{I}_{\pi^{\flat}} \rar R.$$

We reformulate Proposition \ref{specialize} in terms of Equation \eqref{basicpairing}. 
\begin{prop}\label{abstract2var}  Fix an embedding $\CO_\pi\hookrightarrow\IC$ extending the inclusion $E(\pi)\subset\IC$, and let $R$ be a $p$-adic ring containing $\CO_\pi$ and satisfying the conditions of Lemma \ref{measures-characters}.  Let $\phi$ be an admissible $R$-measure on $X_p \times T$ as in Section  \ref{2varmeas}.  Assume
Hypotheses \ref{gor} and \ref{multone}.
Let $\varphi \otimes \varphi^{\flat} \in \hat{I}_{\pi}\otimes \hat{I}_{\pi^{\flat}}$ as above.   
Then there is a unique element $L(\phi,\varphi \otimes \varphi^{\flat}) \in \Lambda_{X_p,R}\hat{\otimes}\TT$ such that, for any classical
$\chi:  X_p \rar R^{\times}$ and any $\lambda \in X(\kap,r,R)$ (with $\kap$ sufficiently regular), the image of $L(\phi,\varphi \otimes \varphi^{\flat})$ under the map
$\Lambda_{X_p,R}\hat{\otimes}\TT \rar R$ induced by the character $\chi\otimes \lambda$ equals $\lambda\circ L(\chi,\phi,r,\kap,\varphi \otimes \varphi^{\flat})$.
\end{prop}

\subsection{Classical pairings in families}\label{classpairings}

The following is essentially obvious.  The notation $\pairS_\kap$ is as in \eqref{SerreC}.

\begin{lem}\label{pair1}  Let $h \in S^{\ord}_{\kap,V}(K_r,\CO)$, $\varphi \in H^{d,\ord}_{\kap^D}(K_r,\CO)[\pi]$, in the notation 
of  Section  \ref{heck}. Then  the map
$$\TT \rar \CO; ~~ A ~\mapsto ~ \la A(h), \varphi\ra_{\kap,K_r}$$
takes $A$ to $\lambda_{\pi}(A)\la h,\varphi\ra_{\kap,K_r}$.
\end{lem}

\begin{proof}  We have
$$\la A(h),\varphi\ra_{\kap,K_r} = \la h, A^{\flat}(\varphi)\ra_{\kap,K_r} = \lambda_{\pi^{\flat}}(A^{\flat})\la h,\varphi\ra_{\kap,K_r} = \lambda_{\pi}(A)\la h,\varphi\ra_{\kap,K_r}.$$
\end{proof}

Note that $h$ is not assumed to be an eigenform in Lemma \ref{pair1}.  However, the pairing with an eigenform for $\lambda_{\pi^{\flat}}$ factors through 
the projection of $h$ on the (dual) $\lambda_{\pi}$-eigenspace.  In general, this projection can only be defined over $\CO[\frac{1}{p}]$.  Extending $\CO$ if necessary to include $\CO_{\pi'}$ for $\pi' \in \CS(K_r,\kap,\pi^{\flat})$, write
$h = \sum_{\pi' \in \CS(K_r,\kap,\pi^{\flat})} a_{\pi'} h_{\pi'}$ where $a_{\pi'} \in \CO[\frac{1}{p}]$ and $h_{\pi'}$ is in the $\lambda_{\pi'}$-eigenspace for $\TT$.
Then under the hypotheses of the lemma,
\begin{equation}\label{isotypic}  \la h,\varphi\ra_{\kap,K_r} = a_{\pi}\la h_{\pi},\varphi\ra_{\kap,K_r}. \end{equation}
where of course $h_{\pi} \in \pi^{\flat}$.  

The denominator of $a_{\pi}$ is bounded by the congruence ideal $C(\pi)$ = $C(\pi^{\flat})$.  In what follows we are making use of 
Corollary \ref{congduality}.

\begin{lem}\label{functional1}   Let $\varphi \in H^{d,\ord}_{\kap^D}(K_r,\CO)[\pi]$.  Then the linear functional
$$h \mapsto L_{\varphi}(h) :=  \la h,\varphi\ra_{\kap,K_r}$$
belongs to $\hat{S}^{\ord}_{\kap,V}(K_r,\CO)[\pi].$
Moreover, the restriction of $L_{\varphi}(h)$ to $S^{\ord}_{\kap,V}(K_r,\CO)[\pi]$ takes values in the congruence ideal $C(\pi) = C(\pi^{\flat}) \subset \CO$.

\end{lem}
\begin{proof}  The claims follow from Lemmas \ref{pair1} and \ref{cflat2}, respectively.  
\end{proof}

The functional in the last lemma can be rewritten as an integral.   Recall that $\hat{I}_{\pi}$ (resp. $\hat{I}_{\pi^{\flat}}$)  was identified with an $\O$-lattice
 in   $\pi_p^{\aord}\otimes \pi_{S^p}^{K^p}$, (resp.  $\pi_p^{\flat,\aord}\otimes \pi_{S^p}^{\flat,K^p}$.
Recall also that we have dropped the subscript $\pi$ for the moment, and so we are writing $\CO$ in place of $\CO_\pi$.
In order to facilitate comparison of the $p$-adic and complex pairings,
we let $R$ be a finite local $\ZZ_{(p)}[\lambda_{\pi}]$-subalgebra of $\C$ that admits an embedding as a dense
subring of $\CO$, and let $\hat{I}_{\pi^{\flat},R}$ and  $\hat{I}_{\pi,R}$ be free $R$-modules given with isomorphisms
$$\hat{I}_{\pi^{\flat},R}\otimes_R \CO  \isoarrow \hat{I}_{\pi^{\flat}}; ~~\hat{I}_{\pi,R}\otimes_R \CO  \isoarrow \hat{I}_{\pi}.$$
The following lemma is then just a restatement of \eqref{L2}.

\begin{lem}\label{functionalintegral} In the notation of the previous lemma, let $\varphi \in \hat{I}_{\pi}$.  
If we identify $h$  as above 
 with an element of $H^0(\grP_h(V),K_h;\CA_0(G)\otimes W_{\kap})^{K_r}$ and $\varphi$ with an element of 
$H^d(\grP_h(-V),K_h;\CA_0(G)\otimes W_{\kap^{D}})^{K_r^\flat}$,  as in Equation \eqref{dbar}, we can rewrite
\begin{equation*}
L_{\varphi}(h) 
 = \frac{1}{\Vol(I_{r,V}^0)\Vol(I_{r,-V}^0)}\int_{G(\Q)Z_G(\R)\backslash G(\A)} [h(g),\varphi(g)]||\nu(g)^{-a(\kap)}|| dg.
\end{equation*} 
\end{lem}

Lemmas \ref{functional1} and \ref{functionalintegral} have variants incorporating the twist by a Hecke character $\chi$, as in \eqref{chitwist} and Section \ref{periodtwist}; we leave the statements to the reader.

\section{Local theory of ordinary forms}\label{localtheory}
\setcounter{subsection}{-1}

\subsection{Parameters}\label{parameter-notation} Throughout this section, following the conventions of Section  \ref{automorphy}, let $\chi$, $\chi_\s$, $m$, $\kap,$ and $(\ub_\s,\uc_\s)$ be associated to one another via: $\chi = ||\bullet||^m\cdot\chi_0$ is an algebraic Hecke character of $\K$ (where
 $m\in \ZZ$ and 
$\chi_{0,\s}(z) = z^{-a\left(\chi_\sigma\right)}\bar{z}^{-b\left(\chi_\sigma\right)}$
 for any archimedean place $\s$), $\chi_\s$ denotes the component of $\chi$ at an archimedean place $\s$, $(\ub_\s,\uc_\s)$ is a tuple of integers as in Inequalities \eqref{parameters}, and $\kap$ is a highest weight defined in terms of $m$, $\chi_\s$, and $(\ub_\s,\uc_\s)$ as in Equation \eqref{C3m}.

\subsection{$p$-adic and $C^{\infty}$-differential operators}\label{maass-section}

Let $(\kap,\chi)$ and $(\ub,\uc)$ be as in Corollary \ref{pluridecomp} and Proposition \ref{holodiffops}.   
The differential operators and restrictions in Parts (a) and (b) of the following proposition are those in \cite[\S6-\S7]{EFMV} (with the choice of a weight $\kappa$ and the differential operator $\Theta^\kappa$ in the notation in \cite{EFMV} corresponding to the choice of a representation of highest weight $(\ub,\uc)$ in the $d$-th tensor product of the standard representation).  That the image is actually cuspidal follows from the description of the action of the differential operators on $q$-expansions.

\begin{prop}\label{padicdiffops}  (a) For $(\ub,\uc)$ and $\chi$ as in Section  \ref{parameter-notation}, and for any prime-to-$p$ level subgroup $K^p$, there is a  differential operator
$$\theta^d_{\chi}(\ub,\uc) = \theta^d_{\chi}(p(\ub,\uc)):  V_{\chi}(G_4,K^p,\CO) \rar V(G_4,K^p,\CO)
$$
compatible with change of level subgroup, and with the following property:    For any level $K^p$, for any form $f \in M_{\chi}(G_4,K^p,\CO)$, and any ordinary CM pair $(J'_0,h_0)$ as in Section  \ref{basepointres}, we have
the identity
$$R_{\kappa,J'_0,h_0} \circ res_{J'_0,h_0}\circ\delta^d_{\chi}(\ub,\uc)(f) 
=     res_{p,J'_0,h_0}\circ \theta^d_{\chi}(\ub,\uc)\circ R_{\kappa,G,X}(f)$$
in the notation of Proposition \ref{CMrest}.

(b) Let $(\kap,\chi)$ be critical as in Corollary \ref{pluridecomp}.  Fix a level subgroup $K_4 \subset G_4(\A_f)$ and a subgroup
$K_1 \times K_2 \subset G_3(\A_f)\cap K_4$.  The composition of $\theta^d_{\chi}(\ub,\uc)$ with the pullback 
$res_3 := (\gamma_{V_p}\circ \iota_3)^*$ defines an operator
$$\theta(\kap,\chi):  V_{\chi}(G_4,K_4^p,\CO)  \rar V_{\kap}(G_1,K_1^p,\CO)\otimes V_{\kap^{\flat}}(G_2,K_2^p,\CO)\otimes \chi\circ \det,$$
where the tensor product with $\chi\circ\det$ is defined by analogy with Proposition \ref{holodiffops}.   

(c) Let $\theta(\kap,\chi)^{\cusp}$ denote the restriction of $\theta(\kap,\chi)$ to $V_{\chi}^{\cusp}(G_4,K_4^p,\CO)$, and let $e_\kap$ denote the ordinary projector 
of \ref{ordprojector} attached to the weight $\kap$, as in \ref{ord-hecke}.  Then the composition $e_\kap \circ \theta(\kap,\chi)^{cusp}$ coincides with the operator $e_\kap\circ \delta^d_{\chi}(\ub,\uc)$ 
upon pullback to functions on $G_4(\ad)$ and restriction to $G_3(\ad)$ (with respect to the maps  \eqref{padicclassical} for $G_3$ and $G_4$):
$$e_\kap \circ  \theta(\kap,\chi)^{cusp} =  e_\kap\circ \delta^d_{\chi}(\ub,\uc):  V_{\chi}^{\cusp}(G_4,K_4^p,\CO) \rar
S_{\kap,V}(K_1)\otimes S_{\kap^{\flat},-V}(K_2)\otimes \chi\circ \det.$$

(d) Under the hypotheses of (a) and (b), there is a differential operator
$$\theta^{hol}(\kap,\chi): V_{\chi}(G_4,K^p,\CO) \rar V(G_4,K^p,\CO)$$
whose composition with the pullback $res_3$ coincides with the operator
$D^{hol}(\kap,\chi)$ of Proposition \ref{holodiffops}, upon pullback to functions on $G_4(\ad)$, restriction to $G_3(\ad)$, and identification of ordinary modular forms with $p$-adic modular forms via \eqref{padicclassical}.
\end{prop}
\begin{proof}  The operators from Part (a) were constructed in \cite[Theorem 5.1.13]{EFMV}. The comparison at CM points follows similarly to \cite[Section 5.1]{kaCM} and is also in \cite[\S10, especially Prop 10.2]{EDiffOps}.   Part (b) follows from \cite[Remark 6.2.7]{EFMV} (and was also present in an earlier form in \cite[Definition 12]{emeasurenondefinite}).  
Since the image of $\theta(\kap,\chi)^{\cusp}$ is contained in 
$V_{\kap}^{\cusp}(G_1,K_1^p,\CO)\otimes V_{\kap^{\flat}}^{\cusp}(G_2,K_2^p,\CO)\otimes \chi\circ \det$, part (c) follows from the control theorem \eqref{classicalitythm1}.

Finally,  part (d) follows from Eischen's construction as well: it follows
(by induction on the size of $\kap$)
from the last part of Corollary \ref{pluridecomp} that the operator $D^{hol}(\kap,\chi)$ is obtained by pullback of the differential operator attached
to a polynomial $P^{hol}(\kap,\chi) \in \otimes_\s \CP(n)_{\sigma}$.  One lets $\theta^{hol}(\kap,\chi)$ be the differential operator
on $p$-adic modular forms attached to the same polynomial.
\end{proof}

The following corollary is the $p$-adic version of the last part of Corollary \ref{pluridecomp}.

\begin{cor}\label{padicdecomp}  Under the hypotheses of the previous proposition,  for all $\kap^{\dag} \leq \kap$ there are differential operators
$\theta(\kap,\lambda):  V_{\chi}(G_4,K^p,\CO) \rar V(G_4,K^p,\CO)$ such that
$$\theta(\kap,\chi) = \sum_{\kap^{\dag} \leq \kap} res_3 \circ \theta(\kap,\kap^{\dag})\circ \theta^{hol}(\kap^{\dag},\chi).$$
\end{cor}

\begin{prop}\label{ordinaryplusclassical}  Let $F \in H^0(Sh(G_4),\CL(\chi)).$

Assume $\kap, (\ub_\s,\uc_\s), m, \chi_\s$ are all as at the beginning of this section, and let $e_{\kap}$ be as above.  Recall the holomorphic projection $pr^{hol}$ from Section  \ref{holoprojection-sec}.  Then

\begin{equation}\label{padicCinfty} (e_{\kap}\circ \theta(\kap,\chi))(F) = e_{\kap}\circ pr^{hol}_{\kap}\circ \delta(\ub_\s,\uc_\s)(F).
\end{equation}
\end{prop}

\begin{proof}   By Corollary \ref{padicdecomp}, the left hand side equals
$$\sum_{\kap^{\dag} \leq \kap} e_{\kap}\circ res_3 \theta(\kap,\kap^{\dag})\circ \theta^{hol}(\kap^{\dag},\chi).$$
By Proposition \ref{CMrest}, it thus suffices to show that, for every ordinary Shimura datum $(J'_0,h_0)$ as in the statement of the proposition, 
\begin{enumerate}
\item[(1)] For $\kap ^{\dag}< \kap$, $e_{\kap}\circ res_3 \theta(\kap,\kap')\circ \theta^{hol}(\kap^{\dag},\chi)(F) = 0$ after composition with $res_{p,J'_0,h_0}$;
\item[(2)] $e_{\kap}\circ res_3 \theta(\kap,\kap)\circ \theta^{hol}(\kap,\chi)(F) - e_{\kap}\circ pr^{hol}_{\kap}\circ \delta(\ub_\s,\uc_\s)(F) = 0$
after composition with $res_{p,J'_0,h_0}$.
\end{enumerate}

Part (2) is a consequence of (c) of \ref{padicdiffops}.   We show that the expression in (1) is arbitrarily divisible by $p$.  More precisely,
 
\begin{lem}  For any $\kap^{\dag} < \kap$, the ordinary projector $e_{\kap} = \varinjlim_N U_{p,\kap}^{N!}$ converges absolutely to $0$ on 
$S_{\kap^{\dag}}(K_r;R)$.
\end{lem}
\begin{proof}  The point is that, for each $w, j$, $U_{w,j,\kap} = |\kap'(t_{w,j})|_p^{-1} U_{w,j}$, with $\kap'$ defined as in \ref{compare-weight}.  Thus
$$U_{p,\kap} = \prod_{w,j}|\kap^{\prime,-1}\cdot\kap^{\dag,\prime}(t_{w,j})|_p \cdot U_{w,j,\kap^{\dag}}.$$
The condition $\kap^{\dag} < \kap$ is equivalent to the condition that the $p$-adic valuation of 
$\prod_{w,j}\kap^{\prime,-1}\cdot\kap^{\dag,\prime}(t_{w,j})$ is positive.  Thus $U_{p,\kap}$ has $p$-adic norm strictly less than $1$ on $S_{\kap^{\dag}}(K_r;R)$, and it follows that $e_{\kap} = \varinjlim_N U_{p,\kap}^{N!}$ acts as $0$ on $S_{\kap^{\dag}}(K_r;R)$.
\end{proof}

Part (1) above now follows from the fact that the ordinary projector commutes with the differential operators.  
\end{proof}

\subsection{Existence of the axiomatic Eisenstein measure}\label{existence-Eisenstein}

Let $\chi_{\unitary}$ be a unitary Hecke character.  Let $\chi = \chi_{\unitary}|\cdot|_{\CK/\IQ}^{-k/2}$.  So $\chi$ is a Hecke character of type $A_0$.  Write $\chi = \prod_w\chi_w$.  We obtain a $p$-adically continuous $\ocp$-valued character $\tilde{\chi}$ on $X_p$ as follows.  Since $\chi$ is of type $A_0$, there are integers $k, \nu_\sigma\in\Z$ such that for each element $a\in \K^\times$,
\begin{align*}
\chi_\infty(a) = \prod_{\sigma\in\Sigma}\chi_\sigma(a) = \prod_{\sigma\in\Sigma}\left(\frac{1}{\sigma(a)}\right)^k\left(\frac{\overline\sigma(a)}{\sigma(a)}\right)^{\nu_\sigma}
\end{align*}
with $\overline\sigma :=\sigma c$.  Let $\tilde\chi_\infty: \left(\K\otimes\Z_p\right)^\times\rightarrow \overline\IQ_p^\times$ be the $p$-adically continuous character such that
\begin{align*}
\tilde\chi_\infty (a) = \incl_p\circ \chi_\infty(a)
\end{align*}
for all $a\in \K$.  So the restriction of $\tilde\chi_\infty$ to $\left(\O\otimes\Z_p\right)^\times$ is a $\ocp^\times$-valued character.  We define
\begin{align*}
\tilde{\chi}: X_p\rightarrow \ocp^\times
\end{align*}
by $\tilde\chi\left(\left(a_w\right)\right) = \tilde\chi_\infty\left(\left(a_w\right)_{w\divides p}\right)\prod_{w\ndivides \infty}\chi_w(a_w)$.  Define $\nu=\left(\nu_\sigma\right)_{\sigma\in\Sigma}$.

For each $\sigma\in\Sigma$, let $n=a_\sigma+b_\sigma$ with $a_\sigma, b_\sigma\geq 0$ be a partition of $n$, and let $a_\sigma = n_{1, \sigma}+\cdots+n_{t(\sigma), \sigma}$ and $b_\sigma=n_{t(\sigma)+1, \sigma}+\cdots + n_{r(\sigma), \sigma}$ be partitions of $a_\sigma$ and $b_\sigma$, respectively.  Let $a=\left(a_\sigma\right)_{\sigma\in\Sigma}$ and $b=\left(b_\sigma\right)_{\sigma\in\Sigma}$.  Let $\psi$ be a finite order character on $T_H\left(\ZZ_p\right)$.  Let $\kappa$ be a dominant character as in Section  \ref{H0-OF-reps}, and
define $\rho$ and $\rho^{\upsilon}$ as in \eqref{parameters}, \eqref{iota}.  We note that $\rho$ and $\kappa$ contain the same information, relative to the shift $(1,\chi)$ which is imposed by the presence of $\chi$ in the Eisenstein measure.

Let $c=\psi\cdot \rho^{\upsilon}$.  We choose $f(\chi, c)$ to be a factorizable Siegel section meeting the conditions of Definition \ref{axiomeis}; the specific local sections will be as in Sections \ref{unrameuler} (local choices for $v\nin S$), \ref{nonarchchoices-section} (local choices for $v\in S$), and \ref{pchoices-section} (local choices for $v\divides p$), and \ref{archimedeanchoices} (local choices for archimedean places).  Note that the choices at $p$ and $\infty$ depend on the signature of the unitary group $G_1$.  When $\rho$ is trivial, the Eisenstein series associated to $f(\chi, c) = f(\chi, \psi)$ is holomorphic; in the notation of \cite{apptoSHL}, it is (a normalization of) the algebraic automorphic form denoted $G_{k, \nu, \chi_{\unitary}, \psi}$ (which arises over $\OK$ but can be viewed over $\IC$ by extending scalars) in \cite[Equation (32)]{apptoSHLvv}.

\begin{rmk}\label{GtoE}  We use the notation $G_{k,\nu,\chi_{\unitary},\psi}$ below in order to cite the construction in \cite{EFMV}.  However, this notation designates one of the Eisenstein series introduced above.  More precisely, the section $f(\chi, \psi)$ is the Siegel section associated to $\chi_{\unitary}$, $k$, $\nu$, and $\psi$ in \cite{apptoSHLvv}, and the associated Eisenstein series $E_{f(\chi, \psi)}(\bullet)$ is the one denoted $E_{k, \nu}\left(\bullet, \chi, \psi, \frac{k}{2}\right)$ in \cite{apptoSHL}.  The Eisenstein series $E_{f(\chi, \psi)}(\bullet)$ is normalized by a factor $D(n, K, \mathfrak{b}, p, k)$ defined in \cite[Proposition 13]{apptoSHL} in order to cancel transcendental factors.  Note that although we do not include $(a, b)$ in the (already long) subscript for the Eisenstein series, the choice of $f(\chi, c)$ (and hence, the associated Eisenstein series) depends on the choice of $(a, b)$.
\end{rmk}

Like in Section  \ref{holoprojection-sec}, let $r_{1,\s} \geq \dots \geq r_{a_\s,\s} \geq r_{a_\s + 1,\s} = 0$, $s_{1,\s} \geq \dots \geq s_{b_\s,\s} \geq s_{b_\s + 1,\s} = 0$ be
descending sequences of integers.   Let $\rho^{\upsilon}_\sigma$ be the corresponding character on the torus $T_H$, and let
$$\tb_{i,\s} = r_{i,\s} - r_{i+1,\s}, i = 1, \dots , a_\s;  \tc_{j,\s} = r_{j,\s} - r_{j+1,\s}, j = 1, \dots , b_\s .$$
Define $\rho^{\upsilon} := \prod_{\sigma\in\Sigma}\rho^{\upsilon}_\sigma$ and
$\phi_\kappa:=\otimes_{\sigma\in\Sigma}p(\ub_\s,\uc_\s)$, with $p(\ub_\s,\uc_\s)$ defined as in Equation \eqref{plurihar} (and identified with a polynomial function on the tangent space of the moduli space).  
 So $p(\ub_\s,\uc_\s)$ is a homogeneous polynomial of degree $d(\sigma)$ for some nonnegative integer $d(\sigma)$.  
 
 Recall the $\ci$ differential operators $\delta^d_{\chi}(\ub,\uc) $ from Section  \ref{ci-do}.  These operators can be realized algebraic geometrically in terms of the Gauss--Manin connection and Kodaira--Spencer isomorphism, e.g. as in the main constructions in \cite[Chapter II]{kaCM}, \cite[\S6-\S9]{EDiffOps}, \cite[\S3]{emeasurenondefinite}, and \cite[\S3-\S6]{EFMV}.  The constructions in those references each build an algebraic differential operator $D$ (which gets applied to automorphic forms on $G_4$) by composing the Gauss--Manin connection and Kodaira--Spencer morphism.  The operator $\delta^d_{\chi}(\ub,\uc)$ can be realized algebraic geometrically by applying the operator $D$ iteratively $d$ times and then projecting onto the highest weight vector corresponding to the choice of highest weights corresponding to $(\chi, (\ub,\uc))$ (after also projecting modulo the anti-holomorphic subspace $H^{0, 1}$ of $H^1_{dR}$).  Each of those references also describes an analogous construction over the Igusa tower, but with $H^{0, 1}$ replaced by the unit root splitting, which yields a $p$-adic differential operator that we denote in the present paper by $\theta^{(\kappa, a, b)}$.  The operator $\theta^{(\kappa, a, b)}$ acts on $p$-adic automorphic forms (over the Igusa tower over ordinary locus of the Shimura variety associated to $G_4$) and outputs $p$-adic automorphic forms of higher weight.  In each case, the operator is applied to an automorphic form on $G_4$ and raises the weight of the automorphic form so that the output takes values in the space generated by the highest weight vector corresponding to the data $(\chi, (\ub,\uc))$.
It follows from \cite[Section  10]{EDiffOps} (which extends \cite[Lemma 5.1.27]{kaCM} to unitary groups) that $\theta^{(\kappa, a, b)}(f)$ and $\delta^d_{\chi}(\ub,\uc)(f)$ (for any algebraic automorphic form $f$) agree at ordinary CM points, up to periods.   See \eqref{padicCinfty} for a more precise statement.

From the $p$-adic $q$-expansion principle and the description of the $q$-expansion coefficients given in \cite[Section  3]{apptoSHLvv}, we obtain the following theorem (similar to \cite[Theorem 7.2.3]{EFMV}).

\begin{thm}[The Eisenstein Measure]\label{eismeasure-thm}
Recall the notation of Equation \eqref{iota}.  There is a measure $\Eisab$ (dependent on $a$ and $b$) on $X_p\times T_H\left(\ZZ_p\right)$ that takes values in the space of $p$-adic modular forms on $G_4$ and that satisfies
\begin{align*}
\int_{X_p\times T_H\left(\ZZ_p\right)}\tilde{\chi}\psi\cdot\rho^{\upsilon}\Eisab = \theta^{(\kappa, a, b)}\left(G_{k, \nu, \chi_{\unitary}, \psi}\right).
\end{align*}
whenever $(\chi,c = \psi\cdot\rho^{\upsilon}) \in Y_H^{class}$.  
\end{thm}
\begin{rmk}When $a_\sigma b_\sigma=0$ for all $\sigma\in\Sigma$ (i.e. in the definite case), the measure in Theorem \ref{eismeasure-thm} is the Eisenstein measure from \cite[Theorem 20]{apptoSHL} and \cite[Section  5]{apptoSHLvv}. 
\end{rmk}

\begin{cor}\label{eismeasureG3}  The measure $d\Eisab$, defined by
\begin{align*}
\int_{X_p\times T_H\left(\ZZ_p\right)}\tilde{\chi}\psi\cdot\rho^{\upsilon}d\Eisab = res_3\theta^{(\kappa, a, b)}\left(G_{k, \nu, \chi_{\unitary}, \psi}\right).
\end{align*}
is an axiomatic Eisenstein measure on values in $V(K_3^p,R)$, with shift $(1,\chi)$.
\end{cor}
\begin{proof} We need to compare the expression in Theorem \ref{eismeasure-thm} with the specifications required in Definition \ref{axiomeis}.   Bearing in mind the translation mentioned in Remark \ref{GtoE}, this comes down to comparing the action of $e_\kap \circ res_3\theta^{(\kappa, a, b)}$ with 
$e_\kap \circ\circ D^{hol}(\kappa,\chi)$.   But this follows from Proposition \ref{ordinaryplusclassical}.
\end{proof}

\begin{rmk}
The set $Y_H^{class}$ for the measure is determined by the conditions in \eqref{C3m2}, together with the relationships between $\chi$, $m$, $(r_{i, \s})_i$, and $(s_{j, \s})_j$ given in Equation \eqref{C3m} and Inequalities \eqref{parameters}.
\end{rmk}

\subsection{(Anti-) Ordinary representations and (anti-) ordinary vectors for $G_1$}\label{ordvec}

For this section, let $G = G_1$.  
For each prime $w\mid p$, let $G_w = \GL_n(\K_w)$.
Recall that by \eqref{G-iso} and \eqref{G-iso2}
there is an identification
\begin{equation}\label{G-iso-p}
G(\Qp) \isoarrow \Q_p^\times \times \prod_{w\in \Sigma_p} G_w.
\end{equation}
Let $B_w\subset \GL_n(K_w)$ be the (non-standard) Borel
consisting of elements $g= \left(\smallmatrix A & B \\ 0 & D\endsmallmatrix\right)$ with $A\in \GL_{a_w}(\K_w)$ upper-triangular and $D\in\GL_{b_w}(\K_w)$ 
lower-triangular. Let $T_w \subset B_w$ be its diagonal subgroup and $B_w^u\subset B_w$ its
unipotent radical. Let $I_{w,r}^0\subset \GL_n(\O_w)$ be the subgroup of elements 
$g$ such that $g\mod p^r =  \left(\smallmatrix A & B \\ 0 & D\endsmallmatrix\right)$ with $A\in \GL_{a_w}(\O_w/p^r\O_w)$ upper-triangular and
$D\in\GL_{b_w}(\O_w/p^r\O_w)$ lower-triangular
(this is the mod $p^r$ Iwahori subgroup relative to the Borel $B_w$). 
Let $I_{w,r}\subset I_{w,r}^0$ be the subgroup consisting of those $g$ such that $A$ and $D$ are unipotent.
Under the identification \eqref{G-iso-p} the subgroups $I_r\subset I_r^0$ of $G(\Zp)$ defined in Section  \ref{levelp} are identified as
\begin{equation}\label{Iwahori-ident}
I_r^0 \isoarrow \Z_p^\times \prod_{w\in\Sigma_p} I_{w,r}^0 \ \ \text{and} \ \  I_r\isoarrow \Z_p^\times\prod_{w\in\Sigma_p} I_{w,r}.
\end{equation}
Let $\delta_w:B_w \rightarrow \C$ be the modulus character: if $t =\diag(t_1,....,t_n) \in T_w$, then 
$\delta_w(t) = |t_1^{n-1} \cdots t_{a_w}^{b_w-a_w} t_{a_w+1}^{1-n}\cdots t_n^{b_w-1-a_w}|_p$.

\subsubsection{Ordinary  representations: local theory}\label{ordinary-localrep}
Let $\pi$ be a cuspidal holomorphic representation of $G (\A)$ of
weight type $(\kap,K)$ as in Section  \ref{autoreps} with $\kap = (\kap_\sigma)_{\sigma\in \Sigma_\K}$, $\kap_\sigma\in \Z^{a_\sigma}$, 
assumed to satisfy:
\begin{equation}\label{holo-wt-ineq}
\kap_{\sigma} + \kap_{\sigma c} \geq n, \ \ \forall \sigma\in\Sigma_\K.
\end{equation}
Let $\kap_{norm} = (\kap_{norm,\sigma})$ with $\kap_{norm,\sigma} = \kap_\sigma - b_\sigma$.

Via the identification \eqref{G-iso-p}, the $p$-constituent $\pi_p$ of $\pi$ is identified with a tensor product
$\pi_p \cong \mu_p\otimes_{w\in\Sigma_p} \pi_w$ with $\mu_p$ a character of $\Q_p^\times$ and
each $\pi_w$ an irreducible admissible representation of $G_w$.  

Recall that the Hecke operators $u_{w,j} = |\kap_{norm}(t_{w,j})|_p^{-1}U_{w,j}$, $w\in\Sigma_p$ and $1\leq j\leq n$,
act on the spaces $\pi_f^{K_r} = \pi_p^{I_r}\otimes(\otimes_{\ell\neq p}\pi_\ell)^{K^p}$ through an action on the spaces $\pi_p^{I_r}$:
$U_{w,j}$ acts on the latter spaces as 
the usual double coset operator $I_rt^+_{w,j} I_r$, and, furthermore,
the generalized eigenvalues of the $u_{w,j}$ are $p$-adically integral (cf.~Section  \ref{Heckeatp}; since $m=1$ the subscript $i$ 
has been dropped from our notation, following our conventions).
In particular, the ordinary projector
$e = \lim_{m\rightarrow \infty} (\prod_{w\in\Sigma_p}\prod_{i=1}^n u_{w,j})^{m!}$ acts on each $\pi_p^{I_r}$. 
From the identification $\pi_p = \mu_p\otimes_{w\in\Sigma_p} \pi_w$ and \eqref{Iwahori-ident} we find that
$u_{w,j}$ acts on $\pi_p^{I_r} = \otimes_{w\in\Sigma}\pi_w^{I_{w,r}}$ via the
action of the Hecke operator $u_{w,j}^\GL = |\kap_{norm}(t_{w,j})|_p^{-1} U_{w,j}^\GL$ on $\pi_w^{I_{w,r}}$, where
 $U_{w,j}^\GL$ acts as the double coset operator $I_{w,r} t_{w,j} I_{w,r}$; here, $t_{w,j} \in T_w$ is the element
defined in Section  \ref{Heckeatp}. It follows that 
 the generalized eigenvalues of the action of the Hecke operators $u_{w,j}^\GL$
 are $p$-adically integral, and $e_w = \lim_{m\rightarrow \infty} (\prod_{j=1}^n u_{w,j}^{\GL})^{m!}$ defines a projector
 on each $\pi_w^{I_{w,r}}$.

Suppose that $\pi$ is ordinary at $p$. Recall that this means $\pi_p^{I_r}\neq 0$ if $r\gg 0$ and that, for any such $r$, 
there is at least one vector $0\neq \phi \in \pi_p^{I_r}$ such that $e\cdot \phi = \phi$. We call such a $\phi$
an {\it ordinary vector}\footnote{But note that this notion depends {\it a priori} on the character $\kap_{norm}$, which in turn
depends on $\kap$ and the signatures $(a_\sigma,b_\sigma)_{\sigma\in \Sigma_\K}$. It turns out that there is
at most one $\kap_{norm}$ with respect to which a given $\pi_p$ can be ordinary, but in general the same $\pi_p$ can appear as local components for unitary groups with various signatures.} {\it for $\pi_p$}.  The existence of an ordinary vector is equivalent to the existence of 
 a $\phi\in \pi_p^{I_r}$, $r\gg 0$, that is a simultaneous eigenvector for the Hecke operators $u_{w,j}$ and having the property that 
 $u_{w,j}\cdot\phi = c_{w,j} \phi$ with $|c_{w,j}|_p =1$. It follows from the identification $\pi_p = \mu_p\otimes_{w\in\Sigma_p} \pi_w$ 
that $\pi_p$ being ordinary at $p$ is equivalent to $\mu_p$ being unramified and each $\pi_w$ being ordinary, in the sense that there 
exists $\phi_w\in \pi_w^{I_{w,r}}$,
$r\gg 0$, such that $e_w\cdot\phi_w = \phi_w$; we call such a $\phi_w$ an {\it ordinary vector for $\pi_w$}. The existence of an ordinary
vector for $\pi_w$ is equivalent to 
\begin{itemize}
\item[(a)] $\pi_w^{I_{w,r}}\neq 0$ for all $r\gg 0$;
\item[(b)] for each $r$ as in (a) there exists $0\neq \phi_w \in \pi_w^{I_{w,r}}$ such that
$\phi_w$ is a simultaneous eigenvector for the $u_{w,j}^\GL$, $1\leq j\leq n$, and having the 
property that $u_{w,j}^\GL\cdot \phi_w = c_{w,j}\phi_w$ with $|c_{w,j}|_p=1$.
\end{itemize}
Note that if $\phi_w\in \pi_w$, $w\in\Sigma_p$, are ordinary vectors and $\mu_p$ is unramified,
then $\phi=\otimes_{w\in\Sigma_p}\phi_w \in \pi_p$ is an ordinary vector for $\pi_p$.

\begin{lem}\label{ord-vect-lem-w} Let $w\in\Sigma_p$.
Suppose $\pi_w$ is an irreducible admissible representation of $G_w$ such that $\mathrm{(a)}$ and $\mathrm{(b)}$ above hold for a weight 
$\kap$ satisfying inequality \eqref{holo-wt-ineq}.
\begin{itemize}
\item[(i)] Up to multiplication by a scalar, there is a unique ordinary vector $\phi_w^\ord\in \pi_w^{I_{w,r}}$; $\phi_w^\ord$ is necessarily independent of $r\gg 0$.
\item[(ii)] There exists a unique character $\alpha_w : T_w \rightarrow \C^\times$ such that $\pi_w\hookrightarrow \Ind_{B_w}^{G_w} \alpha_w$ is the unique
irreducible subrepresentation and
$\phi_w^\ord$ is identified with the unique simultaneous $U_{w,j}^\GL$-eigenvector, $1\leq j \leq n$, with support containing $B_wI_{w,r}$, for $r\gg 0$.
$($In particular, $c_{w,j} = |\kap_{norm}(t_{w,j})|_p^{-1} \delta_w^{-1/2}\alpha_w(t_{w,j})$.$)$
\end{itemize}
\end{lem}

\begin{proof} Our proof is inspired in part by the arguments in \cite[\S5]{H98}.
Let $V$ be the space underlying the irreducible admissible representation $\pi_w$ of $G_w=\GL_n(\K_w)$,
and let $V_{B_w}$ be the Jacquet module of $V$ with respect to the unipotent radical $B_w^u$ of the Borel $B_w$. 
Let $N=\cap_r I_{w,r}$; this is just $B_w^u\cap \GL_n(\O_w)$. For each $j=1,...,n$, let
$$
t_j = \begin{cases} \diag(p1_j, 1_{n-j}) & j\leq a_w \\ 
\diag(p1_{a_w}, 1_{n-j}, p1_{j-a_w}) & j>a_w.
\end{cases}
$$
We let the double coset $U_j = Nt_jN$ act on $V^N=\cup_r V^{I_{w,r}}$ in the usual way: if $Nt_jN= \sqcup_i x_{i,j}N$ then 
$U_j\cdot v = \sum_i x_{i,j}\cdot v$. Then $U_j$ acts on the subspace $V^{I_r}$ as $U_{w,j}^\GL$. By the same arguments explaining
\cite[(5.3)]{H98}, $V^N$ decomposes as $V^N=V^N_{nil}\oplus V^N_{inv}$, where the $U_j$ act nilpotently on $V^N_{nil}$ and
are invertible on $V^N_{inv}$. Then, just as in \cite{H98}, the natural $B_w$-invariant projection $V\stackrel{v\mapsto \bar v}{\twoheadrightarrow} V_{B_w}$ induces an 
isomorphism 
\begin{equation}\label{Jacquetmod-iso}
V^N_{inv}\isoarrow V_{B_v}, \ \ v\mapsto\bar v,
\end{equation}
that is equivariant for the action of the $U_j$. 

Let $\phi\in V^{I_r}$ be an ordinary vector for some $r$: $\phi$ is an eigenvector for each 
$u_j = |\kap_{norm}(t_j)|_p^{-1}U_j$ with eigenvalue $c_j$ such that $|c_j|_p = 1$. In particular,
$\phi\in V^N_{inv}$. As $U_j$ acts on $V_{B_w}$ via $\delta_w(t_j)^{-1}t_j$, 
it then follows from \eqref{Jacquetmod-iso} that there must be a 
$B_w$-quotient 
$$
\iota:V_{B_w}\twoheadrightarrow \C(\lambda)
$$
with $\lambda:T_w\isoarrow B_w/B_w^u\rightarrow \C$ 
is a character such that $\lambda(t_j) = |\kap_{norm}(t_j)|_p\delta(t_j)c_j$ for all $j=1,..,n$. Let $\alpha = \lambda\delta^{-1/2}$
and let $I(\alpha)=\Ind_{B_w}^{G_w}(\alpha)$ be the unitary induction of $\alpha$ to a representation of $G_w$.  
By \cite[Thm.~3.2.4]{Casselman-book}, 
\begin{equation*}\label{Frob-ident}
\Hom_{G_w}(V,I(\alpha)) \isoarrow \Hom_B(V_B,\C(\lambda)),  \ \ \vphi \mapsto (\bar v\mapsto \vphi(v)(1)),
\end{equation*}
is an isomorphism, 
from which we conclude that there exists a non-zero $G_w$-homomorphism
$V\hookrightarrow I(\alpha)$, $v\mapsto f_v$ (which is necessarily an injection since $\pi_w$ is irreducible)
such that
\begin{equation}\label{Frobeval-ident}
\iota(\bar v) = f_v(1).
\end{equation}

By the characterization of $\lambda$, $\beta=|\kap_{norm}|_p^{-1}\delta_w^{-1}\lambda= |\kap_{norm}|_p^{-1}\delta_w^{-1/2}\alpha$
is a continuous character $T_w\rightarrow \C^\times$ such that each
$\beta(t_j)$ is a $p$-adic unit. From the definition of the $t_j$ it then follows easily
that $\beta(t)$ is a $p$-adic unit for all $t\in T_w$. 
Let $W$ be the Weil group of $T_w$ in $G_w$. For $x\in W$, let
$\beta_x = |\kap_{norm}|_p^{-1}\delta_w^{-1/2}\alpha^x$, where $\alpha^x(t)= \alpha(xtx^{-1})$. 
We claim that the values of $\beta_x$ are all $p$-adic units if and only if $x=1$. If the values
of $\beta_x$ are all $p$-adic units, then
$$
\beta_x/\beta^x(t) = |\kap_{norm}(xtx^{-1}t^{-1})|_p^{-1}\delta_w(xtx^{-1}t^{-1})^{-1/2}
$$
is a $p$-adic unit for all $t\in T_w$. As $\delta_w$ is the composition of $|\cdot|_p$ with an algebraic character
of $T_w$, it follows that the above values must all be $1$. That is, the character 
$\theta = |\kap_{norm}|_p\delta_w^{-1/2}$ satisfies $\theta^x = \theta$. Recall that if 
$\kap$ is identified with a dominant tuple $(\kap_0,(\kap_\sigma)_{\sigma\in\Sigma_\K})$
as in \eqref{dom-wt-ineq} then 
$$
\kap_{norm} (\diag(t_1,...,t_n)) = \prod_{\sigma\atop \grp_\sigma = \grp_w} \prod_{i=1}^{a_w} \sigma(t_i)^{\kap_{\sigma,i}-b_w}
\prod_{j=1}^{b_w} \sigma(t_{a_w+j})^{-\kap_{\sigma c,j}+a_w}.
$$
In particular, letting
$$
m_i = \begin{cases} \sum_{\sigma, \grp_\sigma=\grp_w} (\kap_{\sigma,i}-b_w) & i\leq a_w \\
-\sum_{\sigma, \grp_\sigma=\grp_w} (\kap_{\sigma c,i}-a_w) & i>a_w,
\end{cases}
$$
we have
$$
|\kap_{norm}(\diag(t_1,...,t_n))|_p = \prod_{i=1}^{n} |t_i|_p^{m_i}.
$$
It follows that 
$$
\theta(\diag(t_1,...,t_n)) = |t_1^{m_1+\frac{n-1}{2}} \cdots t_{a_w}^{m_{a_w}+\frac{b_w-a_w}{2}} t_{a_w+1}^{m_{a_w+1}+\frac{1-n}{2}}\cdots t_n^{m_n+\frac{a_w-2-b_w}{2}}|_p^{-1}.
$$
From the dominance of $\kap$ and the inequality \eqref{holo-wt-ineq} it follows that
$$
m_1\geq m_2 \geq \cdots \geq m_{a_w} \geq m_n \geq m_{n-1}\geq \cdots \geq m_{a_w+1},
$$
and so
$$
m_1+\frac{n-1}{2} > \cdots > m_{a_w} + \frac{b_w-a_w}{2} > m_n + \frac{b_w-2-a_w}{2} >\cdots > m_{a_w+1}+\frac{1-n}{2}.
$$
That is, $\theta$ is a regular character of $T_w$, and therefore $\theta^x = \theta$ if and only if $x=1$. This completes that proof that the values of
$\beta_x$ are all $p$-adic units if and only $x=1$.

As $\beta_x\neq \beta$ for all $x\neq 1$,  the characters $\alpha^x$, $x\in W$, are all distinct, and hence
the Jacquet module $I(\alpha)_{B_w}$ of $I(\alpha)$ is a semi simple $B_w$-module and isomorphic to the direct sum 
$\oplus_{x\in W} \C(\alpha^x\delta^{1/2})$ (cf. \cite[Prop.~5.4]{H98}). The inclusion $V\hookrightarrow I(\alpha)$, $v\mapsto f_v$,
induces a $B_w$-inclusion 
\begin{equation}\label{Jacquetmod-inject}
V_{B_w}\hookrightarrow I(\alpha)_{B_w} \cong \oplus_{x\in W} \C(\alpha^x\delta_w^{1/2}).
\end{equation}
It then follows from \eqref{Jacquetmod-iso} that $V^N_{inv}$ is a sum of one-dimensional simultaneous eigenspaces 
for the $U_j$ that are in one-to-one correspondence with those characters $\alpha^x\delta^{-1/2}$, $x\in W$, 
that appear in $V_B$ via \eqref{Jacquetmod-inject}; the eigenvalue of $u_j=|\kap_{norm}(t_j)|_p^{-1}U_j$ on the eigenspace corresponding
to $\alpha^x\delta^{1/2}$ is $\beta_x(t_j)$. As the values of $\beta_x$ are not all $p$-adic units if $x\neq 1$, it follows that
the space of ordinary vectors in $V$ is one-dimensional; this proves part (i).  
It further follows that the ordinary eigenspace must project non-trivially to $\C(\lambda) = \C(\alpha\delta^{1/2})$ 
via the composition of \eqref{Jacquetmod-iso} with $\iota$,  and that all other eigenspaces map to $0$ under this
composition. As this composition is just $v\mapsto f_v(1)$ by \eqref{Frobeval-ident}, part (ii) follows easily.
\end{proof}

\begin{coro}\label{ord-vect-cor-p} Suppose $\kap$ satisfies \eqref{holo-wt-ineq} and $\pi_p$ is ordinary. 
Up to multiplication by a scalar, there is a unique ordinary
vector $\phi^\ord \in \pi_p^{I_r}$ for $r\gg 0$; $\phi^\ord$ is necessarily independent of $r$. Furthermore, under the identification 
$\pi_p = \mu_p\otimes_{w\in\Sigma_p}\pi_w$,
$\phi^\ord = \otimes_{w\in\Sigma_p}\phi_w^\ord$, with $\phi_w^\ord$ as in Lemma \ref{ord-vect-lem-w}.
\end{coro}

The following lemma
will aid in the computation of certain local zeta integrals involving ordinary vectors.

\begin{lem}\label{ord-dual-lem-w} Let $w$, $\pi_w$, and $\kap$ be as in Lemma \ref{ord-vect-lem-w}. 
Let $\pi_w^\vee$ be the contragredient of $\pi_w$ and $\pair_w:\pi_w\times\pi_w^\vee \rightarrow \C$ the non-degenerate $G_w$-invariant pairing (unique
up to scalar multiple).
\begin{itemize}
\item[(i)] Let $\alpha_w$ be as in Lemma \ref{ord-vect-lem-w}(ii). Then $\pi_w^\vee$ is isomorphic to the unique irreducible quotient of
$\Ind_{B_w}^{G_w} \alpha^{-1}_w$: $\Ind_{B_w}^{G_w} \alpha_w^{-1} \twoheadrightarrow \pi_w^\vee$.
\item[(ii)] For $r\gg 0$, let $\phi^\vee_{w,r} \in \pi_w^\vee$ be the image of the vector in $\Ind_{B_w}^{G_w}\alpha_w^{-1}$ that is supported on $B_wI_r$.
Then $c(\pi_w,r):=\langle\phi^\ord_w,\phi^\vee_{w,r}\rangle_w$ is non-zero and depends only on $r$.
\item[(iii)] Let $0\neq\phi \in \pi_w^{I_r}$ with $e\cdot \phi = c(\phi)\phi_w^\ord$.  Then 
$$
\langle\phi,\phi^\vee_{w,r}\rangle_w = c(\phi)\langle \phi_w^\ord,\phi^\vee_{w,r}\rangle_w.
$$
\end{itemize}
\end{lem}

\begin{proof} Part (i) follows from the identification of 
$\Ind_{B_w}^{G_w} \alpha^{-1}_w$ as the contragredient of 
$\Ind_{B_w}^{G_w} \alpha_w$ (cf. \cite[Prop.~3.1.2]{Casselman-book}).  The pairing 
$\pair: \Ind_{B_w}^{G_w} \alpha_w\times\Ind_{B_w}^{G_w} \alpha^{-1}_w\rightarrow\C$ corresponding to this 
identification is just integration over $\GL_n(\O_w) \subset G_w$:
$$
\langle\varphi,\varphi'\rangle = \int_{\GL_n(\O_w)} \varphi(k)\varphi'(k)dk, \ \ \ 
\varphi\in \Ind_{B_w}^{G_w} \alpha_w, \varphi'\in\Ind_{B_w}^{G_w} \alpha^{-1}_w, 
$$
(cf. \cite[Prop.~3.1.3]{Casselman-book}). 
For part (ii), let $\varphi^\ord\in \Ind_{B_w}^{G_w} \alpha_w$ correspond to $\phi^\ord_w$ as in Lemma \ref{ord-vect-lem-w}(ii)
and let $\varphi^\vee_r\in \Ind_{B_w}^{G_w} \alpha^{-1}_w$ be the function supported on $B_wI_r$. 
Then 
$$
\langle \phi_w^\ord,\phi_{w,r}^\vee\rangle_w = \int_{\GL_n(\O_w)} \varphi^\ord(k)\varphi^\vee_r(k) dk.
$$
As $B_wI_r\cap\GL_n(\O_w) = I_r^0$, and since for $k = tk'\in I_r^0 = T_w(\O_w)I_r$ we have $\vphi^\ord(k)\vphi^\vee(k) = \alpha_w(t)\alpha_w^{-1}(t)=1$,
it then follows that 
$$
c(\pi_w,r):=\langle \phi_w^\ord,\phi_{w,r}^\vee\rangle_w=\int_{I_w^0} dk = \vol(I_r^0)\neq 0.
$$
This proves part (ii). 

For part (iii), write $\phi$ as a sum of simultaneous generalized $U_{w,j}^\GL$-eigenvectors:
$$
\phi = c(\phi)\phi_w^\ord + \sum_{i=1}^m \phi_i, \ \ e\cdot\phi_i = 0.
$$
Let $\vphi$ (resp.~$\vphi^\ord$, $\vphi_i$) be the function in $\Ind_{B_w}^{G_w} \alpha_w$
that corresponds to $\phi$ (resp.~$\phi_w^\ord$, $\phi_i$) as in Lemma \ref{ord-vect-lem-w}(ii). 
Then, for $r\gg 0$, $\vphi_i|_{I_r^0} = 0$, and so 
\begin{equation*}\begin{split}
\langle \phi,\phi^\vee_{w,r}\rangle_w & = \int_{\GL_n(\O_w)} \vphi(k)\vphi_r^\vee(k)  dk = \int_{I_r^0} \vphi(k)\vphi^\vee_{r}(k) dk \\
& = c(\phi) \int_{I_r^0} \vphi^\ord(k) \vphi^\vee_r(k) = c(\phi)\langle \phi_w^\ord,\phi^\vee_{w,r}\rangle_w.
\end{split}
\end{equation*}
\end{proof}

\subsubsection{Anti-ordinary representations: local theory}\label{aordinarylocal}
Let $\pi$ be an anti-holomorphic representation of $G(\A)$ of 
type $(\kap,K)$ as in \ref{autoreps} with $\kap$ satisfying the inequality \eqref{holo-wt-ineq}.  This the case if and only if 
$\pi^\flat$ is a cuspidal holomorphic representation of type $(\kap,K)$ as 
considered in the preceding section.

For each $r>0$ the Hecke operators $u_{w,j}^- = |\kap_{norm}(t_{w,j})|_p U_{w,j}^-$, $w\in\Sigma_p$ and $1\leq j\leq n$,
act on the space $\pi_f^{K_r} = \pi_p^{I_r}\otimes(\otimes_{\ell\neq p}\pi_\ell)^{K^p}$ through an action on the space $\pi_p^{I_r}$: $U_{w,j}^-$ 
acts on 
$\pi_p^{I_r}$ as the usual double coset operator $I_rt^-_{w,j} I_r$. Furthermore,
the generalized eigenvalues of the $u_{w,j}^-$ are $p$-adically integral. In particular, the {\it anti-ordinary} projector
$e^- = \lim_{m\rightarrow \infty} (\prod_{w\in\Sigma_p}\prod_{i=1}^n u_{w,j}^-)^{m!}$ acts on $\pi_p^{I_r}$. 
From the identification $\pi_p = \mu_p\otimes_{w\in\Sigma_p} \pi_w$ (via the isomorphism \eqref{G-iso-p}) we find that 
$u_{w,j}^-$ acts on $\pi_p^{I_r}  = \otimes_{w\in\Sigma_p} \pi_w^{I_{w,r}}$ via the action of the Hecke operator
$u_{w,j}^{\GL,-} = |\kap_{norm}(t_{w,j})|_p^{-1} U_{w,j}^{\GL,-}$, where
 $U_{w,j}^{\GL,-}$ acts as the double coset operator $I_{w,r} t_{w,j}^{-1} I_{w,r}$; here $t_{w,j}\in T_w$ is
 the element defined in Section  \ref{Heckeatp}.  It follows that the generalized eigenvalues of the action of
 the Hecke operators $u_{w,j}^{\GL,-}$ are $p$-adically integral and 
 $e_w^- = \lim_{m\rightarrow\infty} (\prod_{j=1}^n u_{w,j}^{\GL,-})^{m!}$ defines a projector
 on $\pi_w^{I_r}$.

We say that $\pi$ is {\it anti-ordinary at $p$ of level $r$ } if $\pi_p^{I_r}\neq 0$ and there exists $0\neq \phi\in\pi_p^{I_r}$ such that
$e^-\cdot\phi = \phi$. We say that such a $\phi$ is an {\it anti-ordinary vector for $\pi_p$ of level $r$, with respect to $I_r$}. Similarly, defining $I_r^{\flat} = I_r^{\dag}$, where the latter is as in Section \ref{MVWdag}, we can speak of {\it anti-ordinary vectors for $\pi_p$ of level $r$, with respect to $I_r^{\flat}$}.  Under the identification
$\pi_p = \mu_p\otimes_{w\in\Sigma_p}\pi_w$, the existence of an anti-ordinary vector of level $r$ 
in $\pi$ is equivalent to $\mu_p$ being unramified and, for each $w\in\Sigma_p$, there existing $0\neq \phi_w\in\pi_w^{I_{w,r}}\neq 0$
such that $e_w\cdot\phi_w = \phi_w$; we call such a $\phi_w$ an {\it anti-ordinary vector for $\pi_w$ of level $r$}.

\begin{lem}\label{aord-vect-lem-w} Let $w\in\Sigma_p$ and $\pi_w$ be a constituent of $\pi_p$ as above.
\begin{itemize}
\item[(i)] The representation $\pi_w$ is anti-ordinary of some level $r$ if and only if $\pi_w^\vee$ is ordinary, in which case $\pi_w$ is anti-ordinary of all levels
$r\gg 0$. 
\item[(ii)] If $\pi_w$ is anti-ordinary
of level $r$, then there exists
a unique (up to nonzero scalar multiple) anti-ordinary vector $\phi_{w,r}^{\aord}\in \pi_w^{I_{w,r}}$ of level r; 
it is characterized by $\langle\phi_{w,r}^\aord, \phi_w^{\vee,\ord}\rangle_w \neq 0$ 
and $\langle\phi_{w,r}^\aord, \phi\rangle_w= 0$ for all $\phi\in \pi_w^{I_{w,r}}$ belonging
to a generalized eigenspace of some $u_{w,j}^{\GL,-}$ with non-unit eigenvalue.
\end{itemize}
\end{lem}

\begin{proof} Suppose $\pi_w$ is anti-ordinary of some level $r$. Then $\pi_w^{I_{w,r}}\neq 0$ and there exists a simultaneous eigenvector
$\phi_{w,r}^\aord\in\pi_w^{I_{w,r}}$ for the $u_{w,j}^{\GL,-}$ with $p$-adic unit eigenvalues $a(j,r)$. 
Let $\pair_w:\pi_w\times\pi_w^\vee\rightarrow \C$ be the $G_w$-equivariant pairing. Then we have
\begin{equation}\label{ord-aord-eq-1}
a(j,r)\langle\phi_{w,r}^\aord,\phi\rangle_w = 
\langle u_{w,j}^{\GL,-} \cdot \phi_{w,r}^\aord, \phi\rangle_w = \langle \phi_{w,r}^\aord,u_{w,j}^\GL\cdot\phi\rangle_w
\end{equation}
for all $\phi\in \pi_w^{\vee,I_{w,r}}$. It follows that the action of each $u_{w,j}^\GL$ on $\pi_w^{\vee,I_{w,r}}$ has an
eigenspace with eigenvalue $a(j,r)$ (which is a $p$-adic unit). To see that there exists a simultaneous such eigenspace we use the 
commutativity of the $u_{w,j}^{\GL}$s: Let $V_{j-1}\subset \pi_w^{\vee,I_{w,r}}$ be a maximal subspace that is a simultaneous
eigenspace for $u_{w,1}^\GL,\ldots, u_{w,j-1}^\GL$ with respective eigenvalues $a(1,r),\ldots,a(j-1,r)$. Then by the commutativity
of the $u_{w,j}^\GL$s, the identity \eqref{ord-aord-eq-1} holds for all $\phi\in V_{j-1}$. In particular, there is a non-zero (maximal) 
subspace of $V_j\subset V_{j-1}$ which is an eigenspace for $u_{w,j}^\GL$ with eigenvalue $a(j,r)$. It follows from induction on $j$
that there exists a non-zero simultaneous $u_{w,j}^\GL$-eigenvector $\phi\in \pi_w^{\vee,I_{w,r}}$, $j=1,...,n$, with 
$p$-adic unit eigenvalues $a(j,r)$. 
That is, $\pi_w^\vee$ is ordinary. 

 Conversely, suppose that $\pi^\vee_w$ is ordinary, and let $\phi_w^{\vee,\ord}\in \pi_w^\vee$ be an ordinary vector with $u_{w,j}^{\GL}$-
eigenvalue $c(j)$ (which is a $p$-adic unit). Then for $r\gg 0$ we have
\begin{equation}\label{ord-aord-eq-2}
c(j) \langle \phi, \phi_w^{\vee,\ord}\rangle_w  =
\langle \phi, u_{w,j}^\GL \cdot\phi_w^{\vee,\ord}\rangle_w = \langle u_{w,j}^{\GL,-}\phi,\phi_w^{\vee,\ord}\rangle_w
\end{equation}
for all $\phi\in \pi_w^{I_{w,r}}$. 
It follows from the non-degeneracy of $\pair_w$ that there exists a $u_{w,j}^{\GL,-}$-eigenvector $\phi_{j,r}\in \pi_w^{I_{w,r}}$ with eigenvalue $c(j)$.
Using \eqref{ord-aord-eq-2} and the commutativity of the $u_{w,j}^{\GL,-}$ we find, as in the preceding proof of the ordinarity of $\pi_w^\vee$, that there exists
a non-zero simultaneous $u_{w,j}^{\GL,-}$-eigenvector $\phi\in \pi_w^{I_{w,r}}$, $j=1,...,n$, with 
$p$-adic unit eigenvalues $c(j)$. That is, $\pi_w$ is anti-ordinary of level $r$ for all $r\gg 0$.

Suppose now that $\pi_w$ is anti-ordinary of level $r$, and let $\phi_{w,r}^\aord\in \pi_w^{I_{w,r}}$ be an anti-ordinary vector of level $r$.
As shown above, $\pi_w^\vee$ is ordinary and $\phi_w^{\vee,\ord}\in \pi_w^{\vee,I_{w,r}}$. We note that 
$$
\pi_w^{\vee,I_{w,r}} = \C\phi_w^{\vee,\ord} \oplus V_1\oplus \cdots \oplus V_t
$$
with each $V_i$ a simultaneous generalized $u_{w,j}^\GL$-eigenspace with at least one of the (generalized) eigenvalues not a $p$-adic unit; 
this follows from the uniqueness of the ordinary vector (see Lemma \ref{ord-vect-lem-w}(i)). Since \eqref{ord-aord-eq-1}
holds for all $\phi\in V_i$ it follows that $\langle \phi_{w,r}^\aord, V_i\rangle_w = 0$. This proves that $\phi_{w,r}^\aord\in \pi_w^{I_{w,r}}$ is
characterized (up to non-zero scalar multiple) as stated in part (ii). The uniqueness also follows.

\end{proof}

Using this we can deduce an analog of Lemma \ref{ord-vect-lem-w}(ii):

\begin{lem}\label{aord-vect-lem-w-2} Let $w\in\Sigma_p$ and $\pi_w$ be a constituent of $\pi_p$ as above.
Suppose $\pi_w$ is anti-ordinary. Then there exists a unique character $\beta_w:T_w\rightarrow\C^\times$ such that $\Ind_{B_w}^{G_w}\beta_w\twoheadrightarrow \pi_w$ is the unique irreducible quotient and the anti-ordinary vector $\phi^\aord_{w,r} \in \pi_w^{I_{w,r}}$ of level $r$ is (up to non-zero scalar multiple) 
the image of the vector in $\Ind_{B_w}^{G_w}\beta_w$ with support $B_wI_{w,r}$.
In particular, the $\phi^\aord_{w,r}$, $r\gg 0$, can be chosen to satisfy
$$
\sum_{\gamma\in I_{w,r}/ (I_{w,r'}^0\cap I_{w,r})} \pi_w(\gamma)\phi_{w,r'}^\aord = \phi_{w,r}^\aord, \ \ r'\geq r.
$$
\end{lem}

\begin{proof} Since $\pi_w$ is anti-ordinary, it follows from Lemma \ref{aord-vect-lem-w}(i) that $\pi_w^\vee$ is ordinary. 
Let $\alpha_w$ be the unique character of $B_w$ associated with $\pi_w^\vee$ as in Lemma \ref{ord-vect-lem-w}(ii).
Let $\beta_w = \alpha_w^{-1}$. As $\pi_w^\vee$ is the unique irreducible subrepresentation of $\Ind_{B_w}^{G_w}\alpha_w$,
$\pi_w$ is the unique irreducible quotient of $\Ind_{B_w}^{G_w}\beta_w$. Furthermore, it follows from Lemmas \ref{ord-vect-lem-w}(ii) and \ref{ord-dual-lem-w}(ii-iii) that 
the image of the vector in $\Ind_{B_w}^{G_w}\beta_w$ that is supported on $B_wI_{w,r}$ satisfies the conditions that characterize $\phi^\aord_{w,r}$ in 
Lemma \ref{aord-vect-lem-w}(ii). The uniqueness of $\beta_w$ easily follows from the uniqueness of $\alpha_w$ and Lemma \ref{ord-dual-lem-w}.
\end{proof}

\begin{coro}\label{aord-vect-cor-p} Suppose $\kap$ satisfies Inequality \eqref{holo-wt-ineq}. Then $\pi_p$ is anti-ordinary
if and only if $\pi_p^\flat$ is ordinary, and  up to multiplication by a scalar, there is a unique anti-ordinary
vector $\phi^\aord_r \in \pi_p^{I_{w,r}}$ of level $r$ for each $r\gg 0$. Furthermore, under the identification 
$\pi_p = \mu_p\otimes_{w\in\Sigma_p} \pi_w$,
$\phi_r^\aord = \otimes_{w\in\Sigma_p}\phi_{w,r}^\aord$, with $\phi_{w,r}^\aord$ as in Lemma \ref{aord-vect-lem-w-2}.
\end{coro}

\begin{rmk}\label{aord-trace-remark}
The desired relation between the $\phi_{w,r}^\aord$ for varying $r$ can be made explicit by normalizing 
$\phi_{w,r}^\aord$ to be the image of the vector in $\Ind_{B_w}^{G_w}\beta_w$ with support $B_wI_{w,r}$
and value at $1$ equal to $1$.
\end{rmk} 

\begin{rmk}\label{aord-testvector-remark}
The description of the anti-ordinary vector $\phi_{w,r}^\aord\in\pi_w$ provided by Lemma \ref{ord-vect-lem-w-2} shows that for $r$ sufficiently large, $\phi_w=\phi_{w,r}^\aord\in\pi_w$ satisfies
the conditions \eqref{phiinvariancew} and \eqref{phi'invariancew2} with $R_w = B_w$ (and $n_{i,w}=1$ for all $i$). In particular, $\phi_{w,r}^\aord\in\pi_w$ is a suitable 
`test vector' for the calculations in \ref{maincalculation}.
\end{rmk}

\subsubsection{The Newton polygon}
Let $\pi$ be a holomorphic or anti-holomorphic cuspidal automorphic representation of $G(\A)$, and let $\pi_p = \mu_p\otimes_{w\in\Sigma_p}\pi_w$
be the identification corresponding to \eqref{G-iso-p}. We assume that 
\begin{equation}\label{subquot-w}
\text{each $\pi_w$ is an irreducible subquotient of $\Ind_{B_w}^{G_w} \beta_w$}
\end{equation}
for some character $\beta_w: T\rightarrow \C^\times$. We view $\beta_w$ as
$n$-tuple $\beta_w = (\beta_{w,1},....,\beta_{w,n})$ of characters of $\K_w^\times$, defined by
$\beta_w(\diag(t_1,...,t_n)) = \prod_{i=1}^n \beta_{w,i}(t_i)$; the characters $\beta_{w,i}$ are uniquely
determined up to order. 
We define the {\it total Hecke polynomial} of $\pi$ at $w$ to be 
\begin{equation}\label{heckepoly}
H_w(T) = \prod_{i = 1}^n (1 - \alpha_{w,i}(\varpi_w)T)(1- \alpha_{w,i}^{-1}(\varpi_w)T)
\end{equation}
The {\it Newton polygon} $\mathrm{Newt}(\pi,w)$ of $\pi$ at $w$ is the Newton polygon of $H_w(T)$.  
Note that 
$$
\mathrm{Newt}(\pi,w) = \mathrm{Newt}(\pi^\flat, w).
$$

Let $\Sigma_w = \{\sigma \in \Sigma_\K ~|~ \grp_{\sigma} = \grp_w\}$.  Let
$$\pi_{\Sigma_w} = \otimes_{\sigma \in \Sigma_w} \pi_\s = \otimes_{\s \in \Sigma_w} \Dc(\tau_\s)$$
in the notation of  \eqref{vermadual}.
Define the Hodge polygon $\mathrm{Hodge}(\pi,w)$ to be the polygon in the right half-plane with vertices
$(i,\sum_{\s\in\Sigma_{w}} p_{i,\s})$, where $(p_{i,\s},q_{i,\s})$ are the pairs introduced in Section  \ref{hodgepolygon} for $\Dc(\tau_\s)$.

\begin{prop}\label{panchishkin}  Suppose $\pi$ is (anti-)holomorphic and (anti-)ordinary.  Then 
$\mathrm{Newt}(\pi_w)$ and $\mathrm{Hodge}(\pi_w)$ meet at the midpoint $(n,\sum_{\s \in \Sigma_w} p_{i,\s})$.  
\end{prop}

In motivic terms, this says that the motive obtained by restriction of scalars to $\Q$
of the motive attached to $\Pi$ satisfies the {\it Panchishkin condition}, see \cite{pan}.
The proof is an elementary calculation and is omitted; it will not be used in what follows.   
Details will be provided in a future article, when the results obtained here are related
to standard conjectures on $p$-adic $L$-functions.

\subsection{(Anti-) Ordinary representations and (anti-) ordinary vectors for $G_2$}\label{ordvec-2}
If the group $G_1$ in Section \ref{ordvec} is replaced with $G_2$, then the analysis of ordinary and anti-ordinary
representations and vectors carries over with only a few changes. The most significant of these is that under
the identification \eqref{G-iso-p} the Borel $B_w$ and the groups $I_{w,r}^0$ and $I_{w,r}$ all get replaced
by their transposes $\t B_w$,  $\t I_{w,r}^0$, and $\t I_{w,r}$, respectively (more precisely, $B_w$ should be replaced by the opposite parabolic, which
is just the transpose in this case, and similarly for $I_{w,r}^0$ and $I_{w,r}$).   However, in order to compare with the test vectors
in Section \ref{localinducedreps} and subsequent calculations, we want vectors induced from $B_w$ and not $\t B_w = B_w^\op$. These are
obtained by composing with the standard intertwining operators between $\Ind_{B_w^\op}^{G_w}$ and $\Ind_{B_w}^{G_w}$. 

Suppose $\pi$ is a cuspidal holomorphic representation of $G_1(\A)$ of weight type $(\kap,K)$ as in Section \ref{ordinary-localrep}.
Let $\pi^\flat$ be as in Section \ref{MVWdag}. In particular, $\pi^\flat = \pi^\vee\otimes||\nu||^{a(\kap)}$.
As a representation of $G_2(\A)$, $\pi^\flat$ is cuspidal holomorphic of weight type $(\kap^\flat,K^\flat)$.
(However, as a representation of $G_1(\A)$ it is anti-holomorphic of this weight type!)

\subsubsection{Ordinary representations II: local theory}\label{ordinary-localrep-2}
Suppose that $\pi$ is ordinary at $p$. Then $\pi^\flat$ is also ordinary at $p$ (but with all the changes of conventions that come from
replacing $G_1$ with $G_2$; in particular, the role of the $U_{w,j}^{\GL}$ operator is now played by
$U_{w,j}^{\flat,\GL} = \t I_{w,j} t_{w,j}^{-1} \t I_{w,j}$). Let
$$
\pi^\flat_p = \mu_p^\flat\otimes_{w\in\Sigma_p}\pi_w^\flat
$$
be the decomposition of $\pi^\flat_p$ with respect to the identification \eqref{G-iso-p}.
Then $\pi_w^\flat = \pi_w^\vee$ and $\mu_p^\flat = \mu_p^{-1}|\nu|^{a(\kap)}_p$.

We have the following analogs of Lemma \ref{ord-vect-lem-w}
and Corollary \ref{ord-vect-cor-p}.

\begin{lem}\label{ord-vect-lem-w-2} Let $w\in\Sigma_p$.
Let $r$ be so large that $\pi_w^{\flat,\t I_{w,r}}\neq 0$ (equiv, $\pi_w^{I_{w,r}}\neq 0$).
\begin{itemize}
\item[(i)] Up to multiplication by a scalar, there is a unique ordinary vector $\phi_w^{\flat,\ord}\in \pi_w^{\flat, \t I_{w,r}}$; 
$\phi_w^{\flat,\ord}$ is necessarily independent of $r\gg 0$.
\item[(ii)] There exists a unique character $\alpha^\flat_w : T_w \rightarrow \C^\times$ such that $\Ind_{B_w}^{G_w} \alpha^\flat_w\twoheadrightarrow\pi_w^\flat$ is the unique
irreducible quotient and
$\phi_w^{\flat,\ord}$ is identified with the image of the simultaneous $U_{w,j}^{\flat,\GL}$-eigenvector, $1\leq j \leq n$, with support $B_w \t I_{w,r}$, for $r\gg 0$.
$($In particular, the $u_{w,j}^{\flat,\GL}$-eigenvalue is $c_{w,j}^\flat = c_{w,j}$.)
Furthermore, if $\alpha_w$ is the character as in Lemma \ref{ord-vect-lem-w}(ii), then $\alpha_w^\flat = \alpha_w^{-1}$.
\end{itemize}
\end{lem}

\begin{proof}
The map $\Ind_{B_w}^{G_w}\alpha \rightarrow \Ind_{\t B_w}^{G_w} \alpha^{-1}$, $\phi(g)\mapsto \phi^\vee(g) = \phi(\t g^{-1})$ realizes $\pi_w^\flat = \pi^\vee$
as the image of $\pi_w$ and hence a subrepresentation of $\Ind_{\t B_w}^{G_w} \alpha_w^{-1}$.  As $\phi^\vee\in\pi_w^\flat$ is ordinary if and only if $\phi\in\pi_w$ is, 
part (i) follows immediately, and $\phi_w^{\flat,\ord} = (\phi_w^\ord)^\vee$. Part (ii) follows from noting that $\pi^\flat$ is then the image of the 
standard interwining operator $\Ind_{B_w}^{G_w}\alpha_w^{-1} \rightarrow  \Ind_{B_w^\op}^{G_w} \alpha_w^{-1}$.
The determination of the support of the eigenvector in $\Ind_{B_w}^{G_w}\alpha_w^{-1}$ is an easy computation.
\end{proof}

\begin{coro}\label{ord-vect-cor-p-2} 
Up to multiplication by a scalar, there is a unique ordinary
vector $\phi^{\flat,\ord} \in \pi_p^{\flat, \t I_r}$ for $r\gg 0$; $\phi^{\flat,\ord}$ is necessarily independent of $r$. Furthermore, under the identification 
$\pi^\flat_p = \mu^\flat_p\otimes_{w\in\Sigma_p}\pi^\flat_w$,
$\phi^{\flat,\ord} = \otimes_{w\in\Sigma_p}\phi_w^{\flat,\ord}$, with $\phi_w^{\flat,\ord}$ as in Lemma \ref{ord-vect-lem-w-2}.
\end{coro}

\subsubsection{Anti-ordinary representations II: local theory}\label{aordinary-localrep-2}
Suppose that $\pi$ is anti-ordinary at $p$. Then $\pi^\flat$ is also anti-ordinary at $p$ (but again with all the changes of conventions that come from
replacing $G_1$ with $G_2$; in particular, the role of the $U_{w,j}^{\GL,-}$ operator is now played by
$U_{w,j}^{\flat,\GL,-} = \t I_{w,j} t_{w,j} \t I_{w,j}$). 

As in the ordinary case, we have the following analogs of Lemma \ref{aord-vect-lem-w-2} and Corollary \ref{aord-vect-cor-p}:

\begin{lem}\label{aord-vect-lem-w-3} Let $w\in\Sigma_p$. Suppose $\pi_w^\flat$ is anti-ordinary with respect to $\t I_{w,r}$ (equivalently,
$\pi_w$ is anti-ordinary with respect to the $I_{w,r}$).
There exists a unique character $\beta^\flat_w:T_w\rightarrow\C^\times$ such that $\pi_w^\flat \hookrightarrow \Ind_{B_w}^{G_w}\beta_w^\flat$ is the unique 
irreducible subrepresentation and the anti-ordinary vector $\phi^{\flat,\aord}_{w,r} \in \pi_w^{\flat, \t I_{w,r}}$ of level $r$ is (up to non-zero scalar multiple) 
the unique similtaneous $U_{w,j}^{\flat,\GL,-}$-eigenvector in $\Ind_{B_w}^{G_w}\beta_w^\flat$ with support containing $B_w\t I_{w,r}$.
In particular, the $\phi^{\flat,\aord}_{w,r}$, $r\gg 0$, can be chosen to satisfy
$$
\sum_{\gamma\in \t I_{w,r}/ (\t I_{w,r'}^0\cap \t I_{w,r})} \pi^\flat_w(\gamma)\phi_{w,r'}^{\flat,\aord} = \phi_{w,r}^{\flat, \aord}, \ r'\geq r.
$$
Furthermore, if $\beta_w$ is as in Lemma \ref{aord-vect-lem-w-2}(ii), then $\beta_w^\flat = \beta_w^{-1}$.
\end{lem}

\begin{proof}  Just as for the ordinary case in the proof of Lemma \ref{aord-vect-lem-w-2}, 
$\phi_{w,r}^{\flat,\aord} = (\phi_{w,r}^\aord)^\vee$ is anti-ordinary. Furthermore, it is identified (up to non-zero scalar multiple)
with the image under the $\Ind_{B_w^\op}^{G_w} \alpha_w^{-1}\rightarrow \Ind_{B_w}^{G_w}\alpha_w^{-1}$
of the function $\phi_r'\in \Ind_{B_w^\op}^{G_w} \alpha_w^{-1}$  that is supported on $B_w^\op\t I_{w,r}$ and takes the value $1$ on $1$.  If $\phi_{w,r}^{\flat,\aord}$ is normalized to be equal to the image 
of $\phi_r'$, then $\phi_{w,r}^{\flat,\aord}$ satisfies the trace relation (since $\phi_r'$ does).
\end{proof}

\begin{coro}\label{aord-vect-cor-p-2} 
Let $I_r^\flat = I_r^\dag$.
There is a unique anti-ordinary
vector $\phi^{\flat,\aord}_r \in \pi_p^{\flat,I_r^\flat}$ of level $r$ for each $r\gg 0$. Furthermore, under the identification 
$\pi^\flat_p = \mu_p^\flat\otimes_{w\in\Sigma_p} \pi^\flat_w$,
$\phi^{\flat,\aord}_{-r} = \otimes_{w\in\Sigma_p}\phi_{w,r}^{\flat,\aord}$, with $\phi_{w,r}^{\flat,\aord}$ as in Lemma \ref{aord-vect-lem-w-3}.
\end{coro}

\begin{rmk}\label{aord-trace-remark-2}
Under the normalization in the proof of Lemma \ref{aord-vect-lem-w-3},
$\phi_{w,r}^{\flat,\aord}$ is indentified with a function in $\Ind_{B_w}^{G_w}\beta_w^\flat$ whose value on $1$
is $\Vol(\t I_{w,r}^0)$.
\end{rmk}

\begin{rmk}\label{testvector-remarkII}
The description of the anti-ordinary vector $\phi_{w,r}^{\flat,\aord}\in\pi_w^\flat$ provided by Lemma \ref{ord-vect-lem-w-2} shows that for $r$ sufficiently large, $\tilde\phi_w=\phi_{w,r}^{\flat,\aord}\in\pi_w^\flat$ satisfies
the conditions \eqref{tphiinvariancew} with $R_w = B_w$ (and $n_{i,w}=1$ for all $i$). In particular, $\phi_{w,r}^{\flat,\aord}\in\pi_w$ is also a suitable 
`test vector' for the calculations in \ref{maincalculation}.
\end{rmk}

The next lemma will be a crucial ingredient in our interpretation of the our local zeta integral formulas. Recall that $\pair_w$
is the canonical pairing on $\pi_w\times\pi_w^\flat$ (using
$\pi^\flat = \pi^\vee$). 

We continue with the hypotheses of Lemma \ref{aord-vect-lem-w-3},
namely that $\pi_w^\flat$ is anti-ordinary with respect to $\t I_{w,r}$ and,equivalently, $\pi_w$ is anti-ordinary with respect to the $I_{w,r}$. Note that this implies that $\pi_w = (\pi_w^\flat)^\vee$ is ordinary with respect to the Borel $\t B_w$ and
$\pi_w^\flat = \pi_w^\vee$ is ordinary with respect to the Borel $B_w$. We will write $\phi_{\pi_w}^{\flat,\ord} \in \pi_w^{\t I_{w,r}}$
and $\phi_{\pi_w^\flat}^\ord \in \pi_w^{\flat,I_{w,r}}$ for these ordinary vectors: the first is the vector as in Lemma \ref{ord-vect-lem-w-2} but with $\pi_w^\flat$ replaced with
$(\pi_w^\flat)^\vee=\pi_w$), and the second is the vector as in Lemma \ref{ord-vect-lem-w} but with $\pi_w$ replaced with
$\pi_w^\vee=\pi_w^\flat$).  
We assume that  $\phi_{\pi_w}^{\flat,\ord}$ takes the value 
the value $1$ at $1$ and 
that $\phi_{\pi_w^\flat}^\ord$ takes the value $\Vol(I_{w,r}^0\t I^0_{w,r})$  (which is independent of $r$) at $1$.
We assume that the anti-ordinary vectors are normalized as in Remarks \ref{aord-trace-remark} and \ref{aord-trace-remark-2}

\begin{lem}\label{local-pairing-lem} 
With the preceding conventions,
$$
\frac{\langle \phi_{w,r}^\aord, \phi_{w,r}^{\flat,\aord}\rangle_w}{\Vol(I_{w,r}^0\cap\t I_{w,r}^0)}
 = \frac{\langle \phi_{\pi_w}^{\flat,\ord},\phi_{w,r}^{\flat,\aord}\rangle_w}{\Vol(\t I_{w,r}^0)}
 = \frac{\langle \phi_{w,r}^\aord, \phi_{\pi_w^\flat}^\ord\rangle_w}{\Vol(I_{w,r}^0)}.
$$ 
In particular, the left-hand side is independent of $r$.
\end{lem}

\begin{proof} We can identify $\pair_w$ with the pairing induced by $\Ind_{B_w}^{G_w}(\beta_w)\times\Ind_{B_w}^{G_w}(\beta_w)^{-1}$ given by $(\vphi,\vphi')\mapsto \int_{\GL_n(\O_w)}\vphi(k)\vphi'(k)dk$.
The lemma is then a straightforward calculation: with this normalization of the pairing, each ratio equals $\Vol(I_{w,r}^0\t I_{w,r}^0)$ (which is independent of $r$).
\end{proof}

\subsection{Global consequences of the local theory}\label{globalvectors}
Let $\pi$ be an anti-ordinary cuspidal anti-holomorphic representation $\pi$ of $G_1$ of weight type $(\kappa,K_r)$, $K_r = K^pI_r$.
Let $\pi^\flat$ be as in Section \ref{MVWdag}. Viewed as a representation of $G_2$, $\pi^\flat$ is also 
an anti-ordinary cuspidal anti-holomorphic representation of weight type $(\kappa^\flat,K_r^\flat)$, $K_r^\flat =  (K^p)^\dag \t I_r$.
But viewed as a representation of $G_1$, $\pi^\flat$ is an ordinary holomorphic representation.

Suppose $\pi$ satisfies the Gorenstein, Minimality, and Global Multiplicity One Hypotheses of section 
\ref{BIG}. Let $S$ be the set of finite primes, not dividing $p$, at which $K^p$ is not hyperspecial maximal. 
We summarize the implications of the local theory for the identification of automorphic
forms in $\pi$.  Let $I_{\pi}$ and $\hat{I}_{\pi}$ be as in Section  \ref{locrep}.  
We say that the anti-holomorphic cuspidal representation $\pi'$ of $G_1$ is {\it in the family determined by $\pi$}
if there is a non-trivial character $\lambda_{\pi'}$ of the Hecke algebra $\TT = \TT_{\pi}$ defining the action of the unramified
Hecke operators on $\pi'$.   Any such $\pi'$ is assumed to be given with a factorization \eqref{factor}.  The 
factors $\pi'_w$, for $w \mid p$, are all (tempered) subquotients of principal series representations.

In what follows, the Borel subalgebras $\mathfrak{b}^+_{\sigma}$ are chosen at archimedean places $\sigma$ as in Section  \ref{holoreps}.
The Minimality Hypothesis allows us to  choose $v_S$ uniformly for $\pi'$ in the following proposition.

\begin{prop}\label{tracefac}  Fix an element 
$v_S \in \hat{I}_{\pi}$. 
Let $\pi'$ be any anti-holomorphic representation, of type $(\kappa',K_{r'})$, in the family determined by $\pi$.
Let $\varphi_{\kappa',-}$ denote a 
lowest weight vector in the anti-holomorphic subspace of $\pi'_{\infty}$, as in \eqref{facphiinf}.  For a finite prime $v \notin S \cup \Sigma_p$,
let $\varphi'_v$ be a fixed generator of the spherical subspace of $\pi'_v$ and let $\varphi^{\prime,\flat}_v$ be the dual
generator of the spherical subspace of $\pi^{\prime,\flat}_v$.   Assume $\kappa$ satisfies \eqref{holo-wt-ineq}.
Then

\begin{enumerate}
\item For $r' >>0$, there is, up to scalar multiple, a unique anti-ordinary anti-holomorphic vector $\varphi_{r'}(v_S,\pi') \in (\pi')^{K_{r'}}$
with factorization \eqref{factor} given by
$$fac_{\pi^{\prime,\flat}}(\varphi_{r'}(v_S,\pi')) 
= \varphi_{\kappa',-} \otimes \otimes_{v \notin S \cup \Sigma_p} \varphi'_v \otimes \otimes_{w \mid p} \phi^{a-ord}_{w,r'} \otimes v_S.$$

\item  As $r'$ varies, the $\varphi^{r'}(v_S,\pi') \in \pi'$ can be chosen so that, if $r"\geq r' >>0 $, then
$$
\frac{\#(I_{r'}^0/I_{r'})}{\#(I_{r''}^0/I_{r''})}\sum_{\gamma\in I_{r'}/ I_{r''}} \gamma\cdot \varphi^{r''}(v_S,\pi') = \sum_{\gamma\in I_{r'}/ (I_{r''}^0\cap I_{r'})} \gamma\cdot \varphi^{r''}(v_S,\pi') = \varphi^{r'}(v_S,\pi').
$$
\end{enumerate}
\end{prop}

\begin{proof}  This follows directly from the results in the previous sections, in particular 
Lemmas \ref{aord-vect-lem-w-2} and \ref{ord-vect-lem-w}.  
\end{proof}

\begin{rmk}\label{tracefac-rmk}
To ensure property (2) we adopt the normalization for $\phi^{\aord}_{w,r''}$ described in Remark \ref{aord-trace-remark}.
\end{rmk}

Similarly, from the results in Section \ref{ordvec-2} we deduce that after fixing a $v_S^\flat \in I_\pi = \hat I_{\pi^\flat}$ and letting $\varphi_{(\kappa')^\flat,-}$ denote a 
lowest weight vector in the anti-holomorphic subspace of $(\pi')^\flat_{\infty}$, there is,
up to scalar multiple, a unique anti-ordinary anti-holomorphic vector 
$\varphi_{r'}(v_S^\flat,(\pi')^\flat) \in (\pi')^{\flat,K_{r'}^\flat}$
satisfying the analogs of properties (1) and (2).

For the most general statement, we introduce a twisting character $\chi$ as in \eqref{chitwist} and Section \ref{periodtwist}.  In the following Lemma, the module $H^{d,\ord}_{\kap^\flat}(K_{r'}^\flat, \psi^{',-1},R)_\chi$  and the period 
$Q[\pi',\chi]$ are as defined in Section \ref{normper}.

\begin{lem}\label{Omega-lem}
Let $\chi$ be a Hecke character of type $A_0$.  The ratio 
$$
\frac{\la \varphi_{r'}(v_S,\pi'), \varphi_{r'}(v_S^\flat,(\pi')^\flat)_\chi \ra_{\pi',\chi}}{\Vol(I_{r',V}^0\cap I_{r',-V}^0)} 
$$
is independent of $r'$.  If $\varphi_{r'}(v_S,\pi') \in H^{d,\ord}_\kap(K_{r'},\psi',R)$ and 
$\varphi_{r'}(v_S^\flat,\pi') \in H^{d,\ord}_{\kap^\flat}(K_{r'}^\flat, \psi^{',-1},R)_\chi$, 
then its value is in $R \cdot Q[\pi',\chi]$, and for appropriate choices of $\varphi_{r'}(v_S,\pi')$ and $\varphi_{r'}(v_S^\flat,\pi')$ it is a unit
multiple of $Q[\pi',\chi]$.
\end{lem}

\begin{proof} The independence of $r'$ is a simple consequence $\varphi_{r'}(v_S,\pi')$ and $\varphi_{r'}(v_S^\flat,(\pi')^\flat$ being anti-ordinary vector and Lemma \ref{local-pairing-lem}. 
The remaining claims of the lemma are consequences of the definition of $Q[\pi',\chi]$.
\end{proof}

\section{Construction of $p$-adic $L$-functions}\label{lastchapter}

\subsection*{Review of notation} We recall the notation from the previous sections, because some of it is admittedly counterintuitive.  Our basic Shimura varieties are denoted $Sh(V)$ (attached to $G_1$) and $Sh(-V)$ (attached to $G_2$, which is isomorphic to $G_1$).   Classical points of our Hida families correspond to cuspidal automorphic representations denoted $\pi$ and $\pi^{\flat}$, for $Sh(V)$ and $Sh(-V)$, respectively.  With our conventions, $\pi$ is an {\it anti-holomorphic} automorphic representation of $G_1$, and therefore with respect to the isomorphism $G_2 \isoarrow G_1$ is a {\it holomorphic } automorphic representation of $G_2$.  Correspondingly, $\pi^{\flat}$, which can be identified with the complex conjugate of $\pi$, is a holomorphic representation of $G_1$, and therefore gives rise to a holomorphic modular form -- of weight $\kappa$, in practice -- on $Sh(V)$; but $\pi^{\flat}$ is anti-holomorphic on $G_2$.  The input of the doubling integral is an {\it anti-holomorphic} vector on $G_3$ which comes from a vector $w \in \pi \otimes \pi^{\flat}$, that will be identified shortly; this is paired with the Eisenstein measure, which takes values in the ring of $p$-adic modular forms on $G_4$ and which specializes to classical forms of weight $\kappa \otimes \kappa^{\flat}$ on $G_3$.  We always assume that $\pi$ and $\pi^{\flat}$ are {\it anti-ordinary} at all primes dividing $p$; in particular, the vector $w$ has local components at $p$ that are chosen to be anti-ordinary.  

Since one is in the habit of thinking of Hida theory as a theory of families of holomorphic and ordinary forms, the following lemma may be welcome; in any case, it is implicit in the assumption that both $\pi$ and $\pi^{\flat}$ are anti-ordinary.

\begin{lem}\label{ord-unitary-lem}
  Suppose $\pi$ is an anti-ordinary and anti-holomorphic representation of $G_1$.  Then the $p$-adic component $\pi_p$ of $\pi$ is also ordinary.
\end{lem}

\begin{proof}  The property of being ordinary is preserved under complex conjugation, and by twist by a power of the similitude character.  On the other hand, duality exchanges ordinary with anti-ordinary representations, by Lemma \ref{aord-vect-lem-w}.    Since $\pi$ is essentially unitary, it follows that it is both ordinary and anti-ordinary.
\end{proof}

More precisely still, the anti-ordinary subspace (or submodule) of $\pi\otimes \pi^{\flat}$ is denoted $\hat{I}_{\pi}\otimes \hat{I}_{\pi^{\flat}}$.  However, it is best to view 
$\hat{I}_{\pi}\otimes \hat{I}_{\pi^\flat}$ as a trace compatible system 
\begin{equation}\label{tracecompatible} w_r \in  \hat{S}^{\ord}_{\kappa,V}(K_r,R)[\pi]\otimes \hat{S}^{\ord}_{\kappa^{\flat},-V}(K_r^\flat,R)[\pi^{\flat}]; ~~ \iota^*_r(w_{r+1}) = w_r,
\end{equation}
with notation as in Lemma \ref{trace2}.  Thus, in what follows, $\varphi\otimes \varphi^{\flat} \in \pi\otimes \pi^{\flat}$, viewed as an anti-holomorphic form of level $K_r\times K_r^\flat$ on $Sh(V)\times Sh(-V)$,
is taken to belong in $\hat{I}_{\pi}\otimes \hat{I}_{\pi^{\flat}}$, which we now identify (with respect to the factorization \eqref{facpi}) with the subspace
\begin{equation}\label{aordholomorphic} \bigotimes_{w \mid p} [\phi^{\aord}_{w,r}\otimes \phi^{\flat,\aord}_{w,r}]\otimes \bigotimes_{\sigma \mid \infty} [\varphi_{\kappa_{\sigma},-}\otimes \varphi_{\kappa_{\sigma}^{\flat},-} ] \otimes \pi_{S^p}^{K^p}\otimes \pi_{S^p}^{\flat,\bar{K}^p} \subset \pi\otimes \pi^{\flat}.
\end{equation}
In other words, these test vectors have local components as in \eqref{facphi},  \eqref{facphip}, and \eqref{facphiinf}.
(See also Section \ref{globalvectors} for how we identify anti-holomorphic, anti-ordinary cuspforms with
elements of $\hat I_\pi \otimes \hat I_{\pi^\flat}$.)
Moreover, we take our vector $\varphi\otimes \varphi^{\flat}$ to be integral over $\CO = \CO_\pi$.  By our choice in \eqref{aordholomorphic}, this is then 
the anti-ordinary vector $w = w_r \in \pi\otimes \pi^{\flat}$ to which we referred above.   

Note that the choice of $\varphi\otimes \varphi^{\flat}$, and therefore of $w_r$, depends on the level $K_r$; 
however, the corresponding system $\{w_r\}$ satisfies the trace compatibility relation \eqref{tracecompatible} by Lemmas \ref{aord-vect-lem-w-2}  and \ref{aord-vect-lem-w-3} and Proposition \ref{tracefac}.  In particular, the value of the (normalized) pairing with the Eisenstein measure is independent of this choice, and we can specifically take $r = d \geq 2t$ as in \eqref{dgeq2t}, and as required for the local calculation at primes dividing $p$.

\subsection{Pairings of axiomatic Eisenstein measures with Hida families}

We now apply the considerations of Section  \ref{classpairings} to the integral over $G_3$.  
Given a fixed Hecke character $\chi$, we let the parameters $\kappa$, $\rho$, $\rho^{\upsilon}$ determine one another as in \eqref{parameters}, \eqref{iota}.
Let $\phi_{r,\rho}$ be as in Statement \ref{meas2ee} of Lemma \ref{meas}, a measure on $T_H(\Zp)$ of type $\chi$ for some $p$-adic
Hecke character $\chi$ of $X_p$.  Choose $\Xi \in C_r(T_H(\Zp),R)\rho^{\upsilon} \subset C_r(T_H(\Zp),\CO)\rho^{\upsilon}$
so that (cf. \eqref{phirkap})  
\begin{equation}\label{phirkap2} \phi_{r,\rho}(\Xi) \in S^{\ord}_{\kappa,V}(K_r,R) \otimes S^{\ord}_{\kappa^{\flat},-V}(K_r^\flat,R)\otimes \chi\circ\det,  \end{equation}
For $\varphi \otimes \varphi^{\flat} \in [\hat{I}_{\pi}\otimes \hat{I}_{\pi^{\flat}}] \subset \pi\otimes \pi^{\flat}$ we define (in the obvious notation) 
$$L_{\varphi\otimes \varphi^{\flat}}(\phi_{r,\rho})(\Xi)$$
using the normalized canonical pairing \eqref{pairC-def} of 
$S^{\ord}_{\kappa,V}(K_r,R) \otimes S^{\ord}_{\kappa^{\flat},-V}(K_r^\flat,R)\otimes \chi\circ\det $ with 
\begin{equation}\label{hdhd}
H^{d,\ord}_{\kap^D}(K_r,R)[\pi]\otimes H^{d,\ord}_{\kap^{\flat,D}}(K_r^\flat,R)[\pi^{\flat}]\otimes\chi^{-1}\circ\det \\ \simeq [\hat{I}_{\pi}\otimes \hat{I}_{\pi^{\flat}}]\otimes\chi^{-1}\circ\det
\end{equation}
applied to $\phi_{r,\rho}(\Xi)$ as in \eqref{phirkap2} and $\varphi \otimes \varphi^{\flat}\otimes \chi^{-1}\circ \det$ as in \eqref{hdhd} 
(the characters $\chi$ and $\chi^{-1}$ cancel in the obvious way).

We apply this to the measure $Eis_{r,\rho,\chi}$ attached to 
$$\Xi \mapsto \int_{X_p \times T_H(\Zp)} (\chi,\Xi)  dEis$$ 
by Lemma \ref{meas}, with $dEis$ an axiomatic Eisenstein measure
as in Section  \ref{existence-Eisenstein}.     First, we need to show that the discussion in Section  \ref{classpairings} applies to this situation.  

\subsubsection{Equivariance of the Garrett map}

If $\lambda:  \Tb_{K,\kap,R} \rar \C$ is a character, let $\lambda^{\flat}(T) = \lambda(T^{\flat})$, where $^{\flat}$ is
the involution defined in Lemma \ref{Hecke-isoms} (ii).  

\begin{lem}\label{conjugatelambda}  Let $\pi$ be a cuspidal automorphic representation of $G$ of type $(\kap,K)$.   Then
$$\lambda_{\bar{\pi}} = \lambda_{\pi}^{\flat}.$$
\end{lem}

\begin{proof}  At unramified places the identity follows from \eqref{flatflat}.  By  Hypothesis \ref{multone-pi} and strong multiplicity one, applied to the base change to $GL(n)_\K$, this in turn implies that the local components of $\bar{\pi}$ and $\pi^{\flat}$ are isomorphic at all places split in $\K/\K^+$.  In particular, we have this isomorphism at  places dividing $p$; then the identity is a consequence of the uniqueness of the ordinary and anti-ordinary eigenspaces and the definition of the involution $^\flat$.
\end{proof}

Let $\pi$ be cuspidal of type $(\kap,K)$, and let $\varphi \in \pi^K$ be an anti-holomorphic vector.  We pick a Hecke character
$\chi$ as in Section  \ref{induced}.
In Section  \ref{zetaintegral} we defined the zeta integral
$$
I(\varphi,\varphi',f,s) = \int_{Z_3(\adeles)G_3(\Q)\backslash G_3(\A)} E_f(s,(g_1,g_2))\chi^{-1}(\det g_2)\varphi(g_1)\varphi'(g_2)d(g_1,g_2).
$$
where $\varphi' \in \bar{\pi}$ and $E_f(s,g_1,g_2)$ is an Eisenstein series depending on
a section $f \in I(\chi,s)$.   We specialize $s$ to a point $m$ where $E_f(s,\bullet)$ is
nearly holomorphic, in other words where the archimedean component $f_{\infty}$ of $f$ satisfies the hypotheses of Definition \ref{axiomeis}.
We consider the {\it Garrett map}
\begin{equation}\label{gar} G(f,\varphi)(g_2) =  I(\varphi,f,m)(g_2) :=  \chi^{-1}(\det g_2)\int_{Z_1(\adeles)G_1(\Q)\backslash G_1(\A)} E_f(m,(g_1,g_2))\varphi(g_1)dg_1.\end{equation}
When $f$ is clear from context, we set $G(\varphi):=G(f, \varphi)$.  One of the main observations of \cite{ga, GPSR} is
that if $\varphi \in \pi$ then $I(\varphi,\varphi',f,s) \equiv 0$ unless $\varphi' \in \pi^{\vee}$, in other words that $G(\varphi) \in \Hom(\pi^{\vee},\C) \simeq \pi$:

\begin{thm}  If $\varphi \in \pi$ then $G(\varphi) \in \pi$.
\end{thm}

The forms $\varphi$ and $G(\varphi)$ are on the same group $GU(V) = GU(-V)$ but on different Shimura varieties.
The restriction of $E_f(m,\bullet)$ is a holomorphic form on $Sh(V,-V)$, which means it pairs with an anti-holomorphic
form on $Sh(V)$ to yield a holomorphic form on $Sh(-V)$.  In terms of parameters, this becomes

\begin{cor}\label{garr}  The Garrett map defines a  homomorphism
\begin{equation*}\begin{split}
I(\chi_f,m) &\rar \Hom_{\Tb_{K,\kap}}(H^0_!(_{K}Sh(V),\omega_{\kap})^{\vee},H^0_!(_{K}Sh(-V),\omega_{\kap^{\flat})})),\\
&\rar \Hom_{\Tb_{K,\kap}}(H^d_!(Sh_{K}(V),\omega^D_{\kap}\otimes L(-\kap)),H^0_!(_{K}Sh(-V),\omega_{\kap^{\flat})})),
\end{split}
\end{equation*}
where the Hecke algebras act through the isomorphism in Section \ref{heck}.
Equivalently, 
$(F^\dag)^{-1}\circ G(\bullet,\bullet)$ defines a homomorphism 
 $$I(\chi_f,m) \rar Hom_{\Tb_{K,\kap},\flat}(H^0_!(_{K}Sh(V),\omega_{\kap})^{\vee},H^0_!(_{K^\dag}Sh(V),\omega_{\kap^\dag})).$$
\end{cor}

The factor $L(-\kap)$ was reinserted in the second line in order to respect the Hecke algebra action.  The action of 
$\Tb_{K,\kap}$ on $L(-\kap)$ factors through the similitude map.

\begin{lem}\label{garrettequivariance}  Let $dEis$ be an axiomatic Eisenstein measure as in Definition \ref{axiomeis}.  Then $dEis$ satisfies
the equivariance property of Assumption \ref{doub}.
\end{lem}

\begin{proof}  This corresponds to the equivariance property of the Garrett map stated in Corollary  \ref{garr}.  
\end{proof}

\subsubsection{Pairings, continued}\label{pcontinued}   
Thanks to Lemma \ref{garrettequivariance}, we can proceed as in Section  \ref{classpairings}.  
Henceforward we fix an anti-ordinary representation $\pi \otimes \pi^\flat$ as at the beginning of \S \ref{lastchapter}, and we denote 
by $\pi'$  the elements of $\CS(K_{r^1},\kap^1,\pi)$, as $r^1$ and $\kap^1$ vary.  In order to guarantee that our 
global pairings are compatible with the local calculations in Section  \ref{ESeriesZIntegrals-section}, especially the local calculations at primes dividing $p$, we choose test vectors $\varphi \in \hat{I}_\pi$ and $\varphi^{\flat} \in \hat{I}_{\pi^{\flat}}$ that are anti-holomorphic, anti-ordinary, and integral over $\CO$, as described following Lemma \ref{ord-unitary-lem}.   Proposition \ref{badp} allows us to identify the space $\hat{I}_\pi\otimes\hat{I}_{\pi^{\flat}}$ with the corresponding anti-holomorphic, anti-ordinary subspace of 
$\pi' \otimes \pi^{\prime,\flat}$ for any $\pi'$ as above.  We do so without further comment.  The image of $\varphi \otimes \varphi^\flat$ under this identification
is denoted $\varphi' \otimes \varphi^{\prime,\flat}$ when we need to indicate that the homomorphism $L_{\varphi' \otimes \varphi^{\prime,\flat}}$ is realized by the character
$\lambda_{\pi'} \in X(\kappa^1,r^1,R)$ attached to $\pi'$ (see Proposition \ref{abstract2var}).  Equivalently, we may identify $\varphi$ with the element $1\otimes \varphi'$ of the free $\TT_{K^p_{r^1},\kap^1,\CO,\pi}$-module $\TT_{K^p_{r^1},\kap^1,\CO,\pi}\otimes \hat{I}_{\pi}$, and $\varphi'$ with its specialization at the character $\lambda_{\pi'}$.

Substituting $\psi\rho^{\upsilon}$ for $\Xi$ in the above discussion, with $\psi \in C_r(T_H(\Zp),R)$ for some $R \subset \CO$
and $\rho$ as above, we find that, for any $\pi' \in \CS(K_{r^1},\kap^1,\pi)$, as $r^1$ and $\kap^1$ vary, we have
\begin{equation}\label{preunwind}
L_{\varphi'\otimes \varphi^{\prime,\flat}}\left(\int_{X_p \times T_H(\Zp)} (\chi,\psi\rho^{\upsilon}) dEis\right) = 
D(\chi)\cdot L_{\varphi\otimes \varphi^{\flat}}(res_3 D(\kappa,m,\chi_0)E^{holo}_{\chi_0,\psi\rho^{\upsilon}}(m)).
\end{equation}

\begin{prop}\label{doublingpairing}  Assume $\pi$ satisfies Hypotheses \ref{gor} and \ref{multone}.
Let $\varphi$ and $\varphi^{\flat}$ be respectively elements of 
$\hat{I}_{\pi}$ and $\hat{I}_{\pi^{\flat}}$.  Let $\pi' \in \CS(K_{r^1},\kap^1,\pi)$, for some $r^1$ and $\kap^1$, and let $\varphi' \otimes \varphi^{\prime,\flat}$
be the corresponding element of $\pi' \otimes \pi^{\prime,\flat}$.
Suppose $(\chi,\psi\rho^{\upsilon}) \in Y_H^{class}$, with $\psi \in C_r(T_H(\Zp),R)$ with $\chi = ||\bullet||^m\chi_u$, $m \geq n$.    
Then we have the equality
\begin{equation*}
\begin{aligned}
&L_{\varphi'\otimes \varphi^{\prime,\flat}}\left(\int_{X_p \times T_H(\Zp)} (\chi,\psi\rho^{\upsilon}) dEis\right)  \\
&= D(\chi)\cdot \frac{1}{\Vol(I_{r,V}^0)\Vol(I_{r,-V}^0)}I(\varphi',\varphi^{\prime,\flat},D(\kappa,m,\chi_0)f^{holo}(\chi_u,\psi\rho^{\upsilon}),m)
\end{aligned} 
\end{equation*} 
\end{prop}
\begin{proof}  Abbreviate $[G_3] = G_3(\Q)Z(\R)\backslash G_3(\A)$, 
$dg_2^{\chi} = \chi(\det(g_2))^{-1}dg_2$.
By doubling the formula in Lemma \ref{functionalintegral} -- in other words, by applying it to the group $G_3$ -- we obtain
\begin{equation*}
\begin{split}
&L_{\varphi'\otimes \varphi^{\prime,\flat}}\left(res_3 D(\kappa,m,\chi_0)E^{holo}_{f(\chi,\psi\rho^{\upsilon})}(m)\right) \\
=  &
\frac{1}{\Vol(I_{r,V}^0)\Vol(I_{r,-V}^0)}\int_{[G_3]} D(\kappa,m,\chi_0)E^{holo}_{f(\chi_0,\psi\rho^{\upsilon})}\left((g_1,g_2), m \right) \varphi'(g_1)\varphi^{\prime,\flat}(g_2) ||\nu(g_1)^{a(\kappa)}|| dg_1dg_2^{\chi}.
\end{split}
\end{equation*} 
Comparing this with Equation \eqref{preunwind} and the definition of the zeta integral, we obtain the equality.
\end{proof}

In view of our choices of local vectors in  \eqref{aordholomorphic},
Corollary \ref{Eulerpairing} below is then a consequence of the local computations summarized in Proposition \ref{globaleuler}, and of the axiomatic properties of the Eisenstein measure summarized in Definition \ref{axiomeis} and Corollary \ref{eismeasureG3}.

\begin{cor}\label{Eulerpairing}  Under the hypotheses of Proposition \ref{doublingpairing}, suppose $\varphi \otimes \varphi^{\flat}$ is an element of the space defined in \eqref{aordholomorphic}, and in particular $\varphi$ and $\varphi^\flat$ admit the corresponding factorizations at places dividing $p$ and $\infty$.  Let the parameters $\kappa$, $\rho$, $\rho^{\upsilon}$ determine one another as in Inequalities \eqref{parameters} and Equations \eqref{iota}.  
Then we have the equality
\begin{equation*}
\begin{split}
&L_{\varphi'\otimes \varphi^{\prime,\flat}}\left(\int_{X_p \times T_H(\Zp)} (\chi,\psi\rho^{\upsilon}) dEis\right) 
= D(\chi) \prod_v I_v(\varphi_v,\varphi_v^{\flat},f_v,m)  \\
&= [\Vol(I_{r,V}^0)\Vol(I_{r,-V}^0)]^{-1}\langle \varphi', \varphi^{\prime,\flat}_\chi \rangle_\chi\cdot I_p(\chi,\kappa)I_{\infty}(\chi,\rho^{\upsilon})I_SL^S(m + \frac{1}{2},\pi',\chi_u)
\end{split}
\end{equation*} 
where the factors are defined as in Proposition \ref{globaleuler}.
\end{cor}

\subsection{Statement of the main theorem}\label{eulerfac}

We reinterpret the identity in Corollary \ref{Eulerpairing} in the language of Proposition \ref{abstract2var}.

\begin{cor}\label{2varEuler}  Under the hypotheses of Corollary \ref{Eulerpairing}, there is a unique element 
$L(Eis,\varphi\otimes \varphi^{\flat}) \in \Lambda_{X_p,R}\hat{\otimes}\TT$
such that, for any classical
$\chi:  X_p \rar R^{\times}$ and any $\pi' \in \CS(K_{r^1},\kap^1,\pi)$ for some $r^1$,
the image of $L(Eis,\varphi\otimes \varphi^{\flat})$ under the map
$\Lambda_{X_p,R}\hat{\otimes}\TT \rar R$ induced by the character $\chi\otimes \lambda_{\pi'}$ 
equals 
$$ [\Vol(I_{r^1,V}^0)\Vol(I_{r^1,-V}^0)]^{-1}\langle \varphi, \varphi^{\flat}_\chi \rangle_\chi\cdot I_p(\chi,\kappa^1)I_{\infty}(\chi,\rho^{\upsilon})I_SL^S(m + \frac{1}{2},\pi',\chi_u).$$
Here $\lambda_{\pi'}$ is the character
of $\TT$ defined in Section  \ref{lambdapi}, and the local factors are defined as in Proposition \ref{globaleuler}.
\end{cor}

In the language of Corollary \ref{2varEuler} this admits the following reformulation.
The statement is in terms of the highest
weight $\kappa$ of the (holomorphic) representation dual to $\pi$ and a Hecke character $\chi$.
Let the algebraic characters $\kappa$, $\rho$, $\rho^{\upsilon}$ determine one another, relative to a given $\chi$, as in Inequalities \eqref{parameters} and Equation \eqref{iota}.

\begin{mainthm}\label{axiomaticmainthm2}   Let $\pi$ be a cuspidal anti-holomorphic automorphic representation of $G_1$ which is ordinary of
type $(\kap,K)$, and let $\TT = \TT_{\pi}$ be the corresponding connected component of the
ordinary Hecke algebra.  Let $\varphi$ and $\varphi^{\flat}$ be respectively elements of $R$-bases of $\hat{I}_{\pi}$ and $\hat{I}_{\pi^{\flat}}$.
  Assume $\pi$ satisfies the following Hypotheses:
\begin{enumerate}
\item  Hypothesis \ref{gor} (the Gorenstein Hypothesis)
\item Hypothesis \ref{multone} (the Global Multiplicity One Hypothesis)
\item Proposition-Hypothesis  \ref{badp} (the Minimality Hypothesis)
\end{enumerate}
There is a unique element 
$$L(Eis,\varphi\otimes \varphi^{\flat}) \in \Lambda_{X_p,R}\hat{\otimes}\TT$$
with the following property.  
For any classical
$\chi = ||\bullet||^m\chi_u:  X_p \rar R^{\times}$, and for any $\pi' \in \CS(K_{r^1},\kap^1,\pi)$ for some $r^1$,
the image of $L(Eis,\varphi\otimes \varphi^{\flat})$ under the map
$\Lambda_{X_p,R}\hat{\otimes}\TT \rar R$ induced by the character $\chi\otimes \lambda_{\pi'}$ 
equals 
$$c(\pi',\chi)\cdot \Omega_{\pi',\chi}(\varphi,\varphi^\flat) I_{\infty}(\chi,\kap^1)I_SL_p(m,\ord,\pi',\chi_u)\frac{L^S(m + \frac{1}{2},\pi',\chi_u)}{P_{\pi',\chi}}.$$
Here $\lambda_{\pi'}$ is the character of $\TT$ defined in Section  \ref{lambdapi} and
$$
\Omega_{\pi',\chi}(\varphi,\varphi^\flat) = \frac{ \langle \varphi, \varphi^{\flat}_\chi \rangle_\chi}{ \Vol(I_{r^1,V}^0\cap I_{r^1,-V}^0)\cdot 
{Q}[\pi',\chi]} \ \ \text{and}  \ \ P_{\pi',\chi} = Q_{\pi',\chi}^{-1}.
$$
Finally, the factor $\Omega_{\pi',\chi}(\varphi,\varphi^\flat)$ is independent of $r^1$ and  $p$-integral, and is a $p$-unit for appropriate choice of $\varphi$ and $\varphi^\flat$.
\end{mainthm}
\begin{proof} This follows from Corollary \ref{2varEuler}, after 
we write ${Q}[\pi',\chi] = c(\pi',\chi){Q}_{\pi',\chi}$, as in Lemma \ref{period2chi}.
We have used the expression for the local zeta integral 
$I_p = \prod_w I_w$ given by the formula in Remark \ref{Ip-formula}. In particular, 
the $\Vol(I_{r^1,V}^0)$ and $\Vol(I_{r^1,V}^0)$ terms cancel, leaving the factor
$\Vol(I_{r^1,V}^0\cap I_{r^1,-V}^0)^{-1}$. 
The final claim
follows from Lemma \ref{Omega-lem}.
\end{proof}

\subsection{Comments on the main theorem}

Even in the setting of ordinary families of $p$-adic modular forms on unitary Shimura varieties, this should not be considered the definitive
construction of $p$-adic $L$-functions.  We list some aspects that call for refinement.

\begin{rmk}[The Gorenstein Hypothesis]  It is often possible to verify the Gorenstein hypothesis when the residual Galois representation attached
to $\pi$ has sufficiently general image, using the Taylor-Wiles method.  See \cite{pilloni} and \cite{Harris-TW} for examples.  On the other hand, it
is certainly not valid in complete generality.  Since the Gorenstein condition is an open one, one can obtain a more general statement by replacing 
$\Lambda_{X_p,R}\hat{\otimes}\TT$ by the fraction fields of its irreducible components.  The method of this paper then provides $p$-adic meromorphic
functions on each such components, which specialize at classical points as indicated in the Main Theorem.
\end{rmk}

\begin{rmk}[The Multiplicity One Hypothesis]\label{globalmult1}  For an automorphic representation of a unitary group such as $G_1$ whose base change
to $GL(n)$ is cuspidal, the global multiplicity one hypothesis \ref{multone} is a consequence of \cite{mok,gangof4}.  The version in Hypothesis \ref{multone-pi} is restrictive, however, as already noted in \S \ref{heck}.  Here we sketch an argument for removing this hypothesis.  

\begin{itemize}
\item[(1)]  The first and most difficult step is to find the appropriate notation for the collection of $\pi'$ such that $\lambda_{\pi'} = \lambda_{\pi}$ -- in other words, the global $L$-packet containing  $\pi$.  We let $<\pi>$ denote the set of such $\pi'$.   
\item[(2)]  Note in particular that $<\pi> \subset \CS(K^p,\pi)$ (see \eqref{SKrkap} and the discussion above
Hypothesis \ref{multone}).  Thus the isomorphisms in Lemma \ref{ordveclem} need to be modified.  We have an isomorphism
$$
j_{<\pi>}:  \oplus_{\pi_i \in <\pi>} \pi_p^\ord\otimes \pi_{i,S}^{K_S} \cong \oplus_{\pi_i \in <\pi>} \pi_{i,S}^{K_S}\isoarrow S_\kap^\ord(K_r;E)[\lambda_\pi]\otimes_E\C.
$$
and an isomorphism
$$
\oplus_{\pi_i \in <\pi>}S_\kap^\ord(K_r;R)[\pi_i] := S_\kap^\ord(K_r;R)_\pi \cap S_\kap^\ord(K_r;E)[\lambda_\pi]
$$
which is identified by $j_{<\pi>}$ with an $R$-lattice in $\oplus_{\pi_i \in <\pi>} \pi_p^\ord\otimes\pi_{i,S}^{K_S}\cong \pi_S^{K_S}$.
\item[(3)] Lemma \ref{aordveclem} needs to be modified analogously, but the definition given there of $H^{d,\ord}_{\kap^D}(K_r,R)[\pi]$ defines a lattice in the sum of 
the $\pi_{i,S}^{\flat,K_S}$.  This is not adequate if we want to account for congruences between the different local constituents of the $\pi_i$ at ramified places where the
$L$-packets are not singletons.   On the other hand, if we only care about measuring the congruences between $\pi$ and the $\pi' \in \CS(K^p,\pi)$ that define different global Galois representations, then we can leave  the definition of $H^{d,\ord}_{\kap^D}(K_r,R)[\pi]$ as is, and modify the  isomorphisms in Lemma \ref{aordveclem} as in (2) above.
\item[(4)]  We also need to modify the statement of Lemma \ref{mult1Q}:  the isomorphism in (i) is replaced by
\begin{equation*}\label{facto22} 
j_{<\pi>}:  \oplus_{\pi_i \in <\pi>}  \pi^{K_S}_{i,S} \otimes \pi_p^{I_r}\isoarrow S_{\kap}(K_r,\C)(\pi)
\end{equation*} 
and the isomorphism in (ii) is replaced analogously.
\item[(5)]  The spaces $\hat{I}_\pi$ and $\hat{I}_{\pi^\flat}$ would have to be replaced by direct sums over the $\pi_i$.
\item[(6)]  Finally,  the identification in Lemma \ref{conjugatelambda} only depends on Hypothesis \ref{multone} and not on the stronger Hypothesis
\ref{multone-pi}.
\end{itemize}
\end{rmk}

Point (5) is the most objectionable, because it is not really compatible with the Minimality Hypothesis \ref{badp}.  An alternative approach would be to choose
an idempotent $e_{\pi,S}$ in the Hecke algebra of $G$ relative to $K_S$ that isolates the representations $\pi_v$ at all inert $v \in S$.  Thus for $\pi_i \in <\pi>, \pi_i \neq \pi$, 
$e_{\pi,S}\star \pi_i = 0$.   If we then redefine $S_\kap^\ord(K_r;E)[\lambda_\pi]$, $H^{d,\ord}_{\kap^D}(K_r,R)[\pi]$, etc. to be the image of projection with respect 
to $e_{\pi,S}$, all of the main theorems remain true without modification.  Better still, we can choose $\bar{e}_{\pi,S}$ to be an idempotent modulo $p$ and lift it to an idempotent in characteristic $0$, in order to avoid introducing  extraneous divisibilities by $p$ in the final result and eliminating interesting congruences between members of the $L$-packet.

\begin{rmk}[The Minimality Hypothesis]  This is a consequence of one part of the Gorenstein Hypothesis, and was included in order to work with a module $[\hat{I}_{\pi}\otimes \hat{I}_{\pi^{\flat}}]$ that is locally constant
on the Hida family.  One can easily eliminate this hypothesis, but the statement is no longer so clean.
\end{rmk}

\begin{rmk}[Unspecified local factors]  The volume factor $I_S$ is a placekeeper.  It might be more illuminating to replace $I_S$ by
$$\tilde{I}_S = \prod_{v \in S} L_v(m + \frac{1}{2},\pi_v,\chi_{u,v})^{-1}I_S$$
and write the specialized value of the $L$-function
$$c(\pi)\cdot \Omega_\pi(\varphi,\varphi^\flat)\cdot I_{\infty}(\chi,\kap)\tilde{I}_SL_p(m,\ord,\pi,\chi_u)\frac{L(m + \frac{1}{2},\pi,\chi_u)}{P_{\pi}}.$$
Here $L(s,\pi,\chi_u)$ denotes the standard $L$-function without the archimedean factors.  Written this way, one sees that the inverted local Euler factors
$L_v(m + \frac{1}{2},\pi_v,\chi_{u,v})^{-1}$ can give rise to exceptional zeroes.  

Ideally one would like to choose an optimal vector in $[\hat{I}_{\pi}\otimes \hat{I}_{\pi^{\flat}}]$ and to adapt the local Eisenstein sections at primes in $S$ to this choice.
This would settle the issues of minimality and local factors simultaneously.  At present we do not see how to carry this out.
\end{rmk}

\begin{rmk}[The congruence factors]  It is expected Ñ at least under the Gorenstein hypothesis Ñ that a congruence factor $c(\pi)$ can be chosen to be 
the specialization at $\pi$ of a canonical $p$-adic analytic function $\mathbf{c}$ that interpolates the normalized and $p$-stabilized value at $s = 1$ of the adjoint $L$-function
$L(s,\pi,Ad)$.   The factor $c(\pi)$ that appears in Main Theorem \ref{axiomaticmainthm2} 
depends on the choice of period $Q_{\pi}$, which in turn depends on the choice of $f$ in Lemma \ref{period2}.  As $\pi$ varies, the vector $f$ can be
chosen uniformly in the Hida family, but there is no obvious preferred choice.  For this reason, one can only define the hypothetical analytic function
$\mathbf{c}$ up to a unit in the Hecke algebra.  This is a persistent problem in the theory, and it has been noted by Hida in \cite{Hida-genuine}.
\end{rmk}

\section*{Acknowledgements}
This project has been developing over many years and the authors have benefited from the advice of numerous colleagues and from the hospitality of the institutions -- including UCLA, the Institute for Advanced Study, and Boston University -- that have provided the space to pursue our collaboration.   

The authors wish to reiterate our thanks to Ching-Li Chai, Matthew Emerton, and Eric Urban for their suggestions that have strongly influenced the present paper, and to add our thanks to Barry Mazur for asking questions that have motivated a number of our choices.   We also thank David Hansen, Paul Garrett, Zheng Liu, and Xin Wan for helpful discussions.  We are deeply grateful to Haruzo Hida for answering our questions and for his consistent support.

Finally, we thank the four referees for their patient and careful reading of the manuscript, which helped us enormously to improve the exposition and to eliminate debris that remained from the many previous versions of the paper in spite of our persistent efforts to remove it.

The authors are grateful for support from several funding sources.  E.E.'s research was partially supported by National Science Foundation Grants DMS-1751281, DMS-1559609, and DMS-1249384.  During an early part of the project, her research was was partially supported by an AMS-Simons Travel Grant.  M.H.'s research received funding from the European Research Council under the European Community's Seventh Framework Programme (FP7/2007-2013) / ERC Grant agreement no. 290766 (AAMOT).  M.H. was partially supported by NSF Grant DMS-1404769.  J.L.'s research was partially supported by RGC-GRF grant 16303314 of HKSAR.  C.S.'s research was partially supported by National Science Foundation Grants DMS-0758379 and DMS-1301842.

\bibliographystyle{amsalpha}  
\bibliography{EHLSbibliography}

\newcommand{\etalchar}[1]{$^{#1}$}
\def\Dbar{\leavevmode\lower.6ex\hbox to 0pt{\hskip-.23ex \accent"16\hss}D}
  \def\cfac#1{\ifmmode\setbox7\hbox{$\accent"5E#1$}\else
  \setbox7\hbox{\accent"5E#1}\penalty 10000\relax\fi\raise 1\ht7
  \hbox{\lower1.15ex\hbox to 1\wd7{\hss\accent"13\hss}}\penalty 10000
  \hskip-1\wd7\penalty 10000\box7}
  \def\cftil#1{\ifmmode\setbox7\hbox{$\accent"5E#1$}\else
  \setbox7\hbox{\accent"5E#1}\penalty 10000\relax\fi\raise 1\ht7
  \hbox{\lower1.15ex\hbox to 1\wd7{\hss\accent"7E\hss}}\penalty 10000
  \hskip-1\wd7\penalty 10000\box7} \def\Dbar{\leavevmode\lower.6ex\hbox to
  0pt{\hskip-.23ex \accent"16\hss}D}
  \def\cfac#1{\ifmmode\setbox7\hbox{$\accent"5E#1$}\else
  \setbox7\hbox{\accent"5E#1}\penalty 10000\relax\fi\raise 1\ht7
  \hbox{\lower1.15ex\hbox to 1\wd7{\hss\accent"13\hss}}\penalty 10000
  \hskip-1\wd7\penalty 10000\box7}
  \def\cftil#1{\ifmmode\setbox7\hbox{$\accent"5E#1$}\else
  \setbox7\hbox{\accent"5E#1}\penalty 10000\relax\fi\raise 1\ht7
  \hbox{\lower1.15ex\hbox to 1\wd7{\hss\accent"7E\hss}}\penalty 10000
  \hskip-1\wd7\penalty 10000\box7} \def\Dbar{\leavevmode\lower.6ex\hbox to
  0pt{\hskip-.23ex \accent"16\hss}D}
  \def\cfac#1{\ifmmode\setbox7\hbox{$\accent"5E#1$}\else
  \setbox7\hbox{\accent"5E#1}\penalty 10000\relax\fi\raise 1\ht7
  \hbox{\lower1.15ex\hbox to 1\wd7{\hss\accent"13\hss}}\penalty 10000
  \hskip-1\wd7\penalty 10000\box7}
  \def\cftil#1{\ifmmode\setbox7\hbox{$\accent"5E#1$}\else
  \setbox7\hbox{\accent"5E#1}\penalty 10000\relax\fi\raise 1\ht7
  \hbox{\lower1.15ex\hbox to 1\wd7{\hss\accent"7E\hss}}\penalty 10000
  \hskip-1\wd7\penalty 10000\box7} \def\cprime{$'$}
\providecommand{\bysame}{\leavevmode\hbox to3em{\hrulefill}\thinspace}
\providecommand{\MR}{\relax\ifhmode\unskip\space\fi MR }
\providecommand{\MRhref}[2]{%
  \href{http://www.ams.org/mathscinet-getitem?mr=#1}{#2}
}
\providecommand{\href}[2]{#2}
\begin{thebibliography}{KMSW14}

\bibitem[BHR94]{BHR}
Don Blasius, Michael Harris, and Dinakar Ramakrishnan, \emph{Coherent
  cohomology, limits of discrete series, and {G}alois conjugation}, Duke Math.
  J. \textbf{73} (1994), no.~3, 647--685. \MR{1262930 (95b:11054)}

\bibitem[Cas95]{Casselman-book}
William Casselman, \emph{Introduction to the theory of admissible
  representations of $p$-adic reductive groups}, Unpublished manuscript, 1995,
  \url{https://www.math.ubc.ca/~cass/research/pdf/p-adic-book.pdf}.

\bibitem[CCO14]{ChCoOo}
Ching-Li Chai, Brian Conrad, and Frans Oort, \emph{Complex multiplication and
  lifting problems}, Mathematical Surveys and Monographs, vol. 195, American
  Mathematical Society, Providence, RI, 2014. \MR{3137398}

\bibitem[CEF{\etalchar{+}}16]{CEFMV}
Ana Caraiani, Ellen Eischen, Jessica Fintzen, Elena Mantovan, and Ila Varma,
  \emph{{$p$}-adic {$q$}-expansion principles on unitary {S}himura varieties},
  Directions in number theory, vol.~3, Springer, [Cham], 2016, pp.~197--243.
  \MR{3596581}

\bibitem[Che04]{CH}
Ga{\"e}tan Chenevier, \emph{Familles {$p$}-adiques de formes automorphes pour
  {${\rm GL}\sb n$}}, J. Reine Angew. Math. \textbf{570} (2004), 143--217.
  \MR{MR2075765 (2006b:11046)}

\bibitem[CHT08]{CHT}
Laurent Clozel, Michael Harris, and Richard Taylor, \emph{Automorphy for some
  {$l$}-adic lifts of automorphic mod {$l$} {G}alois representations}, Publ.
  Math. Inst. Hautes \'Etudes Sci. (2008), no.~108, 1--181, With Appendix A,
  summarizing unpublished work of Russ Mann, and Appendix B by Marie-France
  Vign{\'e}ras. \MR{2470687}

\bibitem[Coa89]{coatesII}
John Coates, \emph{On {$p$}-adic {$L$}-functions attached to motives over
  {${\bf Q}$}. {II}}, Bol. Soc. Brasil. Mat. (N.S.) \textbf{20} (1989), no.~1,
  101--112. \MR{1129081 (92j:11060b)}

\bibitem[CPR89]{CPR}
John Coates and Bernadette Perrin-Riou, \emph{On {$p$}-adic {$L$}-functions
  attached to motives over {${\bf Q}$}}, Algebraic number theory, Adv. Stud.
  Pure Math., vol.~17, Academic Press, Boston, MA, 1989, pp.~23--54.
  \MR{1097608 (92j:11060a)}

\bibitem[Del79]{deshimura}
Pierre Deligne, \emph{Vari\'et\'es de {S}himura: interpr\'etation modulaire, et
  techniques de construction de mod\`eles canoniques}, Automorphic forms,
  representations and {$L$}-functions ({P}roc. {S}ympos. {P}ure {M}ath.,
  {O}regon {S}tate {U}niv., {C}orvallis, {O}re., 1977), {P}art 2, Proc. Sympos.
  Pure Math., XXXIII, Amer. Math. Soc., Providence, R.I., 1979, pp.~247--289.
  \MR{546620 (81i:10032)}

\bibitem[EFMV18]{EFMV}
Ellen Eischen, Jessica Fintzen, Elena Mantovan, and Ila Varma,
  \emph{Differential operators and families of automorphic forms on unitary
  groups of arbitrary signature}, Doc. Math. \textbf{23} (2018), 445--495.
  \MR{3846052}

\bibitem[Eis12]{EDiffOps}
Ellen~E. Eischen, \emph{{$p$}-adic differential operators on automorphic forms
  on unitary groups}, Ann. Inst. Fourier (Grenoble) \textbf{62} (2012), no.~1,
  177--243. \MR{2986270}

\bibitem[Eis14]{apptoSHLvv}
Ellen Eischen, \emph{A p-adic {E}isenstein measure for vector-weight
  automorphic forms}, Algebra Number Theory \textbf{8} (2014), no.~10,
  2433--2469. \MR{3298545}

\bibitem[Eis15]{apptoSHL}
Ellen~E. Eischen, \emph{A p-adic {E}isenstein measure for unitary groups}, J.
  Reine Angew. Math. \textbf{699} (2015), 111--142. \MR{3305922}

\bibitem[Eis16]{emeasurenondefinite}
Ellen~Elizabeth Eischen, \emph{Differential operators, pullbacks, and families
  of automorphic forms on unitary groups}, Ann. Math. Qu\'e. \textbf{40}
  (2016), no.~1, 55--82. \MR{3512523}

\bibitem[EM19]{EiMa}
Ellen Eischen and Elena Mantovan, \emph{$p$-adic families of automorphic forms
  in the $\mu$-ordinary setting}, American Journal of Mathematics (2019),
  Accepted for publication.

\bibitem[Gar84]{ga}
Paul~B. Garrett, \emph{Pullbacks of {E}isenstein series; applications},
  Automorphic forms of several variables ({K}atata, 1983), Progr. Math.,
  vol.~46, Birkh\"auser Boston, Boston, MA, 1984, pp.~114--137. \MR{MR763012
  (86f:11039)}

\bibitem[Gar08]{ga06}
Paul Garrett, \emph{Values of {A}rchimedean zeta integrals for unitary groups},
  Eisenstein series and applications, Progr. Math., vol. 258, Birkh\"auser
  Boston, Boston, MA, 2008, pp.~125--148. \MR{2402682 (2009e:11093)}

\bibitem[GPSR87]{GPSR}
Stephen Gelbart, Ilya Piatetski-Shapiro, and Stephen Rallis, \emph{Explicit
  constructions of automorphic {$L$}-functions}, Lecture Notes in Mathematics,
  vol. 1254, Springer-Verlag, Berlin, 1987. \MR{MR892097 (89k:11038)}

\bibitem[GW09]{goodmanwallach}
Roe Goodman and Nolan~R. Wallach, \emph{Symmetry, representations, and
  invariants}, Graduate Texts in Mathematics, vol. 255, Springer, Dordrecht,
  2009. \MR{2522486 (2011a:20119)}

\bibitem[Har86]{harris86}
Michael Harris, \emph{Arithmetic vector bundles and automorphic forms on
  {S}himura varieties. {II}}, Compositio Math. \textbf{60} (1986), no.~3,
  323--378.

\bibitem[Har89]{Harris-ToroidalCompactifications}
\bysame, \emph{Functorial properties of toroidal compactifications of locally
  symmetric varieties}, Proc. London Math. Soc. (3) \textbf{59} (1989), no.~1,
  1--22. \MR{997249 (90h:11048)}

\bibitem[Har90]{H90}
\bysame, \emph{Automorphic forms of {$\overline\partial$}-cohomology type as
  coherent cohomology classes}, J. Differential Geom. \textbf{32} (1990),
  no.~1, 1--63. \MR{1064864 (91g:11064)}

\bibitem[Har97]{harriscrelle}
\bysame, \emph{{$L$}-functions and periods of polarized regular motives}, J.
  Reine Angew. Math. \textbf{483} (1997), 75--161. \MR{1431843 (98b:11070)}

\bibitem[Har08]{Harris-rationality}
\bysame, \emph{A simple proof of rationality of {S}iegel-{W}eil {E}isenstein
  series}, Eisenstein series and applications, Progr. Math., vol. 258,
  Birkh\"auser Boston, Boston, MA, 2008, pp.~149--185. \MR{2402683
  (2009g:11061)}

\bibitem[Har13a]{HarrisIMRN}
\bysame, \emph{Beilinson-{B}ernstein localization over {$\Bbb Q$} and periods
  of automorphic forms}, Int. Math. Res. Not. IMRN (2013), no.~9, 2000--2053.
  \MR{3053412}

\bibitem[Har13b]{Harris-TW}
\bysame, \emph{The {T}aylor-{W}iles method for coherent cohomology}, J. Reine
  Angew. Math. \textbf{679} (2013), 125--153. \MR{3065156}

\bibitem[Hid88]{Hida1988}
Haruzo Hida, \emph{A {$p$}-adic measure attached to the zeta functions
  associated with two elliptic modular forms. {II}}, Ann. Inst. Fourier
  (Grenoble) \textbf{38} (1988), no.~3, 1--83. \MR{976685 (89k:11120)}

\bibitem[Hid96a]{hidasearch}
\bysame, \emph{On the search of genuine {$p$}-adic modular {$L$}-functions for
  {${\rm GL}(n)$}}, M\'em. Soc. Math. Fr. (N.S.) (1996), no.~67, vi+110, With a
  correction to: ``On $p$-adic $L$-functions of ${\rm{G}}L(2){\times{G}}L(2)$
  over totally real fields'' [Ann. Inst. Fourier (Grenoble) {{\bf{4}}1} (1991),
  no. 2, 311--391; MR1137290 (93b:11052)]. \MR{1479362 (98i:11027)}

\bibitem[Hid96b]{Hida-genuine}
\bysame, \emph{On the search of genuine {$p$}-adic modular {$L$}-functions for
  {${\rm GL}(n)$}}, M\'em. Soc. Math. Fr. (N.S.) (1996), no.~67, vi+110, With a
  correction to: ``On $p$-adic $L$-functions of ${\rm{G}}L(2){\times{G}}L(2)$
  over totally real fields'' [Ann. Inst. Fourier (Grenoble) {{\bf{4}}1} (1991),
  no. 2, 311--391; MR1137290 (93b:11052)]. \MR{1479362 (98i:11027)}

\bibitem[Hid98]{H98}
\bysame, \emph{Automorphic induction and {L}eopoldt type conjectures for {${\rm
  GL}(n)$}}, Asian J. Math. \textbf{2} (1998), no.~4, 667--710, Mikio Sato: a
  great Japanese mathematician of the twentieth century. \MR{MR1734126
  (2000k:11064)}

\bibitem[Hid02]{H02}
\bysame, \emph{Control theorems of coherent sheaves on {S}himura varieties of
  {PEL} type}, J. Inst. Math. Jussieu \textbf{1} (2002), no.~1, 1--76.
  \MR{MR1954939 (2003m:11086)}

\bibitem[Hid04]{Hida}
\bysame, \emph{{$p$}-adic automorphic forms on {S}himura varieties}, Springer
  Monographs in Mathematics, Springer-Verlag, New York, 2004. \MR{MR2055355
  (2005e:11054)}

\bibitem[HKS96]{HKS}
Michael Harris, Stephen~S. Kudla, and William~J. Sweet, \emph{Theta dichotomy
  for unitary groups}, J. Amer. Math. Soc. \textbf{9} (1996), no.~4, 941--1004.
  \MR{1327161 (96m:11041)}

\bibitem[HLS05]{HLS05}
Michael Harris, Jian-Shu Li, and Christopher~M. Skinner, \emph{The {R}allis
  inner product formula and {$p$}-adic {$L$}-functions}, Automorphic
  representations, {$L$}-functions and applications: progress and prospects,
  Ohio State Univ. Math. Res. Inst. Publ., vol.~11, de Gruyter, Berlin, 2005,
  pp.~225--255. \MR{2192825}

\bibitem[HLS06]{HLS}
\bysame, \emph{{$p$}-adic {$L$}-functions for unitary {S}himura varieties. {I}.
  {C}onstruction of the {E}isenstein measure}, Doc. Math. (2006), no.~Extra
  Vol., 393--464 (electronic). \MR{MR2290594 (2008d:11042)}

\bibitem[Jac79]{jacquet}
Herv{\'e} Jacquet, \emph{Principal {$L$}-functions of the linear group},
  Automorphic forms, representations and {$L$}-functions ({P}roc. {S}ympos.
  {P}ure {M}ath., {O}regon {S}tate {U}niv., {C}orvallis, {O}re., 1977), {P}art
  2, Proc. Sympos. Pure Math., XXXIII, Amer. Math. Soc., Providence, R.I.,
  1979, pp.~63--86. \MR{546609 (81f:22029)}

\bibitem[Kat78]{kaCM}
Nicholas~M. Katz, \emph{{$p$}-adic {$L$}-functions for {CM} fields}, Invent.
  Math. \textbf{49} (1978), no.~3, 199--297. \MR{MR513095 (80h:10039)}

\bibitem[KMSW14]{gangof4}
Tasho Kaletha, Alberto Minguez, Sug~Woo Shin, and Paul-James White,
  \emph{Endoscopic classification of representations: Inner forms of unitary
  groups}, 2014, arXiv preprint. \url{https://arxiv.org/pdf/1409.3731.pdf}.

\bibitem[Kot92]{Kottwitz}
Robert~E. Kottwitz, \emph{Points on some {S}himura varieties over finite
  fields}, J. Amer. Math. Soc. \textbf{5} (1992), no.~2, 373--444.
  \MR{MR1124982 (93a:11053)}

\bibitem[Lab11]{lab}
J.-P. Labesse, \emph{Changement de base {CM} et s\'eries discr\`etes}, On the
  stabilization of the trace formula, Stab. Trace Formula Shimura Var. Arith.
  Appl., vol.~1, Int. Press, Somerville, MA, 2011, pp.~429--470. \MR{2856380}

\bibitem[Lan12]{lanalgan}
Kai-Wen Lan, \emph{Comparison between analytic and algebraic constructions of
  toroidal compactifications of {PEL}-type {S}himura varieties}, J. Reine
  Angew. Math. \textbf{664} (2012), 163--228. \MR{2980135}

\bibitem[Lan13]{Lan}
\bysame, \emph{Arithmetic compactifications of {PEL}-type shimura varieties},
  London Mathematical Society Monographs, vol.~36, Princeton University Press,
  2013.

\bibitem[Lan16]{KLkoecher}
\bysame, \emph{Higher {K}oecher's principle}, Math. Res. Lett. \textbf{23}
  (2016), no.~1, 163--199. \MR{3512882}

\bibitem[Lan17]{lanIMRN}
\bysame, \emph{Integral models of toroidal compactifications with projective
  cone decompositions}, Int. Math. Res. Not. IMRN (2017), no.~11, 3237--3280.
  \MR{3693649}

\bibitem[Lan18]{Lan2}
\bysame, \emph{Compactifications of {PEL}-type {S}himura varieties and {K}uga
  families with ordinary loci}, World Scientific Publishing Co. Pte. Ltd.,
  Hackensack, NJ, 2018. \MR{3729423}

\bibitem[Li92]{JSLi}
Jian-Shu Li, \emph{Nonvanishing theorems for the cohomology of certain
  arithmetic quotients}, J. Reine Angew. Math. \textbf{428} (1992), 177--217.
  \MR{1166512 (93e:11067)}

\bibitem[Liu19a]{liu19}
Zheng Liu, \emph{The doubling {A}rchimedean zeta integrals for $p$-adic
  interpolation}, Mathematical Research Letters (2019), Accepted for
  publication. Preprint available at \url{https://arxiv.org/abs/1904.07121}.

\bibitem[Liu19b]{ZL}
\bysame, \emph{$p$-adic ${L}$-functions for ordinary families on symplectic
  groups}, Journal of the Institute of Mathematics of Jussieu (2019), 1--61.

\bibitem[LS13]{Lan-Suh-vanishing}
Kai-Wen Lan and Junecue Suh, \emph{Vanishing theorems for torsion automorphic
  sheaves on general {PEL}-type {S}himura varieties}, Adv. Math. \textbf{242}
  (2013), 228--286. \MR{3055995}

\bibitem[Mok15]{mok}
Chung~Pang Mok, \emph{Endoscopic classification of representations of
  quasi-split unitary groups}, Mem. Amer. Math. Soc. \textbf{235} (2015),
  no.~1108, vi+248. \MR{3338302}

\bibitem[Moo04]{moonen}
Ben Moonen, \emph{Serre-{T}ate theory for moduli spaces of {PEL} type}, Ann.
  Sci. \'{E}cole Norm. Sup. (4) \textbf{37} (2004), no.~2, 223--269.
  \MR{2061781}

\bibitem[MVW87]{MVW}
Colette M{\oe}glin, Marie-France Vign{\'e}ras, and Jean-Loup Waldspurger,
  \emph{Correspondances de {H}owe sur un corps {$p$}-adique}, Lecture Notes in
  Mathematics, vol. 1291, Springer-Verlag, Berlin, 1987. \MR{1041060}

\bibitem[Pan94]{pan}
Alexei~A. Panchishkin, \emph{Motives over totally real fields and {$p$}-adic
  {$L$}-functions}, Ann. Inst. Fourier (Grenoble) \textbf{44} (1994), no.~4,
  989--1023. \MR{1306547 (96e:11087)}

\bibitem[Pil11]{pilloni}
Vincent Pilloni, \emph{Prolongement analytique sur les vari\'et\'es de
  {S}iegel}, Duke Math. J. \textbf{157} (2011), no.~1, 167--222. \MR{2783930}

\bibitem[Shi97]{sh}
Goro Shimura, \emph{Euler products and {E}isenstein series}, CBMS Regional
  Conference Series in Mathematics, vol.~93, Published for the Conference Board
  of the Mathematical Sciences, Washington, DC, 1997. \MR{MR1450866
  (98h:11057)}

\bibitem[Shi00]{shar}
\bysame, \emph{Arithmeticity in the theory of automorphic forms}, Mathematical
  Surveys and Monographs, vol.~82, American Mathematical Society, Providence,
  RI, 2000. \MR{MR1780262 (2001k:11086)}

\bibitem[SU02]{SkUr-SKlifts}
Christopher Skinner and Eric Urban, \emph{Sur les d\'eformations {$p$}-adiques
  des formes de {S}aito-{K}urokawa}, C. R. Math. Acad. Sci. Paris \textbf{335}
  (2002), no.~7, 581--586. \MR{1941298 (2003j:11048)}

\bibitem[SU14]{SkUr}
\bysame, \emph{The {I}wasawa main conjectures for {$\rm GL_2$}}, Invent. Math.
  \textbf{195} (2014), no.~1, 1--277. \MR{3148103}

\bibitem[Wan15]{xinwan}
Xin Wan, \emph{Families of nearly ordinary {E}isenstein series on unitary
  groups}, Algebra Number Theory \textbf{9} (2015), no.~9, 1955--2054, With an
  appendix by Kai-Wen Lan. \MR{3435811}

\bibitem[Wed99]{wedhorn}
Torsten Wedhorn, \emph{Ordinariness in good reductions of {S}himura varieties
  of {PEL}-type}, Ann. Sci. \'Ecole Norm. Sup. (4) \textbf{32} (1999), no.~5,
  575--618. \MR{1710754}

\end{thebibliography}
\end{document}